\documentclass[10pt]{amsart}
\usepackage{graphicx,amssymb,amsfonts,amsmath,amsthm,newlfont}
\usepackage{epsfig,url}
\usepackage{color}

\usepackage[all,2cell]{xy} \UseAllTwocells \SilentMatrices

\newtheorem{thm}{Theorem}[section]
\newtheorem{cor}[thm]{Corollary}
\newtheorem{lem}[thm]{Lemma}
\newtheorem{prop}[thm]{Proposition}
\theoremstyle{definition}
\newtheorem{defn}[thm]{Definition}

\newtheorem{example}[thm]{Example}

\theoremstyle{remark}
\newtheorem{rem}[thm]{Remark}
\newtheorem{caveat}[thm]{Caveat}
\numberwithin{equation}{section}

\newcommand{\Z}{\mathbb Z}
\newcommand{\C}{\mathbb C}

\newcommand{\R}{\mathbb R}
\newcommand{\N}{\mathbb N}
\newcommand{\Pro}{\mathbb P}

\newcommand{\eqv}{\mathrm{ev}}

\newcommand{\ax}{\mathbf{a}}
\newcommand{\bx}{\mathbf{b}}
\newcommand{\fx}{\mathbf{f}}

\newcommand{\tp}{\mathrm{top}}

\newcommand{\h}{\mathfrak{h}}
\newcommand{\hol}{\mathrm{hol} }
\newcommand{\Lef}{\mathbb{L} }

\newcommand{\gr}{\mathrm{gr}}
\newcommand{\rr}{\mathfrak{r}}
\newcommand{\ue}{\mathfrak{u}^{\mathrm{geom}}}

\font \rus= wncyr10
\newcommand{\sha}{\, \hbox{\rus x} \,}

\newcommand{\MT}{\mathcal{MT}}
\newcommand{\M}{\mathcal{M}}
\newcommand{\Ss}{\mathcal{S}}
\newcommand{\Eis}{\mathcal{E}}

\newcommand{\bbf}{\overline{b}}
\newcommand{\uu}{\mathfrak{u}}
\newcommand{\rel}{\mathrm{rel}}
\newcommand{\GG}{\mathcal{G} }
\newcommand{\Eq}{\mathcal{E}^{\times}_{\partial/\partial q}}

\newcommand{\Xs}{\mathsf{X}}
\newcommand{\Ys}{\mathsf{Y}}
\newcommand{\ms}{\mathsf{m}}

\newcommand{\es}{\mathsf{e}}

\newcommand{\av}{\mathsf{a}}
\newcommand{\bv}{\mathsf{b}}

\newcommand{\PiU}{ \U^{dR,\hol}_{1,1} }
\newcommand{\Image}{\mathrm{Im}}
\newcommand{\Real}{\mathrm{Re}}

\newcommand{\Hc}{\mathcal{H} }

\newcommand{\FF}{\mathcal{F}}
\newcommand{\DD}{\mathcal{D}}

\newcommand{\CC}{\mathcal{C}}

\newcommand{\MMV}{\mathcal{MMV}}

\newcommand{\zetam}{\zeta^{ \mathfrak{m}}}

\newcommand{\eis}{\mathrm{eis}}
\newcommand{\cusp}{\mathrm{cusp}}

\newcommand{\Q}{\mathbb Q}

\newcommand{\U}{\mathcal{U}}

\newcommand{\Be}{\mathsf{b}}
\newcommand{\s}{\mathbf{s}}

\newcommand{\To}{\longrightarrow}

\newcommand{\A}{\mathbb{A}}
\newcommand{\G}{\mathbb{G}}

\newcommand{\x}{\mathsf{x}}

\newcommand{\bq}{\backslash \!\! \backslash }

\newcommand{\Eb}{\widetilde{\EE}}

\newcommand{\tone}{\overset{\rightarrow}{1}\!}

\newcommand{\Aut}{\mathrm{Aut}}

\newcommand{\dR}{\mathfrak{dr}}

\newcommand{\SL}{\mathrm{SL}}

\newcommand{\Or}{\mathcal{O}}

\newcommand{\g}{\mathfrak{g}}

\newcommand{\V}{\mathcal{V}}
\newcommand{\tv}{\overset{\rightarrow}{v}}

\newcommand{\tinf}{\overset{\rightarrow}{1}_{\!\!\infty}}

\newcommand{\ls}{\mathfrak{ls}}
 
\newcommand{\B}{\mathcal{B} }

\newcommand{\mm}{\mathfrak{m} }

\newcommand{\HH}{\mathfrak{H} }
\newcommand{\al}{\mathsf{a} }
\newcommand{\Al}{\mathsf{A} }

\newcommand{\id}{\mathrm{id} }

\newcommand{\T}{\mathcal{T} }

\newcommand{\e}{\mathbf{e}}
\newcommand{\EE}{\mathbf{E}}

\newcommand{\Lie}{\mathrm{Lie}}

\newcommand{\LL}{\mathbb{L}}

\newcommand{\per}{\mathrm{per}}

\newcommand{\Spec}{\mathrm{Spec} \,}
\newcommand{\PP}{\mathfrak{P}}

\newcommand{\MMM}{\mathcal{MM}_{1,1}}

\newcommand{\Pe}{\mathcal{P}}

\newcommand{\sv}{\mathrm{sv}}

\newcommand{\comp}{\mathrm{comp}}

\newcommand{\tbp}{\tone_{\infty}}

\newcommand{\pp}{\mathsf{p}}

\newcommand{\Hom}{\mathrm{Hom}}

\begin{document}
\author{Francis Brown}
\begin{title}[Multiple Modular Values]{Multiple Modular Values  and the relative completion of the  fundamental group of  $\mathcal{M}_{1,1}$. }\end{title}
\maketitle
 
 \section{Introduction}
 The motivation for this paper is  the question: how can one construct iterated extensions of motives (or compatible systems of
 $\ell$-adic Galois representations) of modular forms? We propose that a large supply of such extensions can be obtained from  the relative completion of the fundamental group   of modular  curves.  In this paper we focus mainly on the case of the moduli space $\M_{1,1}$
 and provide some evidence both for and against the proposition that it generates all expected extensions of  motives attached to  modular forms for the full modular group $\SL_2(\Z)$.  A further reason for considering this space is that it should play a central role in a motivic version of Grothendieck's `Esquisse d'un programme' \cite{Esquisse}, in which he proposed studying the action of the absolute Galois group on the profinite fundamental groupoids of the spaces $\M_{g,n}$.

 \subsection{Background}
We briefly recall the main ingredients in the theory of the pro-unipotent fundamental groupoid of $\M_{0,4}$, which is isomorphic to the projective line minus three points $X=\Pro^1\backslash \{0,1,\infty\}$. The theory was initiated in  \cite{DeP1, Drinfeld, IharaICM}. 
 \begin{description}
 \item[(i)] The unipotent completion of the fundamental groupoid of $X(\C)$  with respect to tangent vectors of length $1$ (respectively -1) at the points $0$ and $1$, is the Betti realisation of an object, called the motivic fundamental groupoid, in an abelian category $\MT(\Z)$ of mixed Tate motives over $\Z$  \cite{DG}.  
 \item[(ii)] The category $\MT(\Z)$ is Tannakian \cite{Levine}, and its de Rham Tannaka group $\GG^{dR}_{\MT(\Z)}$ acts upon the de Rham realisation of the motivic fundamental groupoid. This gives rise to  a homomorphism
 $$ \GG^{dR}_{\MT(\Z)} \To \mathrm{Aut} \,\big( \pi_1^{dR}( X, \tone_0, -\tone_1) \big)$$
 whose image is contained in  certain group of automorphisms which preserve some inertial conditions at the points $0$ and $1$. By a quirk of fate, this automorphism group happens to be isomorphic (as a scheme) to the de Rham fundamental groupoid itself. A formula for this action was computed by Ihara \cite{Ihara}, and the dual formula for the corresponding coaction by Goncharov \cite{GG}.
 \item[(iii)] The structure of the  graded Lie algebra of $\GG^{dR}_{\MT(\Z)}$  is known.  It is isomorphic to a free Lie algebra on certain elements $\sigma_3, \sigma_5, \ldots $, where each $\sigma_{2n+1}$ spans a copy of the Tate motive $\Q(2n+1)$. A key result of \cite{DG} is that 
 $$\big(\mathrm{Lie} \, \GG^{dR}_{\MT(\Z)} \big)^{ab} \To (  \mathrm{Aut} \, \pi_1^{dR}( X, \tone_0, -\tone_1))^{ab} $$
 is injective. In other words, the generators of the Tannaka Lie algebra of $\MT(\Z)$ act non-trivially on the de Rham fundamental groupoid of $X$. This follows from the fact  that the values of the Riemann zeta function $\zeta(3), \zeta(5), \ldots$ can be expressed as  regularised iterated integrals on $X$. 
 \item[(iv)] The images of the $\sigma_{2n+1}$ generate a free Lie algebra. This implies that the motivic fundamental groupoid of $X$ actually generates the category $\MT(\Z)$. 
 \end{description}

\subsection{Genus one} We wish to mimic this story for the moduli space of elliptic curves $\M_{1,1}$ with basepoint the unit tangent vector at the cusp. Its orbifold fundamental group is canonically isomorphic to $\SL_2(\Z)$. 
 The first point is that the unipotent completion of $\SL_2(\Z)$ is trivial, and we must instead consider its relative completion $\GG^B_{1,1}$, 
 which is an affine group scheme over $\Q$, extension of $\SL_2$ by a  pro-unipotent group $\U^B_{1,1}$.  The de Rham theory of  relative completion was carried out in \cite{HaMHS}.   
 The next difficulty is that there is no suitable abelian category of mixed modular motives in which the relative completion should lie. For this reason we must work in a category of systems of realisations, or more precisely, in a category $\Hc$ of  Betti and de Rham realisations  equipped  with a mixed Hodge structure \cite{DeP1}. One expects that there exists a corresponding $\ell$-adic and crystalline theory, but these aspects
 are entirely neglected from the treatment given here.  The relative completions of fundamental groups of modular curves carry a limiting mixed Hodge structure which was computed in \cite{HaGPS}. Thus our main object of study is a version $\GG^{\Hc}_{1,1}$ of relative completion  which is a pro-algebraic group object in the Tannakian category $\Hc$. It is equipped with a weight filtration $M$ and an additional geometric weight filtration $W$ via its limiting mixed Hodge structure.  Its Betti realisation is the group-theoretic completion of $\SL_2(\Z)$ relative to its inclusion in $\SL_2(\Q)$.

The affine ring of $\GG^{\Hc}_{1,1}$ lies in the subcategory of $\Hc$ of objects of mixed-modular type: these are objects whose semi-simplification are in the full Tannakian subcategory of $\Hc$  generated by Tate objects $\Q(n)$ and realizations $M_f$  in $\Hc$ of motives attached to  cusp forms $f$ for $\SL_2(\Z)$ with rational coefficients. It is convenient to work in $\Hc\otimes \overline{\Q}$, in which case the previous category is generated by realisations $V_f$ of motives attached to Hecke eigenforms. These are of rank two and Hodge types $(2n+1,0), (0, 2n+1)$.

 \subsubsection{A group of automorphisms}
 We completely carry out the analogue of $(ii)$ in this situation. For any fiber functor $\omega$ on $\Hc$, the action of the Tannaka group $\GG^{\omega}_{\Hc}$ of the category $\Hc$ acts upon $\GG^{\omega}_{1,1}$ via a certain group of automorphisms which we describe explicitly and  denote by $\A^{\omega}$.  This action  is compatible with a number of geometric constraints, the most important being the inertia at the cusp, and also  the geometric monodromy  representation 
  $$\GG^{\omega}_{1,1} \To \pi_1^{\omega} (\Eq, \tone_0)$$
 where $\Eq$ is the infinitesimal Tate elliptic curve (fiber of the universal elliptic curve over the tangential base point at the cusp). 
 The latter is very closely related to the theory of multiple elliptic polylogarithms \cite{LR}, the elliptic KZB equation \cite{HaKZB, CEE},  the theory of elliptic associators \cite{En},
 Eisenstein symbols,  and the theory of universal mixed elliptic motives \cite{MEM}.
All of this structure  is enocded by a homomorphism:
$$\GG^{\omega}_{\Hc} \To \A^{\omega} $$
which is a genus one analogue of the Ihara action $(ii)$.

\subsubsection{Extensions}
We presently  lack an abelian category of mixed modular motives let alone the required bounds on their extension groups.  Therefore as a substitute for $(iii)$,  we are guided by Beilinson's conjectures,
which predict the dimensions of the ext groups
 $$\mathrm{Ext}_{\mathcal{MM}\otimes \overline{\Q}}^1(\Q; \mathrm{Sym}^{i_1} V_{f_1} \otimes \ldots \otimes  \mathrm{Sym}^{i_r}V_{f_r}(d) )$$
in the hypothetical abelian category $\mathcal{MM}$ of mixed motives over $\Q$.  From this,  one can work out how many elements in the Lie algebra of $\GG^{\omega}_{\Hc}$ should be of motivic origin.  Using our description  of the Hodge theory of $\A$, we can then predict where the expected  extensions   could lie.  Here we find our first surprise:  a tiny but non-zero proportion of  such extensions \emph{cannot} appear in $\GG^{dR}_{1,1}$.  In fact, a tension between the two weight filtrations  $W, M$ and the modular degree $k=i_1+\ldots + i_r$ implies that only  extensions with $d$ sufficiently large relative to $k$ can occur in $\GG^{\Hc}_{1,1}$. This  constraint  is very weak and is in fact vacuous for small values of $k$.

In the positive direction, we conjecture that all other extensions actually arise in $\GG^{\Hc}_{1,1}$ and prove this for $k=0,1$. 
In particular, we construct  non-trivial \emph{zeta elements}:
$$\sigma_{2n+1}  \quad \in \quad  \big(\Lie \, \GG^{dR}_{\Hc})^{ab}$$
corresponding to $\mathrm{Ext}^1_{\Hc}(\Q, \Q(2n+1))$, and  \emph{modular elements}
$$\sigma'_{f}(d) , \sigma''_{f}(d)  \quad  \in \quad  (\Lie \, \GG^{dR}_{\Hc}\otimes \overline{\Q})^{ab}$$
corresponding to $\mathrm{Ext}^1_{\Hc\otimes \overline{\Q}}(\Q, V_f(d))$, for every cuspidal Hecke eigenform $f$, and an integer $d$ greater than or equal to the weight of $f$.

 \begin{thm}   The images of these  elements  in $\big(\mathrm{Lie}\, \A^{dR}\otimes \overline{\Q}\big)^{ab}$ are non-zero.
 Any choice of lift  of these elements   to $\mathrm{Lie}\, \A^{dR}\otimes \overline{\Q}$. 
  acts non-trivially upon $\GG^{dR}_{1,1} \times \overline{\Q}$.  
  \end{thm} 
  
   By abuse of terminology, we sometimes refer to zeta or modular elements as a choice of  lifting of these elements in 
   $\mathrm{Lie}\, \A^{dR}\otimes \overline{\Q}$.
  The proof uses a calculation of certain periods of $\GG^{\Hc}_{1,1}$, and a  Rankin-Selberg type argument. The elements $\sigma_{2n+1}$
 are related to the odd zeta values $\zeta(2n+1)$, and the elements $\sigma_f(d)$ to the non-critical values of the $L$-function $L(f, d)$ of the cusp form $f$. 

Finally, our general description of the automorphism group $\A^{dR}$ together with its mixed Hodge structure allows us to prove a general freeness criterion for elements in its Lie algebra. In particular we deduce the following  theorem, analogue of $(iv)$:

\begin{thm} \label{thmintrofree} Any choice of lift of the elements $\sigma_{2n+1}, \sigma'_f(d),  \sigma''_f(d)$  to $\Lie \, \A^{dR} \times \overline{\Q}$ generate a free Lie algebra. Thus they act freely on $\GG^{dR}_{1,1} \times \overline{\Q}$. 
\end{thm} 
 
 \subsection{Applications}
 The  theory developed here has a number of applications:
 \begin{itemize}
 \item  We prove that quadratic relations in  a certain quotient of the derivation algebra of the fundamental Lie algebra of the infinitesmial Tate curve 
 found by Pollack  \cite{Po} lift to genuine relations, and extend to all depths $\geq 2$. This was also proved in \cite{HaDBCA}, based on the computations in an earlier version of this paper, but using rather different techniques. We also compute the `arithmetic' part of the action of the zeta elements $\sigma_{2n+1}$ on this Lie algebra to leading order.  It involves a quotient of two Bernoulli numbers.
 
 \item We used the methods of this paper to write down a canonical extension of the zeta elements $\sigma_{2n+1}$ to depths 3 and 4, which was the subject of  \cite{BrSigma}. 
 
 \item We provide a geometric explanation for the modular depth-defect between double zeta values \cite{GKZ}.
 
 \item  We construct  `motivic' versions of Manin's iterated Shimura integrals \cite{Ma1,Ma2} in the category $\Hc$ and compute the action of the  Galois group $\GG^B_{\Hc}$ upon them.   This should have applications to perturbative quantum field theory, where   periods of mixed modular type are expected  to arise as Feynman integrals.
  
 \item We construct single-valued or equivariant versions of iterated integrals of modular forms. In particular, we construct linear combinations of iterated integrals of Eisenstein series and their complex conjugates which are non-holomorphic modular forms.  We expect that  the so-called modular graph functions in genus one superstring perturbation theory lie in this new class of modular forms.
 \end{itemize}
 \subsection{Outlook} It is highly likely that the methods of this paper also lead to a construction of, and freeness theorem for, elements $\sigma_{f\otimes g}(d)$ corresponding to the Rankin-Selberg convolution of two cusp forms for $d$ in the range where the Ext group has rank 1. We conjecture that the corresponding elements in the rank 2 case should also arise in $\GG^{\Hc}_{1,1}$ and pinpoint where they should occur. Thus, if we are optimistic,  the theory described here may lead to a proof of Beilinson's conjecture for Rankin-Selberg convolutions of two or more cusp forms. 
 In a different direction, most of the results of this paper should extend without too much difficulty to arbitrary congruence subgroups of $\SL_2(\Z)$. The system of relative completions for modular groups is extremely rich, and will eventually lead,   we hope,   to a modular construction of mixed Tate motives over cyclotomic fields. 
 
 \subsection{Plan} The paper is divided into three parts which use somewhat different techniques. 
 The first part is entirely analytic, and  concerns the periods of the totally holomorphic quotient of the relative completion $\GG^{dR}_{1,1}$. 
 The periods of the latter are  regularised versions of Manin's iterated Shimura integrals  for $\SL_2(\Z)$ and define a non-abelian group cocycle for $\SL_2(\Z)$.  In general, we call the ring of periods generated by $\GG^{\Hc}_{1,1}$ multiple modular values. It is important to note that only a  subspace of the multiple modular values are  given by regularised iterated Shimura integrals, which we call `totally holomorphic'. These are much more accessible than the general periods of relative completion and are studied in the first part. 
  A certain family of relations between these periods plays an important role in the proof of the freeness theorem. 
  It is a generalisation of the Petersson inner product to iterated integrals of modular forms which we call the transference principle. In \S\ref{sectDoubleEisandL}, we compute some periods of double Eisenstein integrals  using a version of the Rankin-Selberg method.
 
 The second part is entirely algebraic and Hodge-theoretic. In \S\ref{sectNew1} we consider a general affine group scheme in a Tannakian category and study the action of the Tannaka group of the latter on the former. These results might be of independent interest. 
 The remainder of this part  is to apply this construction to the  relative completion of $\SL_2(\Z)$, and define its group $\A$ of automorphisms. This uses the limiting mixed Hodge structure on $\GG_{1,1}$ in a fundamental  way. The necessary background  is taken from Hain's papers and  is recalled in \S\ref{sectM11MHS}.
 A major difference with the genus zero situation, which is entirely combinatorial,  is that the Hodge theory provides non-trivial constraints on the action of the Tannaka group $\GG^{dR}_{\Hc}$ upon $\GG^{dR}_{1,1}$.

 The third part of the paper brings the two strands together via the theory of $\Hc$-periods. We combine the Hodge-theoretic results of the second part with the period computations of the first to construct the zeta and modular elements and prove that they are non-zero. Section \S\ref{sectOutline} paints a conjectural panorama of the expected Galois theory of multiple modular values, and the later sections  give applications.
 The single-valued and equivariant versions of  the periods of relative completion are constructed in \S\ref{sectSV}, which may also be of independent interest.  Pollack's relations are studied in \S\ref{sectRelations}, the freeness theorem proved in \S\ref{sectFreeness}, and a decomposition theorem for `motivic' iterated integrals of Eisenstein series proved in \S\ref{sectDecomp}, with applications to the study of relations between double motivic zeta values. Their  existence is directly related to the modular elements $\sigma_f(w_f)$, where  $w_f$ is the modular weight of  a cusp form $f$.

\subsection{Acknowledgement}
 I am greatly indebted  to Andrey Levin for  discussions in 2013 and encouragement, and especially to Richard Hain 
  who patiently explained his results on  the Hodge theory of relative completion,  and his joint work with Matsumoto during a stay at the IAS in 2014 and 2015.   Many thanks also to  Pierre Cartier, Nikos Diamantis, Marc Levine, Yuri I. Manin and Frederico Zerbini  for their interest and corrections. The first part  of this  work was carried out at the IHES, and was partly supported by  ERC grant PAGAP 257638. Numerical checks for the periods    were computed during a stay at Humboldt University in summer 2013. The second and third parts of this paper were completed at All Souls College, Oxford, in the latter half of 2016.

\section{Basic notation and reminders}  \label{sectBasicNotation}

 All tensor products are over $\Q$ unless stated otherwise.

\subsection{Modular forms} 
\subsubsection{} \label{sectGammaST} Let $\Gamma = \mathrm{SL}_2(\Z)$, acting on the left on $\HH= \{ \tau \in \C: \mathrm{Im} (\tau)>0\}$ via
$$
\tau \mapsto {a \tau + b \over c\tau +d} \quad \hbox{ where }  \quad  \gamma  =  \left(
\begin{array}{cc}
  a  & b  \\
   c  &   d 
\end{array}
\right) \in \Gamma \ .$$
Recall that the group $\Gamma$ is generated by matrices $S,T$ defined by 
\begin{equation}\nonumber
S= 
\left(
\begin{array}{cc}
  0   & -1  \\
   1  &   0 
\end{array}
\right)\quad, \quad  T= \left(
\begin{array}{cc}
  1   &  1  \\
   0  &   1 
\end{array}
\right)
\ .
\end{equation} 
If we set $U = T S$, then  $S^2=U^3=-1$. Let $\Gamma_{\infty}$   denote the subgroup of $\Gamma$ consisting of matrices
with a $0$ in the lower left hand corner. It is generated by  $-1,T$ and is the   stabilizer of the 
cusp $\tau = i \infty $. Write   $q = \exp( 2  \pi i  \tau)$ for $\tau \in \HH$.

 \subsubsection{} \label{sectRightSL2act} For $n\geq 0$, let  $V_n$ denote the vector space of homogeneous polynomials  in $X,Y$ of degree $n$ with rational coefficients, and write  $V_{\infty}Ê= \bigoplus_{n\geq 0} V_{n}  \subset \Q[X,Y]$. 
The  graded vector space $V_\infty$ admits the following  right action of   $\SL_2(\Q)$
$$   P(X,Y)\big|_{\gamma}  = P(aX+bY, cX+dY)\quad \hbox{ where } \gamma  =  \left(
\begin{array}{cc}
  a  & b  \\
   c  &   d 
\end{array}
\right)\ .$$
 We shall identify  $V_{\infty}^{\otimes n}$
with the vector space of  (multi-)homogeneous polynomials in $X_1,Y_1,\ldots, X_n,Y_n$. Thus a tensor 
$ X^{i_1}Y^{j_1} \otimes  \ldots \otimes  X^{i_n} Y^{j_n}$ will be denoted by  $X_1^{i_1} Y_1^{j_1} \ldots X_n^{i_n} Y_n^{j_n}.$
We shall view $V_{n}$, $V_{\infty}$, and their various tensor products as trivial bundles over $\HH$, equipped with the   action of $\Gamma$.

In the second and third parts of this paper, we shall put a pure Hodge structure on the vector space $V_{2n}$, and distinguish between its
Betti and de Rham versions, the latter being denoted by sans serif letters such as $\Xs, \Ys$.  All the operators and constructions defined above apply verbatim to these variants. The reader may wish to  bear in mind that every occurence of $V_{2n}, X, Y$ in  Part I does indeed correspond to the Betti version of $V_{2n}$ as the notation suggests (contrary to what one might sometimes expect).

\subsubsection{} \label{secFourierexp} Let $\M_k(\Gamma)$  denote the  vector space  over $\Q$ spanned by modular forms $f(\tau)$  for $\Gamma$ of weight $k$.   Every such modular form admits a Fourier expansion
$$f(q) = \sum_{n\geq 0} a_n(f) \, q^n\quad \hbox{ where } \,\, a_n(f) \in \Q \ . $$
Let $\M_k(\Gamma)  = \Eis_k(\Gamma) \oplus \Ss_k(\Gamma)$ denote the decomposition into Eisenstein series and cusp forms. The Eisenstein series of weight $2k\geq 4$ will be denoted by 
\begin{equation}\nonumber
E_{2k} (q) = - {\Be_{2k} \over 4k} + \sum_{ n \geq 1} \sigma_{2k-1}(n) q^n \ ,
\end{equation}
where $\Be_{2k}$ is the $2k^{\mathrm{th}}$ Bernoulli number, and $\sigma$ denotes the divisor function.
For every modular form $f(\tau)\in \M_{2k}(\Gamma)$ of weight $2k\geq 4$ we shall write:
\begin{equation} \label{underlinefdefinition}
\underline{f}(\tau) = (2\pi i)^{2k-1} f(\tau) (X- \tau Y )^{2k-2} d \tau\ .
\end{equation}
It is viewed as a section of $ \Omega^1(\HH ; V_{2k-2}\otimes \C).$
The reason for choosing this particular normalisation is to simplify formulae in the first part of this paper by making certain periods effective, and for compatibilty with the literature on  modular forms. 
The rational de Rham normalisation  differs by $(2 \pi i)^{2k-2}$ (see $(\ref{QdR})$).

The modularity of $f$ is equivalent  to  
\begin{equation} \label{finvGamma}
 \underline{f}(\gamma(\tau))\big|_{\gamma} = \underline{f}(\tau) \quad \hbox{ for all } \gamma \in \Gamma\ .
 \end{equation}

 \subsubsection{} \label{sectLvalues}
 Let $f \in \M_{2k}(\Gamma)\otimes \C$ with Fourier expansion $f(q) = \sum_{n\geq 0} a_n \, q^n$.
Recall that its $L$-function is   the Dirichlet series, defined for $\Real(s)>2 k$, by 
 \begin{equation}\label{Lfs}
 L(f,s) = \sum_{n\geq 1 } {a_n  \over n^s}\ .
 \end{equation}
 The Mellin transform gives 
the following     equality of meromorphic functions of $s$:
\begin{equation}\label{Mellin} \Lambda(f,s)  = \int_0^{\infty} (f(iy) -a_0) y^{s} {dy \over y} \ .\end{equation}  
  By Hecke, it has a meromorphic continuation to $\C$, and the completed $L$-function 
 $$\Lambda(f,s) = (2 \pi)^{-s} \Gamma(s) L(f,s)$$
 admits a functional equation of the form $\Lambda(f, s) = (-1)^k\Lambda(f,2k-s)$.   This  follows immediately from the following lemma, which may serve as motivation for \S\ref{sect3}. Indeed, the terms in  $(\ref{Lambdafsformula})$ will be interpreted geometrically  using tangential base-points.
 \begin{lem} Let $f$ be  modular of weight $2k$.  Then  
$$\Lambda(f,s)  = R(s) + i^{2k} R(2k-s)  - a_0\Big( { 1 \over s} + {i^{2k} \over 2k-s}\Big)$$
where  $a_0 = a_0(f)$ is its zeroth Fourier coefficient and
$$R(s) =  \int_1^{  \infty} (f(iy )-a_0) y^s {d y \over y} = \sum_{n\geq 1} a_n \int_1^{\infty}  e^{-2 \pi n y} y^{s} {dy \over y} \ .$$
 which converges uniformly for  $\mathrm{Re}(s)\geq K$ for any $K$.  
\end{lem} 
\begin{proof}
For all $\mathrm{Re}\, (s) \gg0$ sufficiently large, decompose the domain of integration in the right-hand side of    $(\ref{Mellin})$ 
into a path from $0$ to $1$ and $1$ to $ \infty$. This gives
$$  \Lambda(f,s) =   \int_{1}^{ \infty} (f( iy )-a_0) y^{s} {dy  \over y}  +  \int_{0}^{1} \big(  f( iy ) y^{s}   +a_0   i^{2k} y^{s-2 k} - a_0 i^{2k} y^{s-2k} - a_0 y^s \big) {d y  \over y} \ . $$
Using  $f( i y^{-1}) = (iy)^{2k} f(iy)$, apply the change of variables $y\mapsto y^{-1}$   to the first two terms in the integrand of the second integral to obtain
\begin{eqnarray}  \Lambda(f,s) & = &   R(s)  +   \int_{1}^{\infty} i^{2k}( f(iy) - a_0 ) y^{2k-s}  {d y \over y}  -  a_0\int_0^1(  i^{2k} y^{s-2k} + y^s ) {dy \over y}   \nonumber \\   \label{Lambdafsformula}
 & = &  \Big( R(s)   - \int_0^1 a_0 y^s {dy \over y} \Big)  + i^{2k} \Big( R(2k-s) -    \int_0^1 a_0 y^{s-2k} {dy \over y} \Big)  \ .
 \end{eqnarray} 
 By analytic continuation, this formula holds for all values of $s\in \C$.
\end{proof}
 The $L$-function of 
 the normalised Eisenstein series is 
 \begin{equation}
 L(E_{2k},s) = \zeta(s) \zeta(s-2k+1)\ .
 \end{equation} 
 When $f$ is a cusp form,  $(\ref{Lfs})$   converges for $\Real(s) > k+1$ and is entire. If $f$ is a Hecke normalised eigenform, then its $L$-function admits an Euler product expansion
 $$L(f,s) = \prod_p (1 - a_p p^{-s}+ p^{k-1-2s})^{-1}$$
which converges for $\Real(s)>k+1$ (\cite{Lang} II, \S2). 
 In particular $L(f,n)$ and hence $\Lambda(f,n)$ does not vanish  for  all integers $n\geq k+2$. Finally, 
  recall Euler's formula 
 for the special values of the Riemann zeta function:  $\zeta(0) = - {1 \over 2}$, and for all $n \geq 1$: 
 $$\zeta(2n) =  -{\Be_{2n}\over 2}  {(2  \pi i )^{2n} \over (2n)! }   \qquad \hbox{ and }  \qquad  \zeta(-n) =  - {\Be_{n+1} \over n+1}   \ .$$
 
 \subsubsection{}  Let $\M_{1,1}$ denote the moduli stack of elliptic curves. Its analytification   $\M^{an}_{1,1}$ is   the orbifold quotient $\Gamma \bq \HH$.  Let $\overline{\M}^{an}_{1,1}$ denote
its  compactification, and denote the cusp, corresponding to the point $i\infty$ on the boundary of $\HH$, by $p$.  There is a canonical  tangential base point at $p$ which we shall denote by 
\begin{equation} \label{introTBdefn} \tbp=  {\partial/ \partial q}\ .\end{equation}  
\subsection{Tensor algebras}

\subsubsection{} \label{Tcnotation} 
Let $W=\bigoplus_{m\geq 0} W_m$ be  a graded  vector space over $\Q$ whose graded pieces $W_m$ are finite-dimensional. Its graded dual is
defined to be $W^{\vee} = \bigoplus_{m\geq 0} W_m^{\vee}$. All infinite-dimensional vector spaces considered in this paper will be of this type.
Let 
$$T(W) = \bigoplus_{n\geq 0} W^{\otimes n}$$
denote the tensor algebra on $W$. It is a graded Hopf algebra for the  grading  given  by the length of tensors, and the  coproduct for which each $w\in W$ is primitive.
Its graded dual (in the above sense, i.e., using the grading $W_m$ on $W$) is the tensor coalgebra 
$$T^c(W)  \quad (\hbox{sometimes denoted by } \Q\langle W \rangle) $$
 which is a commutative graded Hopf algebra whose generators  will be denoted using the bar notation $[w_1 | \ldots | w_n]$, where $w_i \in W$. The coproduct is 
 $$\Delta ([w_1|\ldotsÊ| w_n]) = \sum_{0 \leq i\leq n} [w_1|\ldots |w_i] \otimes [w_{i+1}| \ldots | w_n]\ .\ $$
The antipode is the linear map defined on generators by
$$ S: [w_1| \ldots | w_n]  \mapsto (-1)^n [w_n |\ldots | w_1]\ .$$
The multiplication on $T^c(W)$ is given by the shuffle product, denoted by $\sha$ \cite{Ca}.

\subsubsection{}   Often it is convenient to work with a basis $X = \bigcup_{m\geq 0} X_m $ of $W= \bigoplus_{m\geq 0} W_m$. 
 Then we shall sometimes denote by $T(X)$ (or $T^c(X)$)  the tensor algebra (or tensor coalgebra) on the  vector space  $W$ 
generated by $X$ over $\Q$.

The topological dual of $T^c(X)$ is isomorphic to the ring 
$$\Q\langle \langle X \rangle \rangle = \{ S= \sum_{w \in X^*} S_w w, \hbox{ where } S_w \in \Q \} $$
of non-commutative formal power series in $X$, where  $X^*$ denotes the free monoid generated by $X$. It is a complete Hopf algebra equipped with the coproduct
 for which the elements of $X$ are primitive. A series $S$ in $\Q\langle \langle X \rangle \rangle$ is invertible if and only if $S_{1}\neq 0$, where $1\in X^{*}$ denotes the empty word.
A series $S$ is group-like if and only if   its coefficients satisfy the shuffle equations:  the linear map defined on generators by
$$w \mapsto S_{w} : T^c(X) \To \Q$$
is a homomorphism for the shuffle product $\sha$.

By the previous paragraph, $\Spec T^c(X)$ is an affine group scheme over $\Q$. It is pro-unipotent. 
 For any commutative unitary ring $R$, its group of $R$ points is
 $$ \{ S \in R\langle \langle X\rangle \rangle^{\times} : S \hbox{ is group-like} \}\ .$$

\subsubsection{}  \label{sectAbelian}  Let $W$ be a vector space over $\Q$ as above. The  algebra 
 $\mathrm{Sym}(W)$ defines a commutative and  cocommutative Hopf subalgebra
 \begin{eqnarray} 
 \mathrm{Sym} (W)  & \subset & T^c(W) \nonumber \\
 w_1\ldots w_n & \mapsto &\sum_{\sigma} w_{\sigma(1)} \otimes \ldots w_{\sigma(n)} \nonumber
 \end{eqnarray} 
 where the sum is over all permutations of $n$ letters, and  $\mathrm{Sym}(W)$ is
 equipped with the coproduct for which the elements of $W$ are primitive. The affine  group scheme 
  $\Spec (\mathrm{Sym} W)$ can be identified with the abelianization of $\Spec T^c(W)$. 
Its group  of $R$-points 
  is the abelian group $\Hom(W,R)$.

\subsection{Group cohomology}

\subsubsection{}\label{sectGroupCohomReminders} Let $G$ be a (finitely-generated) group, and let $V$ be a right $G$-module over a $\Q$-algebra  $R$. Recall that the group of $i$-cochains for $G$ is the abelian group generated by 
maps from the product of $i$ copies of $G$ to $V$:
$$ C^i(G;  V) = \langle f: G^i \To V\rangle_R \ .$$
These  form  a complex with respect to differentials $\delta^i :  C^i(G ; V)  \rightarrow C^{i+1}(G; V)$, whose  $i^{\mathrm{th}}$ homology group is denoted
$H^i(G; V)$. The group of $i$ cocycles is denoted $Z^i (G;V)$. We shall only need the following special cases:
 \begin{itemize}
 \item A $0$-cochain is an element $v \in V$. Its coboundary is $$\delta^0(v)(g) =    v|_g-v \ .$$
In particular $H^0(G;V) \cong  Z^0(G;V) \cong V^G$, the group of $G$-invariants of $V$.
\vspace{0.1in}

\item A $1$-cochain is a map $f: G \rightarrow V$. Its coboundary is
$$\delta^1 f(g,h) =   f(gh)-   f(g) \big|_h - f(h)  \ . $$
\end{itemize}
 
We shall often denote the value of a cochain $f$ on $g \in G$ by a subscript $f_g$.

\subsubsection{Cup products}  \label{sectCupproducts}
There is a cup product on cochains
$$\cup : C^i(G;V_1) \otimes_R C^j (G;V_2) \To C^{i+j}(G;V_1 \otimes_R V_2)\ ,$$
which satisfies a version of the Leibniz rule 
$ \delta (\alpha \cup \beta ) = (-1)^\beta \delta(\alpha) \cup \beta +  \alpha \cup \delta(\beta)$.
In particular, cup products of cocycles are cocycles. Some special cases: 
$$ 
\begin{array}{ccrl}
 (i,j)=(0,1)  :&  \qquad   ( v\cup \phi)(g)  &=& v|_g \otimes \phi(g)   \\
 (i,j)= (1,0)  :&    \qquad   ( \phi \cup v)(g) &=& \phi(g) \otimes v      \\
  (i,j)=(1,1)  :&    \qquad   ( \phi_1  \cup  \phi_2)(g, h)& = & \phi_1(g)|_h \otimes \phi_2(h)   \ .
\end{array}
$$

\subsubsection{Relative cohomology}  \label{sectRelcohom}  Let $H\leq G$ be a subgroup, and let $C^i( G, H ; V)$ denote the cone of the  restriction morphism:
$$ i^*: C^i(G, V) \To C^i(H, V)\ .$$
Denote the homology of $C^i(G,H;V)$ by $H^i(G,H;V)$. 
Chains in   $C^i( G,H ; V)$    can be represented by pairs
 $(\alpha, \beta)$, where  $\alpha \in C^i(G; V)$ and $\beta \in  C^{i-1}(H;V),$ with differential
 $$\delta( \alpha, \beta) = (\delta \alpha, i^*\alpha - \delta \beta)$$
 where $i^*$ denotes restriction to $H$.  There is a long exact cohomology sequence
  \begin{equation} \label{longexactH} \cdots  \rightarrow H^i(G; V) \rightarrow H^i(H; V)  \rightarrow  H^{i+1} ( G, H;  V)   \rightarrow H^{i+1}(G; V) \rightarrow \cdots \ .
 \end{equation}

 \subsection{Representations of $\mathrm{SL}_2$}

\subsubsection{Tensor products} \label{sectdeltakdef}
Let $m,n \geq 0$. There is an isomorphism of $\mathrm{SL}_2$-representations 
$$ V_m \otimes V_n \overset{\sim}{\longrightarrow} V_{m+n} \oplus V_{m+n-2} \oplus \ldots \oplus V_{|m-n|}$$
Identifying $V_m  = \bigoplus_{i+j=m} X^iY^j \Q$, we can  define an explicit  $\mathrm{SL}_2$-equivariant map
$\partial^k : V_m \otimes V_n \rightarrow V_{m+n-2k}$ for all $k\geq 0$
as follows. First of all, let us denote the projection onto the top component 
\begin{equation} \label{piddefn}
 \pi_d : V_{m_1} \otimes \cdots \otimes V_{m_n} \To V_{m_1+ \ldots +m_n}
\end{equation}
It is given by  the diagonal map $\Q[X_1,\ldots, X_n, Y_1,\ldots, Y_n] \To \Q[X,Y]$ which sends every $(X_i,Y_i)$ to $(X,Y)$.
Now define
$$
\partial^k: \Q[X_1,X_2,Y_1,Y_2]  \To \Q[X,Y]$$ 
to be the operator $ \pi_d(\partial_{12})^k$ where
$$ \partial_{12} =  {\partial \over \partial X_1}{\partial \over \partial Y_2} - {\partial \over \partial Y_1}{\partial \over \partial X_2}  \ .
$$
The operator $\partial^k$ decreases the degree by $2k$ and is evidently $\mathrm{SL}_2$-equivariant.
 It is $(-1)^k$ symmetric with respect to the  involution $v\otimes w \mapsto w\otimes v: V_m \otimes V_n \overset{\sim}{\rightarrow} V_n \otimes V_m.$

 \subsubsection{Equivariant inner product} \label{sectIP} In particular,
the operator $(k!)^2 \partial^k: V_k \otimes V_k \rightarrow V_0$ defines a 
 $\Gamma$-invariant pairing commonly denoted by 
$$\langle \  , \  \rangle:  V_k\otimes V_k \To \Q\ .$$
It is uniquely determined by the property that for all $P(X,Y) \in V_k$  
\begin{equation}  \label{PIPformula}
\langle P , (aX+bY)^k\rangle = P(-b,a)\ .
\end{equation}
In particular  $\langle P|_{\gamma}, Q|_{\gamma} \rangle = \langle P, Q \rangle $ for all $\gamma \in\Gamma$ and $P, Q\in V_k$.

Now suppose that $P, Q: \Gamma \rightarrow V_k\otimes \C$ are two $\Gamma$-cocycles, and suppose that $Q$ is cuspidal (i.e., $Q_T=0)$.
Define the Peterssen-Haberlund pairing \cite{KZ,PaPo} by
\begin{equation}  \label{curlypairing}
\{P, Q\} = \langle P_S , Q _S\big|_{T-T^{-1}} \rangle -  2\, \langle P_T, Q_S\Big|_{1+T} \rangle
\end{equation}
It will be derived in  \S\ref{sectPairingandcup}   and \S\ref{Haberlundviageometry}. It is  antisymmetric when $P_T=0$, i.e., $P$ and $Q$ are both cuspidal, but not otherwise. 
It has the property that $\{P,Q\}=0$ whenever $P$ is  the cocycle of a Hecke normalised  Eisenstein series (proved in \S\ref{sectLengthone}).

\subsubsection{Highest and lowest weight vectors} 
We frequently use the notation 
$$\varepsilon^{\vee}_0 = X { \partial \over \partial Y} \quad \hbox{ and } \quad \varepsilon_0 = Y { \partial \over \partial X} \ . $$
These encode the action of the Lie algebra $\mathfrak{sl}_2$ on $V_{2n}$.  Note that $\varepsilon_0$ is the logarithm of $T:(X,Y) \mapsto (X+Y, Y)$.
There is an exact sequence of   $\Q$-vector spaces:
\begin{equation} \label{Tlongexact}  0 \To Y^{2k} \Q \To  V_{2k} \overset{T-1}{\To}   V_{2k} \To X^{2k} \Q \To 0 \  . 
\end{equation}
With our conventions, the space of highest weight vectors in $V_{2n}$ is  $\Q X^{2n}$, the space of lowest weight vectors is $\Q Y^{2n}$. 
It follows directly  from the definition that  the map $f\mapsto f_{T}: Z^1(\Gamma_{\infty}, V_{2k}) \cong V_{2k}$  is an isomorphism, and the set of
coboundaries $B^1(\Gamma_{\infty}; V_{2k})$ is the cokernel of $T-1$. Therefore
\begin{equation} \label{H01Gammainfinity} H^0(\Gamma_{\infty}; V_{2k}) = Y^{2k} \Q \qquad \hbox{ and } \qquad H^1(\Gamma_{\infty}; V_{2k}) \cong X^{2k} \Q\ . 
\end{equation}
In the tensor product $V_{2m} \otimes V_{2n}$, the lowest weight vectors are spanned by 
$$(X_1 Y_2 - X_2 Y_1)^k Y_1^{2m-k} Y_2^{2n-k} \qquad \hbox{ for } \quad 0\leq k \leq \min\{2m,2n\}\ .$$
The one-dimensional $\Q$-vector space generated by the previous  element corresponds via $\partial^k : V_{2m} \otimes V_{2n} \overset{\sim}{\rightarrow} V_{2m+2n-2k}$ 
to $\Q Y^{2m+2n-2n}$.

 \newpage
 
 \begin{center}
 \Large{\bf{ Part I: Analysis:  Iterated integrals of modular forms.}}
 \end{center}

\section{Iterated  Shimura integrals} \label{sect2}

We recall some basic  properties of iterated Shimura integrals  on modular curves which are essentially contained in Manin's paper \cite{Ma1}. We only consider  the  special case $\Gamma=\SL_2(\Z)$ and work entirely on the universal covering space $\HH$. 

\subsection{Generalities on iterated integrals} \label{GeneralitiesItInt}
Let $\omega_1, \ldots, \omega_n$ be smooth 1-forms  on  a differentiable manifold $M$.
For any  piecewise smooth path $\gamma: [0,1] \rightarrow M$,  the iterated integral of $\omega_1,\ldots, \omega_n$ along $\gamma$
 is defined  by
$$ \int_{\gamma} \omega_1 \ldots \omega_n = \int_{0< t_1 < \ldots < t_n < 1} \gamma^*(\omega_1)(t_1) \ldots  \gamma^*(\omega_n)(t_n)\ .$$
The empty iterated integral $n=0$ is defined to be the constant $1$.
Well-known results due to Chen \cite{Ch}   state that there is the composition of paths formula:
\begin{equation} \label{compospaths}
 \int_{\gamma_1 \gamma_2} \omega_1 \ldots \omega_n =   \sum_{i=0}^n \int_{\gamma_1 }\omega_1 \ldots \omega_i \int_{\gamma_2} \omega_{i+1}\ldots \omega_n \ ,
 \end{equation}
whenever $\gamma_1(1) = \gamma_2(0)$ and $\gamma_1\gamma_2$ denotes the path $\gamma_1$ followed by $\gamma_2$.
The shuffle product formula states that iterated integration along a path $\gamma$ is a homomorphism for the shuffle product. Extending the definition by linearity, this reads $$\int_{\gamma} \omega_1 \ldots \omega_m \int_{\gamma} \omega'_1 \ldots \omega'_n = \int_{\gamma} \omega_1\ldots \omega_{m} \sha \omega'_1 \ldots \omega'_n\ .$$
Finally, recall that the reversal of paths formula states that 
$$\int_{\gamma^{-1}} \omega_1 \ldots \omega_n = (-1)^n \int_{\gamma} \omega_n \ldots \omega_1$$
where $\gamma^{-1}$ denotes the reversed path $t \mapsto \gamma(1-t)$.  Many basic properties of iterated integrals can be found in \cite{Ch}.
One often  writes  iterated integrals using bar notation
$$\int_{\gamma} \omega_1 \ldots \omega_n = \int_{\gamma} [\omega_1 | \ldots | \omega_n]\ .$$

It is convenient to work with generating series of iterated integrals, indexed by non-commuting symbols, as follows.

\subsection{Notations}  \label{sectNotations3.2}
Most of the constructions in this paper will be  defined intrinsically, but it can be useful to  fix a rational basis  $\B$ of $\M(\Gamma) =\bigoplus_k \M_k(\Gamma)$.  We assume that $\B=\cup_k \B_k$ where $\B_k$ is a basis of $\M_k(\Gamma)$, and that $\B_k$ is compatible with the action of Hecke operators. This means that $\B_k$ is a disjoint union of sets $\B_{k,g}$, each of which is a
basis for generalised eigenspaces $g$ with respect to the action of Hecke operators.  
Define a $\Q$-vector space with a  basis  consisting of  certain symbols  indexed by $\B_{k,g}$:
$$M_{k,g}   = \langle \al_f:  f \in \B_{k,g}  \rangle_{\Q} \ , \quad \hbox{ and set } \quad M_k= \bigoplus_g M_{k,g}     .$$
  In order to distinguish between vector spaces and their duals, we shall reserve upper case letters (to be consistent with \cite{Ma1,Ma2}) for the dual vector spaces
$$M_{k,g}^{\vee}   = \langle \Al_f:  f \in \B_{k,g}  \rangle_{\Q}   \quad   \hbox{ and }    \quad M_k^{\vee} = \bigoplus M_{k,g}^{\vee}$$
where 
$\langle \al_f, \Al_{g}\rangle = \delta_{f,g}$, and $\delta$ is the Kronecker delta.
 We can assume $\B_{2n}$ contains the Hecke normalised Eisenstein series $E_{2n}$, and write more simply
\begin{equation} \label{ennotation} 
\e_{2n} \quad \hbox{ for  } \quad  \al_{E_{2n}} \quad \ , \quad  \hbox{ and } \quad   \EE_{2n} \quad  \hbox{ for  } \quad \Al_{E_{2n}}
\end{equation}
  Consider the  graded  right $\SL_2$-module
 $$M^{\vee} = \bigoplus_{k\geq0}  M^{\vee}_k \otimes V_{k-2}$$
 which has one copy of $V_{k-2}$ for every element of $\B_k$.
 For any commutative unitary $\Q$-algebra $R$,  let  $ R\langle \langle M^{\vee} \rangle \rangle$
 denote the ring of formal power series in $M^{\vee}$.
  It is a complete Hopf algebra with respect to  the coproduct  which makes every element of $M^{\vee}$ 
 primitive. Its elements  can be represented by infinite $R$-linear combinations of 
  \begin{equation} \label{AlfXY}
   \Al_{f_1} \ldots \Al_{f_n}  \otimes  X_1^{i_1-1} Y_1^{k_1-i_1-1}
  \cdots X_n^{i_n-1} Y_n^{k_n-i_n-1} 
  \end{equation} where  $f_j \in \B_{k_j}$   and  $1\leq i_j \leq k_j-1$.
  
  \begin{rem} Hain's notations are equivalent but slightly different. Given a Hecke eigenform $f$ of weight $n$ he writes
   $S^{n-2}(e_f)$ for  the  $\mathrm{SL}_2$-representation $\Al_f\otimes V_{n-2}$, where $e_f$ denotes
   the   highest weight vector  $\Al_f \otimes X^{n-2}$. Note, however, that he works with left $\SL_2$-modules as opposed to the right modules we consider here.
    \end{rem} 
    In the second and third parts of this paper, the symbols  $\al_f$ will be interpreted as elements in the graded Lie algebra of the unipotent radical of the de Rham completion of the relative fundamental group of $\M_{1,1}$.

 \subsection{Iterated Shimura integrals}  \label{sect2.2}
The trivial vector bundle 
 $\Or_{\HH}\langle\langle M^{\vee} \rangle \rangle$ on  $\HH$ can be equipped with the connection
 $$ \nabla:  \Or_{\HH} \langle \langle M^{\vee} \rangle \rangle   \To \Omega^1_{\HH} \langle \langle M^{\vee} \rangle \rangle
$$ defined by 
  $\nabla= d + \Omega(\tau)$, where $d(\Al_f)=0$,  
  \begin{equation}  \label{Omegadefn} \Omega (\tau)=  \sum_{f \in \B}  \Al_f  \,  \underline{f}(\tau)  \ ,
  \end{equation}
 and $\Al_f$ acts on $ \C \langle \langle M^{\vee} \rangle \rangle $ by concatenation on the left. Clearly $\nabla$ is flat 
 because $d \Omega(\tau)=0$ and $\Omega(\tau) \wedge \Omega(\tau)=0$.    
 By the invariance $(\ref{finvGamma})$  of $\underline{f}(\tau)$, we have 
  $$ \Omega(\gamma(\tau)) \big|_{\gamma}= \Omega(\tau) \hbox{ for all } \gamma \in \Gamma\ . $$ 
  Horizontal sections of this  vector bundle can be written down using iterated integrals. 
 Let $\gamma :[0,1]\rightarrow \HH$ denote a piecewise smooth path,  with endpoints $\gamma(0)=\tau_0$, and $\gamma(1)=\tau_1$, and  consider the iterated integral
\begin{equation} \label{Idefn} 
I_{\gamma} = 1 + \int_{\gamma} \Omega(\tau) + \int_{\gamma} \Omega(\tau) \Omega(\tau)  + \ldots   
\end{equation}
Since the connection $\nabla$ is flat,
$I_{\gamma}$ only depends on the homotopy class of $\gamma$ relative to its endpoints. Since  $\HH$ is simply connected, $I_{\gamma}$ only depends on the endpoints of $\gamma$ and we can write
$$I(\tau_0;\tau_1) \in \C\langle \langle M^{\vee}\rangle \rangle\ . $$
It is a well-defined function on $\HH\times \HH$, and  for all $\tau_1 \in \HH$, the map $\tau \mapsto I(\tau; \tau_1)$ defines a horizontal section of the bundle $(\Or_{\HH}\langle \langle M^{\vee} \rangleÊ\rangle, \nabla)$.

\subsection{Properties} 
\begin{prop} \label{propPropertiesI} The integrals $I(\tau_0;\tau_1)$ have the following properties:

i). (Differential equation).
$$ d  I(\tau_0;\tau_1) =     I(\tau_0;\tau_1) \, \Omega(\tau_1) - \Omega(\tau_0)\,I(\tau_0;\tau_1)\ .$$

ii).  (Composition of paths).  For all $\tau_0, \tau_1,\tau_2 \in \HH$, 
$$I(\tau_0;\tau_2) = I(\tau_0;\tau_1) I(\tau_1;\tau_2)\ .$$

iii). (Shuffle product).  
$$ I(\tau_0 ;  \tau_1) \in \C\langle \langle  M^{\vee}\rangle \rangle \hbox{ is invertible and  group-like} \ .$$

iv). ($\Gamma$-invariance). For all $\gamma \in \Gamma$, and $\tau_0, \tau_1 \in\HH$, we have
$$ I(\gamma(\tau_0); \gamma(\tau_1))\big|_{\gamma} = I(\tau_0 ;\tau_1)\ .$$

\end{prop}
\begin{proof} 
Properties $i)$-$iii)$ are general properties of iterated integrals.  The last property 
$iv)$ follows from the $\Gamma$-invariance of  $\Omega$. For any $\tau_1\in \HH$,   $  I(\gamma(\tau) ; \gamma(\tau_1))|_{\gamma}$ 
satisfies the  differential equation $\nabla F=0$, as  does $I(\tau ; \tau_1)$. It follows that 
  $  I(\gamma(\tau) ; \gamma(\tau_1))|_{\gamma} = I(\tau ; \tau_1) C$ for some constant series $C \in \C\langle \langle M^{\vee} \rangle \rangle$. Since both sides 
 are equal to $1$ when $\tau= \tau_1$, the series $C$ is equal to $1$.
   \end{proof}

 \subsection{A group scheme.} 
Consider the following graded ring and its dual
$$M= \bigoplus_{k\geq 1}  M_{2k+2} \otimes V_{2k}^{\vee} \quad \hbox{ and } \quad  M^{\vee}= \bigoplus_{k\geq 1}  M^{\vee}_{2k+2} \otimes V_{2k}  \ . $$
Then $M$ is a graded left $\SL_2$-module, and $M^{\vee}$ is a graded right $\SL_2$-module.
Let $T^c(M)$ denote the tensor coalgebra on $M$. 
It is a graded Hopf algebra over $\Q$ whose graded pieces are finite-dimensional  left $\SL_2$-representations.
Let us define
\begin{equation} \label{Pidef} 
\PiU = \Spec(T^c(M))\ . \end{equation}
The justification for this notation will be given in the second part of this paper. 
It is a non-commutative pro-unipotent affine group scheme over $\Q$, and for 
 any commutative $\Q$-algebra $R$,  its group of $R$-points is given  by formal power series
$$\PiU(R) = \{ÊS \in R\langle \langle M^{\vee} \rangle \rangle^{\times}Ê\hbox{ such that } S \hbox{ is group-like} \}\ .$$
The group $\PiU(R)$  admits a right action of  $\SL_2$, and hence $\Gamma$, which we write
$$ S\big|_{\gamma} T\big|_{\gamma} = ST\big|_{\gamma} \quad \hbox{ for } \quad S,T \in \PiU(R)\ .$$
 Property $iii)$ of proposition $\ref{propPropertiesI}$ states that the elements 
$I(\tau_0;\tau_1) \in \PiU(\C)$
for all $\tau_0,\tau_1 \in \HH\times \HH$, and in fact the iterated integral
$I : \HH \times \HH \rightarrow \PiU(\C)$ defines, by property $ii)$, an element  of the constant groupoid over $\HH$ with fibers $\PiU(\C)$.

\subsection{Representation as linear maps}
  Any element $S\in R\langle \langle M^{\vee} \rangle \rangle$   can be viewed as a collection of maps (also denoted by $S$):
\begin{equation} \label{Smap}
S: M_{2k_1+2} \otimes \ldots \otimes M_{2k_n+2} \To V_{2k_1} \otimes \ldots \otimes V_{2k_n} \otimes R
\end{equation}
 which to any $n$-tuple of modular forms associates a multi-homogeneous polynomial in $n$ pairs of variables. The right-hand side carries a right action of $\SL_2$.
 This map sends  $\al_{f_1}\ldots \al_{f_n}$ to the coefficient of $\Al_{f_1} \ldots \Al_{f_n}$ in $S$.
 A series $S$  is group-like if and only if the following shuffle relation holds
 \begin{multline} \label{phishuffle}
 S ( \al_{f_1}  \ldots \al_{f_p} )(X_1,\ldots, X_p) \,  S ( \al_{f_{p+1}}  \ldots \al_{f_{p+q}})(X_{p+1},\ldots, X_{p+q}) \\ 
 = \sum_{\sigma\in \mathfrak{S}_{p,q}} S ( \al_{f_{\sigma(1)}} \ldots  \al_{f_{\sigma(p+q)}} )(X_{\sigma(1)},\ldots, X_{\sigma(p+q)})
  \end{multline}
  and if the leading term of $S$ is $1$. In this formula, $ \mathfrak{S}_{p,q}$  denotes the set of shuffles of type $p,q$, and we dropped the variables $Y_i$ for simplicity.  Note, for example, that the polynomial  $S(\al_{f} \al_{f})$  in four variables $X_1,Y_1,X_2,Y_2$ is not completely determined by $S(\al_{f})(X_1,Y_1)$ by the  relation $(\ref{phishuffle})$; however,  its image under $\pi_d$ $(\ref{piddefn})$ is.

\section{Regularization} \label{sect3}
We explain how to regularize the iterated integrals of \S\ref{sect2} at  a tangential base point at infinity. This defines canonical iterated Eichler integrals, or higher period polynomials, for any sequence of modular forms.  The construction is simplified by exploiting the  explicit universal covering spaces that we have at our disposal.

\subsection{Tangential base points and iterated integrals} \label{GeneralTBpoint}
Let $\overline{C}$ be a smooth complex curve,  $p\in \overline{C}$ a point, and  $C= \overline{C} \backslash p$ the punctured curve.
Let $T_p$ denote the tangent space of $\overline{C}$ at the point $p$, and $T^{\times}_p = T_p \backslash \{0\}$ the punctured tangent space.

 A tangential base point on $C$ at the point $p$ is  an element $\tv \in T_p^{\times}$ (\cite{DeP1}, \S15.3-15.12).  A convenient  way to think of the tangential base point is to choose a germ of an analytic isomorphism $\Phi: (T_p,0) \rightarrow (\overline{C},p)$  such that $d \Phi: T_p \rightarrow T_p$
is the identity. One can glue the space $T_p^{\times}$ to $C$ along the map $\Phi$ to obtain a space $$T^{\times}_p \cup_{\Phi} C$$ which is
homotopy equivalent to $C$. The tangential base point $\tv$ is simply  an ordinary base point on this enlarged space.
  A path from a point $x \in C$ to this  tangential base point can be thought of as a path in $\overline{C}$ from $x$ 
to a point $\Phi(\varepsilon)$ close to $p$, followed by a path from $\varepsilon$ to $\tv$ in the tangent space $T_p$. This is pictured below.

\begin{figure}[h!]
  \begin{center}
   \epsfxsize=8cm \epsfbox{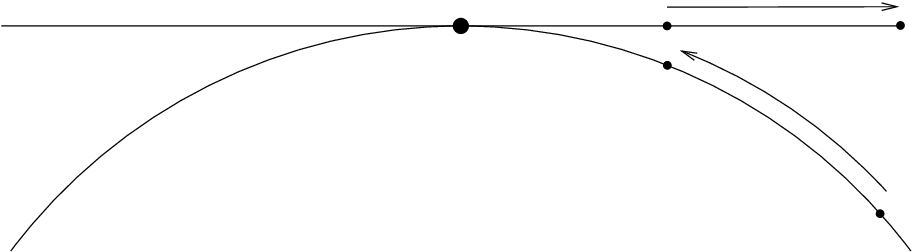}
  \label{Polygons}
  \put(-250,0){{   $\overline{C}$ }}
  \put(-250,55){{   $T_p$ }}
  \put(-248,25){{   $\downarrow \, \Phi$ }}
    \put(-122,45){{   $p$ }}
      \put(-75,35){{ \small  $\Phi(\varepsilon)$ }}
            \put(-70,65){{ \small  $\varepsilon$ }}
  \put(-5,57){{ \small  $\tv$ }}
 \put(-10,8){{ \small  $x$ }}
   \caption{}
  \end{center}
\end{figure}
Now let $\omega$ be a meromorphic one-form on $C$ with at most a logarithmic singularity at $p$. 
If we choose a linear function $q$ on $T_p$, we can locally write
$$\Phi^*(\omega) = \sum_{n\geq 0} \alpha_n q^n {dq \over q}$$
and define the polar part $P\Phi^*(\omega)$ to be the one-form $\alpha_0 {dq \over q}$ on $T^\times_p$.  It does not depend
on the choice of  function $q$. The line integral
of $\omega$ along a path from $x$ to $\tv$ is  defined to be
$$\int_x^{\tv} \omega = \lim_{\varepsilon \rightarrow p} \Big( \int_x^{\Phi(\varepsilon)} \omega + \int_{\varepsilon}^{\tv} P \Phi^*(\omega) \Big)$$
It is straightforward to verify that  the limit is finite and does not depend on $\Phi$. The analogue for iterated integrals is given by the composition of paths formula $(\ref{compospaths})$.
If $\omega_1,\ldots, \omega_n$ are  closed holomorphic one forms with logarithmic singularities at $p$,  let
$$\int_x^{\tv} \omega_1\ldots \omega_n  = \lim_{\varepsilon \rightarrow p} \Big(  \sum_{k=0}^n \int_x^{\Phi(\varepsilon)} \omega_1\ldots \omega_k   \int_{\varepsilon}^{\tv} P \Phi^*(\omega_{k+1}) \ldots  P \Phi^*(\omega_{n}) \Big)$$
The iterated integral  is finite and is  independent of the choice of $\Phi$. It only depends on $x$ and $\tv$ in the sense
that homotopy equivalent paths from $x$ to $\tv$ give rise to the same integral (since $\omega_i \wedge \omega_j=0$ for all $i,j$). The integrals
in the right-hand factors  of the right-hand side are performed on  $T^{\times}_p$, those on the left on $C$.
\\

We are interested in the  case $C= \M^{an}_{1,1}$, $\overline{C} = \overline{\M}^{an}_{1,1}$ and $p$ the cusp (image of $i \infty$). The punctured tangent 
space $T_p^{\times}$ is isomorphic to the punctured disc with coordinate $q$.  The tangential base point corresponding to $1\in T^{\times}_p$ 
is given by $(\ref{introTBdefn})$.   
\begin{rem} There are many equivalent ways to think of tangential base points. A better way is to view $\tv$ as a point on the exceptional locus
of the   blow-up of $\overline{C}$ at $p$. This makes the independence of $\Phi$ obvious. In our setting, however, we have a \emph{canonical} map $\Phi$ (given by the $q$-disc) so the presentation above is more convenient.

A more general version of regularisation exists for vector bundles with flat connections, using Deligne's canonical extension (\cite{DeP1}, \S15.3-15.12).  Instead of presenting 
this approach, we prefer to adapt the above construction for   universal covering spaces, which gives a more direct route to the same answer.

\end{rem}

\subsection{Universal covering space at ${\partial \over \partial q}$ }
The  punctured tangent space $T^{\times}_p$ of $\overline{\M}^{an}_{1,1}$  is isomorphic to  $\C^{\times}$. Its universal covering space
is $\C$   with the  covering map
$$  \tau \mapsto \exp(2\pi i \tau ) : \C \rightarrow \C^{\times}\  , $$
which sends  $0$ to $1$. 
We can therefore glue a copy of $\C$ to $\HH$ via  the natural inclusion map $i_{\infty} : \HH \rightarrow \C$ to define 
a  space 
$ \HH  \cup_{i_{\infty}} \C$ pictured below.

\begin{figure}[h!]
  \begin{center}
   \epsfxsize=8cm \epsfbox{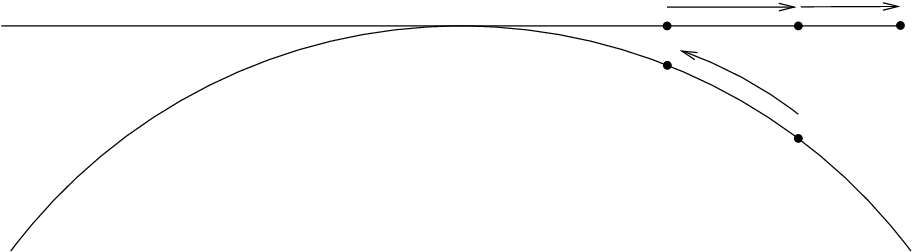}
  \label{Polygons}
  { \color{white} \put(-120,54){\circle*{8}}} 
  \put(-250,0){{   $\HH$ }}
  \put(-250,55){{   $\C$ }}
    \put(-128,45){\small {   $ i\infty $ }}
      \put(-73,38){{ \small  $\varepsilon$ }}
            \put(-75,63){{ \small  $\varepsilon $ }}
   \put(-39,65){{ \small  $\tau$ }}
  \put(-5,55){{ \small  $0$ }}
 \put(-42,19){{ \small  $\tau$ }}
   \caption{}
  \end{center}
\end{figure}
In the previous notations, $\Phi$ is a local inverse to $i_{\infty}$. 
A path from  $\tau \in \HH$ to $\tbp$ can be thought of as the compositum of the  following  two path segments 
on $\HH \cup_{i_{\infty}} \C$:
\vspace{0.05in}

(i)  a path  from $\tau$ to a point $\varepsilon \in \HH$ infinitely close to  $i \infty$, 

(ii)   a path from $i_{\infty}(\varepsilon)$ to  the point $0$ in  $\C$.
\vspace{0.05in}

As shown in the picture, the latter path can be divided into two segments, from $i_{\infty}(\varepsilon)$ to $i_{\infty}(\tau)$ and from $i_{\infty}(\tau)$ to $0$. Recombining
 these   in a different way gives
\vspace{0.05in}

(i)$'$  a path  from $\tau$ to a point $\varepsilon$, followed by a path from $i_{\infty}(\varepsilon)$  to $i_{\infty}(\tau)$. 

(ii)$'$   a path from $ i_{\infty}(\tau)$   to  the point $0$ in  $\C$.
\vspace{0.05in}

Later we shall identify $\HH$ with its image in $\C$, which means that we drop all $i_{\infty}$'s from the notation (as in figure 2 above)
and compute all integrals on $\C$.

\begin{rem} \label{remUnivCov} The universal covering space  of $\M^{an}_{1,1} \cup_{\Phi} \C^{\times}$ at the   base point  $\tbp$ 
is the space $\HH \cup_{\infty} \C$, where $\C$ is  glued  to $\HH$ along  a germ of the  map $i_{\infty}^{-1}$. 
One can  repeat this construction by gluing a copy of $\C$ at every cusp (rational point) along  the boundary of $\HH$. This gives rise to a space  $  \HH  \cup_{\Q \cup \{\infty\}} \C$, which now carries  an action of $\Gamma$.  Its orbifold quotient is $\M^{an}_{1,1}\cup_{\Phi} \C^{\times}$.
\end{rem}

\subsection{Iterated integrals on the tangent space}
In \S\ref{GeneralTBpoint},  the divergent part of $\omega$   corresponded to the form ${dq \over q}$ on $T^{\times}_p$.
On a universal covering space of  $T^{\times}_p$, the divergent parts  correspond to   iterated integrals in ${dq\over q}$, namely, polynomials in $\tau$ times $d\tau$.

\begin{defn} Let $f \in \M_{2k}(\Gamma)$, and denote the constant term in its Fourier expansion  by  $ a_0(f)$. Define the tangential component (polar part) of 
$\underline{f}(\tau)$ to be
\begin{eqnarray} \label{omegainfinity}
\underline{f}^{\infty}(\tau) =    (2 \pi i)^{2k-1}  a_0(f) (X- \tau Y)^{2k-2} d \tau \ . \end{eqnarray}
It  is to be viewed as a section of $\Omega^1(\C; V_{2k-2}\otimes \C)$ on the tangent space $\C \subset \HH\cup_{i_{\infty}} \C$.
Clearly, $f$ is a cusp form if and only if  $\underline{f}^{\infty}(\tau)$ vanishes.
\end{defn}
One can repeat the discussion of \S\ref{sect2.2} with  the trivial bundle  $\C\langle \langle M^{\vee} \rangle \rangle$  viewed this time over $\C$, and replacing $\nabla$ with the connection  $ \nabla_{\infty} = d + \Omega^{\infty}(\tau)$, where
 \begin{equation} \Omega^{\infty} (\tau)=  \sum_{f \in \B}  \Al_f  \,  \underline{f}^{\infty}(\tau)  \ ,
  \end{equation}
For any pair of points $a,b \in \C$, define $I^{\infty}(a;b) \in \C\langle \langle M^{\vee} \rangle \rangle$ to be the iterated integral
\begin{equation} \label{Iinfdef}
I^{\infty}(a; b) = 1+ \int_{\gamma} \Omega^{\infty} +  \int_{\gamma} \Omega^{\infty}  \,  \Omega^{\infty}  + \ldots 
\end{equation}
along any piecewise smooth path $\gamma:[0,1] \rightarrow \C$ such that $\gamma(0)=a, \gamma(1)=b$.
It only depends on the endpoints $a,b$ for similar reasons to proposition \ref{propPropertiesI}. In particular, the composition of paths formula
 $I^{\infty}(a; c) =I^{\infty}(a; b) \,I^{\infty}(b; c)$ holds for all $a,b,c \in \C$, and $I^{\infty}(a;b) \in \PiU(\C)$.
We have a similar equivariance property
$$\Omega^{\infty}(\gamma(\tau))\big|_{\gamma}  = \Omega^{\infty}(\tau) \quad \hbox{ for all } \gamma \in\Gamma_{\infty}\ .$$

\subsection{Iterated Eichler integrals} \label{sectIteratedEichler}
As in figure $2$,  we integrate the  form $\Omega(\tau)$ along the first path segment $(i)$ on  $\HH$, and integrate  $\Omega^{\infty}(\tau)$
along  the second segment $(ii)$ on $\C$.  Since composition of paths corresponds to the concatenation product of generating series of iterated integrals, one arrives at the
following definition.

\begin{defn}  \label{defniteratedEichler} The  iterated Eichler integral from  $\tau \in \HH$ to $\tbp$ is
$$I(\tau ; \infty) = \lim_{\varepsilon \rightarrow i \infty} \big(   I(\tau; \varepsilon)  \,  \, I^{\infty}(i_{\infty}(\varepsilon);  0) \big) \qquad \in \quad  \PiU(\C) \ \subset \  \C\langle \langle M^{\vee} \rangle \rangle \ , $$
where  $i_{\infty}: \HH \rightarrow \C$ is the inclusion. \end{defn}

The right-hand integral $I^{\infty}$ in the definition is viewed on the tangent space  $  \C$, the left-hand one 
 on $\HH$. However, using the gluing map $i_{\infty}: \HH \rightarrow \C$, we can compute both kinds of iterated integral  on a single copy of $\C$: in short we can   drop all occurrences of $i_{\infty}$ from the notation and henceforth work entirely on $\C$.

To verify the finiteness of the iterated Eichler integral, we first define,  for $\tau_0, \tau_1 \in \HH$,  the regularized iterated integral to be
$$ RI(\tau_0; \tau_1)  =  I(\tau_0; \tau_1)  I^{\infty}(\tau_1; \tau_0) \ .$$

\begin{lem}  
$ RI(\tau; x) $
is finite as $x\rightarrow i \infty$ and converges like $O(e^{2 \pi i x})$.
\end{lem} 
\begin{proof} From the differential equation for $I$ (Proposition \ref{propPropertiesI} $i)$), we check that
$$ {\partial \over \partial x }RI(\tau; x)  =   I(\tau; x) \Big(  \Omega(x) - \Omega^{\infty}(x)    \Big)  I^{\infty}(x; \tau)\ .$$ 
For each $\omega \in \M_k(\Gamma)$,   the form $\underline{\omega}(x)$ grows at most polynomially in $x$ near $\infty$.  Therefore each term in 
$ I(\tau_0; x)$, and $I^{\infty}(x; \tau_0)$,  is of polynomial growth in $x$. On the other hand 
$$\Omega(x) - \Omega^{\infty}(x) = O( \exp(2  \pi i x )) \qquad \hbox{ as } \quad x \rightarrow i \infty \ ,$$
which follows from the Fourier expansion \S\ref{secFourierexp}.  
This proves the lemma.
\end{proof}
As a consequence, we define
\begin{equation} \label{RIdef}
  RI(\tau) = \lim_{x \rightarrow i \infty}  RI(\tau; x)\ . 
  \end{equation} 
Recombining the paths in figure 2 into the two parts $(i)'$ and $(ii)'$ leads to the following formula for the generating series
of iterated Eichler integrals.

\begin{cor}  The iterated Eichler integral is a product
\begin{equation}  
\label{IisRIinf} I(\tau ;  \infty)  = RI (\tau)   \, I^{\infty}(\tau;0)\ .
\end{equation}
\end{cor}
\begin{proof}
By the composition of paths formula for $I^{\infty}$, we have
$$I(\tau;\infty) = \lim_{x \rightarrow i \infty} \big(   I(\tau;  x)  \,  \, I^{\infty}(x ; \tau)  \big)   I^{\infty}(\tau;0)  =  RI (\tau)   \, I^{\infty}(\tau,0) \ . $$ 
\end{proof}
\subsection{Properties}
The following properties are almost immediate from definition $\ref{defniteratedEichler}$. 

\begin{prop} \label{propEichler}
The iterated Eichler integrals $I(\tau ;\infty )$ have the following properties:

i). (Differential equation).
$$ {d\over d\tau  } I(\tau ;\infty) =- \Omega(\tau )\,I(\tau; \infty) \ .  $$

ii). (Composition of paths). For any $\tau_1, \tau_2 \in \HH$, 
$$ I(\tau_1 ;\infty) =  I(\tau_1 ;\tau_2) \,I(\tau_2; \infty)  \ . $$

iii). (Shuffle product). $I(\tau ;  \infty) \in \PiU(\C)$, or equivalently, 
$$ I(\tau ;  \infty) \in \C\langle \langle  M^{\vee}\rangle \rangle \hbox{ is invertible and  group-like} \ .$$
\end{prop}
\begin{proof} To verify $i)$, observe that 
$$ {\partial \over \partial \tau }     I(\tau; x)  \,   I^{\infty}(i_{\infty}(x);  0)  =   -\Omega(\tau )\,      I(\tau; x)  \,   I^{\infty}(i_{\infty}(x);  0)$$
and  take the limit as $x \rightarrow i \infty$, according to definition \ref{defniteratedEichler}. The remaining properties are straightforward and follow in a similar manner to the proof of proposition \ref{propPropertiesI}.
\end{proof}

\subsection{Explicit formulae}
Let $\omega \in \M_k(\Gamma)$, and write
\begin{equation} \label{omega0defn} 
\underline{\omega}^0(\tau) = \underline{\omega}(\tau)  -  \underline{\omega}^{\infty}(\tau)\ ,
\end{equation}
where $\underline{\omega}^0, \underline{\omega}, \underline{\omega}^{\infty}$  are viewed as sections of $\Omega^1(\C; V_{k-2}\otimes \C)$. We have seen that  $\underline{\omega}^0(\tau)$ tends to zero like $e^{ 2 \pi i \tau}$,  as $\tau$ tends to $i \infty$
along the imaginary axis.
In order to write down compact formulae for iterated Eichler integrals as integrals of absolutely convergent forms, we use the following notation. Let $W$ be a vector space together with an isomorphism
$$ (\pi^0, \pi^{\infty}): W \overset{\sim}{\To} W^0 \oplus W^{\infty}\ .$$ 
We shall also write $w^0, w^{\infty}$ for $\pi^0(w), \pi^{\infty}(w)$.   Consider the convolution   product
$$R = \sha \circ ( \id \otimes  \pi^{\infty} S ) \circ \Delta : T^c(W) \To T^c(W)$$ 
where $S, \Delta$, were defined in \S\ref{Tcnotation}, and $\sha$ is the shuffle multiplication on $T^c(V)$. 
Explicitly, the map $R$ is given for $\omega_1, \ldots, \omega_n \in W$ by 
\begin{equation} \label{Requation1} 
 R    [ \omega_1| \ldots | \omega_n  ]  = \sum_{i=0}^n (-1)^{n-i}  [\omega_1| \ldots  |\omega_i] \sha [\omega^{\infty}_n | \ldots  |\omega^{\infty}_{i+1}] \ .
\end{equation} 

\begin{lem} For any elements $\omega_1, \ldots, \omega_n \in W$ we have
\begin{equation} \label{Requation2} 
R [ \omega_1| \ldots | \omega_n  ]  = \sum_{i=1}^{n} (-1)^{n-i }\Big[  [\omega_1| \ldots  |\omega_{i-1}] \sha  [ \omega^{\infty}_{n} | \ldots | \omega^{\infty}_{i+1} ] \, \Big| \, \omega^{0}_i \Big] \ .
\end{equation}
\end{lem}
\begin{proof} By replacing the final $\omega^{0}_i$ in $(\ref{Requation2})$   by $\omega_i-\omega^{\infty}_i$,  we can view
both $( \ref{Requation1})$ and $(\ref{Requation2})$  as formal expressions inside $T^c(W \oplus W^{\infty})$.  They satisfy the formulae $R(1)=1$ and
\begin{eqnarray}  \partial_{\omega_i}  R  [ \omega_1| \ldots | \omega_n  ]   &= &\delta_{i1}  R  [ \omega_2| \ldots | \omega_n  ]  \nonumber \\
\partial_{\omega^{\infty}_i}  R  [ \omega_1| \ldots | \omega_n  ]   &=& -   R  [ \omega_1| \ldots | \omega_{n-1}  ]  \delta_{in}  \ , \nonumber 
\end{eqnarray}
where  $\partial_a$ is the differential operator on $T^c(W\oplus W^{\infty})$ defined by $ \partial_{\omega_i}    [ \omega_1| \ldots | \omega_n  ]  = \delta_{i1}$, 
and $\delta$ is the Kronecker delta.
These equations uniquely determine $R$. \end{proof} 

\begin{example} In lengths 1 and 2, 
\begin{eqnarray} \label{Rexlength1}
R [\omega_1] &=& [\omega_1] - [\omega^{\infty}_1]  \\
&  = &   [\omega^0_1]\ . \nonumber \\ \label{Rexlength2}
R [\omega_1 | \omega_2] &= & [\omega_1 | \omega_2] - [\omega_1 ] \sha  [\omega^{\infty}_2] +[  \omega^{\infty}_2 |\omega^{\infty}_1] \\
& = & [ \omega_1 | \omega^0_2] - [\omega^{\infty}_2 | \omega^0_1 ] \ .\nonumber
\end{eqnarray}
\end{example}
Applying the above to the subspace $W\subset  \Gamma^1(\C; \Omega^1_{\C} \otimes V)$  spanned by $\underline{f}(\tau)$ $(\ref{underlinefdefinition})$ for $f\in \M(\Gamma)\otimes \C$,  and  combining   with $(\ref{IisRIinf})$ leads to the following
formula:
\begin{eqnarray}  \label{EichlerConv}
\int_{\tau}^{\tbp} [\omega_1 | \ldots | \omega_n] & = &  \sum_{i=0}^{n}  \int_{\tau}^{\infty} R [ \omega_1 | \ldots | \omega_i ]  \int_{\tau}^0 [\omega^{\infty}_{i+1} | \ldots  |\omega^{\infty}_{n}]  \\
& = & \sum_{i=0}^{n} (-1)^{n-i} \int_{\tau}^{\infty} R [ \omega_1 | \ldots | \omega_i ]  \int_0^{\tau} [\omega^{\infty}_{n} | \ldots  |\omega^{\infty}_{i+1}] \nonumber
\end{eqnarray}
Each right-hand factor (the integral from $0$ to $\tau$) is simply a polynomial in $\tau$, and each left-hand factor (the integral from $\tau$ to $\infty$) converges exponentially 
fast in $\tau$.  The second line of $(\ref{EichlerConv})$ follows from the first by the reversal of paths formula $\S\ref{GeneralitiesItInt}$.

\begin{example}
In length 1, this gives for $\omega$ a modular form of weight $k$ by $(\ref{Rexlength1})$,
\begin{equation} \label{ppinlength1} \int_{\tau}^{\tbp} \underline{\omega}(\tau) =   \int_{\tau}^{\infty} \underline{\omega}
^0(\tau)      - (2 \pi i)^{k-1}  \int^{\tau}_{0} a_0(\omega) (X- \tau Y)^{k-2}  \ . 
\end{equation}

In length 2, with $\omega_1, \omega_2 \in \M(\Gamma)$, formula $(\ref{EichlerConv})$ combined with $(\ref{Rexlength1})$, $(\ref{Rexlength2})$ gives the  following four    rapidly-convergent integrals, for any $\tau \in i\R^{>0}$:
\begin{multline}  \int_{\tau\leq \tau_1 \leq \tau_2 \leq \infty}   \underline{\omega_1}(\tau_1) \underline{\omega^0_2}(\tau_2) -  \underline{\omega^{\infty}_2}(\tau_1)  \underline{\omega^0_1}(\tau_2)  \\ 
-  \int_{\tau}^{\infty}  \underline{\omega^0_1}(\tau) \int_{0}^{\tau} \underline{\omega^{\infty}_2}(\tau)    +  \int_{0 \leq  \tau_2 \leq \tau_1 \leq \tau}   \underline{\omega^{\infty}_2}(\tau_2) \underline{\omega^{\infty}_1}(\tau_1) 
\end{multline}
 \end{example}
Because of the exponentially fast convergence of the integrals, these formulae lend themselves very well 
to numerical computations.

\begin{rem} Alternative approaches to the regularisation of iterated integrals of modular forms have been suggested  independently by Enriquez, Horozov, and Manin. However, the  essential point of using tangential base points is to ensure that the  regularised iterated integrals defined in that manner  are indeed the periods of the relative completions of the associated  fundamental groups (and in particular, defined over $\Q$).  
\end{rem}

 \section{The canonical holomorphic  $\Gamma$-cocyle} \label{sectcocyles}

\subsection{Definition} Let $I(\tau;\infty)$ denote the non-commutative generating series of iterated Eichler integrals 
defined in \S\ref{sectIteratedEichler}.

\begin{lem} For every $\gamma \in \Gamma$, there exists a series $\CC_{\gamma}\in \PiU(\C)$,  such that
\begin{equation} \label{Cgamdef}  I(\tau;\infty)= I(\gamma(\tau);\infty)\big|_{\gamma} \CC_{\gamma}  
\end{equation}
It does not depend on $\tau$. It satisfies the cocycle relation
\begin{equation}\label{Ccocy}
\CC_{gh} = \CC_g\big|_h \CC_h \qquad \hbox{ for all } g, h \in \Gamma \ .
\end{equation} 
\end{lem} 
\begin{proof}Let $\gamma \in \Gamma$.  It follows from the $\Gamma$-invariance of $\Omega(\tau)$ 
that   $ I(\tau; \infty)$ and    $ I(\gamma(\tau) ; \infty)|_{\gamma}$
are two solutions to the differential equation
$ { \partial \over \partial\tau  } L(\tau ) =-  \Omega(\tau )\, L(\tau) $
where $L(\tau) \in \PiU(\C)$. They therefore  differ by multiplication on the  right by a constant series $C_{\gamma} \in \PiU(\C)$ which does not depend on $\tau$.  The proof of $(\ref{Ccocy})$ is standard. Put $\gamma=g$ in $(\ref{Cgamdef})$, replace $\tau$ with $h(\tau)$, and act on the right by $h$. This  gives
 $$I(h(\tau);\infty)\big|_h = I(gh(\tau);\infty)\big|_{gh} \CC_g\big|_h\ .$$
Substituting this equation into $(\ref{Cgamdef})$ with $\gamma=h$ gives
 $$I(\tau;\infty) = I(gh(\tau);\infty)\big|_{gh} \CC_g\big|_h\CC_h\ .$$
The cocycle relation then follows from definition of $\CC_{gh}$.
\end{proof}
Equation  $(\ref{Ccocy})$ can be interpreted, via remark $\ref{remZ1isHom}$, as a homomorphism  of groups $\gamma \mapsto (\gamma, C_{\gamma}): \Gamma \rightarrow  \Gamma \ltimes \PiU(\C) $. 
\begin{defn}  Define the ring of holomorphic multiple modular values  $\mathcal{MMV}^{\hol}_{\Gamma}$ for $\Gamma$ to be the $\Q$-algebra generated by the coefficients of $(\ref{AlfXY})$  in 
$\CC_{\gamma}$ for all  $\gamma \in \Gamma$. 
\end{defn}
It is a subring of the ring of all periods of the relative completion of the fundamental group of $\M_{1,1}$. 
Setting $\tau = \gamma^{-1}(\infty)$ in equation $(\ref{Cgamdef})$ gives the formula
\begin{equation} \label{Cgammabetweencusps}
\CC_{\gamma} = I(\gamma^{-1}(\infty); \infty) \ .
\end{equation}
To make sense of this formula, one must define iterated
integrals $I(a;b)$ regularised with respect to two tangential base  points $a$ and $b$. But this follows easily from the previous construction
using  the formula $I(a;b) = I(\tau;a)^{-1} I(\tau;b)$, for  any $\tau \in \HH$. 
\subsection{Non-abelian cocycles} \label{sectNonabcocycles}
Let $G$ be a group, and let $A$ be a group with a right $G$-action. This means that
$ab|_g = a|_g b|_g$ for all $a,b\in A$ and $g\in G$, and 
$$a|_{gh} = (a|_g)|_h$$
for all $a\in A$, and $g,h\in G$.  The set of cocycles of $G$ in $A$ is defined by 
$$Z^1(G,A) = \{C: G \rightarrow A \hbox{ such that } C_{gh} = C_g\big|_h C_h  \hbox{ for all } g, h \in G \}\ .$$ 
Two such cocycles $C,C'$ differ by a coboundary if  there exists  a $B \in A$ such that
$$ C'_g = B^{-1}|_g C_g B$$
This defines an equivalence relation on cocycles, and the set of equivalence classes  is denoted by $H^1(G,A)$.
It has a  distinguished element $1: g \mapsto 1$. 

\begin{rem} \label{remZ1isHom} Let $\Hom_G(G, G\ltimes A)$ denote the set of  group homomorphisms 
from $G$ to $G\ltimes A$ whose composition with the  projection $G\ltimes A \rightarrow G$ is the identity on $G\rightarrow G$. As is well known, there is a canonical bijection 
\begin{eqnarray} Z^1(G,A)  &=  &\Hom_G(G, G\ltimes A)  \nonumber \\
z& \mapsto & (g \mapsto (g,z_g)) \nonumber 
\end{eqnarray}
\end{rem}

The canonical cocycle $\CC$  defines an element 
$$\CC \in Z^1(\Gamma; \PiU(\C))\ .$$
Since $\Gamma$ is generated by $S$ and $T$ (\S\ref{sectGammaST}), the cocycle $\CC$ is completely determined by $\CC_S$ and $\CC_T$. Since $i \in \HH$ is fixed by $S$, 
formula   $(\ref{Cgamdef})$ gives the following formula for $\CC_S$:
\begin{equation} \label{CinfSformula}
\CC_S  = I( i;\infty) \big|_S^{-1} I( i;\infty)  \ .
\end{equation}   
The series  $\CC_T$ will be computed  explicitly in the next paragraph. Its coefficients are rational multiples of powers of $2\pi i$. 
Therefore the ring $\mathcal{MMV}^{\hol}_{\Gamma}$ is generated by the coefficients of $\CC_S$ and $2  \pi i$.

\begin{rem} For every point $\tau_1 \in \HH$, one obtains a cocycle
 $C(\tau_1) \in  Z^1(\Gamma;\PiU(\C))$ defined by 
 $I(\tau ;\tau_1)= I(\gamma(\tau) ; \tau_1)|_{\gamma}  C_{\gamma}(\tau_1)$. 
 The  composition of paths formula for $I$  implies that the cocyles $C_{\gamma}(\tau_1)$, for varying $\tau_1$, define the same cohomology class
 $$[C_{\tau_1}]  \in H_1(\Gamma ;\U^{dR, \hol}_{1,1})\ .$$ 
Manin called this class  the non-commutative modular symbol in \cite{Ma2}. The cocycle $\CC_{\gamma}$ is a canonical representative of this class.
\end{rem}

\subsection{Equations} To simplify notations, let  $Z^1(\Gamma; \PiU)$ denote the functor on commutative unitary $\Q$-algebras 
$R \mapsto Z^1(\Gamma; \PiU(R))$.

\begin{lem} \label{lem3eq}
An element $C \in Z^1(\Gamma, \PiU)$ is uniquely determined by a pair   $C_S, C_T  \in \PiU$  
satisfying  the relations:
\begin{eqnarray}
1 &= & C_S\big|_S \,  C_S \nonumber \\
1 &= & C_U\big|_{U^2} \,  C_U\big|_U  \, C_U \nonumber 
\end{eqnarray}
where $C_U = C_T\big|_S \, C_S$.
\end{lem}
\begin{proof} 
 Since all modular forms for $\Gamma$ have even weight, it follows from the definition of $\PiU$ that 
the image of  the maps $(\ref{Smap})$ for any element of $\PiU$  have even weight ($-1$ acts trivially).  Therefore  $C_{-1}=1$ for any cocycle $C\in Z^1(\Gamma, \PiU)$ and thus
$$  Z^1(\Gamma/\{\pm 1\}, \PiU) \overset{\sim}{\To}    Z^1(\Gamma, \PiU)    \ .$$
  The left-hand side is $\mathrm{Hom}(\Gamma/ \{\pm 1\}, \Gamma / \{\pm 1\} \ltimes \PiU)$ by   remark \ref{remZ1isHom}.
It is well-known (\S\ref{sectGammaST}) that  $\Gamma/\{\pm 1\} = \langle S, T, U :   U=TS, U^3=S^2=1 \rangle$, so such a homomorphism
is defined by the above equations.   A computational proof was  given  in \cite{Ma2}, \S1.2.1.
   \end{proof}  

These equations can be made more explicit by the following observation.
Consider  $C \in Z^1(\Gamma, \PiU(R))$. Since $C_{\gamma} \in \PiU(R)$, its leading term is $1$, and we can define 
$$C' : \Gamma \To R\langle \langle M^{\vee} \rangle \rangle$$
by the equation $C'=C-1$. The element $C'$ satisfies 
$$C'_{gh} - C'_g \big|_h - C'_h = C'_g\big|_h C'_h$$
for all $g, h\in \Gamma$.  If we interpret  $C_{\gamma}$ as a morphism via $(\ref{Smap})$, we can write the previous equation for all $n \geq 1$,  as
a system of cochain equations (\`a la Massey)
\begin{equation}  \label{phiMassey}
 \delta C_{\al_1  \ldots  \al_{n}}   =   \sum_{i=1}^{n-1} C_{\al_1  \ldots \al_{i}}   \cup  C_{\al_{i+1}  \ldots  \al_{n}}   \ ,
 \end{equation} 
where $\al_i \in M$ and  where $\delta^1(C) (g,h) = C_{gh}  - C_g\big|_h - C_h$ and $(A \cup B) (g,h) = A_g|_h \otimes  B_h$
are the coboundary and cup product  for $\Gamma$-cochains (see \S\ref{sectGroupCohomReminders}, \S\ref{sectCupproducts}).

\begin{caveat} The conditions for $C$, viewed as a series of polynomials $(\ref{Smap})$, to be a cocycle are equivalent
 to the shuffle equation $(\ref{phishuffle})$, together with the equations $(\ref{phiMassey})$ evaluated at the pairs 
$(S,S)$, $(T, S)$, $(U,U^2)$  by lemma \ref{lem3eq}.  They are unobstructed in the sense that they can be solved recursively in the length: the $C_{\al_1}$
are ordinary abelian cocyles, and so on. This is because $\Gamma$ has cohomological dimension 1.

However, we will need to  constrain the value of $C_T$ which leads to non-trivial obstructions to solving $(\ref{phiMassey})$. These obstructions are the object of study of \S\ref{sectTranference}.

\end{caveat}

\subsection{Complex conjugation}  \label{sectRealStructure}
Consider the matrix 
\begin{equation}\label{epsilondef}
\epsilon  = 
\left(
\begin{array}{cc}
 1 &  0     \\
  0 &  -1   
\end{array}
\right)\ .
\end{equation}
It acts on the right on $V_{\infty}$ via $(X,Y) \mapsto (X,-Y)$  and acts diagonally  on $T(V_{\infty})$. 
It defines an involution on $\PiU$ by acting trivially on the elements $\Al_f$.

\begin{lem} Let $\CC$ denote the canonical cocycle. Then
 \begin{equation} \label{FinfonCC}
 \overline{\CC}_{\gamma} = \CC_{\epsilon \gamma \epsilon^{-1}}\big|_{\epsilon} \ .
 \end{equation} 
In particular,   $\overline{\CC}_S= \CC_S\big|_{\epsilon} $. 
\end{lem}
\begin{proof} Since the local analytic coordinate near the cusp is $q = e^{2 i \pi \tau}$, complex conjugation acts upon $\M^{an}_{1,1}\cup_{\Phi} T^*_{p}$ via the  map $\tau \mapsto - \overline{\tau}$ on $\HH\cup_{i_{\infty}} \C$.  It  satisfies
$$-\overline{\gamma(\tau)} = \epsilon \gamma \epsilon^{-1}(-\overline{\tau})$$ 
for all $\tau \in \HH, \gamma \in \Gamma$.  Therefore the induced action  on $\Gamma=\pi_1(\M_{1,1}^{an},\tbp)$ is  by conjugation by $\epsilon$.  A similar formula holds for  $\tau\in \C$ in the tangent space at the cusp, and $\gamma \in \Gamma_{\infty}$. 
Now  let $f \in \M(\Gamma)$ be a modular form with rational (and in particular,  real)  Fourier coefficients. Then it follows from the definition $(\ref{underlinefdefinition})$
that 
$$ \underline{f}(-\overline{\tau}) = \overline{\underline{f}(\tau)}\big|_{\epsilon} \ .$$
There is a similar equation on replacing $\underline{f}$ with $\underline{f}^{\infty}$. Thus the action of  complex conjugation  on differential forms $\underline{f}(\tau)$ is by right action by $\epsilon$, and taking the complex conjugate of coefficients.
We deduce from   its definition as an iterated integral that $\overline{I(\tau; \infty)} = I(-\overline{\tau} ;\infty)\big|_{\epsilon}$.  Therefore
by $(\ref{Cgammabetweencusps})$,
$$\overline{\CC}_{\gamma} = \overline{I(\gamma^{-1}(\infty);\infty)}\big|_{\epsilon} = I(-\overline{\gamma^{-1} (\infty)}; \infty) \big|_{\epsilon} = I(\epsilon \gamma^{-1} \epsilon^{-1} (\infty); \infty)\big|_{\epsilon} = \CC_{\epsilon \gamma \epsilon^{-1}} \big|_{\epsilon}$$
For the last part, observe that $\epsilon S \epsilon^{-1} = -S$. Since  $\CC_{-1} =1$, we deduce from the cocycle equations that  $\CC_{-S}= \CC_S$.  \end{proof}

If $F_{\infty}$ denotes the real Frobenius involution, we have shown that $F_{\infty}$ acts on $\Gamma$ by conjugation by $\epsilon$, and acts on $V_{2n}$ (which is the Betti version of $V_{2n}^{dR}$ to be introduced later), via right action by $\epsilon$.

Finally, observe that
\begin{equation} \label{partialkandepsilon} 
\partial^k (\epsilon \otimes \epsilon) = (-1)^k \epsilon  \, \partial^k\ .
\end{equation} 
which  follows immediately from the definition of $\partial^k$,  $\S\ref{sectdeltakdef}$. It follows that  
$\langle \  , \  \rangle $ is equivariant with respect to $\epsilon$.

\section{Cocycle at the cusp} \label{sectCocyatT}
It is straightforward to compute the image of the canonical cocycle $\CC$ under the map
\begin{equation}\label{resZ1map} 
Z^1(\Gamma; \PiU) \To Z^1(\Gamma_{\infty} ; \PiU)\ .
\end{equation}

\subsection{Local monodromy} 
 Since $\Gamma_{\infty}$ is generated by $-1$ and $T$, and $\CC_{-1}=1$, the image of $\CC$ under  $(\ref{resZ1map})$ is determined 
 by  $\CC_T$.
\begin{lem}We have the following formula  for $\CC_{T}$:
\begin{equation}  \label{Cinfequation} 
 \CC_{T} =  I^{\infty}(-1; 0 )  \ .
\end{equation} 
In particular, $\CC_T$ has coefficients in $\Q[2 \pi i]$ (see below for an explicit formula).
\end{lem}

\begin{proof}  Set $\gamma = T$ in  $(\ref{Cgamdef})$ and $\tau = T^{-1} \tbp$, to obtain $\CC_{T} = I ( T^{-1} {\tbp} ; \tbp)$. On the universal covering $\HH \cup_{i_{\infty}} \C$ this is simply the path from 
$-1$ to $0$ on  the tangent space $\C$. Formula $(\ref{Cinfequation})$ is immediate from the discussion of \S\ref{sectIteratedEichler}, as $I$ restricts to $I^{\infty}$ on the tangent space. 
The second statement  follows from the observation that the coefficients of $\Omega^{\infty}(\tau)$ 
are given by the zeroth Fourier coefficients of Eisenstein series (multiplied by a power of $2 \pi i$). By  \S\ref{secFourierexp}, the latter are  rational. 
\end{proof}
\begin{rem} In Part II, the map $ (\ref{resZ1map})$ will be interpreted as a local monodromy
$$ \pi_1( T^{\times}_{p}, 1) \To \pi_1(  \M^{an}_{1,1}, \tbp) \ , $$
corresponding to the inclusion of $\Gamma_{\infty}$ into $\Gamma$.  
The coefficients of $\CC_T$  are   periods of the unipotent fundamental group of $T^{\times}_{p} \cong \G_m$, which are in $\Q[2 \pi i]$. \end{rem}

If we view $\CC_T\in \PiU(\C)$ as a linear map from a sequence of modular forms to polynomials via $(\ref{Smap})$, then it follows from the above discussion that 
\begin{equation} \label{CTzeroforcusp}
\CC_T(\al_{f_1}  \ldots  \al_{f_n}) = 0
\end{equation}
whenever any $f_i$ is a cusp form (since $\underline{f}_i^{\infty}$ vanishes in that case).  The only non-zero contributions
to $\CC_T$ come from iterated integrals of Eisenstein series.

\subsection{Formula for $\CC_T$} \label{sectCTformula}
In order to write down $\CC_T$ it is convenient to rescale the Eisenstein series as follows.
By comparing with \S\ref{secFourierexp}, we define normalized letters 
$$  \Eb_{2k} = {1 \over (2 \pi i)^{2k-1}} { - 4 k\over   \Be_{2k} (2k-2) !} \, \EE_{2k}\ , \hbox{ for } k \geq 2\ .$$
The rational factor is chosen so that in this alphabet, 
\begin{equation}\label{OmegainfnewEs} \Omega^{\infty}(\tau)  = \sum_{k\geq 2} {\Eb_{2k} \over (2k-2)!} (X-Y\tau )^{2k-2}\ .
\end{equation}
With this choice of normalisation, we can write down the cocycle explicitly as follows.
\begin{lem} \label{propexpformula} The coefficient of  $\,\Eb_{2k_1} \ldots \Eb_{2k_n}$ in $\CC_T $  is equal to  the coefficient of 
$s_1^{2 k_1-2} \ldots s_n^{2 k_n-2}$ in the commutative generating series
\begin{equation} \label{Exponentialformula} 
 e^{ s_1 X_1 + \ldots + s_n X_n} \Big( \sum_{i=0}^{n} { (-1)^{n-i} \over \pi^L(s_1Y_1, \ldots, s_{i} Y_{i})  } { e^{s_{1} Y_{1} + \ldots +  s_{i} Y_{i}}  \over
\pi^R(s_{i+1}Y_{i+1}, \ldots, s_n Y_n) }\Big)\ .
\end{equation}
\end{lem}
\begin{proof} See proof of proposition \ref{propTrivV} below. \end{proof} 
\noindent Here we use the  notation `pile up on the left or right':
\begin{eqnarray}
\pi^L(z_1, \ldots, z_n ) &  = & (z_1 +\ldots + z_n)   \cdots   (z_{n-1}+z_n)  z_n \nonumber \\
\pi^R(z_1, \ldots, z_n ) &  =  & z_1 (  z_1+ z_2)\cdots (z_1 + \ldots +z_n) \nonumber 
\end{eqnarray}
For clarity,  formula $(\ref{Exponentialformula})$ in lengths $1$ and $2$, and with $s_1=s_2=1$  reads
$$  e^{ X_1}\Big( {e^{Y_1} \over Y_1} - {1 \over Y_1}  \Big) \qquad \hbox{ and} \qquad 
 e^{ X_1+ X_2}\Big( {e^{Y_1+Y_2} \over(Y_1 + Y_2)Y_2 } - {e^{Y_1}  \over Y_1 Y_2 } + {1 \over Y_1(Y_1+Y_2) }   \Big)
 $$
Note that $(\ref{Exponentialformula})$, despite appearances, has no poles. It is  clearly  defined over $\Q$.

\subsection{Trivialisation}  We can formally trivialise the restriction   of
$\CC $ to  $Z^1(\Gamma_{\infty}, \PiU(\C))$
by enlarging the space of coefficients in the following way. By $(\ref{Smap})$, we can regard  $\CC_T$ as a map from sequences
of modular forms into the space $T(V_{\infty})$ of polynomials in infinitely many variables $X_i,Y_i$. Enlarge it by defining
$$ \widehat{ T(V_{\infty})}  = \Q  [ X_1, Y_1, X_2, Y_2,  \ldots ]\big[ {1\over Y_1}, {1 \over Y_1+Y_2}, \ldots , {1 \over Y_{i}  + \ldots + Y_{i+r}}\big]$$
to be the space  of polynomials in $X_i, Y_i$ with denominators  in  $Y_i+Y_{i+1} +\ldots + Y_{i+r}$.
Since the elements $Y_i$ are fixed by  $T$, this space  inherits  an action of $\Gamma_{\infty}$ by \S$\ref{sectGammaST}$.
 
 \begin{prop} \label{propTrivV}  There exists a series   $ \V\in \widehat{ T(V_{\infty})} \langle \langle \Eb_{2n} \rangle \rangle$  which  trivialises $\CC_T$, i.e.,
\begin{equation} \label{CTasV} \CC_T = \V|_T \, \V^{-1}\  . 
\end{equation}
It is not unique. A   representative is given by the series   whose 
 coefficient of $\Eb_{2k_1} \ldots \Eb_{2k_n}$ is  the coefficient of 
$s_1^{2 k_1-2} \ldots s_n^{2 k_n-2}$ in the commutative generating series 
\begin{equation} \label{ExponentialformulaforV} 
 v(s_1,\ldots, s_n)= { e^{ s_1 X_1 + \ldots + s_n X_n}   \over  (s_1Y _1+\ldots + s_nY_n) \ldots (s_{n-1} Y_{n-1}+s_nY_n) s_n Y_n } 
\end{equation}
expanded in the sector $0 \ll s_1 \ll \ldots \ll s_n $.
\end{prop}
\begin{proof} 
By $(\ref{Cgamdef})$, restricted to the tangent space $\C$ of $\HH \cup_{i_{\infty}} \C$, we have
$$   I^{\infty}(\tau ; 0 )  =   I^{\infty}(\tau+1 ; 0 )  \big|_{T} \CC_{T} \ . $$
Since $\CC_{T}$ does not depend on $\tau$, we can  set $\V = \lim_{\tau \rightarrow \infty}  I^{\infty}(\tau ; 0 )^{-1}$ to be a regularised limit.  For this, consider the part of the series  $ I^{\infty}(0 ;\tau)$ in length  $n$, and 
view it as a commutative formal power series by replacing the words $\Eb_{2k_1} \ldots \Eb_{2k_n}$ with
$s_1^{2 k_1-2} \ldots s_n^{2 k_n-2}$, for $r\leq n$. Since $ I^{\infty}( 0;\tau )$ is the iterated integral of 
 $\Omega^{\infty}(\tau)$ by $(\ref{Iinfdef})$, the coefficients of $I^{\infty}( 0;\tau )$ are represented via $(\ref{OmegainfnewEs})$ by  
$$   \int_{0}^{\tau}   [ e^{ (X_1- \tau Y_1)s_1}  d\tau | \ldots |  e^{  (X_n - \tau Y_n )s_n}   d\tau ] $$
Take a regularised limit as  $\tau \rightarrow \infty$  by thinking of   $Y_n$ as positive real numbers. This gives, by the reversal of paths formula, 
\begin{multline} (-1)^n  \lim_{\tau \rightarrow \infty} \int_{\tau}^0   [ e^{ (X_n- \tau Y_n)s_n}  d\tau | \ldots |  e^{  (X_1 - \tau Y_1 )s_1}   d\tau ]   \nonumber \\
 = { e^{s_n X_n}  \over s_n Y_n } (-1)^{n-1}    \lim_{\tau \rightarrow \infty} \int_{\tau}^0   [ e^{ (X_n- \tau Y_n)s_n+(X_{n-1}- \tau Y_{n-1})s_{n-1}}  d\tau | \ldots |  e^{  (X_1 - \tau Y_1 )s_1}   d\tau ]  \
\end{multline}
which yields 
  $(\ref{ExponentialformulaforV})$ by induction. 
From this we deduce that $(\ref{CTasV})$ holds, as a function of the parameters $s_i$.  
Translating $(\ref{CTasV})$ into commutative generating series in the $s_i$ leads to the following formula for the coefficients of $\CC_{T}$:
$$\sum_{i=0}^n (-1)^{n-i} v(s_1,\ldots, s_i)\big|_T v(s_{n}, \ldots, s_{i+1}) $$
 This gives exactly   $(\ref{Exponentialformula})$.  
\end{proof}

 Expanding $(\ref{ExponentialformulaforV})$ in a different sector gives rise to a different choice of trivialisation for the restriction of 
 $\CC$ to $\Gamma_{\infty}$. However, after projecting 
 $$\pi_d: \widehat{ T(V_{\infty})}   \rightarrow \Q  [ X, Y, {1\over Y}] $$
by sending $(X_i, Y_i)$ to $(X, Y)$, we obtain a canonical  trivialisation 
 from $(\ref{ExponentialformulaforV})$ 
$$ { e^{(s_1+\ldots +s_n)X} \over  Y^n (s_1+\ldots +s_n) \ldots  (s_{n-1}+s_n) s_n   }  $$
which can be uniquely  expanded as a Laurent power series in  the $s_i$ (in any sector).

\begin{rem}
 Zagier's  `extended period polynomials'
for Eisenstein series are  the  coefficients of $\EE_{2k}$  in the  cocycle $\gamma\mapsto \V|^{-1}_{\gamma} \CC_{\gamma} \V$  (viewed 
as a cocycle whose coefficients are in the field of rational functions in $X_i,Y_i$) applied to $\gamma = S$. In other words, by modifying 
the cocycle of the Eisenstein series by adding a coboundary with 
poles in $Y_i$, he forces  it to vanish at $T$.  \end{rem} 
 
 \begin{rem}
 A different approach to computing the local monodromy will be discussed in the second part of this paper. It will follow from lemma \ref{lemNdRto1storder} that
 $$(T, \CC_T ) = \exp \big(\epsilon_0^{\vee} , \sum_{n \geq 1}  {\Be_{2n+2}  \over 4n+4} X^{2n} \big)$$ 
computed in the semi-direct product $\SL_2 \ltimes \PiU$.    \end{rem}
 A recurring circle of ideas in this paper is that the existence of poles in $\mathcal{V}$, the non-vanishing of $H^2(\Gamma,\Gamma_{\infty};\Q)$, the transference principle (\S\ref{sectTranference}),  and the inertial condition at the cusp 
 (\S\ref{sectConstraints})   are all, more or less, equivalent. Computing the poles in $\mathcal{V}$ gives an alternative method for studying this constraint on the  structure of the Galois group $\A^{dR}_{\U}$, but that we shall not pursue any further here. 
 
 \section{The abelianised cocycle and the  Eichler-Shimura theorem} \label{sectES}
We  compute the image of the canonical cocycle $\CC$ under the map
$$Z^1(\Gamma; \PiU(\C)) \To Z^1(\Gamma; (\PiU)^{ab}(\C))\ . $$
The results of this section are  well-known, but are recalled here for convenience.

\subsection{Abelianization of $\CC$}
For any commutative $\Q$-algebra $R$ we have $\S\ref{sectAbelian}$
$$(\PiU)^{ab}(R)   \cong \Hom(M, R) =  \prod_{k} M_{2k+2}^{\vee} \otimes V_{2k} \otimes R    \ . $$
The natural map  $\PiU \rightarrow (\PiU)^{ab}$  therefore induces a map
$$Z^1(\Gamma; \PiU)  \To Z^1(\Gamma; (\PiU)^{ab}) \cong \prod_k M^{\vee}_{2k+2} \otimes  Z^1(\Gamma; V_{2k}) \  .$$
This can be written
$$Z^1(\Gamma; \PiU)  \To \prod_k \Hom(M_{2k+2}, Z^1(\Gamma; V_{2k}))\  .$$
In particular, for     $f \in \B_{2k+2}$, the coefficient of $\Al_f$ in $\CC$, which is denoted by 
$\CC(\al_f) $  (see $(\ref{Smap})$), is a $\Gamma$-cocycle in $V_{2k}$.  
The canonical cocycle therefore  defines a linear map
$$\pp: \M_{2k+2}(\Gamma)\otimes \C \To Z^1(\Gamma; V_{2k})\otimes \C$$
which we call $\pp$ for period. It is the abelianization of  the canonical cocycle $\CC$. 
Explicit formulae  for $\pp$  are obtained from $(\ref{CinfSformula})$  and $(\ref{ppinlength1})$. 
\subsection{Periods of cusp forms}  \label{sectPeriodsofscusp}
 For   any cusp form   $f\in S_{2k+2}(\Gamma)$ of weight $2k+2$,
\begin{eqnarray}
\pp(f)_T & = 0&  \nonumber  \\
\pp(f)_S  & = &   (2  \pi i)^{2k+1}  \int_{0}^{i \infty}  f(\tau) (X- \tau Y)^{2k} d\tau \ . \nonumber
\end{eqnarray}
A binomial expansion of the second equation and $(\ref{Mellin})$ yields the formula
$$ \pp(f)_S    =  (2\pi i)^{2k+1} \sum_{r=1}^{2k+1}  (-i)^{r} \binom{2k}{r-1}\Lambda (f,r) X^{2k+1-r} Y^{r-1} \ . $$
In particular,  the numbers  $ (2  \pi i )^{2k+1}  i^{r} \Lambda(f,r)$ are in $\MMV^{\hol}_{\Gamma}$ for all values of $r$ inside the critical strip $1\leq r \leq 2k+1$.
If $f$ is a normalised Hecke eigenform, Manin showed \cite{Ma0} that there exist two  numbers 
$$\omega^+_f   \in \R \   , \    \omega^-_f \in i \R\ ,$$ 
 called the periods of $f$, such that 
$$  \pp(f)_S = (2 i \pi)^{2k+1} \big( \omega^+_f P_{f}^+(X,Y) +  \omega^-_f P_f^-(X,Y)\big) $$
where $P_{f}^{\pm}(X,Y) \in V^{\pm}_{2k}\otimes K_f$,   $K_f$ is the number field generated by the Fourier coefficients of $f$,
and $\pm$ denotes the (anti)-invariants  with respect to $\varepsilon$. One can normalise the polynomials $P_f$ in such a way 
that    $ \sigma(P^{\pm}_f ) = P^{\pm}_{\sigma(f)}$ for all $\sigma \in\mathrm{Aut}_{\Q}(K_f)$.

\subsection{Period polynomials} \label{sectPeriodPolys}
The cocycle conditions for $c\in Z^1(\Gamma; V_{2k})$ are:
\begin{eqnarray}
c_S\big|_S + c_S & = & 0   \label{cocycleabelianrel}  \\
c_U \big|_{U^2} + c_U \big|_{U} + c_U  &=  & 0 \ , \nonumber 
\end{eqnarray} 
where $c_U= c_T \big|_S + c_S$. If $c_T$ vanishes, then $c_U = c_S$, and these two equations translate into the pair of equations for $c_S= P(X,Y)$: 
\begin{eqnarray}
P(X,Y) + P(-Y,X) & = & 0 \label{ppequations} \\
P(X,Y) + P(X-Y,X) +P(-Y,X-Y)  & = & 0 \ .\nonumber 
\end{eqnarray} 
The space $W_{2k} \subset V_{2k}$ of solutions to these equations is called the space of period polynomials, and the right action of $\epsilon$ decomposes it into a sum of two eigenspaces $W_{2k}^{\pm}$. 
Let  $k\geq 2$, and let
\begin{equation}\label{Zcusp} Z_{\cusp}^1(\Gamma; V_{2k}) = \ker( Z^1(\Gamma;V_{2k}) \To Z^1(\Gamma_{\infty};V_{2k})) 
\end{equation} 
denote the subspace of cuspidal cocycles.  There is  an isomorphism 
\begin{equation}  \label{Z1toW} c \mapsto c_S: Z_{\cusp}^1(\Gamma; V_{2k} )  \overset{\sim}{\rightarrow} W_{2k}. 
\end{equation} 
In particular,  $P_f^{\pm}(X,Y)$ lies in the subspace $W_{2k}^{\pm}$  if $f$ has weight $2k+2$. 

\subsection{Periods of  Eisenstein series} \label{sectRatEis}
  Let
 \begin{equation} \label{cxdefn} 
 c(x) = {1 \over e^x-1} +{1\over 2}- {1 \over x}\ .
 \end{equation}
 Define a set of rational cocycles $ e_{2k}^0 \in Z^1(\Gamma; V_{2k-2})$ via their generating series
  $$e^0=\sum_{k\geq 2} {2 \over (2k-2)!} e_{2k}^0$$
where $e^0$ is the unique cocycle  in $V_{\infty}$ defined on $\Gamma$  by 
 \begin{eqnarray} \label{e^0defn}
 e^0(S)  &=&   c(X)c(Y) \\ 
  e^0(T)  &= & \textstyle{1\over Y}(c(X+Y)-c(X))  - {1 \over 12}  \ .\nonumber
 \end{eqnarray}
One  verifies that the $e_{2k}^0$ do indeed satisfy the  abelian cocycle relations $(\ref{cocycleabelianrel})$ using  the  following well-known functional equation for  $b(x) = c(x) + {1\over x}$:
$$   b(x_1)b(x_2) - b(x_1)b(x_2-x_1) +b(x_2) b(x_2-x_1)={1 \over 4} \ . 
  $$
 This cocycle  is given explicitly for $k\geq 2$ by 
  \begin{eqnarray} \label{e2k0S} e_{2k}^0(S)   &= &  {(2k-2)! \over 2} \, \sum_{i=1}^{k-1}    {\Be_{2i} \over (2i)!}{\Be_{2k-2i}  \over (2k-2i)!}  X^{2i-1} Y^{2k-2i-1} \ ,   \\
  e_{2k}^0(T)   &= &  {(2k-2)! \over 2}  {\Be_{2k} \over (2k)!} \Big(  { (X+Y)^{2k-1} - X^{2k-1}   \over Y}\Big)\ .  \nonumber 
  \end{eqnarray}  
The  following lemma is probably equivalent to  facts which are essentially well-known to experts, but I could not find 
the precise statement in the literature.

 \begin{lem}   \label{propE2k} The  cocycles  of Eisenstein series are
$$ \pp(E_{2k}) = (2\pi i )^{2k-1} e_{2k}^0  + {(2k-2)!\over 2} \zeta(2k-1)   \delta^0 (Y^{2k-2})\ , $$
where  $\delta^0$ is the boundary \S\ref{sectGroupCohomReminders} and $k\geq 2$. The coboundary term
 $\delta^0 (Y^{2k-2})$ is the cocycle which sends $T$ to $0$ and $S$ to $X^{2k-2}- Y^{2k-2}$. 
\end{lem} 
\begin{proof} For any $f\in \M_{2k}(\Gamma)$, the value of the cocycle $\CC^{ab}(f)$ on $S$ is given by:
 $$\CC^{ab}(f)_S = (2 \pi i)^{2k-1} \Big( \int_{i}^{\tbp} f(\tau)(X-\tau Y)^{2k-2} d\tau  \Big)\Big|_{S-1}$$
by $(\ref{CinfSformula})$. Now  equation $(\ref{Lambdafsformula})$ implies via  $(\ref{ppinlength1})$ that
$$ i^r \Lambda(f,r) = \int_i^{\tbp} f(\tau) \tau^{r-1} {d\tau}  -  (-1)^r \int_i^{\tbp} f(\tau) \tau^{2k-r-1} {d\tau }$$
for any integer $1 \leq r \leq 2k-1$. 
  Now expand  $(X-\tau Y)^{2k-2}$ in the first equation  using the binomial formula, and identify  the  terms  with values of the completed $L$-function of $f$ using the previous equation. 
This gives
$$   \CC^{ab}(f)_S =  -  (2\pi i)^{2k-1} \sum_{r=1}^{2k-1}  i^{r} \binom{2k-2}{r-1}\Lambda (f,r) X^{2k-1-r} Y^{r-1} \ .$$
By $\S\ref{sectLvalues}$ we have in the case $f= E_{2k}$,
$$ \Lambda(E_{2k}, r) = (2 \pi)^{-r} \Gamma(r) \zeta(r) \zeta(r-2k+1) \ .$$
This vanishes for odd $3 \leq r \leq 2k-3$. For even $r$, this produces the product of Bernoulli 
numbers in $(\ref{e2k0S})$ by  Euler's formula  $\S\ref{sectLvalues}$:
$$ i^{2a} \Lambda(E_{2k}, 2a) = - {1 \over 2} {\Be_{2a} \over 2a} {\Be_{2k-2a} \over 2k-2a}\ .$$
 Finally, at $r=2k-1$ it gives
$$ \Lambda(E_{2k}, 2k-1) = (2 \pi)^{1-2k}  (2k-2)! \zeta(2k-1) \zeta(0) \ , $$
which exactly produces the  coefficient of $Y^{2k-1}$ since $\zeta(0) = - {1 \over 2}$.  The case $r=1$ (or coefficient of $X^{2k-1}$) can be deduced from the first equation of 
$(\ref{cocycleabelianrel})$, or from the functional equation of $\Lambda$. 
 The value of $\pp(E_{2k})$  on $T$ follows from \S \ref{sectCTformula}.
\end{proof} 
The coefficients of the cocycle $\pp(E_{2k})$ lie in $\zeta(2k-1) \Q  +  (2 \pi i)^{2k-1} \Q$.
We have
\begin{eqnarray}  \label{RatEismap}
[\pp]: \Eis_{2k}(\Gamma) &  \To  & H^1(\Gamma; V_{2k-2}) \otimes (2 \pi i )^{2k-1}\Q\ . \\
E_{2k} & \mapsto &  (2i \pi)^{2k-1}[e^0_{2k}] \nonumber
\end{eqnarray} 
The cohomology class of the Eisenstein cocycle is rational up to a power of $2 \pi i$, although the cocycle itself is not, due to the 
presence of the odd zeta value.  This simple observation has far-reaching consequences (e.g. $(\ref{DecompofEis})$).

\subsection{Eichler-Shimura isomorphism}   \label{sectESisom}
 We have $H^0(\Gamma; V_{\infty}) =V_{\infty}^{\Gamma}=  \Q$, and  furthermore,  $\Gamma$ is of virtual cohomological dimension $1$ since $\M^{an}_{1,1}(\C) =  \Gamma \backslash \!\! \backslash  \HH$ is  of real dimension 2 and non-compact, so  $H^i(\Gamma; V_n)$ vanishes for all $i\geq 2$. The group $H^1(\Gamma; V_n)$ is described by the 
 Eichler-Shimura isomorphism.
  
  By $(\ref{FinfonCC})$ the action of complex conjugation on coefficients is equivalent to    right action by  $\epsilon$  on $V_n$, and conjugation by  $\epsilon$ on the group  $\Gamma$. 
  This defines the following action on   cochains:
  \begin{eqnarray}
  C^i(\Gamma; V_n) &  \To &  C^i(\Gamma;V_n) \\
  \phi  &\mapsto&  \big( \, (g_1,\ldots, g_n) \mapsto \phi ( \epsilon g_1 \epsilon^{-1}, \ldots, \epsilon g_n \epsilon^{-1})\big|_{\epsilon}\big) \nonumber
  \end{eqnarray} 
  It is a morphism of complexes, and therefore induces an action on cohomology. Denote the eigenspaces 
  of  $H^1(\Gamma ; V_n) $  and $Z^1(\Gamma; V_n)$ for this action  by $\pm$. Thus elements of 
$Z^1(\Gamma ; V_n)^{\pm}$  can be represented by cocycles satisfying
  $$   C_{\epsilon \gamma \epsilon^{-1} }  \big|_{\epsilon}=  \pm C_{\gamma} \ .$$
For example,  $C_S|_{\varepsilon} = C_S$ if and only if $C_S$ is even in   $Y$ (an even period polynomial)
 and $C_S|_{\varepsilon} = -C_S$ if and only if $C_S$ is odd in $Y$, hence
 $$C\mapsto C_S: Z^1_{\cusp}(\Gamma; V_{2n}) \overset{\sim}{\To} W_{2n}^{\pm}\ .$$
\begin{thm}(Eichler-Shimura) For all  $n \geq 2$, integration defines isomorphisms
$$\begin{array}{ccccc} 
   [ \pp^+] & :  &  S_{2n+2}(\Gamma) &  \overset{\sim}{\To }  & H^1(\Gamma ; V_{2n})^{+} \otimes \R  \ ,     \\
  {[}\pp^-] & : &  \M_{2n+2}(\Gamma) &  \overset{\sim}{\To }  &  H^1(\Gamma ; V_{2n} )^{-}\otimes \R \ .   \nonumber  
 \end{array} 
$$
where $\pp^+ = \Real\, \pp$ and  $\pp^-= \Image\,  \pp$. 
In particular, for all $n\geq 2$ 
 $$\dim_{\Q} H^1(\Gamma ; V_{2n}) = \dim_{\Q} \Eis_{2n+2}(\Gamma)+ 2\,  \dim_{\Q} \Ss_{2n+2}(\Gamma)\ .$$ 
 \end{thm} 

The restriction map induced from the inclusion $i$ of  $\Gamma_{\infty}$ in $\Gamma$ is
$$i^*: H^1(\Gamma;V_{2n}) \rightarrow H^1(\Gamma_{\infty};V_{2n})$$ 
Denote the kernel of this map by $H^1_{\cusp} (\Gamma;V_{2n})  \subset  H^1(\Gamma;V_{2n})$.

\subsection{Hecke-equivariant splitting} \label{sectHeckeinvs}
  The subspace of coboundaries in $Z^1_{\cusp}( \Gamma; V_{2k})$ is generated by $\delta^0 v$, where $ v\in V_{2k}$, 
 such that $\delta^0 v(T)= v\big|_T -v=0$. This is   one-dimensional, 
spanned by $\delta^0 Y^{2k}$ by  $(\ref{Tlongexact})$.
 Since the cocycle of a cusp form vanishes on $T$,  we have 
  $$\pp^{\pm}: S_{2k+2}(\Gamma) \To Z^1_{\cusp}(\Gamma, V_{2k})^{\pm} \otimes \R \To  H^1_{\cusp}(\Gamma, V_{2k})^{\pm}\otimes \R\ .$$
Manin defined \cite{Ma0} the action of Hecke operators onto $ Z^1_{\cusp}(\Gamma, V_{2k} )^{\pm}$  and proved that $\pp^{\pm}$
commutes with  this action.  Linear algebra implies the following lemma.

\begin{lem}  \label{lemsplitting} There is a canonical splitting over $\Q$
\begin{equation} \label{sHeckesplitting} s: H_{\cusp}^1(\Gamma;V_{2k}) \rightarrow Z_{\cusp}^1(\Gamma;V_{2k})
\end{equation} 
which is equivariant for the action of Hecke operators.
We have
$$ Z^1_{\cusp}( \Gamma; V_{2k}) =   \delta^0 Y^{2k}  \Q \oplus s(H_{\cusp}^1(\Gamma;V_{2k})) \ .$$
\end{lem} 
\begin{proof}  The map $s$ can be written explicitly by noting that the space $s(H_{\cusp}^1(\Gamma;V_{2k}))$
is orthogonal to the space of Eisenstein cocycles $e^0_{2k+2}$ with respect to the inner product  $\{\ , \}$ defined in  $(\ref{curlypairing})$, which is equivariant for the 
action of Hecke operators \cite{KZ, PaPo}.   Since a cuspidal cocycle $C$ (or its cohomology class) is uniquely determined by the  polynomial $C_S \in V_{2k}$,  we can simply define $s(C)_T=0$ and 
$$s(C)_S = C_S + \alpha (X^{2k}-Y^{2k}) $$
where $\alpha$ is determined by $\{e^0_{2k+2}, C_S\} + \alpha \{e^0_{2k+2},  \delta^0 Y^{2k} \}=0$. This can be solved for $\alpha$ 
since the coefficient of $\alpha$ is invertible, by the following lemma.
 \end{proof} 

\begin{lem}  Let $e^0_{2k}$ denote the rational cocycle defined above. Then
\begin{equation} \label{e0cuppedwithboundary} 
\{ e^0_{2k}, \delta^0 Y^{2k-2}\}=   { 3  \Be_{2k}  \over 2k} \quad  \hbox{ for } \quad k \geq 2\ .
\end{equation} 
\end{lem} 
\begin{proof}
Applying  definition $(\ref{curlypairing})$ gives
$$\langle e^0_{2k}(S), (X+Y)^{2k-2} - (X-Y)^{2k-2}\rangle -2 \langle e_{2k}^0(T),  (X^{2k-2} -Y^{2k-2})\big|_{1+T} \rangle $$
We deduce from the  definition $(\ref{e^0defn})$ that    $e_{2k}^0(T)|_{T^{-1}} = e_{2k}^0(T)|_{\epsilon}$. Using the $\Gamma$ and $\epsilon$-invariance of $\langle \ ,\  \rangle$, the previous expression  
becomes
$$\langle e^0_{2k}(S), (X+Y)^{2k-2}  - (X-Y)^{2k-2}\rangle  \ - \ 4 \langle e_{2k}^0(T),  X^{2k-2} -Y^{2k-2} \rangle $$
Replacing $e_{2k}^0$ with its generating series $(\ref{e^0defn})$, and applying   $(\ref{PIPformula})$, the previous quantity  is  ${(2k-2)! \over 2}$ times the coefficient of $t^{2k-2}$ in the expression
\begin{multline}    c(X)c(Y) \big|_{(X,Y)= (t,-t)} - c(X)c(Y)\big|_{(X,Y) = (t,t)} - {4 \over Y}\big( c(X+Y)-c(X) ) \big|_{(X,Y)=(0,t)}  
 \nonumber \\
 +  \lim_{Y\rightarrow 0} {4 \over Y}\big( c(X+Y)-c(X)  ) \big|_{X=t}  =c(t)c(-t)- c(t) c(t)  + 4 \Big( c'(t)- {c(t) \over t}  \Big) \ . 
  \end{multline}
One verifies using $(\ref{cxdefn})$ that 
  this  is in turn  equal to $   6 c'(t) -{ 1\over 2 }$, which proves $(\ref{e0cuppedwithboundary})$. \end{proof} 
  
  The previous lemma implies that the rational  Eisenstein cocycle $e^0_{2k}$ and the cuspidal  coboundary cocycles  $\delta^0 Y^{2k-2}$ are dual to each other.

In summary,  the following diagram is commutative:
$$
\begin{array}{ccc}
 H_{\cusp}^1(\Gamma;V_{2k})^{\pm} \otimes \R &  \overset{s\otimes \R}{\To} &  Z_{\cusp}^1(\Gamma;V_{2k})^{\pm} \otimes \R  \\
 \uparrow_{[\pp^{\pm}]} &   &  ||   \\
S_{2k+2}(\Gamma)  &  \overset{\pp^{\pm}}{\To} &   Z_{\cusp}^1(\Gamma;V_{2k})^{\pm}\otimes \R
\end{array}
$$ 
Since  the Haberlund-Peterssen inner product is non-degenerate, 
we can uniquely  determine  elements in $Z^1_{\cusp}(\Gamma;V_{2k})$ by pairing    with the cocycles of cusp forms \S\ref{sectPeriodsofscusp} and Eisenstein series \S\ref{sectRatEis}  with respect to $\{ , \}$.

\section{Transference of periods} \label{sectTranference}
The non-vanishing of  $H^2(\Gamma, \Gamma_{\infty};\Q )$ 
leads to non-trivial identities between periods of iterated Eichler integrals. It gives rise to a  kind of `transference principle'
whereby periods of iterated integrals of  certain modular forms are related  to periods of iterated integrals of  different modular forms. 

\subsection{Relative $H^2$}  \label{sectRelcohomGamma} 
The group $\Gamma$ is of cohomological dimension $1$. The  cohomology of $\Gamma$ relative to $\Gamma_{\infty}$ ($\S\ref{sectRelcohom}$), however, satisfies  
   \begin{eqnarray} \label{H2}  H^2 ( \Gamma, \Gamma_{\infty};  V_{n}) = \begin{cases}
   \Q & \text{if } n =0\ , \\
    0 & \text{if }  n \hbox{ even}>0 \ ,
  \end{cases}
 \end{eqnarray} 
  corresponding to the compactly supported cohomology of $\M^{an}_{1,1}$.  Define a map
  \begin{equation} h: H^2 ( \Gamma,  \Gamma_{\infty};  \Q) \overset{\sim}{\To} \Q 
  \end{equation}
  as follows. By $(\ref{longexactH})$, there is a long exact cohomology sequence
  $$H^1(\Gamma; \Q) \rightarrow H^1(\Gamma_{\infty};\Q) \rightarrow H^2(\Gamma,\Gamma_{\infty};\Q) \rightarrow H^2(\Gamma;\Q)$$
and since $H^1(\Gamma; \Q)=H^2(\Gamma ;\Q)= 0$ the boundary map is an isomorphism
\begin{equation} \label{H1atinfversusH^2rel}  H^1(\Gamma_{\infty};\Q) \overset{\sim}{\rightarrow} H^2(\Gamma,\Gamma_{\infty};\Q)\ .
\end{equation} 
Evaluation of cocycles at $T$ defines an isomorphism $ H^1(\Gamma_{\infty};\Q)\overset{\sim}{\rightarrow} \Q$  (see $(\ref{H01Gammainfinity})$).
We therefore define $h $ to be the inverse of $(\ref{H1atinfversusH^2rel})$ followed by evaluation at $T$.

In order to compute  $h$, note that    an element in $H^2 ( \Gamma,  \Gamma_{\infty};  \Q)$
can be  represented by a pair $(\alpha, \beta)$, where $\alpha \in Z^2( \Gamma;\Q)$, $\beta \in C^1( \Gamma_{\infty};\Q)$ and $\alpha|_{\Gamma_{\infty}} = \delta^1 \beta$.

\begin{lem} 
Let $(\alpha, \beta) \in Z^2 ( \Gamma,  \Gamma_{\infty};  \Q) $ as above. Then
\begin{equation} \label{hequation} 
h((\alpha, \beta) )=   \beta_T + { 1\over 6 } \big( 2 \alpha_{(U,U)}  +  2 \alpha_{(U^2,U)} + 6 \alpha_{(T,S)} - 3 \alpha_{(S,S)} \big)  \ . \end{equation}

\end{lem}

\begin{proof} 
The isomorphism $(\ref{H1atinfversusH^2rel})$ is induced by the  map   
$v \mapsto (0,v)$ on cocycles.  Therefore $h([0,v])= v_T$.
For a general cocycle $(\alpha, \beta)$, it suffices to express it in the form $(0,v)$ modulo a coboundary.
Since $H^2(\Gamma;\Q)=0$ there exists a cocycle  $f\in C^1(\Gamma;\Q)$ such that 
$\alpha = -  \delta^1 f$.  Since $\Gamma$ acts trivially on $\Q$, we have by \S\ref{sectGroupCohomReminders}
$$\alpha(g,h) =  f(g) + f(h)- f(gh)\ , $$
for $g,h \in \Gamma$. Setting $g, h =\pm 1$ implies that $f(\pm 1)=0$. 
To compute  $f_T$, evaluate the  previous equation  on  pairs in $\Gamma \times \Gamma$  to get:
$$
\begin{array}{cclcrcl}
 \alpha_{(S,S)}  &  =   &  2 f_{S} \qquad  &, &  \qquad \alpha_{(T,S)} & = & f_{S}+ f_{T}- f_{U}\\
 \alpha_{(U,U)} & =    & 2 f_{U}  - f_{U^2} & ,  &   \alpha_{(U^2,U)} & = &  f_{U} + f_{U^2}\   .
\end{array}
$$ 
Combining these equations gives
$$ 6 f_{T}  =  2 \alpha_{(U,U)}  +  2 \alpha_{(U^2,U)}  + 6 \alpha_{(T,S)}- 3 \alpha_{(S,S)}\ .$$
Denote the inclusion of $\Gamma_{\infty}$ by $i:\Gamma_{\infty} \rightarrow \Gamma$. The element
$$ (\alpha, \beta) + \delta (f, 0) = (0, \beta +  i^*f  )$$
is cohomologous to $(\alpha,\beta)$,  and so  $h([\alpha, \beta]) =  \beta_T +  f_T $, which gives $(\ref{hequation})$.
\end{proof}

\subsection{Relative $H^1$.} 
The group $H^0(\Gamma, \Gamma_{\infty}; V_{2n})$ vanishes for all $n$.
\begin{lem} Let $n \geq 1$. Then $H^2(\Gamma, \Gamma_{\infty}; V_{2n})=0$ and  there is an isomorphism 
$$H^1(\Gamma, \Gamma_{\infty}; V_{2n}) \cong H^1_{\cusp}(\Gamma; V_{2n}) \oplus \Q\ .$$
The cohomology class corresponding to the second component is $[(0,Y^{2n})]$.
\end{lem} 
\begin{proof}  By the long exact cohomology sequence $(\ref{longexactH})$, we have
\begin{multline}  0 \rightarrow H^0(\Gamma_{\infty}; V_{2n}) \rightarrow  H^1(\Gamma, \Gamma_{\infty}; V_{2n} ) \nonumber \\
\rightarrow H^1(\Gamma; V_{2n}) \rightarrow H^1(\Gamma_{\infty}; V_{2n}) \rightarrow H^2(\Gamma, \Gamma_{\infty}; V_{2n}) \rightarrow 0 \ . \nonumber
\end{multline}
By $(\ref{H01Gammainfinity})$,  $H^1(\Gamma_{\infty};V_{2n}) \cong \Q X^{2n}$, and the map 
$H^1(\Gamma; V_{2n}) \rightarrow H^1(\Gamma_{\infty}; V_{2n})$ is evaluation at $T$ followed by projection $Y\mapsto 0$.  By 
$(\ref{e2k0S})$, this map is surjective, since the  cocycles $e^0_{2n+2}$ have a non-zero coefficient of $X^{2n}$.   We deduce that $H^2(\Gamma, \Gamma_{\infty}; V_{2n})=0$, and the previous long exact sequence reduces to 
$$0 \To \Q Y^{2n} \To H^1(\Gamma, \Gamma_{\infty}; V_{2n} )  \To H_{\cusp}^1(\Gamma; V_{2n}) \To 0\ .$$
This splits canonically  by composing the Hecke equivariant map $(\ref{sHeckesplitting})$ with 
$$ c \mapsto [(c,0)]: Z^1_{\cusp}(\Gamma; V_{2n}) \rightarrow H^1(\Gamma, \Gamma_{\infty}; V_{2n} ) \ . $$
The last statement follows from the definition of the boundary map. 
\end{proof} 

\begin{rem} Zagier's extended period polynomials, which have poles in $Y$, can be interpreted  as follows.
 Define  a graded vector space  $\widehat{V}_{\infty} = \oplus_{n\geq 0}  \widehat{V}_{2n}$, where 
$$\widehat{V}_{2n} \subset {1 \over Y} \Q[X,Y] $$
denotes the space of  rational functions in $X,Y$ with only simple poles in $Y$,  which are homogeneous  of  degree $2n$. Since $\Gamma_{\infty}$ fixes $Y$,
it inherits a right $\Gamma_{\infty}$-action.

There is a natural map of $\Gamma_{\infty}$-modules $V_{\infty} \rightarrow \widehat{V}_{\infty}$. 
Let $C^i( \Gamma, \widehat{\Gamma}_{\infty} ; V_{2n})$ denote the cone of 
$ C^i(\Gamma, V_{2n}) \To C^i(\Gamma_{\infty}, \widehat{V}_{2n})$.  In addition to the generators of $H^1(\Gamma, \Gamma_{\infty}; V_{2n})$, the cohomology $H^i(\Gamma; \widehat{\Gamma}_{\infty};V_{2n})$ possesses Eisenstein classes
$ [ ( e^0_{2k},  v_{2k}) ]$, for  $k \geq 2$,  where $v_{2k}$ is the trivialising element:
 $$v_{2k}    =  {\Be_{2k} \over 4k (2k-1) }  {  X^{2k-1}  \over Y} \quad \hbox{ which satisfies }  \quad  e^0_{2k}(T) = \delta^0 \, v_{2k}(T)\ .$$ 
 Zagier's extended period polynomials for Eisenstein series 
 are formally given by  the elements $e^0_{2k}(S) - \delta^0(v_{2k})(S)$ where $\widehat{V}_{\infty}$ is `illegally' viewed as a $\Gamma$-representation. These are not to be confused with the actual cocycle corresponding to Eisenstein series \S\ref{sectRatEis}.
 \end{rem}

\subsection{Poincar\'e duality} \label{sectPairingandcup}
There is a cup product 
\begin{eqnarray} 
Z^1(\Gamma;  V_{2n})   \times Z^1(\Gamma, \Gamma_{\infty};  V_{2n})&  \overset{\cup}{ \To} & Z^2 ( \Gamma, \Gamma_{\infty};  V_{2n} \otimes V_{2n})  \nonumber \\
\gamma \cup (\alpha, \beta) & = & (\gamma \cup \alpha, \gamma \cup \beta) \nonumber 
\end{eqnarray}
Composing with the projection  $V_{2n}\otimes V_{2n} \rightarrow V_0 \cong \Q$  of \S\ref{sectIP},  taking cohomology, and 
applying the map $h$ of $(\ref{hequation})$ yields a pairing between cocycles and relative cocycles:
$$  Z^1(\Gamma;  V_{2n})   \times Z^1(\Gamma, \Gamma_{\infty};  V_{2n} )   \To  \Q\ .  $$
Via the map  $ \alpha \mapsto (\alpha,0): Z_{\cusp}^1(\Gamma;V_{2n}) \rightarrow  Z^1(\Gamma; \Gamma_{\infty};  V_{2n})$,  it induces a pairing
$$\{ \ , \ \}:  Z^1(\Gamma;V_{2n}) \times Z_{\cusp}^1(\Gamma;V_{2n}) \To \Q \ .$$
It follows immediately that $\{P,Q\}$ vanishes if $P$ is a coboundary, but not if $Q$ is, since a coboundary in $Z_{\cusp}^1(\Gamma;V_{2n})$ is not necessarily a relative coboundary. 
We can lift  this pairing to cochains (non-uniquely) by  
substituting \S\ref{sectGroupCohomReminders}  into   $(\ref{hequation})$.
\begin{defn} Define a bilinear pairing of \emph{cochains}
\begin{equation} \h:  C^1(\Gamma; V_{2m}) \otimes C^1(\Gamma; V_{2n})  \To    V_{2m} \otimes V_{2n}   \\
  \end{equation} 
by the formula $\h (\alpha', \alpha ) =  h ( \alpha' \cup \alpha)$. Explicitly, by \S\ref{sectCupproducts}, and  $(\ref{hequation})$
\begin{equation}
\label{IPdef} \h( \alpha', \alpha )  =   {1 \over 3}  \big(  \alpha'_U + \alpha'_{U^2}\big)\big|_U \otimes \alpha_U    + \big(   \alpha'_T -  {1 \over 2} \alpha'_{S}\big)\big|_S  \otimes  \alpha_S \ .
\end{equation}
\end{defn}
The pairing $\h$ is a precursor to the Peterssen-Haberlund inner product.

\begin{lem} \label{lemhiscurly} If  $f\in Z^1(\Gamma;V_{2k})$ and  $g \in Z^1_{\cusp}(\Gamma;V_{2k})$   then 
$$\{ f, g\} = - 6  \langle \h(g,f) \rangle $$
where the bracket $\{ \ , \}$ was defined in $(\ref{curlypairing})$.
\end{lem}
\begin{proof}
For any cocycle $c$, we have $0= c_U+ c_{U^2}|_U$ since $U^3=1$, and also $c_U=c_S+c_T|_S$ since  $U=TS$.
Because  $g_T=0$, we have furthermore $g_U=g_S$. Therefore by $(\ref{IPdef})$
$$\langle \h(g,f)\rangle= {1 \over 3}  \langle    g_S \big|_{TS} - g_S , f_S + f_T\big|_S  \rangle     -{1\over 2}   \langle   g_{S}\big|_S   ,f_S \rangle\ .$$
Using the $\Gamma$-invariance of the inner-product,  the equation $c_S|_S = -c_S$, and re-grouping terms paired with $f_S$ on the left, and those paired with $f_T$ on the right, we obtain:
\begin{equation}\label{penultimate6hg}
6 \langle \h(g,f)\rangle= \langle g_S- 2 g_S|_T,  f_S \rangle+ 2 \langle g_S \big|_{1+T}, f_T  \rangle  \ .
\end{equation}
On the other hand, for any cocycle $c$ we have  $c_U + c_U|_{U} + c_U|_{U^2}=0$, which, applied to $g$ gives
$g_S + g_S|_{TS} + g_S|_{S^{-1}T^{-1}}=0$. Pairing with $f_S$ leads to the equation
$$\langle g_S, f_S \rangle = \langle  g_S\big|_T, f_S \rangle + \langle  g_S\big|_{T^{-1}}, f_S\rangle$$
since $\langle g_S |_{S^{-1}T^{-1}}, f_S \rangle = \langle g_S|_{S^{-1}}, f_S|_T\rangle = -\langle g_S, f_S |_T\rangle = - \langle g_S|_{T^{-1}}, f_S\rangle$. 
Substituting into  $(\ref{penultimate6hg})$ and using  the fact that $\langle \ , \rangle$ is symmetric on $V_{2k} \otimes V_{2k}$ gives back the  formula written down in $(\ref{curlypairing})$.  
\end{proof} 
We shall give a geometric interpretation in \S\ref{Haberlundviageometry}. 
\subsection{Transference principle}

Let $\CC$ denote the canonical  holomorphic cocycle. By $(\ref{Smap})$,  we shall view  $\CC$ as a collection of  cochains 
$$\CC: M_{2k_1+2} \otimes \ldots \otimes M_{2k_r+2} \To C^1(\Gamma; V_{2k_1} \otimes \ldots \otimes V_{2k_r})$$
The  vector space on the left has a basis  given by words  $w=\al_{f_1} \ldots \al_{f_r}$ where $f_i \in \B_{2k_i+2}$.  Let $\CC(w)$ denote the corresponding  $\Gamma$-cochain.

\begin{thm} \label{proptransfer} Let $\pi: V_{2k_1} \otimes \ldots \otimes V_{2k_r} \rightarrow V_0$ denote any $\mathrm{SL}_2$-equivariant
projection onto a copy of  $V_0\cong \Q$. The coefficients of $\CC $ satisfy  an equation
\begin{equation} \label{obstructeqn} 
\pi \, \Big(\sum_{uv = w}  \h(\CC(u), \CC(v)) + \CC(w)_T\Big) = 0
\end{equation} 
for any word   $w$ in the $\al_{f}$,   where the sum is over strict factorisations of $w$.
If $w$ contains at least one letter $\al_{f}$ where $f$ is a cusp form, then
$$
\pi \, \sum_{uv = w}  \h(\CC(u), \CC(v))  = 0\ .
$$
\end{thm}

\begin{proof} Denote the restriction of $\CC _w$ to $\Gamma_{\infty}$ by $i^* \CC _w$. Then by $(\ref{phiMassey})$,
$$\delta^1 (\CC(w), 0) =   ( \sum_{w=uv} \CC(u) \cup \CC(v), i^*\CC(w)) \quad \in \  Z^2 (\Gamma, \Gamma_{\infty}; T^c V_{\infty})  \ .$$
This is a relative coboundary, so   its image under $\pi$ is zero in
$H^2(\Gamma;\Gamma_{\infty}, \Q)$.  We have $h(c_1 \cup c_2, \beta) = \beta_T + \h(c_1 \otimes c_2)$ by definition of $\h$, 
so  $h\circ \pi \delta^1 (\CC(w), 0)$ vanishes, and this gives exactly  $(\ref{obstructeqn})$ since $\pi$ is $\Gamma$-equivariant and hence commutes with  $\h$.
The last equation follows immediately  on applying $(\ref{CTzeroforcusp})$.
\end{proof}
One can view relation $(\ref{obstructeqn})$ as  a pairing between non-abelian cochains.  Equation $(\ref{obstructeqn})$  implies relations between iterated Eichler integrals of length $n$
coming from the existence of iterated Eichler integrals  of length $n+1$.

\subsection{Length one} \label{sectLengthone}
Let $n\geq 2$ and let $\al_1, \al_2  \in M_{2n}$ where $\al_1$ corresponds to a cusp form.
Then $\CC(\al_1 \al_2)$ is cuspidal (vanishes on $T$), and we deduce that 
$$\langle  \h(\CC(\al_1), \CC(\al_2))\rangle =0\ ,$$
which implies by lemma $\ref{lemhiscurly}$      that 
$\{ \CC(\al_2), \CC(\al_1) \} =0$ since the $\CC(\al_i)$ are cocycles. In particular, if $f$ is a cusp form of weight $2n$, then $\CC(\al_f)$ is $\pp(f)$ and $\CC(\e_{2n})$ is, by \S\ref{sectRatEis}, a  multiple of the rational cocycle $e^0_{2n}$ plus a coboundary term. 
It follows  immediately from lemma $\ref{lemhiscurly}$  that the cocycles of cusp forms satisfy
\begin{equation} \label{pfperptoe2n}
\{  e^0_{2n}, \pp(f)\}=0\ .
\end{equation}
This is of course well-known \cite{KZ}.

\subsection{Examples in length two}
 Let $p,q,r \in \N$  be a  triangle:  $$|p-q| \leq r \leq p+q$$
 and let $ \al_1 \in M_{2p+2}, \al_2\in M_{2q+2},\al_3 \in M_{2r+2}$.   Then we have
$$  \langle \h( \CC(\al_1), \partial^{q+r-p} \, \CC(\al_2 \al_3 ))  \rangle  + \alpha \langle     \h(  \partial^{p+q-r} \CC(\al_1 \al_2)     ,  \CC(\al_3) )\rangle \quad \in \quad \Q (2  \pi i )^{2p+2q+2r+3}\ $$
for some $\alpha \in \Q^{\times}$. 
The left-hand side vanishes if $\al_1,\al_2,\al_3$ are not all Eisenstein series.

On the other hand, when $r=p+q$, and $\al_1,\al_2,\al_3$ are Eisenstein series, we obtain:
$$\langle \partial^0 \h(  \CC(\e_{m}  \e_{n} ) ,    \CC(\e_{m+n-2}) )\rangle+ \alpha \langle  \h(\CC(\e_{m}), \partial^{n-2}  \CC(\e_{n}  \e_{m+n-2} )) \rangle \quad  \in \quad \Q (2 \pi i)^{2m+2n-5}  $$ from the previous formula with $m=2p+2, n=2q+2$.
Since we know the cocycles $ \CC(\e_{m})$ explicitly, this gives a  relation between the highest-weight and lowest-weight parts of 
double Eisenstein cocycles
$$  \partial^0  \CC(\e_{m}  \e_{n} ) \qquad \hbox{ and } \qquad   \partial^{n-2}  \CC(\e_{n}  \e_{m+n-2} )$$
This are precisely the two places  where 
we obtain  non-trivial multiple zeta value coefficients (as opposed to single zeta values).

More strikingly, if $\al_1, \al_2$ are Eisenstein series and $\al_3$ corresponds to a cusp form, we find non-trivial relations
between the periods of double Eisenstein integrals $\CC(\e_{m}\e_n)$  and the iterated integral $\CC(\e_n \al_f)$ of an Eisenstein 
series and a cusp form. This fact will be crucial for proving  theorem \ref{thmintrofree}.

\section{Double Eisenstein Integrals and $L$-values} \label{sectDoubleEisandL}

We can determine the imaginary part of the regularised iterated 
integrals of two Eisenstein series. It involves special values of $L$-functions of modular forms  outside the critical strip
and will prove that the latter are periods of the relative completion of the fundamental group of $\M_{1,1}$.

\subsection{Statement} \label{sectStatement}
Let $a,b\geq 2$.   For all $k\geq 0$, define
\begin{equation}
I^k_{2a,2b}= \partial^k  \Image   \big(  \CC_{\e_{2a}\e_{2b}}    +   \bbf_{2a} \cup   \overline{e}^0_{2b}  - \overline{e}^0_{2a} \cup \bbf_{2b}     \big)    
\end{equation}
where $\CC_{\e_{2a}\e_{2b}}$  is the coefficient of 
 $\e_{2a}\e_{2b}$ in the canonical cocycle $\CC$, and for $k\geq 2$,
 \begin{eqnarray} \label{bbardefn}
 \bbf_{2k}  &= &   { (2k-2)! \over 2   }  \zeta(2k-1)  Y^{2k-2}      \\
 \overline{e}^0_{2k} &  = &     (2\pi i)^{2k-1}   e^0_{2k}      \ . \nonumber
 \end{eqnarray}
 
 \begin{lem}  \label{lemIkiscocycle} The cochain $I^k_{2a,2b}$ is a cocycle, i.e., $I_{2a,2b}^k \in Z^1(\Gamma; V_{2a+2b-2k-4})$.
   \end{lem}  
   \begin{proof}
   We showed in $(\ref{phiMassey})$ that $\delta \CC_{\e_{2a}\e_{2b}}= \CC_{\e_{2a}} \cup \CC_{\e_{2b}}$, and in 
 \S\ref{sectRatEis}  that the cocycle $\CC_{e_{2n}}$ is equal to $\overline{e}^0_{2n} + \delta   \bbf_{2n}$. Therefore 
 $\Image\,  \delta  \CC_{\e_{2a}\e_{2b}}    =  \overline{e}^0_{2a} \cup \delta  \bbf_{2b} + \delta \bbf_{2a} \cup  \overline{e}^0_{2b}$ and it follows that $\delta I^k_{2a,2b}=0$ by the Leibniz rule of \S\ref{sectCupproducts}, since $\delta \overline{e}^0_{2n}=0$. 
   \end{proof} 
 It is   $(-1)^k$ invariant with respect to  $\epsilon$, and the shuffle product $(\ref{phishuffle})$ for iterated integrals implies the symmetry  
 $I^k_{2a,2b}=(-1)^{k-1} I^k_{2b,2a}.$

\begin{thm} \label{thmImEab}  Let $k\geq 0$ and let $g$ be a 
Hecke normalised cusp eigenform  for $\Gamma$ of weight $w=2a+2b-2k-2$, and let $C_g$ denote its cocycle (\S\ref{sectES}) .  Then 
\begin{equation}\label{IkpariedwithCg}
 \{ I^k_{2a,2b}, C_{g} \}  = 3 A^k_{a,b}  (2\pi i)^{w+k-1}  \Lambda(g, 2a-k-1) \Lambda(g, w+k) 
 \end{equation} 
 where   $\Lambda(s) = (2 \pi)^{-s} \Gamma(s) L(g,s)$ and  
 $$A^k_{a,b} =  (-1)^a \textstyle{\binom{2a-2}{k} \binom{2b-2}{k}} (k!)^3\ . $$ 
 \end{thm}
Note that the functional equation of the $L$-series of $g$ implies that formula $(\ref{IkpariedwithCg})$ is compatible with the 
 symmetry  $I^k_{2a,2b}=(-1)^{k-1} I^k_{2b,2a}$. 

The strategy of proof is the following: first  we relate the coefficient of $\EE_{2a}\EE_{2b}$  in  the indefinite  iterated integral $\Image( I(\tau;\infty))$ to the product of a holomorphic  Eisenstein series with
a certain real analytic Eisenstein series. Then the Petersson inner product of its cocycle with that of an arbitrary cusp form $g$ can be expressed
as an integral over a fundamental domain via a  generalisation of Haberlund's formula. This can in turn be computed using a version of the Rankin-Selberg
method. When $g$ is a Hecke eigenform, the  final answer is a convolution $L$-function.

\begin{cor} For fixed $a,b,k$ as above, we can write
\begin{equation} \label{Ikformulaversion2} 
I^k_{2a,2b}(S)\equiv \sum_{\{g\}} (2  \pi i )^k  \Lambda(g,w+k)  P_g^{\pm} \pmod  { \delta^0(V_{w-2}\otimes \C)_S} 
\end{equation} 
where the sum ranges over a basis of Hecke normalised  cusp eigenforms of weight $w$, and
$P_g^{\pm} \in P_{w-2} \otimes K_g$ are Hecke-invariant period polynomials \S\ref{sectPeriodsofscusp}. Here, 
 $\pm$ denotes  $\epsilon$-invariants    if $k$ is odd,  and $\epsilon$-anti-invariants if $k$ is even, 
and $K_g$ is the field generated by the Fourier coefficients of $g$. We can can assume $\sigma(P^{\pm}_g) = P^{\pm}_{\sigma(g)}$ 
for $\sigma \in \Aut_{\Q}(K_g)$.
  \end{cor} 

\begin{proof}  By \S\ref{sectPeriodsofscusp}, we can choose the  period   $\omega^{\mp}_g$ (opposite parity to $\pm$ in the statement)  to be the quantity $(2 \pi i)^{w-1}  \Lambda(g,2a-k-1)$. The   polynomials $P^{\pm}_g \in P_{w-2}^{\pm}\otimes K_g $ can  be assumed to be $ \Aut_{\Q}(K_g)$ equivariant.  Now write  $I^k_{2a,2b} =\sum_{\{f\}} \alpha_f P^{\pm}_f$ for $f$ a basis of Hecke eigenforms of weight $w$.   Plugging into $(\ref{IkpariedwithCg})$ implies that 
 $$\{P^+_g, P^-_g\} \alpha_g \in  (2 i \pi)^k \Lambda(g, w+k)  \Q $$
 where the rational multiple only depends on $w,a,b$ and not $g$ itself. Since $\{P^+_g, P^-_g\}\in K_g$ is non-zero, 
we can  rescale  either $P^{+}_g$ or $P^{-}_g$ as appropriate by a multiple of   $\{P^+_g, P^-_g\}^{-1}$ to obtain the required statement.
  \end{proof}

We deduce that 
$(2i \pi)^{-w} L(g,n )$ for all  $n\geq w$
 can be expressed  as  $\overline{\Q}$-linear combinations of  double integrals of Eisenstein series.
 
  \subsubsection{Restriction to $\Gamma_{\infty}$} 
  
  The following theorem implies that  $I^k_{2a,2b}$ is cuspidal, except when $k=2 \min\{a,b\}-2$. 
  
 \begin{thm} Let $i: \Gamma_{\infty} \hookrightarrow \Gamma$. Then 
 $i^* [I^k_{2a,2b}] \in H^1( \Gamma_{\infty}; V_{2a+2b-4-2k}) $
vanishes unless $k=2 \min\{a,b\}-2$. In this case, assuming $a<b$,
$$i^* [ I^{2a-2}_{2a,2b}] =  \overline{\lambda}_{b-a+1}^{a,b} i^* [ e^0_{ 2b-2a+2 }]  \ ,$$
where 
$$ \overline{\lambda}_{b-a+1}^{a,b}=  (-1)^{a+b} {b-a +1 \over b} {(2b-2)! \over (2a-2b)!} {\Be_{2b} \over \Be_{2b-2a+2}} \zeta(2a-1) (2 \pi)^{2a+2b-2}\ .$$
If we interchange $a$ and $b$  the same formula holds  with a minus sign in front of $\overline{\lambda}$.  \end{thm} 

\begin{proof}   A direct way to see this is that the value of the cochain 
$  \CC_{\e_{2a}\e_{2b}}$ on $T$ lies in $\Q (2\pi i)^{2a-2b-2}$, so its imaginary part is zero. 
Therefore $i^*[I^k_{2a,2b}]$ is equal to the $\Gamma_{\infty}$-cohomology class of      $i^*  \partial^k  \Image   \big(   \overline{e}^0_{2a} \cup \bbf_{2b} -  \bbf_{2a} \cup   \overline{e}^0_{2b}    \big) $.  This can be computed by evaluating at $T$, and projecting onto lowest weight vectors
$(\ref{H01Gammainfinity})$ by setting $Y=0$ $(\ref{Tlongexact})$.  This uses our explicit formulae for $e^0_{2n}(T)$. The calculation is elementary but tedious, and is omitted since the theorem 
actually follows from  theorem \ref{thmphieis} via \S\ref{SectEquivDoubleEis}.
\end{proof} 
 
  This, together with theorem \ref{thmImEab}, completely determines the  class 
   $[I^k_{2a,2b}]$. One can go further and determine the corresponding cocycle, but this is not required here.

\subsection{Double Eisenstein integrals}

\subsubsection{Real analytic Eisenstein series}

 \begin{defn} For any integers $i,j \geq 0$, and $s\in \C$ such that  $i+j+ 2 \mathrm{Re}(s) >2$, define a real analytic Eisenstein series
 for $z=x+iy \in \HH$  by
 \begin{equation} 
 \mathcal{E}_{ij}^s( z) = { 1\over 2} \sum_{(m,n)    } { y^s \over (mz+n)^{i+s} (m\overline{z} +n)^{j+s}}
  \end{equation}
 where the sum is over pairs $(m,n)$ of coprime integers such that $(m,n) \neq (0,0)$.
  \end{defn} 

Clearly $\mathcal{E}_{ij}^s(z) =  \mathcal{E}_{ji}^s(\overline{z})$. If $i=j+k$, where $k\geq 0$, then 
$$2 y^j \zeta(i+j+2s) \mathcal{E}_{ij}^s(z)  = \sum_{m,n\in \Z^2\backslash \{(0,0)\}} {y^{j+s}   \over (mz+n)^{k} |m z +n|^{2j+2s}}
 $$
 is the series considered in \cite{Shimura}, (9.1), and has a meromorphic continuation with respect to $s$ to the entire complex plane (\cite{Shimura}, 9.7).
  The same is therefore true of $ \mathcal{E}_{ij}^s(z) $.

  For any element $\gamma \in \Gamma$, we have the transformation formula
  \begin{equation} \label{transform} 
  \mathcal{E}_{ij}^s( \gamma(z)) = (cz + d)^{i} (c \overline{z} +d)^j \, \mathcal{E}_{ij}^s( \gamma(z)) \ .  
  \end{equation} 
 It can be useful to think of $ \mathcal{E}_{ij}^s$ as a modular form of `weights' $(i,j)$.
  
\subsubsection{Primitives of Eisenstein series}
Let $w \geq 4$ and consider the  following real analytic function on $\HH$ taking values  in $V_{w-2} \otimes \C$:
\begin{equation}\label{Erabarwdef}
\underline{\mathcal{E}}_w(z) =  \pi^{-1}{ \zeta(w)(w-2)! }  \sum_{i+j= w-2}   \mathcal{E}^1_{i,j}(z) (X -z Y)^i (X-\overline{z} Y)^j
\end{equation}
where the sum is over $i,j\geq 0$. It is modular invariant:
$$  \underline{\mathcal{E}}_w(\gamma(z))|_{\gamma} =\underline{\mathcal{E}}_w(z)  \hbox{ for  } \gamma \in \Gamma\ .$$
\begin{lem}  $  d\underline{\mathcal{E}}_w(z) = \textstyle{1 \over 2}( \underline{E_w}(z)  - \underline{E_w}(\overline{z}) ) $
\end{lem}
\begin{proof} 
Writing out the definition of $\underline{\mathcal{E}}_w(z)$ gives
$$ \underline{\mathcal{E}}_w(z) = {(w-1)!    \over  4 \pi i(w-1)}  \sum'_{m,n\in \Z } \sum_{i+j= w-2}   {(z-\overline{z}) (X-z Y)^i (X-\overline{z} Y)^j\over (mz+n)^{i+1}(m  \overline{z}+n )^{j+1} }  $$
where the first sum is over $(m,n) \in \Z^2$ such that $(m,n) \neq (0,0)$.  The lemma follows from the following elementary identity, 
and its complex conjugate:
$${\partial \over \partial z}   \Big( \sum_{i+j= w-2}   {(z-\overline{z}) (X-z Y)^i (X-\overline{z} Y)^j\over (mz+n)^{i+1}(m  \overline{z}+n )^{j+1} } \Big)  =(w-1)
 {(X - zY)^{w-2} \over (mz+n)^w} $$
 The formula follows from the definition  of $\underline{E}_w(z)$:
 $$ \underline{E}_w(z) = {(w-1)! \over 2 \pi i}    \sum'_{m,n\in \Z} { (X-zY)^{w-2} \over (mz+n)^w} \  dz\ ,$$
 which is verified by  observing that the constant term of the inner sum at $z=i \infty$ is  $2 \zeta(w)$, 
 which, by Euler's formula,  is $- (2\pi i )^w \Be_{w}  (w!)^{-1}$. 
 \end{proof}
Hereafter we use the following simplified notation for  the  iterated integrals
$$[\underline{E}_{2a}](z) = \int_z^{\tbp} \underline{E}_{2a}(\tau)$$
$$[\underline{E}_{2a}|\underline{E}_{2b}](z) = \int_z^{\tbp} \underline{E}_{2a}(\tau)\underline{E}_{2b}(\tau)$$

 \begin{lem} \label{lemImpart} For all $a,b\geq 2$,  we have the identities 
$$\mathrm{Re} \big(  [\underline{E}_{2a}](z)   \big) = \underline{\mathcal{E}}_{2a}(z)  - \bbf_{2a} $$
where  $\bbf$ is defined in $(\ref{bbardefn})$, and 
\begin{align}  \nonumber
d  \, \big(\mathrm{Im} [\underline{E}_{2a}|\underline{E}_{2b}] - \mathrm{Re} [ \underline{E}_{2a}] \mathrm{Im}[\underline{E}_{2b}] \big)    =  & \quad   \big(\underline{\mathcal{E}}_{2a}(z)  -\bbf_{2a}\big)(X_1,Y_1)  \,  \mathrm{Im} ( \underline{E}_{2b}(z)(X_2,Y_2) ) \\
& -\mathrm{Im} ( \underline{E}_{2a}(z)(X_1,Y_1) ) \,\big(\underline{\mathcal{E}}_{2b}(z) -\bbf_{2b}\big)(X_2,Y_2) 
 \label{ImeabasEis}
\end{align}
\end{lem} 
\begin{proof}
Recall that $\CC_{e_{2a}}$ denotes the $\Gamma$-cocyle associated to $[\underline{E}_{2a}](z)$. 
Since $E_{2a}(q)$ has real Fourier coefficients, the previous lemma gives
$$  \mathrm{Re} [\underline{E}_{2a}](z) =  \underline{\mathcal{E}}_w(z)  + P_{2a}$$
for some constant polynomial $P_{2a}\in V_{2a-2} \otimes \C$. By 
$(\ref{transform})$, the real analytic Eisenstein series is modular invariant 
 $ \underline{\mathcal{E}}_w(\gamma(z))|_{\gamma} =  \underline{\mathcal{E}}_w(z) $. It follows that
 $$\mathrm{Re } (C_{\e_{2a}})_{\gamma}  =   P_{2a}\big|_{\gamma} - P_{2a} \ . $$
Now since $\Gamma$ acts without fixed points on $V_{2a-2}$, this uniquely determines $P_{2a}$ from $\Real (C_{\e_{2a}})$. From 
lemma $\ref{propE2k}$,  it follows that  $P_{2a} =  -\bbf_{2a}$. 
The second equation follows from the general identity, for iterated integrals $[\omega_1| \omega_2]$ of two closed $1$-forms $\omega_1, \omega_2$
$$ d \big(\mathrm{Im}\, [ \omega_1 | \omega_2] - \mathrm{Re}\, [ \omega_1] \mathrm{Im}\, [  \omega_2]) \big)  =     \mathrm{Re}\, [ \omega_1] \mathrm{Im}\,  ( \omega_2) -  \mathrm{Im}\, (  \omega_1)   \mathrm{Re}\, [ \omega_2] 
  $$ 
which follows from $d [\omega_1 | \omega_2 ] =-\omega_1 [\omega_2] $ and $ d[\omega_i] = -\omega_i$ for $i=1,2$. 
Applying it to $\omega_1=  \underline{E}_{2a}(z)(X_1,Y_1)$, and  $\omega_2=  \underline{E}_{2b}(z)(X_2,Y_2)$ gives the required identity.
\end{proof}

\subsubsection{Double Eisenstein cocyle} \label{sectDoubleEisensteinFF}
For all $a, b\geq 2$ define a 1-form
$$\FF_{2a,2b}(z) = \Image(\underline{E_{2a}}(X_1,Y_1) )\,  \underline{\mathcal{E}_{2b}}(z)(X_2,Y_2) - \Image(\underline{E_{2b}}(X_2,Y_2))\,   \underline{\mathcal{E}_{2a}}(z)(X_1,Y_1)$$
It is modular invariant: $\FF_{2a,2b}(\gamma(z))  |_{\gamma}= \FF_{2a,2b}(z)$ for all $\gamma \in \Gamma$.
Furthermore, it has at most logarithmic singularities (with respect to the coordinate $q= e^{2\pi i z}$) at the cusp and therefore we can define the regularised integral
$$ [\FF_{2a,2b}](z)= \int_z^{\tinf} \FF_{2a, 2b}(z) \ .$$
Since $\FF_{2a,2b}(z)$ is a closed $1$-form, the integral only depends on $z$ and not the choice of path.
 Denote the corresponding $\Gamma$-cocycle by
\begin{eqnarray}
\DD_{2a,2b} : \Gamma  & \longrightarrow  &  \C[X_1,Y_1,X_2,Y_2]    \nonumber \\
\DD_{2a,2b} (\gamma) &  = &    [\FF_{2a,2b}](\gamma(z))|_{\gamma}  - [\FF_{2a,2b}](z)  \ .\nonumber
\end{eqnarray}
It follows from equation $(\ref{ImeabasEis})$ and lemma \ref{propE2k} that  
$$ \DD_{2a,2b}= \Image   \big(  \CC_{[e_{2a}|e_{2b}]}    +    \bbf_{2a} \cup   e^0_{2b}  - e^0_{2a} \cup \bbf_{2b}     \big) \  .  $$
It is a  cocycle by lemma \ref{lemIkiscocycle}. Our goal is to determine its cohomology class.
\subsection{Haberlund's formula}

\subsubsection{ }
Let $k\geq 0$ and  $a,b\geq 4$ as above. Define two differential forms
\begin{eqnarray}   \label{omega1andomega2values}
\omega_1(z,w)  & = & \langle \partial^k \FF_{2a,2b}(z) , (X_1 - \overline{w} Y_1)^{2a+2b-2k-4}\rangle  \\
\omega_2(w)  & = & \overline{g(w)} d\overline{w}  \nonumber  
\end{eqnarray} 
where $g$ is any cusp form of weight $2a+2b-2k-2$. The differential form $\omega_1$ is a  polynomial  in $\overline{w}$ whose coefficients are 
closed $1$-forms in $dz$ and $d \overline{z}$. Then 
$$\omega_1(z,w) \wedge \omega_2(w) = \langle  \partial^k \FF_{2a,2b}(z) \, , \,  \overline{g(w)}  (X_1 - \overline{w} Y_1)^{2a+2b-2k-4}d\overline{w} \rangle  $$
is   $\Gamma$-invariant  for the diagonal action of $\Gamma$ on $(w,z) \in \HH \times \HH$, by the $\Gamma$-invariance of the inner product.
Since $g$ vanishes at the cusp, the $2$-form $\omega_1(z,z) \wedge \omega_2(z)$ is clearly integrable on the standard fundamental domain for $\Gamma$ on $\HH$.

The following result is a corollary of a version of Haberlund's theorem.

\begin{cor}\label{corofHaberland} Let $C_g$ be the  $\Gamma$-cocycle corresponding to the cusp form $g$. Then 
$$\{ \partial^k \DD_{2a,2b} , C_g \} =6  \int_{\mathcal{D}} \omega_1(z, z) \wedge \omega_2(z) $$
where $\mathcal{D}$ is the standard fundamental domain for $\Gamma$ in $\HH$.
\end{cor} 
The right-hand side can be interpreted as a kind of Petersson product.

\begin{lem}  \label{lemwedgeasJab} With the above notations
$$ \omega_1(z,z) \wedge \omega_2(z) =  J_{2a,2b}   - (-1)^{k} J_{2b,2a}$$
where $J_{2a,2b}$ is  given explicitly by
\begin{multline} \label{Jabfirst} J_{2a,2b} = {1 \over 2i}(2\pi i)^{2a-1}   {(2a-2)! k! \over (2a-2-k)!}    \,\big(\pi^{-1} \zeta(2b) (2b-2)! \big)\\
\times \Big(   (z-\overline{z})^{2a+2b-k-4}    E_{2a}(z) \mathcal{E}^1_{2b-2-k,k}(z) \overline{g(z)}   \Big) dz \wedge d\overline{z}  
\end{multline}
\end{lem} 
\begin{proof}
First check that for any $r,i,j,k\in \Z$, we have
\begin{align}
 \partial^k  \Big[ ( aX_1 +bY_1)^r  ( aX_2 +bY_2)^i  ( cX_2 +dY_2)^j\Big] \quad   \qquad \qquad   \qquad \qquad   \label{inproofidpartialk} \\
 \qquad   \qquad \qquad  = \qquad    {r! j! (ad-bc)^k \over (r-k)!(j-k)!}    ( aX_1 +bY_1)^{r+i-k} ( cX_1 +dY_1)^{j-k} \nonumber
 \end{align}
To see this,  simply apply the definition of $\partial^k$ to both sides of the expression
\begin{multline} \big((  \lambda aX_1 +\lambda bY_1 +(\mu a +\nu c) X_2 + (\mu b+ \nu d)Y_2\big)^{N}  \nonumber \\ 
 = \sum_{\alpha+\beta+\gamma=N}   {(\alpha+\beta+\gamma)! \over \alpha!\beta!\gamma!}\lambda^{\alpha}
  \mu^{\beta} \nu^{\gamma}( aX_1 +bY_1)^{\alpha}  ( aX_2 +bY_2)^{\beta}  ( cX_2 +dY_2)^{\gamma}
 \end{multline}
 and read off the coefficients of $\lambda^r \mu^i \nu^j$.
Suppose that  $m= r+i+j-k\geq 0$.    For any $P\in V_m$  we have
$$\langle   P(X_1,Y_1),   (X_1-t Y_1)^{m}\rangle= P(t,1)$$ 
by definition of the inner product. Now apply the  identity  $(\ref{inproofidpartialk})$  to the expression $\partial^k \big(( X_1 - z Y_1)^{r} (X_2-z Y_2)^{i}   (X_2- \overline{z} Y_2)^{j} \big)  $ and put $X_1=\overline{z}$, $Y_1=1$.  This gives
$$
\langle  \partial^k \big(( X_1 \!-\! z Y_1)^{r} (X_2 \!-\! z Y_2)^{i}   (X_2 \!-\! \overline{z} Y_2)^{j} \big)    , (X_1 \!-\!\overline{z}Y_1)^{m}  \rangle 
=  \delta_{j,k}   { (-1)^{m} r! k! \over (r-k)! }   ( z\!-\!\overline{z})^{m+k}  \nonumber
$$
where $\delta_{j,k}$ is the Kronecker delta. Applying this identity to the definition of $\FF_{2a,2b}$ and keeping track of the factors (using $(\ref{Erabarwdef})$) gives
the required expression.
\end{proof}
\subsubsection{Haberlund's formula} \label{Haberlundviageometry}
Suppose, as above,  that  we have two differential forms 
$$ \omega_1 (z,w) \ , \  \omega_2( w)  $$
where $\omega_1$ is a polynomial in $\overline{w}$ whose coefficients are  closed $1$-forms in $z$ and $ \overline{z}$ and
$\omega_2(w)$ is closed and vanishes at the cusp $w=i\infty$. Suppose furthermore that 
$$\gamma^*( \omega_1  \wedge \omega_2) = \omega_1 \wedge \omega_2$$
for all $\gamma\in \Gamma$, where $\gamma$ acts on $(z,w) \in \HH\times \HH$ diagonally.
We also assume that  $\omega_1(z,w)$ has, for all $w  \in \HH$,  at most logarithmic singularities 
in $q=\exp(2 \pi i z)$ at $z= i \infty$ (and likewise for all cusps $\gamma (i \infty)$, for $\gamma \in \Gamma$.). Therefore the  following integral with respect to $z$   exists
$$F(w) = \int_{w}^{\tinf} \omega_1(z,w) \ ,  $$
and defines a real analytic function of $w\in \HH$. 
Since $\omega_1$ is closed,  it only depends on $w$ and not the choice of path from $w$ to $\tinf$.  For 
all $\gamma \in \Gamma$, denote by 
$$C_{\gamma}^F  (w) =  \int_{\gamma \tinf }^{\tinf} \omega_1(z,w)\ , $$
where the integral is with respect to $z$ and regularised with respect to the tangential base points at the cusps. It exists by the previous assumptions on $ \omega_1(z,w)$.

\begin{lem} For all $\alpha,\beta \in  \HH \cup \Q \cup \{i \infty\}$, and $\gamma \in \Gamma$,
\begin{equation} \label{Fomegatrans1} 
\int_{\alpha}^{\beta} F \omega_2 =  \int_{\gamma(\alpha)}^{\gamma(\beta)} F\omega_2 - \int_{\gamma(\alpha)}^{\gamma(\beta)}  C_{\gamma}^F   \omega_2 \ . 
\end{equation} 
\end{lem}

\begin{proof} 
First of all, there is the following identity (of $1$-forms in $w$):
\begin{equation} \label{Fid} 
\gamma^{*} (F  \omega_2)  =   F  \omega_2 - C_{\gamma^{-1}}^F\,  \omega_2
\end{equation}
To see this, note that the left-hand side is equal to
$$F(\gamma(w))\wedge \gamma^*(\omega_2) =   \int_{\gamma(w)}^{\tinf} \omega_1(z,  \gamma(w)) \, \wedge\,  \gamma^*( \omega_2) = \int_{w}^{\gamma^{-1}\tinf}  \gamma^*( \omega_1 \wedge \omega_2)  $$
by changing variables in $z$.  But $\omega_1 \wedge \omega_2$ is $\Gamma$-invariant, and the domain of integration on the right-hand side can 
be written as a composition of paths:
$$\int_{w}^{\gamma^{-1}\tinf}   \omega_1 \wedge \omega_2 =   \int_{w}^{\tinf}  \omega_1 \wedge \omega_2- \int_{\gamma^{-1}\tinf}^{\tinf }  \omega_1 \wedge \omega_2$$
where all integrals are with respect to $z$.  This gives $(\ref{Fid})$.  Replacing $\gamma$ with $\gamma^{-1}$ in  $(\ref{Fid})$ and integrating  from $\alpha$ to $\beta$  in the $w$ plane gives $(\ref{Fomegatrans1})$. 
\end{proof}

\begin{prop} Let $\mathcal{D}\subset \HH$ denote the standard  fundamental domain for  $\Gamma$.
Then with the above assumptions on $\omega_1, \omega_2$, 
$$ 6 \int_{\mathcal{D}} \omega_1 (z,z) \wedge \omega_2(z) =  \int_{T^{-1}p-T p} C_S^F \omega_2  + 2  \int_{p} (C_T^F-C_{T^{-1}}^F)  \omega_2  $$
where $p$ denotes the geodesic path from $S(\tinf)$ to $\tinf$.
\end{prop}

\begin{proof}
Consider the domain $\mathcal{D}'$ enclosed by  the geodesic square with corners $-1,0,1,\infty$. 
We  also shall denote the following tangential base points 
$$\tone_{\infty} \quad \ , \quad S(\tone_{\infty}) \quad , \quad  TS (\tone_{\infty})  \quad , \quad  T^{-1}S (\tone_{\infty}) $$
by $\infty, 0,1,-1$, respectively.  The beautiful idea for taking the domain $\mathcal{D}'$, as opposed to $\mathcal{D}$, is due to Pasol and Popa \cite{PaPo}.
It is covered by exactly $6$ copies of $\mathcal{D}$. 
Applying Stokes' formula to $\mathcal{D}'$ gives  
$$  \int_{\mathcal{D}'} \omega_1(w,w) \wedge \omega_2(w) = \int_{\mathcal{D'}} d (F \wedge \omega_2) =  \int_{\partial\mathcal{D}'} F \omega_2\ .$$
All integrals converge  because  $\omega_2(w)$ was assumed to vanish at the cusp.
  The boundary of $\mathcal{D}'$ consists of four  geodesic  path segments, 
from $\infty$ to $-1$ to $0$ to $1$ and back to $\infty$. 
Denote the geodesic path from $0$ to $\infty$ by $p$.   Each side of the square is a path $\gamma p$ for some $\gamma \in \Gamma$. Writing $-p$ for $Sp$,  we have
$$\int_{\partial\mathcal{D}'} F  \omega_2     =    \Big(\int_{ -T^{-1} p} +  \int_{ ST p} + \int_{ -ST^{-1}  p} +  \int_{ T p} \Big) F\omega_2  $$
 Applying $(\ref{Fomegatrans1})$ to each   term gives, for example
 $$ 
\int_{ T^{-1}  p}  F\omega_2  = \int_p F\omega_2 - \int_p C_{T}^F \omega_2 
 $$
 and applying it twice to the second term gives (since $S^2=1$),
 $$ 
\int_{ S T p}  F\omega_2  = \int_{Tp} F\omega_2 - \int_{Tp} C_S^F \omega_2 = \int_{p} F\omega_2 - \int_p C_{T^{-1}}^F \omega_2 - \int_{Tp} C_S^F \omega_2 \ .
 $$
 Adding all four contributions together gives
 $$\int_{\partial\mathcal{D}'} F  \omega_2     = \int_{T^{-1}p-T p} C_S^F \omega_2  + 2  \int_{p} (C_T^F-C_{T^{-1}}^F)  \omega_2 \ $$
 as required.
\end{proof}
In order to prove corollary $\ref{corofHaberland}$,  substitute the values $(\ref{omega1andomega2values})$ for $\omega_1$, $\omega_2$ into the previous formula. For example,
\begin{align} 
\int_{T p} C_S^F \omega_2   & = \int_p \int_{T p} \langle \partial^k \FF_{2a,2b}(z),  \overline{g(w)} (X_1 - \overline{w} Y_1)^m d\overline{w}\rangle \nonumber \\
&=  \langle \int_p  \partial^k \FF_{2a, 2b}(z),  \int_{T p }  \overline{g(w)} (X_1 - \overline{w} Y_1)^m d\overline{w}\rangle \nonumber \\
& = - \langle  \partial^k \DD^S_{2a,2b} , (C_g)_S \big|_{T^{-1}} \rangle \nonumber 
\end{align}
In the third line we used the $\Gamma$-invariance of $ \overline{g(w)} (X_1 - \overline{w} Y_1)^m d\overline{w}$ and the   formula
$$\DD^{\gamma}_{2a,2b} = -\int^{\tone_{\infty}}_{\gamma^{-1} (\tone_{\infty})}  \partial^k \FF_{2a, 2b}(z)$$
which follows from the definition of  $\DD$. 
The other terms similarly give a total of
$$
\langle P^S ,Q^S\big|_T- Q^S\big|_{T^{-1}}\rangle  + 2 \langle P^{T^{-1}} - P^T ,Q^S \rangle  
$$
where $P = \partial^k \DD_{2a,2b}$ and $Q=C_g$. Since $P$ is a $\Gamma$-cocycle, $ P^{T^{-1}}+ P^{T} \big|_{T^{-1}}=0$, and the
previous expression reduces to  $\{P,Q\}$  by the $\Gamma$-equivariance of $\langle  \ ,\  \rangle$.

\begin{rem} The identical argument, applied  in the case $\omega_1 = f(z) (z-\overline{w})^{k-2} dz$ and $\omega_2 = \overline{g(w)} d\overline{w}$
where $f$ is a modular form of weight $k$, and $g$ a cusp form of weight $k$, gives the generalisation of Haberlund's formula of Kohnen and Zagier \cite{KZ}.
\end{rem} 
 
\subsection{Rankin-Selberg Method} 
  
 Let  $f\in M_{k}(\Gamma)$ be a modular form of weight $k$ and let $g\in S_{\ell}(\Gamma)$ be a cusp form of weight $\ell$.
 Let $m \geq \max(k,\ell)$ and $\mathrm{Re}\, s $ large. Then 
    $$   f(z) \mathcal{E}^s_{m-k,m-\ell} (z) \overline{g(z)} \,  y^{m-2}  {dx dy}  $$
    is invariant under $\Gamma$ and the integral 
  $$\langle f   \mathcal{E}^s_{m-k,m-\ell} , g\rangle = \int_{\mathcal{D} }  f(z) \mathcal{E}^s_{m-k,m-\ell} (z) \overline{g(z)}  y^{m-2} dxdy$$
  where $\mathcal{D}\subset \HH$ is the standard fundamental domain for $\Gamma$, 
  converges. This is because,     as $y \rightarrow \infty$,  $g(z)$ is exponentially small in $y$,
  whereas $\mathcal{E}^s_{ij}(z)$ and $f(z)$ are  of  polynomial growth in $y$. 
   In particular, it admits a meromorphic continuation to $\C$.
  
  \begin{prop} If $f(z)= \sum_{n\geq 0} a_n e^{2 \pi in z}$ and  $g(z)= \sum_{n\geq 1} b_n e^{2 \pi in z}$ then 
$$\langle f   \mathcal{E}^s_{m-k,m-\ell} , g\rangle  = (4 \pi)^{-(s+m-1)} \Gamma(s+m-1) \sum_{n\geq 1} { a_n \overline{b_n} \over n^{s+m-1}} $$
for all $Re (s) $ sufficiently large, and hence for all $s\in \C$, by meromorphic continuation.
\end{prop}
\begin{proof} The proof is a standard application of the Rankin-Selberg method. For the convenience of the reader, we sketch the argument here. 
Let
$$\phi^s(z) =  f(z)  \overline{g(z)} y^{s+m}\ .$$
It is invariant under $\Gamma_{\infty}$. 
When $\mathrm{Re}(s)$ is sufficiently large, unfolding gives
$$\int_{\Gamma_{\infty} \backslash \HH} \phi^s(z) { dx dy \over y^2} = \int_{\Gamma \backslash \HH} \sum_{\gamma \in   \Gamma_{\infty} \backslash \Gamma} \phi^s(\gamma(z)) {dx dy \over y^2}  $$
and the right-hand integral reduces to $\langle f   \mathcal{E}^s_{m-k,m-\ell} , g\rangle$.  A fundamental domain for $\Gamma_{\infty} \backslash \HH$
 is given by $(x,y) \in [0,1] \times \R^{>0}$ and the left-hand integral gives
 $$\sum_{p\geq 0, q\geq 1} a_p \overline{b_q} \int_{0\leq x \leq 1} e^{2 i \pi (p-q)x} dx  \int_0^{\infty} e^{-2\pi(p+q)y} y^{s+m-2} dy$$
 It converges for $\mathrm{Re}(s)$ large. After doing the $x$ integral, only the terms with $p=q$ survive, and the previous expression reduces
 to 
  $$ (4 \pi)^{-(s+m-1)} \Gamma(s+m-1) \sum_{n\geq 1} { a_n \overline{b_n} \over n^{s+m-1}}\ . $$
 \end{proof}

\begin{cor} \label{corofRS} Suppose that $f=E_{2a}$ is the Hecke normalised  Eisenstein series of weight $2a$  and $g$ is a Hecke normalised cusp form
of weight $2c$. Then, for any $m\geq 2a,2c$, and writing $s'=s+m$, we have
\begin{multline} 
  \qquad   \zeta(2s'-2a-2c)\langle f   \mathcal{E}^s_{m-2a,m-2c} , g\rangle =     (4 \pi)^{-(s'-1)} \Gamma(s'-1)   \\
  \times  L(g,s'-1) L(g,s'-2a)   \  . \qquad
   \end{multline}
\end{cor} 
\begin{proof} Assume $\mathrm{Re}(s)$ is large.
For any Hecke eigenform $f$ of weight $k$, let us write
$$L(f,s) = \sum_{n\geq 1} {a_n(f)  \over n^s} = \prod_p { 1 \over  (1-\alpha^f_p p^{-s} ) (1-\beta^f_p p^{-s} ) } $$
where  $\{\alpha^f_p,\beta^f_p\}$ are solutions to the equations:
$ \alpha_p^f+ \beta_p^f = a^f_p$ and $\alpha_p^f\beta_p^f = p^{k-1}$.
It is well-known that for $f,g$ Hecke normalised eigenfunctions of weights $k, \ell$, 
$$\sum_{n\geq 1} {a_n(f) a_n(g) \over n^s} = \zeta(2s+2-k-\ell)^{-1}  L(f\otimes g, s)$$
where the tensor product $L$-function is defined by 
$$L(f\otimes g, s) = \prod_p { 1 \over  (1-\alpha^f_p \alpha^g_p p^{-s} ) (1-\alpha^f_p \beta^g_p p^{-s} )(1-\beta^f_p \alpha^g_p p^{-s} ) (1-\beta^f_p \beta^f_g p^{-s} ) } $$
On the other hand, for an Eisenstein series of weight $2a$, we have: 
$$L(E_{2a}, s) = \zeta(s) \zeta(s-2a+1) =\prod_{p} {1 \over (1-p^{-s}) (1- p^{2a-1} p^{-s})}\ .$$
In particular,
$$L(E_{2a} \otimes g, s) = L(g,s) L(g,s-2a+1)$$
Therefore if $f=E_{2a}$ and $g$ has weight $2c$, we have 
 $$\sum_{n\geq 1} {a_n(f)a_n(g) \over n^s} = \zeta(2s+2-2a-2c)^{-1}  L(g,s) L(g,s-2a+1)$$
 Since a Hecke eigenfunction has real Fourier coefficients, applying this formula to the conclusion of the previous proposition gives the statement of the corollary.
\end{proof}

\subsubsection{Proof of theorem \ref{thmImEab}}
Putting all the pieces together, we let $a,b,k$ and $g$ be as in the statement of theorem $\ref{thmImEab}$.
Then  $ I^k_{2a,2b} = \partial^k \DD_{2a,2b}$ and so 
$$ \{ I^k_{2a,2b}, \CC_{g} \}  =  6 \int_{\DD} J_{2a,2b}- (-1)^k J_{2b,2a}$$
by corollary $\ref{corofHaberland}$ and lemma $\ref{lemwedgeasJab}$, where $J_{2a,2b}$ is given by 
\begin{multline} J_{2a,2b} = (2\pi i)^{2a-1}   {(2a-2)! k! \over (2a-2-k)!}    \,\big(\pi^{-1} \zeta(2b) (2b-2)! \big)\\
\times (2i)^{2a+2b-k-4} \Big(   y^{2a+2b-k-4}    E_{2a}(z) \mathcal{E}^1_{2b-k-2,k}(z) \overline{g(z)}  \Big)   dx dy
\end{multline}
using $(\ref{Jabfirst})$. Now plug $m= 2a+2b-k-2$, $2c=2a+2b-2k-2$,  and $s=1$, into the statement of corollary $\ref{corofRS}$. It gives
$$ \zeta(2b)\,  \langle f   \mathcal{E}^1_{2b-k-2,k} , g\rangle =     2^{-m}  \Lambda(g,m ) L(g, 2b-k-1) $$
using the fact that $\Lambda(g,s) = (2 \pi)^{-s} \Gamma(s) L(g,s)$. Using this same expression to replace $L(g,2b-k-1)$ with $\Lambda(g,2b-k-1)$
and combining with the above gives
$$
J_{2a,2b} = (2\pi i)^{m-1}   {(2a-2)! k!  (2b-2)! \over (2a-2-k)! (2b-2-k)!}  \\
\times  \Lambda(g,m ) \Lambda(g, 2b-k-1)   
$$
Finally, writing $m=w+k$, and using the functional equation $$\Lambda(g,2b-k-1) = (-1)^{a+b-k-1} \Lambda(g,2a-k-1)$$
since $g$ is of weight $w=2a+2b-2k-2$, 
 gives
$$
J_{2a,2b} ={1 \over 2}  (2\pi i)^{w+k-1}  A^k_{a,b}   \\
 \Lambda(g,m ) \Lambda(g, 2a-k-1)   
$$
By the remark following theorem $\ref{thmImEab}$, the quantity  $(-1)^{k-1} J_{2b,2a}$ gives an identical contribution.

 \newpage
 
 \begin{center}
 \Large{\bf{ Part II: Hodge and Tannakian theory of $\pi^{\mathrm{rel}}_1(\mathcal{M}_{1,1},\tone_{\infty})$}}
 \end{center}

\section{Algebraic groups in a Tannakian category} \label{sectNew1}

Let $\GG$ be a pro-algebraic group in a Tannakian category. We describe how the Tannaka group acts on $\GG$ by automorphisms.

\begin{rem} The  Tannaka group is usually understood to act on the left of the affine ring $\Or(\GG)$, and hence on the right of $\GG$. Unfortunately, the convention of writing from left to right  is ill-adapted for denoting right-actions, so for this reason we have chosen to consider only  groups of \emph{left}-automorphisms 
in this section. See \S\ref{sectLtoR}. 
\end{rem}

\subsection{Notations for semi-direct products}  \label{sectsemidirect} 
Let $A, B$ be groups.  We shall write $A\ltimes B$ for a semi-direct product where  $A$ acts on $B$ on the \emph{right}:  $ (b,a) \mapsto b^a: B \times A \rightarrow B. $
Its underlying set is  $A\times B$ with the  composition law  
$$(a_1,b_1 ) (a_2, b_2) = (a_1a_2 , b_1^{a_2} \,b_2)\ .$$
We  write $B \rtimes A$ when $A$ acts on $B$ on the \emph{left}:
$ (a,b) \mapsto {}^a b: A \times B \rightarrow B$.
  Its underlying set is  $B \times A$ with the composition law 
$$(b_1,a_1 ) (b_2, a_2) = (b_1   {}^{a_1}\! b_2 , a_1 a_2)\ .$$

\subsection{Automorphism groups}
Let $\GG$ be a pro-algebraic affine  group scheme over a field $k$ of characteristic zero, equipped with a  morphism $ \pi :\GG \rightarrow S$ defined over $k$,
where $S$ is a pro-reductive affine  group scheme  over $k$. Denote its  kernel  by $\U = \ker (\pi)$, and suppose that it is pro-unipotent.  
Thus there is an exact sequence
\begin{equation} \label{GGexactsequence}  1 \To \U \To \GG \overset{\pi}{\To} S \To 1\ .
 \end{equation}

 \begin{defn} Let us denote by 
 \begin{eqnarray} 
 \mathrm{Aut}_{\U}(\GG) & = &  \{ \alpha: \GG  \overset{\sim}{\rightarrow} \GG \hbox{ such that }   \alpha(\U)\subset \U \} \nonumber \\
 \mathrm{Aut}_{\pi}(\GG) & = &  \{ \alpha: \GG  \overset{\sim}{\rightarrow} \GG \hbox{ such that }   \pi \,\alpha  = \pi  \} \nonumber 
 \end{eqnarray} 
 the group of left automorphisms of $\GG$ which preserve $\U$,  or respect  $\pi$, respectively.   They   are functors from commutative $k$-algebras to groups. 
 \end{defn} 
 In our applications,  these functors will be  representable  and  define  affine group schemes over $k$.\footnote{See \url{arxiv:1704.00555}, \S6.4 for some sufficient conditions for representability.}
 There are natural  maps 
 \begin{eqnarray} \label{restrictionAut} r: \mathrm{Aut}_{\U}(\GG) & \To& \mathrm{Aut}(\U)\ ,  \\ 
  q:  \mathrm{Aut}_{\U}(\GG) & \To& \mathrm{Aut}(S)\  .   \nonumber  
 \end{eqnarray} 
 (for `restriction', and reductive `quotient'). The restriction map $r$ 
 will not in general be surjective (see remark \ref{remimagerest} below).  There is an exact sequence
 \begin{equation} \label{Autexactseq}  1 \To \mathrm{Aut}_{\pi}(\GG) \To \mathrm{Aut}_{\U}(\GG) \overset{q}{\To} \mathrm{Aut}(S)\ .
 \end{equation}
 
 In order to describe these groups it is convenient to assume that $\GG$ is a semi-direct product of $\U$ and $S$. In other words, 
 we shall fix a splitting of the exact sequence $(\ref{GGexactsequence})$.  This is guaranteed by a   version of Mostow's theorem.
 In our applications to relative completion, it follows  from 
 \cite{HaGPS}, \S3.1, where it is proved on the level of points.
 
 \begin{thm}  \label{Levisplit} There is a splitting  of $(\ref{GGexactsequence})$, i.e., an isomorphism of affine group schemes
$\GG \cong S \ltimes \U$.  Any two  splittings are conjugate by an element of $\U(k)$.
 \end{thm} 
 
 \subsection{Automorphisms of a semi-direct product}
 Fix  a right action of $S$ on $\U$, so that we can  form the semi-direct product $S\ltimes \U$.  Let $\pi:  S \ltimes \U \rightarrow S$ denote the natural map. We first  compute $\mathrm{Aut}_{\pi}( S \ltimes \U)$ and proceed to $\mathrm{Aut}_{\U}(S\ltimes \U)$ in \S\ref{sectAutUSltimesU}.

 \begin{defn} Denote the  $S$-equivariant automorphisms of $\U$ by
$$ \mathrm{Aut}(\U)^S = \{ \phi \in \mathrm{Aut}(\U) \hbox{ such  that } \phi(u^{s}) = \phi(u)^s \hbox{ for all } u \in \U, s\in S  \}$$
It is the functor $R \mapsto  \mathrm{Aut}(\U)^S(R)$ from commutative $k$-algebras $R$ to groups.   Its elements act on  $\U(R)$ on the left.
\end{defn}

In particular, we can form the semi-direct product
$ \U \rtimes  \mathrm{Aut}(\U)^S.$
If $\U^S$ denotes the subgroup of $S$-invariants of $\U$, then $\U^S$ acts by right-multiplication on $\U$. It also defines a right action  by conjugation 
on $\mathrm{Aut}(\U)^S$ as follows: set
$$\phi_a( u) = a^{-1} \phi(u) a \qquad \hbox{ for } a\in \U^S \ , \ \phi \in \Aut(\U)^S \ , \ u \in \U\ .$$
We  view all objects $\U, \U^S, \Aut(\U)$ and so on as functors from commutative unitary $k$-algebras  $R$ to groups
and write, for example, $u \in \U$ to denote $u \in \U(R)$.  

\begin{defn} Denote by 
\begin{equation} \label{UrtimesAutquot} \U \rtimes^{\U^S} \mathrm{Aut}(\U)^S
\end{equation} 
the functor from commutative $k$-algebras to groups whose points are given by the  set of pairs $(b, \phi) \in   \U \times  \mathrm{Aut}(\U)^S$ modulo the equivalence relation
$$(b, \phi) \sim (ba, \phi_a) \quad \hbox{ for any } a \in \U^S\ .$$
Denote the equivalence class of $(b, \phi)$ by $[(b,\phi)]$.  There is an exact sequence
\begin{equation}\label{USexactsequence}  1 \To \U^S \overset{*}{\To} \U \rtimes \mathrm{Aut}(\U)^S \To \U \rtimes^{\U^S} \mathrm{Aut}(\U)^S \To 1
\end{equation}
where $*$ is $a\mapsto (a, \id_a)$.  Note that $(a, \id_a) . (b, \phi) = (a \id_a b, \id_a \phi) = (ba, \phi_a)$. 
\end{defn}

We define a left action of $ \U \times \mathrm{Aut}(\U)^S$ on $S \ltimes \U$ via
\begin{equation} \label{bphiactiononsemidDef} (b, \phi) \circ (s, u) = (s, b^s \phi(u) b^{-1})\ .
\end{equation}
The following verifications are straightforward, but are included for the convenience of the reader due to the lack of a suitable reference.
\begin{enumerate}
\item The image of $\U^S$  under  the second map in $(\ref{USexactsequence})$ is the subgroup of  $\U \times \mathrm{Aut}(\U)^S$  which acts trivially on $S\ltimes \U$. In particular, it is a normal subgroup.

To see this, check that
$$(b, \phi) \circ (s, u) = (s, b^s \phi(u) b^{-1}) =  (s, u)$$
for all $(s, u)$ if and only if   $b^s b^{-1}=1$ for all $s\in S$ (set $u=1$), and $\phi(u) = b^{-1} u b$,  for all $u\in \U$ (set $s=1$). Thus 
$(b, \phi) = (b , \id_b)$ where $b \in \U^S$. Conversely, if $b\in \U^S$ then  $(b, \id_b)$ acts trivially on $S \ltimes \U$. 

\item The action $(\ref{bphiactiononsemidDef})$ respects the group law on $S \ltimes \U$, i.e., 
$$(b, \phi) \circ \big( (s,u). (s',u')\big) = (b, \phi) \circ (ss', u^{s'} u') = (ss', b^{ss'} \phi(u^{s'} u') b^{-1})$$
On the other hand, 
\begin{eqnarray}  \big((b, \phi) \circ (s,u) \big) . \big((b, \phi) \circ (s',u') \big) & = & (s, b^s \phi(u)b^{-1}). (s', b^{s'} \phi(u')b^{-1})  \nonumber  \\
& = &  (ss',   b^{ss'} \phi(u)^{s'}(b^{-1})^{s'} b^{s'} \phi(u')b^{-1})  \nonumber 
\end{eqnarray} 
which coincides with the previous formula because $\phi(u^{s'} u') = \phi(u)^{s'} \phi(u')$, i.e., since $\phi$ is $S$-invariant and a homomorphism.
\item The composition is given by the semi-direct product.  Check that 
$$(b', \phi') \circ \big( (b, \phi) \circ (s,u)\big) = (b', \phi') \circ (s, b^s \phi(u) b^{-1})  = (s, b'^s \phi'(b^s) \phi' \phi(u) \phi'(b^{-1}) b'^{-1}) $$
which indeed coincides with 
$$\big((b', \phi') .   (b, \phi) \big) \circ (s,u) = (b' \phi'(b), \phi'\phi) \circ (s,u)  = (s, ( b' \phi'(b))^s \phi' \phi(u)  (b'\phi'(b))^{-1})\ . $$
\end{enumerate}
We have thus defined a   morphism  (of functors from $k$-algebras to groups)
\begin{equation} \label{UrtimesAutUtoS}
 \U \rtimes^{\U^S} \mathrm{Aut}(\U)^S \To \mathrm{Aut}_{\pi} (S \ltimes \U)\ ,
 \end{equation} 
via the  following formula, which is well-defined by $(1)$: 
$$[(b, \phi)] \circ (s, u) = (b,\phi) \circ (s,u) \ .$$

\begin{prop} \label{propAutassemidirec}  The map $(\ref{UrtimesAutUtoS})$ is an isomorphism.
\end{prop} 

\begin{proof}
We  construct the inverse as follows. Let 
 $\alpha \in \mathrm{Aut}_{\pi}(S \ltimes \U)$, i.e.,   an isomorphism $\alpha: S\ltimes U \overset{\sim}{\rightarrow} S \ltimes U$ preserving $S$. 
 Let us write  $\alpha(s, 1) = (s, \alpha_s)$.
 This defines a homomorphism
$$s\mapsto (s, \alpha_s):  S \To S\ltimes \U$$ and hence  $\alpha_s\in Z^1(S, \U)$ is a right cocycle (remark \ref{remZ1isHom}). Since $S$ is pro-reductive, $H^1(S,\U)$ is trivial, and therefore there exists a $b\in \U$ such that 
$\alpha_s = b^s\, b^{-1}$. Furthermore, $b$ is unique up to right-multiplication by an element of $\U^{S}$.   Define  an isomorphism (for the time being,  of schemes only and not necessarily of groups)
$$\phi: \U \overset{\sim}{\To} \U $$
by $\phi(u) = b^{-1}  \alpha(u) b$, where $\alpha(u)\in \U$. Now verify that 
$$\alpha( s,u) = \alpha \big( ( s,1). (1,u)  \big)= (s, \alpha_s) . (1, \alpha(u)) = (s, \alpha_s \alpha(u))  $$
and  by definition of $\phi$, 
$$\alpha_s \alpha(u) = \alpha_s b \phi(u) b^{-1} = b^s \phi(u) b^{-1}\ , $$
since $\alpha_s = b^s b^{-1}$. We have shown that 
$$\alpha(s, u ) = (s, b^s \phi(u) b^{-1})\ .$$
It follows from the fact that $\alpha$ respects the group law on $S\ltimes U$, and an 
essentially  identical calculation to $(2)$ above, that $\phi$ is a  homomorphism and $S$-equivariant. 
Finally, the equivalence class $[(b, \phi)] \in  \U \rtimes^{\U^S} \mathrm{Aut}(\U)^S$ is independent of the choice of $b$,  
 so we have constructed  a well-defined map 
 $$ \mathrm{Aut}_{\pi} (S \ltimes \U) \To \U \rtimes^{\U^S} \mathrm{Aut}(\U)^S \ . $$
It is easily checked to be the inverse to the map $(\ref{UrtimesAutUtoS})$.
\end{proof} 
\subsubsection{Representations}
The  natural map $r: \mathrm{Aut}_{\pi}( S \ltimes \U)\rightarrow \mathrm{Aut} (\U)$ is given by 
\begin{eqnarray} \label{Autsemi-dtoAut}     \U \rtimes^{\U^S} \mathrm{Aut}(\U)^S &  \To & \mathrm{Aut} (\U)    \\
{[}(b, \phi)] & \mapsto & b \,\phi \, b^{-1} \ . \nonumber 
\end{eqnarray}

\begin{defn} The proof of the preceding proposition defines a homomorphism
\begin{eqnarray}  \label{Autpitoquotient}  \mathrm{Aut}_{\pi} (S \ltimes \U)  & \To &  \mathrm{Aut}(\U)^S / \U^S   \\
{[}(b,\phi)] & \mapsto & [\phi]\ ,  \nonumber 
\end{eqnarray} 
where $[\phi]$ denotes the equivalence class with respect to  $\phi \sim \phi_a $ for any $a \in \U^S$. 
We shall also consider, for any point $s\in S(k)$, the morphism of schemes
\begin{eqnarray} s:  \mathrm{Aut}_{\pi} (S \ltimes \U)(k)  &  \To &   \U(k)  \ , \label{AuttoSmap} \\
    {[}(b,\phi)] & \mapsto&   b^s b^{-1} \ \nonumber 
\end{eqnarray} 
which is just the action on $(s,1) \in S\ltimes \U$. 
It is \emph{not} a homomorphism of groups. 
\end{defn}

\begin{cor}  Let $\GG$ be as above. Given  an isomorphism 
$\sigma: \GG \overset{\sim}{\To}  S\ltimes \U\ , $
there is a canonical isomorphism (depending on $\sigma$):
$$\mathrm{Aut}_{\pi}(\GG) \cong   \U \rtimes_{\U^S} \mathrm{Aut}(\U)^S\ .$$
\end{cor} 

\begin{rem}  In our examples, we shall refer to $b$ as the `\emph{geometric}' component of a representative $(b,\phi)$ of an automorphism $[(b,\phi)]$ and $\phi$ 
as its `\emph{arithmetic}' component.
\end{rem}

\begin{rem} \label{remimagerest} 
We can characterize the points in the image of the  restriction map $(\ref{restrictionAut})$ as follows.
Call  $\alpha \in \mathrm{Aut}(\U)$  \emph{essentially  $S$-invariant}, if, for one, and hence any,  choice $\sigma: S \rightarrow \GG$ of splitting, there exists a cocycle $c \in Z^1(S,\U)$ such that
\begin{equation} \label{alphaimagecond2}
\alpha(u)^s = c_s \alpha(u^s) c_s^{-1} \ . \end{equation} 
The image of  the points of $\mathrm{Aut}_{\pi}(\GG)$ under $(\ref{restrictionAut})$ is the set of essentially $S$-invariant automorphisms.
To see this,   there exists a $b\in \U$ such that $c_s = b^s b^{-1}$,  since $H^1(S,\U)$ is trivial. If we define
  $\phi(u) = b^{-1} \alpha(u) b$ then  $[(b,\phi)] $ is a well-defined element of    $ \U \rtimes_{\U^S} \mathrm{Aut}(\U)^S$ 
  which restricts to  $\alpha$.  Conversely, given $[(b, \phi)] \in \U \times \mathrm{Aut}(\U)^S$, its restriction to $\U$ is the map 
  $\alpha(u) = b \phi(u) b^{-1}$, which satisfies $(\ref{alphaimagecond2})$ with $c_s =  b^s b^{-1}$.
   \end{rem}

\subsection{Description of $\mathrm{Aut}_{\U}(S \ltimes \U)$}  \label{sectAutUSltimesU} We now describe a general automorphism, or equivalently, the fibers of the map 
$$q: \mathrm{Aut}_{\U}(S \ltimes \U) \To \Aut(S)\ ,$$
which admit a left and right action by  $\mathrm{Aut}_{\pi}(S \ltimes \U)$.
For any $\chi \in \mathrm{Aut}(S)$,  define
\begin{equation}\label{AutUchi} \mathrm{Aut}(\U)^{S,\chi} = \{ \phi\in \mathrm{Aut}(\U) \hbox{ such that }  \phi(u)^{\chi(s)} = \phi(u^s) \} \ .
\end{equation} 
Composition of automorphisms defines a map 
$$\phi, \phi' \mapsto \phi \phi': \mathrm{Aut}(\U)^{S,\chi} \times \mathrm{Aut}(\U)^{S,\chi'}\To \mathrm{Aut}(\U)^{S,\chi\chi'}  \ .$$
In particular, $(\ref{AutUchi})$ is stable under pre- or post-composition by  $\mathrm{Aut}(\U)^S =\Aut(\U)^{S,\id}$. 
It is also stable by conjugation by an $S$-invariant element of $\U$, via the map $\phi \mapsto \phi_a$ for $a\in \U^S$.   
Therefore, we can  define
$$ \U \times^{\U^S} \mathrm{Aut}(\U)^{S,\chi} $$
to be the space of equivalence classes $[(b,\phi)]$, where $b\in \U$ and $\phi \in \mathrm{Aut}(\U)^{S,\chi}$ modulo the action of $\U^S$. It is a functor from commutative algebras to sets.   Any such equivalence class defines an automorphism of $S\ltimes \U$ via the formula
\begin{equation} \label{bphichiaction}
{[}(b, \phi)] \circ (s,u)   =  (  \chi(s) ,   b^{\chi(s)} \phi(u) b^{-1}) \ .
\end{equation} 
It is straightforward to check that $(\ref{bphichiaction})$ is well-defined and  an automorphism. 
\begin{prop} The map $(\ref{bphichiaction})$ defines an  isomorphism  
$$ \U \times^{\U^S} \mathrm{Aut}(\U)^{S,\chi}  \overset{\sim}{\To} q^{-1} (\chi)\ . $$
\end{prop} 
\begin{proof}
Let $\alpha \in \mathrm{Aut}_{\U}(S \ltimes \U) $
such that $q(\alpha) =\chi\in \mathrm{Aut}(S)$. The map $c: S \rightarrow U$ defined by  $\alpha(s,1) = (\chi(s), c_s)$  is a cocycle in  $Z^1(S, \U)$, where the action of $S$  on $\U$ is twisted by  $\chi$, i.e., $c_{st} = c_s^{\chi(t)} c_t$ for all $s,t\in S$.  Therefore $c_s= b^{\chi(s)} b^{-1}$ for some $b\in \U$, and the proof proceeds in essentially the same manner as proposition \ref{propAutassemidirec}. 
\end{proof}

In particular, if  $\chi$ is the image of an element of 
$\mathrm{Aut}_{\U}(S \ltimes \U)$,  then $(\ref{AutUchi})$ is non-empty and hence a 
 $\mathrm{Aut}(\U)^S$-torsor. 

Therefore, given any splitting $\GG \cong S \ltimes \U$, the proposition provides an explicit description of the fibers of  the map $q$ defined in  $(\ref{restrictionAut}).$

The composition law on fibers is formally given by 
\begin{eqnarray}   \big( \U \times^{\U^S} \mathrm{Aut}(\U)^{S,\chi} \big)  \times  \big( \U \times^{\U^S} \mathrm{Aut}(\U)^{S,\chi'}\big)  & \To  & \big( \U \times^{\U^S} \mathrm{Aut}(\U)^{S,\chi\chi'} \big) \nonumber \\
{[}(b,\phi)] \circ [(b', \phi')] & = & [(b \phi(b'), \phi \phi')]\ . \nonumber
\end{eqnarray}

\subsection{Automorphisms tangent to the identity}
In the applications,  $\GG$  is an algebraic group in a Tannakian category of mixed Hodge realisations, and 
the affine rings of both its reductive quotient $S$, and also the   abelianization $\U^{ab} = \U/[\U, \U]$ (but not $S \ltimes \U^{ab}$) will be groups in the subcategory of semi-simple objects.   For this reason, we wish to  consider automorphisms which act trivially on $ S$ and $\U^{ab}$. 

\begin{defn} Denote by 
\begin{eqnarray} 
\Aut'(\U)  &= & \ker ( \Aut(\U) \To \Aut(\U^{ab}) ) \nonumber \\ 
\Aut_{\pi}'(\GG) &= & \ker ( \Aut_{\pi}(\GG) \To \Aut(\U^{ab}) ) \nonumber
 \end{eqnarray} 
the schemes  of automorphisms whose restriction to the abelianization of $\U$ is trivial, where the map on the second line
is  $(\ref{restrictionAut})$ followed by restriction to $\U^{ab}$. 
\end{defn} 

\begin{cor} \label{corAutprimedecomp}  Any splitting $\sigma: \GG \cong S \ltimes \U$ gives rise to an isomorphism 
\begin{equation} 
\mathrm{Aut}'_{\pi}(\GG) \cong  \U \rtimes^{\U^{S}} \mathrm{Aut}'(\U)^S \ .
\end{equation} 
\end{cor} 
\begin{proof} Let  $[(b,\phi)] \in \U \rtimes^{\U^S} \Aut(\U)^S$ represent an element of $\Aut_{\pi}(\GG)$.   Its restriction to  $\Aut(\U)$   is $u \mapsto b \phi(u) b^{-1}$. The class of  this automorphism in  $\Aut(\U^{ab})$ is the identity if and only if  the image of $\phi$ in $\Aut(\U^{ab})$ is the identity.
\end{proof}

\subsection{Lie algebras of derivations} We now turn to  the infinitesimal versions of the above automorphism groups. Fix a splitting $\GG \cong S \ltimes \U$. Let $\uu$ denote the Lie algebra of $\U$.  Recall that since  $\U$ is pro-unipotent,  the exponential map defines an isomorphism of affine schemes
$\uu \overset{\sim}{\rightarrow} \U$.  The action of $S$ on $\U$ induces a right action of $S$ on $\uu$.

Let $\mathrm{Der} (\uu)^S$ denote the Lie algebra of $S$-equivariant derivations on $\uu$. It is the  $k$-vector space of  linear maps 
$\delta : \uu \rightarrow \uu$ such that $\delta(x^s) = \delta(x)^s$ for all $s\in S$, $x \in \uu$,  and such that $\delta [x,y]= [\delta(x), y] + [x, \delta(y)]$ for 
all $x,y \in \uu$. It is  equipped with the bracket $[\delta, \delta'] = \delta \delta' - \delta' \delta$.  Recall that  the semi-direct product
$$ \uu \rtimes \mathrm{Der}(\uu)^S$$
is the Lie algebra whose underlying vector space is $\uu \oplus  \mathrm{Der}(\uu)^S$, equipped with the Lie bracket which is  given by the formula
\begin{equation} \label{semidLieformula} \big[ (b, \delta) , (b', \delta') \big] = ( [b,b'] + \delta(b') - \delta'(b),   [\delta, \delta'])\ .
\end{equation}
There is a natural map 
\begin{eqnarray} \uu^S  &\To &   \uu \rtimes \mathrm{Der}(\uu)^S \nonumber \\
a & \mapsto & ( a, - \mathrm{ad}(a)) \nonumber 
\end{eqnarray} 
where $\mathrm{ad}(a)(x) = [a,x]$ for $x \in \uu$. 
Denote its cokernel by $ \uu \rtimes^{\uu^S} \mathrm{Der}(\uu)^S$. It is the  vector space of pairs $(b, \delta) \in  \uu \oplus  \mathrm{Der}(\uu)^S$ modulo the 
equivalence relation 
\begin{equation} \label{bdeltaequivrel} 
(b, \delta ) \sim (  a  + b , \delta - \mathrm{ad}(a)    )\qquad \hbox{ for } a \in \uu^S \  . 
\end{equation} 
We shall denote the equivalence classes by $[(b,\delta)]$. In the examples, we  call $b$ the `geometric' part and $\delta$ the `arithmetic' parts of a representative of such  a derivation. 
\begin{lem}
An isomorphism  $\sigma: \GG \cong S \ltimes \U$ induces an isomorphism
$$ \mathrm{Lie}\, \mathrm{Aut}_{\pi}(\GG) \cong \uu \rtimes^{\uu^S} \mathrm{Der}(\uu)^S$$
\end{lem}
\begin{proof}
Apply the Lie algebra functor to $(\ref{USexactsequence})$  to obtain an exact sequence
$$ 0 \To \uu^S \To \uu \oplus \mathrm{Der}(\uu)^S \To \mathrm{Lie}\, \mathrm{Aut}_{\pi}(\GG)  \To 0 \ .$$
The Lie algebra of a semi-direct product is the semi-direct product of Lie algebras. 
\end{proof} 

\subsubsection{Representations} We shall use three methods to detect elements in this Lie algebra. Firstly, 
the derivative of the  restriction map $(\ref{restrictionAut})$ gives a representation 
\begin{eqnarray} \label{usemidirrepresentation} 
\uu \rtimes^{\uu^S} \mathrm{Der}(\uu)^S & \To  & \mathrm{Der}(\uu)   \\
{[}(b, \delta)] &\mapsto &  \mathrm{ad}(b) + \delta  \nonumber
\end{eqnarray}
which is evidently well-defined. 
Secondly, the differential of  $(\ref{Autpitoquotient})$ is 
\begin{eqnarray}  \label{usemidirrep2}
\uu \rtimes^{\uu^S} \mathrm{Der}(\uu)^S & \To  & \mathrm{Der}(\uu)^S / \uu^S   \\
{[}(b, \delta)] &\mapsto &  [ \delta] \ .  \nonumber
\end{eqnarray}
Thirdly, an element $s \in S(k)$ defines a linear map
\begin{eqnarray}  \label{usemidirrep3}
s: \uu \rtimes^{\uu^S} \mathrm{Der}(\uu)^S & \To  & \uu^S    \\
{[}(b, \delta)] &\mapsto &  b^s -b \ .  \nonumber
\end{eqnarray}

The Lie algebra of $\mathrm{Aut}'(\U)$ is the Lie subalgebra of  derivations 
$$\mathrm{Der}'(\uu) = \{ \delta \in \mathrm{Der}(\uu):  \delta(x) \equiv x \mod \,  [\uu,\uu] \}$$
which are trivial on $\uu^{ab}$.  Given a splitting as in corollary  $\ref{corAutprimedecomp}$ above,  one has
$$
\mathrm{Lie} \, \mathrm{Aut}'_{\pi}(\GG) \cong  \uu \rtimes^{\uu^{S}} \mathrm{Der}'(\uu)^S \ .$$
\begin{rem}
The automorphisms of a free Lie algebra have been studied in \cite{Reut}. In the case when $\uu$ is free, 
 a derivation $\delta \in \mathrm{Der}'(\uu)^S$ is uniquely determined by  the images of generators of $\uu$ under $\delta - \id$.
Thus 
$$\mathrm{Der}'(\uu)^S = \mathrm{Hom}_S( \uu^{ab}, [\uu, \uu])$$
and also
$
\mathrm{Lie} \, \mathrm{Aut}'_{\pi}(\GG) \cong  \uu \rtimes^{\uu^{S}} \mathrm{Hom}_S( \uu^{ab}, [\uu, \uu]).$
\end{rem}

\subsection{A lower central series filtration} \label{sectLCS}
The various automorphism groups in this section can be described using the lower central series.
 Let 
 $$L^0 \U = \U  \ , \   L^1 \U  = [\U, \U] \ , \  \ldots \ , \  L^{n+1} \U = [\U, L^{n}\U]$$ denote the lower central series of $\U$, and define  $ \GG_n = \GG/ L^{n}\U$.
Since the lower central series is stable under automorphisms, elements of $\mathrm{Aut}_{\U}(\GG)$ preserve $L^n\U$ and hence act upon each $\GG_n$.  Define  a  decreasing filtration of closed subgroup schemes
$$L^n  \mathrm{Aut}_{\U}(\GG) = \ker \big( \mathrm{Aut}_{\U}(\GG) \rightarrow \mathrm{Aut}(\GG_n) \big)\ .$$
Observe that  $\GG_0 = S$, and that there is an exact sequence
$$ 1 \To \U^{ab} \To \GG_1 \To S \To 1\ .$$
We conclude that  $L^0 \Aut_{\U}(\GG) =  \Aut_{\pi}(\GG) $ and 
$$ L^1 \Aut_{\U}(\GG)   \  \leq \ \Aut'_{\pi}(\GG) $$
The corresponding filtration on $\mathrm{Aut}_{\pi}(\GG)$ satisfies $L^k \mathrm{Aut}_{\U}(\GG)= L^k \mathrm{Aut}_{\pi}(\GG)$ for $k\geq 0$.

\begin{lem} \label{lemLCS} Fix a splitting $\GG = S \ltimes \U$. Let $k\geq 0$.  The set of points  of $L^k \Aut_{\pi}(\GG)$ can be represented by  pairs $(b,\phi)$
such that $b \equiv 1 \pmod{ L^k \U}$ and $\phi \equiv \id \pmod{ L^{k} \U}$.
\end{lem}
\begin{proof} Suppose $[(b,\phi)]$ acts trivially on $S\ltimes \U/L^k \U$. 
By $[(b,\phi)]\circ (s,1) = (s, b^s b^{-1})$, this implies that  $b^s b^{-1} \equiv 1 \mod L^k \U$, and hence 
$b^s =  c_s b$ for some $c_s \in L^k \U$. Then $c_s$ is a cocycle:  $c_s \in Z^1(S, L^k \U)$. Since $S$ is pro-reductive,  $c_s$ is a coboundary, and 
there exists $a\in L^k \U$ such that  $c_s = a^s a^{-1}$. Modify $[(b,\phi)]$ by the $S$-invariant element $b^{-1}a$, to obtain 
$[(b,\phi)] = [(a, \phi_{b^{-1}a})]$. 
Therefore we can assume that  $b\equiv 1 \pmod{ L^k \U}$.  In  this case $[(b,\phi)]\circ (1 , u ) = ( 1 , b \phi(u) b^{-1})$ and $b \phi(u) b^{-1} \equiv \phi(u) \pmod{L^k \U}$. Since $[(b,\phi)]$ acts trivially on $S\ltimes \U/L^k \U$, we have $\phi(u) \equiv u \pmod{L^k \U}$ for all $u \in U$.%
\end{proof}

\subsection{Left and right actions} \label{sectLtoR}
The entire discussion can be repeated with right instead of left actions. If all automorphisms now act  on the right,  
the action of   $\mathrm{Aut}_{\U}(\GG)$ on $(s,u) \in S\ltimes \U$ can be expressed by the formula
$$(s, u) \circ [(b,\phi)]   =   (\chi(s),   b^{-1}\big|_{\chi(s)} \phi(u) b)$$
where $\phi \in \mathrm{Aut}(\U)^{S,\chi}$, is now viewed as a  right automorphism (strictly speaking, one should write $(u)\phi$ or $u|_{\phi}$ instead of $\phi(u)$).  The equivalence relation on pairs $(b,\phi)$ is now via left-action of $\U^S$, namely
$(b, \phi) \sim (ab, \phi_{a^{-1}})$.
On the level of Lie algebras, this will only affect the previous  formulae  by a largely unimportant sign. 

\section{Relative completion and cocycles} \label{sectNew2}

In the case when the group scheme $\GG$ is the relative completion of  a group $\Gamma$, the results of the  previous section 
can be translated in terms of $\Gamma$-cocycles. We first recall some background on relative completion of a group. 
The main reference is \cite{HaMHS}.

\subsection{Relative completion of a group} \label{sectRelCompGroup}
Let $\Gamma$ be a  group, $k$ a field of characteristic zero, and $S$ a (pro-)reductive affine group scheme. Consider a homomorphism
$$\rho: \Gamma \To S(k)$$
which we assume to be Zariski-dense. To this data one associates the relative completion $\GG_{\Gamma}$, which is an affine group scheme over $k$, equipped with a projection 
\begin{equation} \label{piGGGammtoS} \pi : \GG_{\Gamma} \To S
\end{equation}
whose kernel $\U_{\Gamma}$ is pro-unipotent. It is equipped with  a natural map $\widetilde{\rho}: \Gamma \rightarrow \GG_{\Gamma}(k)$ which is Zariski-dense, and whose composition with $\pi$ is $\rho$ on $k$-points. 

Relative completion  satisfies the following universal property. 
Let $G$ be any affine group scheme over $k$,  extension of $S$ by a pro-unipotent affine group scheme $U$
$$ 1 \To U \To G \overset{\pi}{\To} S \To 1\ .$$
Given any homomorphism $\alpha:\Gamma \rightarrow G(k)$ such that $\pi \alpha = \rho$, there exists a unique morphism of affine group schemes over $k$
$$\widetilde{\alpha}: \GG_{\Gamma} \To G$$  
such that $\widetilde{\alpha}\widetilde{\rho} = \alpha$ on $k$-points. In the case when $S=1$ is the trivial group, the relative completion $\GG_{\Gamma}$ is  the pro-unipotent (Mal\v{c}ev) completion of $\Gamma$. 

\subsubsection{Tannakian definition} Relative completion can be defined as follows. 
Consider the category $\mathrm{Rep}_{\Gamma,\rho}$ whose objects are finite-dimensional $\Gamma$-representations  $V$ over $k$, equipped with  a finite increasing filtration  by  sub-representations $V_i:$
$$0 \subset  V_1 \subset V_2 \subset   \ldots \subset  V_n =V$$
with the property that the successive quotients $V_i/V_{i+1}$ are $S$-modules, and the action of $\Gamma$ upon them factors through the map $\rho$.  Then 
$\mathrm{Rep}_{\Gamma,\rho}$  is a neutral Tannakian category, with  fiber functor $\omega: \mathrm{Rep}_{\Gamma,\rho} \rightarrow \mathrm{Vec}_k$ defined by forgetting the filtration and $\Gamma$-action. We define $\GG_{\Gamma} =\mathrm{Aut}_{\mathrm{Rep}_{\Gamma,\rho}}^{\otimes}(\omega)$.  The properties of relative completion are easily deduced from this definition. For example, since $\Gamma$ acts on every $V$, it defines an automorphism of the fiber functor $\omega$ and we deduce the natural map 
$$\widetilde{\rho} : \Gamma \To \GG_{\Gamma}(k)\ .$$
This map is Zariski-dense because  an object of $\mathrm{Rep}_{\Gamma,\rho}$ is trivial if $\Gamma$ acts trivially upon it. 
Likewise, the category of $S$-representations  (equipped with the trivial filtration) defines a full Tannakian subcategory of  $\mathrm{Rep}_{\Gamma,\rho}$ and hence a morphism $\GG_{\Gamma} \rightarrow S$.

\subsection{Structure of relative completion} In this section, assume that $\Gamma$ is finitely-generated.
The functor  $\mathrm{Rep}_{\Gamma, \rho} \rightarrow \mathrm{Rep}_{\Gamma}$ which forgets the filtration gives  a map
$$H^k(\GG_{\Gamma}; V) = \mathrm{Ext}^k_{   \mathrm{Rep}_{\Gamma, \rho}}(k;V) \To \mathrm{Ext}^k_{   \mathrm{Rep}_{\Gamma}}(k;V) = H^k(\Gamma; V)\ , $$
for any  object  $V$ of $\mathrm{Rep}_{\Gamma, \rho} $. The following results are stated in \cite{HaGPS}, \S3.2.
\begin{prop}  The
 induced  map on cohomology
$$H^n (\GG_{\Gamma}; V) \To H^n(\Gamma; V)$$
 is an isomorphism for $n=1$ and injective for $n=2$. 
\end{prop}
This can be proved by hand using the universal property and  the definition of $\mathrm{Ext}^k$, for $k=1, 2$. Let $\U_{\Gamma}$ denote the pro-unipotent radical of $\GG_{\Gamma}$. There is an  exact sequence 
$$1 \To \U_{\Gamma} \To \GG_{\Gamma} \To S \To 1\  .$$
Let $\uu_{\Gamma} = \Lie \, \U_{\Gamma}$. By a Hoschild-Serre spectral sequence, one deduces the
\begin{cor}  \label{corstructureofrelcomp} Suppose that every irreducible $S$-representation is absolutely irreducible.  Then there is an isomorphism 
$$H_1(\uu_{\Gamma};k) \cong \prod_{\lambda}  H^1(\Gamma; V_{\lambda})^{\vee} \otimes_{k} V_{\lambda}$$
where $V_{\lambda}$ ranges over a family of representatives for the irreducible $S$-representations over $k$.
If $\Gamma$ has cohomological dimension $1$ then  $H_n(\uu_{\Gamma};k) =0$ for all $n \geq 2$. 
\end{cor}

\subsection{Cocycles}
Now let  $U$ be any pro-unipotent  affine group scheme over $k$ equipped with a right $S$-action, and hence  a $\Gamma$-action via $\rho$.  Consider the functor of cocycles
$Z^1(\Gamma, U)$ from the category of commutative $k$-algebras $R$ to sets. Its $R$-points consists of maps $c: \Gamma \rightarrow U(R)$ satisfying the cocycle condition 
$$ c_{gh} = c_g^{h}\,  c_h \qquad \hbox{ for all } g,h \in \Gamma \ , $$
Since this  condition is algebraic, $Z^1(\Gamma, U)$ is an affine scheme over $k$.
By remark \ref{remZ1isHom}, there is an isomorphism of schemes 
$$ \mathrm{Hom}_{\Gamma}(\Gamma, \Gamma \ltimes U) \overset{\sim}{\To} Z^1(\Gamma, U)\ .$$

\begin{lem} \label{HomsGgammaCocycles} Restriction along $\widetilde{\rho}$  defines an isomorphism of schemes
\begin{equation} \label{Z1UisHomsG} 
 \mathrm{Hom}_{\pi}(\GG_{\Gamma},  S\ltimes U)  \overset{\sim}{\To} Z^1(\Gamma; U)    \ ,
 \end{equation} 
 where the points of the scheme on the left consists of homomorphisms from $\GG_{\Gamma}$ to $S \ltimes U$ whose projection onto $S$
is the map  $ \pi$ $(\ref{piGGGammtoS})$.   Via this correspondence, 
the projection  $ \pi: \GG_{\Gamma} \rightarrow S $ maps to the trivial cocycle. 
\end{lem} 

\begin{proof}  By the universal property of relative completion, the restriction map 
$$ \mathrm{Hom}_{\pi}(\GG_{\Gamma},  S\ltimes U) \rightarrow \mathrm{Hom}_{\rho}(\Gamma, S\ltimes U)  $$
is  an isomorphism of schemes over $k$. The $R$-points of the right-hand scheme are homomorphisms $\Gamma \rightarrow (S \ltimes U) (R)$ 
whose projection onto $S(R)$ is $\rho$.  The lemma follows from the canonical isomorphism 
 $ \mathrm{Hom}_{ \rho}(\Gamma, S\ltimes  U ) \overset{\sim}{\rightarrow} Z^1(\Gamma; U)  $
 obtained by composing with the morphism of schemes $S\ltimes U \rightarrow U$. 
   \end{proof}
   
\subsection{Action of automorphisms}  
The group of automorphisms $\mathrm{Aut}_{\pi}(S\ltimes U)$ acts on  $\mathrm{Hom}_{\pi}(\GG_{\Gamma},  S\ltimes U)$  on the left. 
Consequently, by $(\ref{Z1UisHomsG})$,     the group   $U  \rtimes^{U^{S}} \mathrm{Aut}(U)^{S}$
also acts on the left on $Z^1(\Gamma, U)$. It does so via  the formula 
\begin{equation} \label{[bphi]oncocyc} [(b, \phi)] \circ c_g = b^g \phi(c_g) b^{-1}\ , 
\end{equation}
where $b \in U$, and $\phi \in \mathrm{Aut}(U)^{S}$.  This  action preserves the cocycle condition. 

\begin{rem} We can think of this action as follows.   There is a natural transformation  of functors from commutative $R$-algebras to sets sending a cocycle to its equivalence class (note that $H^1(\Gamma,U)$ has no reason in general to be representable):
$$ Z^1(\Gamma, U) \To H^1(\Gamma, U)\ . $$
Recall that  cocycles $c,c'$ are equivalent $c \sim c'$ if there exists $b \in U$ with $c'_g = b^g c_g b^{-1}$. 
 If we think of $Z^1(\Gamma,U)$ as the total space, and $H^1(\Gamma, U)$ as the base, then  $\mathrm{Aut}(U)^S$
 acts upon  $H^1(\Gamma, U)$, and  $U$ acts on the fibers via the equivalence  relation for  cocycles. In proposition \ref{propAutassemidirec}, we think of the `arithmetic' component $\phi$ as an automorphism of the `base' $ H^1(\Gamma, U)$ and the `geometric' component $b$ as an automorphism of the fibers.
 \end{rem}

   \subsection{Cocycles with a tangency condition}  \label{sectCocyclesWithTangency}
   Now suppose that $\GG^{'}$ is an affine group scheme with pro-reductive quotient $S'$ and pro-unipotent radical $\U'$. 
   Consider an isomorphism 
   $\alpha :  \GG_{\Gamma}   \overset{\sim}{\rightarrow} \GG'$
   which respects the unipotent radicals  $\alpha: \U_{\Gamma}   \cong \U' $.  It induces  isomorphisms\footnote{Note that $(\alpha_S, \alpha^{ab}) : S \times \U^{ab}_{\Gamma}   \overset{\sim}{\rightarrow}   S' \times (\U')^{ab}_{\Gamma} $ will not in general respect the group structure on $S\ltimes \U^{ab}_{\Gamma}$, i.e., it does not coincide with the morphism 
    $\alpha_1 :  (\GG_{\Gamma})_1   \overset{\sim}{\rightarrow} \GG'_1$ as defined in  \S\ref{sectLCS}.}
  $\alpha_S : S \overset{\sim}{\rightarrow} S'$ and $\alpha^{ab}: \U^{ab}_{\Gamma} \overset{\sim}{\rightarrow} (\U')^{ab}_{\Gamma}$.     
 \begin{defn} For any isomorphism
 $$\psi= (\psi_S ,\psi^{ab})  \  : \  S \times    \U_{\Gamma}^{ab} \overset{\sim}{\To}  S' \times  (\U')^{ab} $$
 let 
$  \mathrm{Isom}_{\psi}(\GG_{\Gamma},  \GG')$ denote
  the scheme  whose points are  isomorphisms  $\GG_{\Gamma} \rightarrow \GG'$ which map $\U_{\Gamma}$ to $\U'$ and induce the maps
  $\psi^{ab}$ and $\psi_S$ on  $\U^{ab}_{\Gamma}$ and $S$, respectively. 
   \end{defn}

Let us fix a splitting of $\GG' = S' \ltimes \U'$.  Then there is a map
$$ \mathrm{Isom}_{\psi}(\GG_{\Gamma}, \GG') \To \mathrm{Hom}_{\psi_S \circ \rho}(\Gamma, S' \ltimes \U') = Z^1(\Gamma, \U')$$
using the notation of  lemma \ref{HomsGgammaCocycles}. Note that  $\Gamma$ acts on  $\U' $  via $\psi_S  \circ \rho$.

\begin{defn} \label{defnCocycleTangency} Let 
$Z^1_{\psi}(\Gamma, \U')$  be the  image of $  \mathrm{Isom}_{\psi}(\GG_{\Gamma},  \GG')$. \end{defn}

The space $Z^1_{\psi}(\Gamma,\U')$ never contains the trivial cocycle.  

\begin{cor} \label{corZ1torsor} If $\psi=(\alpha_S, \alpha^{ab})$ is the restriction of an isomorphism 
$\alpha :  \GG_{\Gamma}   \overset{\sim}{\rightarrow} \GG'$ as above, then $Z^1_{\psi}(\Gamma; \U')$
is a torsor over $\U' \rtimes_{\U'^{S'}} \mathrm{Aut}'(\U')^{S'} $.
\end{cor}

\begin{proof}   Follows immediately from the definition and proposition \ref{propAutassemidirec}.
\end{proof}

We can give a different  characterization of $Z^1_{\psi}(\Gamma; \U')$ purely in terms of cocycles. 
Consider  the composition of the  natural maps:
$$Z^1(\Gamma; \U') \rightarrow   Z^1(\Gamma; (\U')^{ab})  \rightarrow   H^1(\Gamma; (\U')^{ab})\ . $$
Compose with the isomorphism $(\psi^{ab})^{-1} : (\U')^{ab}\rightarrow \U_{\Gamma}^{ab}$ to obtain   
$$ Z^1(\Gamma; \U') \rightarrow H^1(\Gamma; (\U')^{ab}) \To   H^1(\Gamma; \U_\Gamma^{ab})\ .$$ 
Now  by corollary \ref{corstructureofrelcomp} the latter space is isomorphic to  
$$H^1(\Gamma; \U_{\Gamma}^{ab}) \cong \prod_{\lambda}  H^1(\Gamma; V_{\lambda})^{\vee} \otimes_k  H^1(\Gamma; V_{\lambda}) \ .$$
Then $Z^1_{\psi}(\Gamma; \U') \subset Z^1(\Gamma; \U')$ is the subspace of cocycles which maps to the identity in 
$\mathrm{End}( H^1(\Gamma; V_{\lambda}) ) =  H^1(\Gamma; V_{\lambda})^{\vee} \otimes_k  H^1(\Gamma; V_{\lambda}) $ in every component $\lambda$.

\section{Relative completion of $\pi_1$} \label{sectNew3}

We consider the relative Betti and de Rham versions of the fundamental group. 
Let $X$ be a smooth geometrically connected scheme over  a field $k\subset \C$. For any  point $x\in X(\C)$, denote the  topological fundamental group  by $\pi^{\tp}_1(X,x) = \pi_1(X(\C),x)$.  

\subsection{Betti and de Rham  completions} \label{sectRelpi1} Let $x \in X(k)$. 
  Suppose  we are given:
\begin{description}
\item[(B)] a full semi-simple Tannakian subcategory $\mathcal{S}^B$ of the Tannakian category   of local systems of finite-dimensional $k$-vector spaces, with fiber functor given by  $\omega_x$ the `fiber at $x$'.
Denote its Tannaka group by  $S^B = \mathrm{Aut}_{\mathcal{S}^B}^{\otimes}( \omega_x)$.  It is  a pro-reductive affine group scheme over $k$.  Since  a  local system  is equivalent to  a $\pi^{\tp}_1(X,x)$-representation, there is a natural  Zariski dense homomorphism 
\begin{equation}\label{pi1toSmap} \pi^{\tp}_1(X,x) \To S^B(k)\ .
\end{equation} 
\item[(dR)]  a  full semi-simple  Tannakian subcategory  $\mathcal{S}^{dR}$ of the Tannakian category   of  algebraic vector bundles on $X$ equipped  with an integrable connection, and regular singularities at infinity. Pull-back along  $x: \Spec(k) \rightarrow X$ defines a fiber functor  $\omega_x$. Denote its Tannaka group by  $S^{dR}= \mathrm{Aut}_{\mathcal{S}^{dR}}^{\otimes}(\omega_x)$. 
\end{description}
 Suppose furthermore that the  Riemann-Hilbert correspondence  induces an equivalence of categories $\mathcal{S}^B \otimes \C \sim \mathcal{S}^{dR} \otimes \C$. In particular, there is an isomorphism 
 \begin{equation}  \label{isomonS}   \mathrm{comp}:  S^B \times \C \overset{\sim}{\rightarrow} S^{dR} \times \C \ .\end{equation}
 Now we  define the relative completion of the fundamental groupoid of $X$.
\begin{description}
\item [(B)] Consider  the category $\mathcal{L}(X,\mathcal{S}^B)$ of local systems $V$ of $k$-vector spaces on $X$, equipped with a finite increasing filtration $0= V_0 \subset V_1 \subset \ldots \subset V_n$  by local systems, with the property that the successive quotients 
$V_i/V_{i-1}$ are isomorphic to objects of $\mathcal{S}^B$.
This forms a Tannakian category, which contains $\mathcal{S}^B$ as a full subcategory.  For any points $x,y \in X(\C)$,  the fibers at $x, y$ define fiber functors and we can set
$$\pi_1^{B, S} (X, x,y ) = \mathrm{Isom}_{\mathcal{L}(X,\mathcal{S}^B)}^{\otimes} (  \omega_x, \omega_y)\ .$$
It is an affine scheme over $k$. There is a natural groupoid structure
$$\pi_1^{B, S} (X, x,y )  \times \pi_1^{B, S} (X, y,z ) \To \pi_1^{B, S} (X, x,z ) $$
for any three points $x,y,z\in X(\C)$.
Denote the homotopy classes of paths in $X(\C)$ from $x$ to $y$ by $\pi^{\tp}_1(X,x,y)$. Since the unit interval is contractible, pull-back along a smooth path  $\gamma: [0,1] \rightarrow X(\C)$ defines an isomorphism of fiber functors which gives rise to 
a natural map
\begin{equation} \label{pitptoBetti} \pi^{\tp}_1 (X,x,y) \To \pi_1^{B, S} (X, x,y )(k)  \end{equation} 
compatible with the groupoid structure. In particular, taking $x=y$, it follows from the Tannakian definition of relative completion given in 
\S\ref{sectRelCompGroup} that
$\pi_1^{B, S} (X,x) = \GG_{\Gamma}$
is  the completion of $\Gamma = \pi^{\tp}_1(X,x)$ relative to $(\ref{pi1toSmap})$.  
\vspace{0.1in} 

\item [(dR)] Consider the category $\mathcal{A}(X,\mathcal{S}^{dR})$ of algebraic vector bundles on $X$, equipped with an integrable connection with regular singularities at infinity, and a finite increasing filtration $0 = V_0 \subset V_1 \subset \ldots \subset V_n$ by  flat algebraic sub-bundles, 
such that the successive quotients $V_i/V_{i+1}$ are isomorphic to objects of $\mathcal{S}^{dR}$. This forms a Tannakian category, containing $\mathcal{S}^{dR}$ as a full subcategory. A rational point $x \in X$ defines a neutral  fiber functor to the category of vector spaces over $k$. Define the  relative de Rham fundamental groupoid by
$$ \pi_1^{dR, S}(X,x,y) = \mathrm{Isom}_{\mathcal{A}(X,\mathcal{S}^{dR})}^{\otimes}(\omega_x, \omega_y)\ ,$$
for any pairs of points $x, y \in X$. These form a groupoid  in the category of affine schemes over $k$ as above.

\item[(comp)]  The Riemann-Hilbert correspondence gives an equivalence of categories
$$\mathcal{A}(X,\mathcal{S}^{dR}) \otimes \C  \sim \mathcal{L}(X,\mathcal{S}^{B}) \otimes \C$$
and hence a canonical comparison isomorphism  of schemes
$$\mathrm{comp}_{B, dR} : \pi_1^{B,S}(X, x, y) \times \C \overset{\sim}{\To}   \pi_1^{dR,S}(X, x, y) \times \C \ .$$
\end{description}

One  can replace $x$ and $y$ by tangential base points over $k$. Any such  tangential base point defines a fiber functor  on $\mathcal{L}(X,\mathcal{S}^B)$ and  $\mathcal{A}(X,\mathcal{S}^{dR})$, and the definitions above pass through without any essential modifications  (\cite{DeP1}, \S15).

\subsubsection{Unipotent radicals} Since  $\mathcal{S}^B$ and $\mathcal{S}^{dR}$ are full Tannakian subcategories of  $\mathcal{L}(X,\mathcal{S}^B)$ and  $\mathcal{A}(X,\mathcal{S}^{dR})$ respectively, there are natural projections
$$  \pi :    \pi_1^{B, S} (X,x) \To S^B \qquad \hbox{ and } \qquad  \pi :    \pi_1^{dR, S} (X,x) \To S^{dR}\ .$$
The comparison isomorphism $\mathrm{comp}_{B,dR}$  induces  $(\ref{isomonS})$.

Let us denote by $\U^{\bullet, S}_{X,x}$ the kernel of the map $\pi$ for $\bullet = B, dR$. The elements of $\U^{\bullet,S }_{X,x}$ act trivially on the  fibers   $\omega_x(V_i/V_{i+1})$.  It follows that $\U^{\bullet, S}_{X,x}$ is a pro-unipotent affine group scheme over $k$, and since $S^{\bullet}$ is pro-reductive, it is the pro-unipotent radical. 

Thus we have an exact sequence
\begin{equation}\label{Urelexactseq}  1 \To  \U^{\bullet, S}_{X,x} \To \pi_1^{\bullet, S} (X,x)  \To S^{\bullet} \To 1
\end{equation} 
where $\bullet = B, dR$,   to which we can apply the results of \S\ref{sectNew1} and \S\ref{sectNew2}. 

\subsubsection{Unipotent completion} \label{sectUnipcase}
Consider the special case when: 
\begin{itemize}
\item $\mathcal{S}^B$ is the category of constant local systems over $k$.  Then $S^B=1$.  
\item $\mathcal{S}^{dR}$ is the Tannakian category of vector bundles with connection generated by the trivial object $(\Or_X, d)$.  The group $S^{dR}=1$. 
\end{itemize} 
In this case,  $\mathcal{L}(X,\mathcal{S}^B)$ is the category of unipotent local systems on $X$,  and  $\mathcal{A}(X,\mathcal{S}^{dR})$   the category of unipotent vector bundles with integrable connection on $X$. We retrieve the unipotent Betti and de Rham fundamental groups:
$$\pi_1^{B, S}(X,x) = \pi_1^{B}(X,x) \qquad \hbox{ and } \qquad \pi_1^{dR,S}(X,x) = \pi_1^{dR}(X,x)\ ,$$
where $\pi_1^B(X,x)$ is the unipotent completion of $\pi_1^{\tp}(X,x)$.

\subsection{Relative completion in a Tannakian category} Consider the $k$-linear category $\T$ whose objects are triples 
$(V_{B}, V_{dR}, c) $
where $V_B, V_{dR}$ are finite dimensional vector spaces over $k\subset \C$, $c: V_{dR} \otimes \C \overset{\sim}{\rightarrow} V_{B} \otimes \C$ 
is an isomorphism, and the morphisms between objects respect this data. This defines a neutral Tannakian category with two fiber functors $\omega_B, \omega_{dR}$ which send $(V_B, V_{dR}, c)$ to  $V_B, V_{dR}$ respectively.

The affine rings of  relative completion  define an Ind-object
$$ \Or(\pi^{\rel,S}_1(X,x,y))  = \big( \   \Or(\pi_1^{B,S}(X,x,y)) \ , \   \Or(\pi_1^{dR,S}(X,x,y)) \  , \  \mathrm{comp}_{B, dR}  \big) $$
in the category $\T$.  This data  defines (the fibers of) a groupoid $\pi_1^{\rel,S}(X,x,y)$ in  $\T$. 
As a consequence, we obtain a right action of the Tannaka group
$$\GG^{\omega}_{\T} = \mathrm{Aut}^{\otimes}_{\T}(\omega)$$
on $\pi_1^{\omega,S}(X,x,y)$, where $\omega = B, dR$. This action respects the exact sequences $(\ref{Urelexactseq})$ 
and
so we deduce a canonical homomorphism
\begin{equation} \label{generalprogmap}  \GG^{\omega}_{\T}  \To  \mathrm{Aut}_{\U^{\omega, S}_{X,x}} \big(  \pi^{\omega,S}_1(X,x) \big)  \ , 
\end{equation}
where the group on the right is the group of right-automorphisms of \S\ref{sectNew1}.
A general programme is
to try to describe the image of $\GG^{\omega}_{\T}$ in the automorphism group on the right-hand side, by finding natural constraints upon its image.
These constraints provide relations between  periods (the coefficients of $\mathrm{comp}_{B,dR}$).

\begin{example}  \label{examplesofpi1relH} The following  examples are of interest. \begin{enumerate}
 \item Let  $k=\Q$,  $X= \Pro^1\backslash \{0,1,\infty\}$, $x=\tone_0$, $y= -\tone_1$ the tangent vectors $1$ (resp. $-1$) at $0$ (resp. $1$), and  let $S^{B/dR}$
 be as in  \S \ref{sectUnipcase}.  Then 
 $$\pi_1^{\rel,S}(X,\tone_0, - \tone_1) $$
 is  the image of the Betti/de Rham realisations  in $\T$ of the motivic fundamental torsor of paths of $X$ \cite{DG}. In this case we can replace $\T$ with the category of mixed Tate motives over $\Z$. 
The study of the map $(\ref{generalprogmap})$ in this case    is equivalent to the study of motivic multiple zeta values. 

 \item  The same theory as $(1)$ can be obtained by considering the relative completion of the braid group  on three strands $B_3$ relative 
 to $\Sigma_3$,  the symmetric  group on three letters. The underlying space is  $X= \Pro^1 \backslash \{0,1,\infty\} / \!\!/\Sigma_3$ which goes slightly beyond the scope of the previous set-up, but, I claim, enables one to retrieve the main theorems  describing the action of the motivic Galois group on the motivic fundamental group of $\Pro^1 \backslash \{0,1,\infty\}$  using  the results of \S\ref{sectNew1}. 
 \item Let $X$ be a modular curve over a number field $k$,  $x$ a suitable tangential base point at a cusp or CM-point, and $S^{B/dR}$ the 
categories generated by the cohomology of the universal elliptic curve.   The maps $(\ref{generalprogmap})$
are compatible with  morphisms between modular curves, and leads to an extremely rich theory. In this paper,  we shall focus only on the 
simplest possible case.  \end{enumerate}
  \end{example}

\section{Relative completion of  $\pi_1(\mathcal{M}_{1,1})$ and its mixed Hodge structure} \label{sectM11MHS}

We  focus on the special case  $X= \mathcal{M}_{1,1}$, $k= \Q$, $ x= \tone_{\infty}$.  Although the previous set-up is probably sufficient for our purposes,  computations are simplified by exploiting the existence of a limiting mixed Hodge structure on the relative fundamental group. 

\subsection{A category of realisations} Instead of $\T$,  we work in the subcategory $\mathcal{H}$ considered in \cite{DeP1, NotesMot}. Its  objects are triples  $(V_{B}, V_{dR}, c)$
where $V_{B}, V_{dR}$ are finite-dimensional vector spaces over $\Q$, and $c: V_{dR} \otimes \C \overset{\sim}{\rightarrow} V_B \otimes \C$
is an isomorphism. Furthermore,   $V_{dR}, V_B$ are equipped with finite increasing filtrations $M$ over $\Q$
such that $c: M_n V_{dR} \otimes \C \overset{\sim}{\rightarrow} M_n V_{B} \otimes \C$ for all $n$, and  $V_{dR}$ is equipped with a finite decreasing filtration $F$  over $\Q$  such that $(V_B, M, cF)$ defines a $\Q$-mixed Hodge structure.
Finally, we demand that $V_B$ be equipped with a real Frobenius involution 
$F_{\infty} : V_{B} \overset{\sim}{\rightarrow} V_B$
such that the following diagram commutes:
$$ \begin{array}{ccc}
 V_{dR} \otimes \C   & \overset{c}{ \To}    & V_B \otimes \C   \\
  \downarrow_{\id \otimes -} &   & \downarrow_{F_{\infty} \otimes -}    \\
 V_{dR} \otimes \C  &  \overset{c}{\To} &   V_B \otimes \C
\end{array}
$$
where $- : \C \rightarrow \C$ denotes complex conjugation.  The morphisms in $\mathcal{H}$ are the morphisms respecting this  data.  By \cite{DeP1}, the category $\mathcal{H}$ is Tannakian, and is  equipped with two neutral fiber functors 
$\omega_B, \omega_{dR} : \mathcal{H} \rightarrow \mathrm{Vec}_{\Q}$.
Since the weight filtration is strict, the functor $\gr^M$ is exact. 
For any fiber functor $\omega$ on $\Hc$, let us write
$$\GG^{\omega}_{\Hc} = \mathrm{Aut}^{\otimes}_{\Hc}(\omega) \ .$$

\begin{rem} We have denoted the weight filtration by $M$. The objects considered below have limiting mixed Hodge structures, and possess in particular  a second,  geometric, weight filtration $W$.  Rather than  setting up a Tannakian category of limiting mixed Hodge structures,  we shall simply consider $W$-filtered objects in $\Hc$, without  any loss of information in our particular situation. 
\end{rem}
\subsubsection{Semi-simple objects} Let $\Hc^{ss}$ denote the full  Tannakian subcategory  of $\Hc$ generated by simple objects. Denote its Tannaka groups by 
$S_{\Hc}^{\omega}= \mathrm{Aut}_{\Hc^{ss}}^{\otimes}(\omega)$, where $\omega$ is any fiber functor on $\Hc$. There is a short exact sequence
\begin{equation} 1 \To \U_{\Hc}^{\omega} \To \GG_{\Hc}^{\omega} \To S_{\Hc}^{\omega} \To 1\ ,
\end{equation}
where $\U_{\Hc}^{\omega}$ is the pro-unipotent radical of $\GG^{\omega}_{\Hc}$. 

\subsubsection{Tate objects} Define the (dual)  Tate object in $\Hc^{ss}$   to be 
$$\Q(-1) =   (\Q , \Q,  1\mapsto 2\pi i)   \ .$$
It  generates a Tannakian subcategory of $\Hc^{ss}$ of semi-simple (or `split')  Tate objects, which are directs sums of  $\Q(n) = \Q(-1)^{\otimes  -n}$ for $n\in \Z$.  Recall that for any object $V$ in $\Hc$ its Tate twists $V(n)$ are defined by $V\otimes \Q(n)$. 
 The action  of $\GG_{\Hc}^{\omega}$ on $\omega(\Q(-1)) = \Q$ defines a character 
 $ \chi: \GG_{\Hc}^{\omega} \rightarrow \G_m $ and hence an exact sequence
 \begin{equation} \label{kerchidef} 1 \To \GG_{\Hc}^{\omega'}  \To \GG_{\Hc}^{\omega}  \overset{\chi}{\To} \G_m \To 1
 \end{equation}
   where $\GG_{\Hc}^{\omega'} = \mathrm{ker} \,\chi$.  An object in $\Hc$ is called mixed Tate if its associated $M$-graded is a direct sum of Tate objects $\Q(n)$. For such an object, the $M$-filtration on its de Rham component is canonically split by the Hodge filtration $F$.

\subsection{The main  objects} \label{sectMainOb}
With the notations of \S\ref{sectRelpi1}, we shall only consider the following three cases. The field $k=\Q$ in each case. 
Let  $\mathcal{E}_{\partial/\partial q}$ denote the fiber of the universal elliptic curve over the tangential basepoint  $(\ref{introTBdefn})$ of $\M_{1,1}$.  

\begin{enumerate} \setlength\itemsep{0.02in}
\item $X= \G_m$,  $x=1$,  and we are in the setting   \S\ref{sectUnipcase}. Thus $S=1$, and the relative completion  $\pi_1^{\rel,S}(\G_m,1)$  is  the unipotent fundamental group  $\pi_1^{un}(\G_m,1)$.

\item $X=\mathcal{E}_{\partial/\partial q}^{\times}$, $x= \tone_0$ and $S$ is trivial via     \S\ref{sectUnipcase}. This is the unipotent completion of the fundamental group of   the punctured  infinitesmial Tate elliptic curve. The basepoint is the unit tangent vector at the origin and is well-defined up 
to a sign, which does not affect the unipotent completion  \cite{MEM}.

\item $X= \mathcal{M}_{1,1}$, the base-point $x=\tbp$ is the unit tangent vector at the cusp. The category $\mathcal{S}^B$ is the Tannakian category  of local systems generated by 
$\underline{H}_B= R^1 \pi_* \Q$, the homology of the universal elliptic curve $\pi: \mathcal{E}\rightarrow \mathcal{M}_{1,1}$, and $\mathcal{S}^{dR}$ the category generated by its relative de Rham cohomology  $\underline{H}_{dR}=H^1_{dR}(\mathcal{E}/\mathcal{M}_{1,1};\Q)$ equipped with the Gauss-Manin connection. We have 
$$S^B \cong \mathrm{SL}_2 \qquad \hbox{ and }  \qquad S^{dR} \cong \mathrm{SL}_2$$
The relative Betti fundamental group is the completion of $\SL_2(\Z)$ relative to $\SL_2$. The main reference  is \cite{HaGPS}.

\end{enumerate}

\begin{rem}
The conjunction of $(2)$ and $(3)$ is equivalent to studying  the relative completion of the fundamental group of $\mathcal{M}_{1,2}$ and  its relation to $\mathcal{M}_{1,1}$ (\cite{HaGPS}, 3.7). 
\end{rem}
 
 Hain has shown \cite{HaGPS} that the relative completions above have a natural limiting mixed Hodge structure. We shall encode this as follows.  
 In each of the three cases 
$$(X,x) \quad  = \quad (\G_m, 1)  \quad ,  \quad  (\mathcal{E}^{\times}_{\partial/\partial q}, \tone)    \quad , \quad  (\mathcal{M}_{1,1}, \tone_{\infty})\ ,$$ 
we have  a $W$-filtered pro-object $\pi_1^{\rel, \Hc}(X, x)$  in $\mathcal{H}$. That is to say,  Ind-objects
$$ \Or(\pi_1^{\rel, \Hc}(X, x)) =  ( \  \Or(\pi_1^{ B,S}(X, x)) \ , \  \Or(\pi_1^{dR,S}(X, x)) \ , \  \mathrm{comp}_{B,dR} \ )  $$
in the category $\mathcal{H}$, equipped with an increasing filtration $W$, and the structure of a Hopf algebra in $\mathcal{H}$ compatible with $W$.  In particular, this data consists of:

\begin{itemize}
\item  an affine group scheme  $\pi_1^{B,S}(X, x)$ over $\Q$. It is the relative completion of the topological fundamental group $\pi^{\tp}_1(X,x)\rightarrow S(\Q)$, and  equipped with a weight filtration $M$, a geometric filtration $W$ and a real Frobenius  $F_{\infty}$. 
\item  an affine group scheme $\pi_1^{dR,S}(X, x)$  over $\Q$, equipped with a  weight filtration $M$, a geometric weight filtration $W$,  and a   Hodge filtration $F$.
\item a comparison isomorphism which respects both $W$ and $M$:
$$\pi_1^{ B,S}(X,x) \times \C \overset{\sim}{\To} \pi_1^{dR,S}(X, x) \times \C$$
Its periods  can,  in principle, be computed by iterated integrals \cite{HaMHS}. 
\end{itemize} 
The existence of a limiting mixed Hodge structure actually gives much more, including a local monodromy operator, and the fact that $\gr^W$ is an $\mathrm{SL}_2$-representation in the category  $\Hc$. However, this extra structure  will follow automatically from the explicit description of the Hodge structures provided below. 

\begin{defn} \label{DefnGGs} In order to simplify the notations, let us denote the above relative completions  by $ \GG_{X}^{\omega}= \pi_1^{\rel, \bullet}(X,x)$, and their unipotent radicals  by  $\U_X^{\omega}$, i.e., 
$$ \U_{\G_m}^{\bullet} =  \GG_{\G_m}^{\bullet}\quad , \quad      \U_{\Eq}^{\bullet} =\GG_{\Eq}^{\bullet} \quad , \quad   \U_{1,1}^{\bullet}  \leq \GG_{1,1}^{\bullet} $$
where $\bullet = B, dR,$ or $ \Hc$. There is an exact sequence
$$1 \To    \U_{1,1}^{\bullet} \To  \GG_{1,1}^{\bullet} \To S^{\bullet} \To 1\ ,$$
where $S^{\bullet} \cong \mathrm{SL}_2$.  
 Denote their respective Lie algebras  by $\uu_{X}^{\bullet}$ and $\mathfrak{g}_X^{\bullet}$. 
 Our convention is to write Betti elements  in normal font, 
 and  de Rham elements in sans serif. 
\end{defn}
We now give reformulation of the results of \cite{HaGPS} in  terms of the category $\Hc$.

\subsection{The reductive quotient $S$} \label{sectSMHS}
Let $H^{\vee}=\Q(0) \oplus \Q(1)$ in the category $\Hc$ (it is essentially the dual of the object $(\ref{H^1EinH})$ defined below).  Define, for $n\geq 0$, 
$$V^{\Hc}_n = \mathrm{Sym}^n (H^{\vee}) \cong \Q(0) \oplus \Q(1) \oplus \ldots \oplus \Q(n)\ .$$ 
Write $V^{\mathcal{H}}_n  = ( V_n , V^{dR}_n, \mathrm{comp}_{B,dR})$.  Let $X,Y$ (resp. $\Xs, \Ys$) denote  Betti (resp. de Rham) generators
of $H^{\vee}= \Q(0) \oplus \Q(1)$, where   $X$ is in $M$-degree $0$ and spans $\Q(0)$, and $Y$ is in $M$-degree $-2$ and  spans $\Q(1)$. They are related by:
\begin{equation} \label{componXandY} 
\comp_{B,dR} \Xs = X  \quad   \hbox{ and } \quad \comp_{B,dR} \Ys = (2\pi i)^{-1} Y \ .
\end{equation}
We shall place $H^{\vee}$, and hence $V^{\mathcal{H}}_n$, in $W$-degree $0$. We can write 
$$\Q[X,Y] \cong \bigoplus_n V_n \quad  \hbox{ and }  \quad \Q[\Xs,\Ys] \cong \bigoplus_n V^{dR}_n\ ,$$
 where $V_n, V_n^{dR}$ are  the subspaces of homogeneous polynomials of degree $n+1$.  
For convenience,  the filtrations $M,W, F$ on $V^{dR}_n$ are given by the table:
$$
\begin{array}{|c|ccc|} \hline
dR & \,\ W &  \,\, M \,\,  &\,\,  F\,\,   \\ \hline
 \Xs & 0  & 0 &  0 \\
 \Ys & 0   & -2 & -1  \\ \hline
\end{array}
$$
The choice of basis for $H^{\vee}$ gives an identification of 
$S^B \cong  \mathrm{SL}_2$   and $S^{dR} \cong \mathrm{SL}_2$.
These act on the right of $V_n$ and $V_n^{dR}$  in the  manner of \S \ref{sectRightSL2act}.
They are the B, dR images  of an object $S^{\Hc}$ in $\Hc^{ss}$, whose affine ring  $\Or(S^{\Hc})$ is given by  
$$ \Or(S^{\Hc}) \cong  \bigoplus_{n\geq0}  (V^{\Hc}_n)^{\vee} \otimes V_n^{\Hc}  \ .$$
It is a Hopf algebra in $\Hc^{ss}$. Since $\Or(S^{\Hc})$ is split Tate,  the action of the group $\GG^{\omega}_{\Hc}$  on $S$ factors through  its quotient $\chi: \GG^{\omega}_{\Hc}\rightarrow \G_m$. Precisely, $\lambda \in \G_m(\Q) = \Q^{\times}$ acts via $(X,Y) \mapsto (X, \lambda^{-1} Y)$ and hence on  points of $S^{\omega}(\Q)$ via the formula 
\begin{equation}\label{EqnActiononSL2} 
\mathrm{SL}_2(\Q) \quad  \ni \quad \begin{pmatrix} a & b \\ c & d \end{pmatrix}  \quad  \mapsto  \quad  \begin{pmatrix} a & \lambda  b \\  \lambda^{-1} c & d \end{pmatrix} \ .
\end{equation}

The natural map $\pi^{\tp}_1(\mathcal{M}_{1,1}, \tbp) \rightarrow S^B(\Q)$ is the inclusion $\mathrm{SL}_2(\Z) \hookrightarrow \mathrm{SL}_2(\Q)$.

\subsection{The multiplicative group $\G_m$} The geometric weight filtration $W$ plays no role here: all objects described in this section lie in $W$-degree $0$. Define 
$$H^1(\G_m) = (  \ H_{B}^1(\G_m;\Q) \ , \ H_{dR}^1(\G_m;\Q) \ , \ \mathrm{comp}_{B,dR} \ )  = \Q(-1)   $$
viewed as an object of $\mathcal{H}$. The Lie algebra $\uu^{\Hc}_{\G_m}$ is free and satisfies
$$(\uu^{\Hc}_{\G_m})^{ab}  \cong   H^1(\G_m)^{\vee}\ .$$
Since $M = -2  L$, where $L$ is the lower central series filtration, we have 
$$\gr^M  \uu^{\Hc}_{\G_m} \cong \LL( H^1(\G_m)^{\vee})$$
the free graded Lie algebra on $H^1(\G_m)^{\vee}$.   Denote the de Rham generator of $H^1_{dR}(\G_m)^{\vee}$ by $\x_0$. It is dual to the generator  $[{dz\over z}]$ in $H_{dR}^1(\G_m;\Q)$. 
Then 
$$\gr^M \uu^{dR}_{\G_m}  \cong \LL (\x_0)  \ .$$
The generator $\x_0$ spans a copy of $\omega_{dR}(\Q(1))$: it  has $M$ degree $-2$, and lies in $F^{-1}$. Note that since
$ \uu^{dR}_{\G_m} $ is Tate, its $M$-filtration is canonically split by the Hodge filtration. 
The action of $\GG^{\omega}_{\Hc}$  on $\U^{\omega}_{\G_m}$ factors
through its quotient $\chi: \GG^{\omega}_{\Hc} \rightarrow \G_m$.

\subsection{The moduli space $\mathcal{M}_{1,1}$}  \label{sectpi1M11structure}

 Since $\mathrm{SL}_2(\Z)$ has cohomological dimension $1$,
corollary \ref{corstructureofrelcomp}  implies that 
$\uu^{\Hc}_{1,1}$ is  non-canonically isomorphic to the free completed Lie algebra on its abelianization  $H_1(\uu^{\Hc}_{1,1}) = (\uu^{\Hc}_{1,1})^{ab} $.  In particular, 
$$(\uu^B_{1,1})^{ab} \cong \prod_{n\geq 0} H^1(\Gamma; V_{2n})^{\vee} \otimes V_{2n}\ ,$$
where $V_{2n}$ is the Betti component of $V_{2n}^{\Hc}$ defined above. 
The action of Hecke operators in turn gives a decomposition:
$$H^1(\Gamma; V_{2n})^{\vee} \otimes V_{2n}  \cong   \big( H_{\mathrm{cusp}}^1(\Gamma; V_{2n})^{\vee} \otimes V_{2n} \big) \oplus  \big( H_{\mathrm{eis}}^1(\Gamma; V_{2n})^{\vee} \otimes V_{2n} \big)$$
where $H^1_{\mathrm{eis}}(\Gamma; V_{2n})$ is generated by the cocycle $e^0_{2n}$ of $(\ref{e^0defn})$.
The part  $H_{\mathrm{cusp}}^1(\Gamma; V_{2n})^{\vee}$ lies in $W_{-1} M_{2n-1}$, and  the Eisenstein part  
$H_{\mathrm{eis}}^1(\Gamma; V_{2n})^{\vee}$ lies  in $W_{-2n-2} M_{-2}$.

The de Rham realization satisfies
$$H_1(\uu^{dR}_{1,1}) \cong \prod_{n\geq 0}  H^1_{dR}(\mathcal{M}_{1,1};  \mathrm{Sym}^{2n} \underline{H}_{dR}  )^{\vee} \otimes V_{2n}^{dR} \ .$$

\noindent
The object $(\uu^{\Hc}_{1,1})^{ab}$ is a pro-object of the semi-simple category $\Hc^{ss}$  and admits an action of Hecke operators. As such, it admits a decomposition
$$  (\uu^{\Hc}_{1,1})^{ab} \cong \prod_{n\geq 0}   \big(e^{\Hc}_{2n+2} \oplus \bigoplus_{f }  M^{\Hc}_f \big)(1) \otimes V_{2n}^{\Hc}$$
where $e^{\Hc}_{2n+2}$ is a copy of $\Q(0)$ corresponding to the Eisenstein series of weight $2n+2$, and  $M_f^{\Hc}$ is the $\Hc$-realisation of the 
motive \cite{Scholl} of $f$, where $f$ ranges over generalised Hecke eigenspaces  in the space of cusp forms of weight $2n+2$ over $\Q$.  It has rank $2$.

After extending scalars to $\overline{\Q}$, each object $M^{\Hc}_f$ splits 
$$M^{\Hc}_f \otimes \overline{\Q} \cong \bigoplus_{f_i}   V^{\Hc}_{f_i} \otimes \overline{\Q}$$
where $f_i$ are a basis of Hecke eigenforms of weight $2n+2$ for the generalised eigenspace $f$, and $V^{\Hc}_{f_i}$ denotes the $\Hc\otimes \overline{\Q}$-realisation of  the motive
 of $f_i$. It is in fact defined over the field generated by the Fourier coefficients of $f_i$. 

From now on, we shall  incorporate the Tate twist $\Q(1)$ into our notations for simplicity. 
Therefore, the de Rham elements $\es_{2n+2}, \ms_f, \es_f$ are of type
$$\es_{2n+2} \cong \Q_{dR}(1)  \quad  , \quad   \ms_f \cong  M^{dR}_f (1) \quad , \quad \es_f \cong V^{dR}_f(1)\ .$$

In summary, there is a canonical isomorphism  
\begin{equation}
\label{bigradudrM11} \gr^M \gr^W \uu^{dR}_{1,1}\otimes \overline{\Q}  \ =\  \LL(\es_f \otimes V^{dR}_{2n},  \es_{2n+2} \otimes V_{2n}^{dR})\ ,
\end{equation}
where the right-hand side is the free bigraded Lie algebra on generators $\es_{2n+2}\Xs^i \Ys^{2n-i}$ for every $n\geq 1$ and $0\leq i \leq 2n$, and 
  $\es_f \Xs^i \Ys^{2n-i}$ for $f$ every Hecke eigenform of weight $2n+2$ and $0 \leq i \leq 2n$. The various filtrations are summarised in the following table, where the Hodge numbers are with respect to $M$ and $F$:
$$
\begin{array}{|c||c|cc|c|} \hline
dR           &   \hbox{rank}  &  \,\, W \,\,  & M  &   \hbox{ Hodge numbers } \     \\ \hline
 \es_f       & 2                    & -1               & 2n-1 &   (2n,-1) \oplus (-1,2n)   \\
 \es_{2n+2}  & 1              & -2n-2              & - 2&  (-1,-1)    \\ \hline
\end{array}
$$
\begin{rem} \label{reme'fe''f} The element $\es_f$ has rank two. We can choose a basis $\e'_f, \e''_f$ such that  $\e'_f \in F^{2n} V_{f}^{dR}(1)$, and normalise it via the formula 
\begin{eqnarray} \label{QdR} \e'_f   &=  &    f(q) (\Xs-   \Ys \log q)^{2n} d \log q    \\ 
   & = & 2\pi i f(\tau) (\Xs - 2 \pi i \tau \Ys)^{2n} d \tau  \nonumber \\ 
    & = &  2\pi i    f(\tau ) ( X- \tau Y)^{2n} d\tau  \ . \nonumber
\end{eqnarray} 
where,  in passing from the second to the third line, we replace the de Rham generators of $V_{2n}^{dR}$ with their Betti versions using the comparison isomorphism $(\ref{componXandY})$.
Thus the $\Q$-de Rham normalisations can be compared with those of $(\ref{underlinefdefinition})$ by 
\begin{equation} \underline{f}(\tau) = (2\pi i)^{2n} \e'_f \qquad \hbox{ and likewise, } \qquad  \underline{E_{2n+2}}(\tau) = (2\pi i)^{2n} \e_{2n+2}\ .
\end{equation} 
The elements $\e''_f$ are not canonically defined. 
\end{rem}

If $\mathfrak{sl}_2$ is the Lie algebra of $S^{dR}$ we have a short exact sequence
$$ 0 \To \uu^{dR}_{1,1} \To \g^{dR}_{1,1} \To \mathfrak{sl}_2 \To 0 \ ,$$
where $W_{-1} \g^{dR}_{1,1} = \uu^{dR}_{1,1}$, and $\mathfrak{sl}_2 \cong \gr^W_0  \g^{dR}_{1,1}$.
It follows that the associated $W$-graded of $\g^{dR}_{1,1}$ is canonically split:
$$\gr^W \mathfrak{g}^{dR}_{1,1} = \mathfrak{sl}_2 \ltimes \gr^W \uu^{dR}_{1,1}  \ ,$$
and consequently, any choice of splitting of the $W$-filtration provides a splitting 
\begin{equation} \label{WsplitGivesGsplit} \GG^{dR}_{1,1} \cong S^{dR} \ltimes \U^{dR}_{1,1}\ .
\end{equation}

\subsection{The infinitesimal Tate elliptic curve $\mathcal{E}^{\times}_{\partial/\partial q}$ } 
Define
\begin{equation} \label{H^1EinH} 
H^1(\mathcal{E}^{\times}_{\partial/\partial q})= (\ H_B^1(\mathcal{E}^{\times}_{\partial/\partial q};\Q)\ , \  H_{dR}^1(\mathcal{E}^{\times}_{\partial/\partial q};\Q) \ , \  \mathrm{comp_{B, dR}}\ ) \cong \Q(0) \oplus \Q(-1)
\end{equation}   to be the cohomology of $\mathcal{E}^{\times}_{\partial/\partial q}$ in  $\mathcal{H}$.   
It is placed in $W$-degree $-1$.  Its dual $H^1(\Eq) ^{\vee}$ has  Betti (resp. de Rham)  generators 
$a, b$ (resp. $\av, \bv$) satisfying
$$\mathrm{comp}_{B,dR} (\av) = (2\pi i)^{-1}\, a \qquad \hbox{ and } \qquad \mathrm{comp}_{B,dR} (\bv) =  b \ ,$$
where $\av$ is a generator of $\Q(1)$, and $\bv$ of $\Q(0)$.

 The abelianization of  the de Rham realisation of  $\uu^{\Hc}_{\Eq}$ satisfies 
$$\uu_{\Eq}^{ab} = H_1( \uu^{\Hc}_{\Eq}) \cong  H^1(\Eq)^{\vee}\ ,$$
and since $H^2(\mathcal{E}^{\times}_{\partial/\partial q})$ vanishes,   $\uu^{\Hc}_{\Eq}$ is non-canonically isomorphic to the completion of the free Lie algebra on $H^1(\Eq)^{\vee}$.  The $W$-filtration coincides with the negative of the  lower central series.  Therefore
  $$\gr^W  \uu^{\Hc}_{\Eq} \cong \LL(  H^1(\mathcal{E}^{\times}_{\partial/\partial q})^{\vee})\ ,$$
the free $W$-graded Lie algebra generated by the dual of $(\ref{H^1EinH})$. In particular, 
it is  mixed Tate. Its de Rham realisation is the free graded Lie algebra on $\av, \bv$:
$$ \gr^W \uu_{\mathcal{E}^{\times}}^{dR} \cong \LL(\av,\bv)\ ,$$
and admits a right action by $S^{dR}$.  Explicitly it is given by 
\begin{eqnarray} \LL(\av,\bv)  \times  S^{dR} & \To  & \LL(\av, \bv) \nonumber   \\ 
 (\av, \bv) \big|_{\left(\begin{smallmatrix} a & b\\c & d 
  \end{smallmatrix}\right)}
 & =  & (a  \, \av + b\,  \bv, c\, \av + d\, \bv)  \nonumber 
 \end{eqnarray} 
Its  $M$-filtration is automatically split by the $F$-filtration since it is of Tate type. In particular, $\LL(\av, \bv)$ is bigraded.  In summary:
$$
\begin{array}{|c||ccc|} \hline
dR & \,\, W \,\, & \,\, M \,\,  &\,\,  F\,\,   \\ \hline
 \av & -1  & -2 &  -1 \\
 \bv & -1  & 0 & 0  \\ \hline
\end{array}
$$

\subsection{Totally holomorphic quotient} \label{sectTotHolQuotient}
Iterated integrals of holomorphic differential forms are the periods of a certain Hopf subalgebra of the affine ring of $\U^{dR}_{1,1}$. 

\begin{defn} Define the totally holomorphic  quotient of $\U^{dR}_{1,1}$  to be its quotient by the normaliser of  the subgroup $F^0 \U^{dR}_{1,1}$ in  $\U^{dR}_{1,1}$. \end{defn} 

\begin{lem} \label{lemTotHol} It is isomorphic to $\PiU$ defined in $(\ref{Pidef})$.
\end{lem} 
\begin{proof} This follows from the explicit description of the Hodge structure of $\uu^{dR}_{1,1}$.
\end{proof} 
The choice of basis of $\PiU$ in $(\ref{Pidef})$ provides a splitting of $M, W, F$ on $\PiU$.

\subsection{Compatibilities}  \label{SectCompat} We have the following compatibilities, where morphisms are in the category $\mathcal{H}$ and furthermore respect  the filtration $W$:

\vspace{0.05in}

(i)  (Local monodromy around the cusp). There is a  morphism 
$$\kappa^{\Hc} \quad : \quad \pi^{\mathcal{H}}_1(\G_m, 1) \To \pi_1^{\mathcal{H}}(\mathcal{M}_{1,1}, \tone_{\infty})$$
which is induced by the local monodromy  homomorphism of topological groups:
$$\Z = \pi^{\tp}_1(\G_m; 1) \To   \pi^{\tp}_1(\mathcal{M}_{1,1}, \tone_{\infty})  = \mathrm{SL}_2(\Z) $$
which sends the generator $1$ (given by the class of a small loop winding around $0$ in the positive direction in $\C^{\times}$) to the element $T$. 

\vspace{0.05in}

(ii)  (Geometric monodromy on fibers). There is a right action  
$$  \pi_1^{\mathcal{H}}(\Eq, \tone_{0})  \times \pi_1^{\mathcal{H}}(\mathcal{M}_{1,1}, \tone_{\infty})\To\pi_1^{\mathcal{H}}(\Eq, \tone_{0}) \ , $$
or equivalently, a morphism:
$$ \mu \quad : \quad  \pi_1^{\mathcal{H}}(\mathcal{M}_{1,1}, \tone_{\infty})   \To \mathrm{Aut} \big( \pi_1^{\mathcal{H}}(\mathcal{E}^{\times}_{\partial/\partial q}, \tone_{0}) \big) \ , $$
where the right-hand side is the group of right automorphisms. It is  induced by the monodromy action of  $\pi^{\tp}_1(\mathcal{M}_{1,1}, \tbp) $ on $ \pi^{\tp}_1(\mathcal{E}^{\times}_{\partial/\partial q}, \tone_{0})$  and its unipotent completion. By the universal property  of relative completion,  the action on the latter factors through the relative completion of the former.  See \cite{HaGPS}, \S3.7.

\subsection{Splittings} \label{sectSplittings} In order to write down formulae, it is useful to choose splittings of the $W$ and  $M$-filtrations which are consistent with the 
morphisms above. This can always be done by the argument of \cite{MEM}, Appendix B. In brief,  consider the morphism
$$\GG^{\omega}_{\Hc} \ltimes \GG^{\omega}_{1,1}  \ \overset{\id \times \pi}{\To}  \ \GG^{\omega}_{\Hc} \ltimes S^{\omega}  \ \overset{\chi \times \id }{\To}  \  \G_m \ltimes S^{\omega}\ .\ $$ 
Choose any splitting $\G_m \ltimes S^{\omega} \rightarrow \GG^{\omega}_{\Hc} \ltimes \GG^{\omega}_{1,1}$. It exists by  theorem \ref{Levisplit}. It provides  in particular a splitting of  $(\ref{kerchidef})$, which is equivalent to a choice of splitting of the $M$-filtration for the $\omega$-realisation of all objects in the category $\Hc$. 
It also provides a splitting   $S^{\omega} \rightarrow \GG^{\omega}_{1,1}$ of $\pi$, which gives an isomorphism
$\GG^{\omega}_{1,1} = S^{\omega} \ltimes \U^{\omega}_{1,1}$, and an action of $S^{\omega}$ upon $\U^{\omega}_{1,1}$. The latter
splits the $W$-filtration on $\U^{\omega}_{1,1}$, since the $W$-degree is uniquely determined from the $\mathrm{SL}_2$-degrees and $M$-degrees (see $(\ref{grWuuandM})$ below).  
Since $W_0 \GG^{\omega}_{1,1} = S^{\omega}$, this provides the required splitting of the $W$-filtration on $\GG^{\omega}_{1,1}$.

Likewise, via the geometric monodromy  $(ii)$, the splitting of $\pi$ provides an action 
$$ \GG^{\omega}_{\Eq} \times S^{\omega} \To \GG^{\omega}_{\Eq}$$
 which splits the $W$-filtration on $\GG^{\omega}_{\Eq}$
 for the same reasons.

\subsection{Equivariance and $M=W$}
Let us write 
$$\uu =  \uu^{\omega}_{1,1} \quad  \hbox{ or } \quad  \uu^{\omega}_{\Eq} \ .$$
 Let us split the $M$ and $W$ filtrations, as in \S\ref{sectSplittings}.  Then $\uu$ admits a right action of $S^{\omega}$. 
 A general property  of limiting mixed Hodge structures implies that
\begin{equation} \label{grWuuandM} \gr^W_n \uu \cong  \bigoplus_{m\geq 0 } \alpha_{m+n} \otimes V^{\omega}_m
\end{equation}  where $\alpha_{m+n}$ is of $M$-degree $m+n$.
This can also be verified directly from the explicit presentations  above, since $\uu$ is free in both cases, and so it suffices to check the property 
on generators.  The following corollary is a useful mnemonic.

\begin{cor}  \label{corHwandLw} All lowest weight vectors in $\uu$ lie in the region $M\leq W$. All highest weight vectors in $\uu$ lie in $M\geq W$. 
All $S^{\omega}$-invariants lie on the diagonal $M=W$.
\end{cor} 
\begin{proof} Use $(\ref{grWuuandM})$ and the fact that highest weight vectors in $V^{\omega}_{2m}$ lie in $X^{2m}\Q$ and have $M$ degree $0$;  lowest weight vectors lie in $Y^{2m}\Q$ and have $M$ degree $-4m$.
\end{proof}

  Given $\delta \in \mathrm{Der} \, \uu$ we can uniquely decompose it according to its $(M,W)$-bidegrees:
$\delta = \sum_{m,w}  \delta_{m,w}$. 
Each component $\delta_{m,w}$ is a derivation.  From  $(\ref{grWuuandM})$, we deduce: 

\begin{cor}  \label{cordeltaM=W} If $\delta$ is $S^{\omega}$-equivariant, then $\delta_{m,w} =0$ unless $m=w$. 
\end{cor}
In other words, an $S^{\omega}$-equivariant derivation lies along the diagonal $\deg_M = \deg_W$. 

\subsection{Compatibilities on the level of Lie algebras}
 
\subsubsection{Operator $N$} \label{sectOperatorN}  The local monodromy $(i)$ gives rise to  a morphism 
\begin{equation}  \label{localmonodrN} \uu^{\Hc}_{\G_m} \To \g^{\Hc}_{1,1} \ . \end{equation} 
The data of $(i)$ is entirely determined by an element  $N \in  \g^{\Hc}_{1,1}$ which is the image of a generator of $\uu^{\Hc}_{\G_m}$.
If we choose $M$ and $W$ splittings as above, we can take the de Rham realisation  $N^{dR}$ to be the image of $\x_0$, and write it in the form
\begin{eqnarray} \gr^W \uu^{dR}_{\G_m} \cong  \LL(\x_0) & \To &  \mathfrak{sl}_2 \ltimes   \uu^{dR}_{1,1}   \\
\x_0 & \mapsto & (\varepsilon_0^{\vee},  N^{dR}_+)   \nonumber
\end{eqnarray} 
where $\varepsilon_0^{\vee}=\Xs\, {\partial/\partial \Ys}$.   Thus the data of $(i)$ is entirely determined by an element 
\begin{equation} \label{Ndef} N^{dR}_+ \in  \uu^{dR}_{1,1}\ .
\end{equation} We shall partially compute this element in \S\ref{sectNdR}.

\subsubsection{Monodromy representation}
The monodromy representation provides  a homomorphism of Lie algebras 
$$ \g^{\Hc}_{1,1} \To \mathrm{Der} \, \uu^{\Hc}_{\Eq}\ .$$
 Choose $M$ and $W$ splittings as above to identify $ \uu^{dR}_{\Eq}\cong \LL(\av,\bv)^{\wedge}$. In order to write down the image of $\uu^{dR}_{1,1}$ we require the following derivations,  first written down by Tsunogai \cite{Tsunogai}.  For all $n\geq -1$, there exists a uniqe derivation 
 $$\varepsilon_{2n+2}^{\vee}  \quad \in \quad  L_{2n+2} \,\mathrm{Der}\,  \LL(\av, \bv)$$
 where $L$ denotes the lower central series, with  the properties 
\begin{equation} \varepsilon_{2n+2}^{\vee} (\av)  =   \mathrm{ad} (\av)^{2n+2}  \bv   \qquad \hbox{ and } \qquad 
   \varepsilon_{2n+2}^{\vee} [\av,\bv ] = 0 \ .
 \end{equation} 
The special case $n=-1$ defines a derivation
 $\varepsilon^{\vee}_0= \bv {\partial  \over \partial \av}   $
 which is in  the image of a generator of $\mathfrak{sl}_2$.  In general, $ \varepsilon_{2n+2}^{\vee}$ lies in $W$-degree $-2n-2$ and $M$ degree $-4n-2$.
 
 The monodromy representation  gives rise to an $S^{dR}$-equivariant  morphism
\begin{eqnarray} \label{Liemonodromy}     \uu^{dR}_{1,1} &  \To  & \mathrm{Der} \, \LL(\av,\bv)  \\
\ms_f & \mapsto & 0   \nonumber \\
\es_{2n+2}  \Ys^{2n}  & \mapsto & {2 \over (2n)!} \varepsilon_{2n+2}^{\vee}      \nonumber 
\end{eqnarray}
The first equation follows immediately from the fact that   $(\ref{Liemonodromy})$ is $\mathfrak{sl}_2$ equivariant and preserves the $W$-filtration: since the generators $\ms_f\otimes V_{2n}^{dR}$ lie in $W=-1$, their image must lie in $W_{-1} \LL(\av,\bv) \cong \av\Q \oplus \bv\Q$, which is impossible.  Alternatively, we can use the fact that $\LL(\av,\bv)$, and hence its algebra of derivations, is mixed Tate. The elements $\ms_f$ must  therefore necessarily map to zero.    The third line of $(\ref{Liemonodromy})$   is proved in \cite{HaGPS}, \S15.

\begin{defn}  \label{defue} Let $\ue$ denote the image of $\uu^{dR}_{1,1}$  under $(\ref{Liemonodromy})$.
\end{defn} 

\subsubsection{Relations amongst the $\varepsilon_{2n}^{\vee}$}

\begin{defn} Define the \emph{ideal of relations} $R^{\Hc} \subset    \uu^{\Hc}_{1,1}  $ to be the kernel of the unipotent part of the  monodromy  morphism
\begin{equation} \label{Rdefn} R^{\Hc} = \ker  \Big(     \uu^{\Hc}_{1,1}   \To  \mathrm{Der} \,   \uu^{\Hc}_{\Eq}    \Big)\ .
\end{equation}
\end{defn} 
Let us choose $M$ and $W$-splittings as in \S\ref{sectSplittings}. Since $R^{dR}$ contains the ideal generated by the cuspidal elements $\ms_f$, 
it is natural to  define
\begin{equation}  \label{RdREis} R^{dR}_{\mathrm{eis}} = \ker  \Big(  \LL(\es_{2n+2} \otimes V^{dR}_{2n}, n\geq 1)     \To  \mathrm{Der} \,   \uu^{\Hc}_{\Eq}    \Big)\ .
\end{equation} 
The structure of these spaces will be studied in \S\ref{sectRelations}.

\subsubsection{Relation with the projective line minus 3 points} \label{sectHainmorphism}
To the three main actors  of \S\ref{sectMainOb}, one can add the unipotent fundamental group of the projective line minus three points, as in   example \ref{examplesofpi1relH}. Hain showed \cite{HaKZB} \S16-18,  that there is a morphism 
$$\Phi^{\Hc}: \pi^{\Hc}_1(\Pro^1\backslash \{0,1,\infty\}, -\tone_1) \To \pi^{\Hc}_1(\Eq, \tone_0)  $$ 
and computed it in the de Rham realisation.  This provides a mechanism to relate the periods of the infinitesmial Tate curve with those
of the projective line minus three points, i.e., multiple zeta values. We shall not make much use of this here.

\section{A group of automorphisms}
Let $\omega$ be a  fiber functor on $\mathcal{H}$. 
The Tannaka group $\GG^{\omega}_{\mathcal{H}}$ acts on  the  objects  
\begin{equation} \label{Thethreeobjects}
\U_{1,1}^{\omega} \leq  \GG_{1,1}^{\omega} \quad , \quad \GG^{\omega}_{\G_m}  \quad  , \quad  \GG_{\Eq}^{\omega}  \end{equation}
of definition \ref{DefnGGs}
in a compatible manner.  This action factors through a certain group of automorphisms which we shall define in \S \ref{sectautogroupdefn}. 

\subsection{Categories of mixed modular type}   \label{SectHeckeSS}  First of all,  the group  $\GG^{\omega}_{\mathcal{H}}$ acts on $(\ref{Thethreeobjects})$ through a certain quotient,
which by the Tannakian theorem, corresponds to a certain sub-category of $\Hc$. It is the sub-category generated by the affine rings of $(\ref{Thethreeobjects})$. 

\begin{defn} Define $\Hc_{\MMM}$ to be the full  Tannakian subcategory of $\Hc$  generated by the affine rings of $(\ref{Thethreeobjects})$.
Denote its Tannaka group   by $\GG^{\omega}_{\MMM} = \Aut^{\otimes}_{\MMM}(\omega)$, for $\omega$ any fiber functor on $\Hc$. 
Let $\Hc_{\MMM}^{ss}$ denote the  Tannakian subcategory of $\Hc$ generated by the simple objects of $\Hc_{\MMM}$.  Denote its Tannaka group by $S^{\omega}_{\MMM}$. 
\end{defn}
Let $\U^{\omega}_{\MMM}$ denote the pro-unipotent radical of $\GG^{\omega}_{\MMM}$. There is an exact sequence 
$$1 \To \U^{\omega}_{\MMM} \To \GG^{\omega}_{\MMM}\To S^{\omega}_{\MMM}\To 1\ .$$
where $S^{\omega}_{\MMM}$ is pro-reductive. Our goal  is to investigate the structure of $\U^{\omega}_{\MMM}$.

\subsubsection{Semi-simple objects and enhancement by Hecke action} \label{sectEnhanceHecke}
\begin{lem} The category $\Hc_{\MMM}^{ss}$ is the  full Tannakian subcategory of $\Hc$ generated by the objects $\Q(-1)$ and $ M^{\Hc}_f $, where  $f$ is a generalised eigenspace of cusp forms over $\Q$  with respect to the Hecke operators. 
\end{lem}

\begin{proof}
The functor $\gr^M : \Hc \rightarrow \Hc^{ss}$ is exact. Since there are no extensions in $\Hc^{ss}$, 
every exact sequence splits. In particular, applying $\gr^M$ splits the  $W$-filtration and we have
by $(\ref{WsplitGivesGsplit})$
$$ \gr^M \Or(\GG^{\Hc}_{1,1}) \cong \gr^M \Or(S^{\Hc}) \otimes \gr^M \Or(\U^{\Hc}_{1,1})\ .$$
The ring $\gr^M \Or(S^{\Hc})$ is a direct sum of Tate objects $\Q(n)$. By the structural results of  \S \ref{sectpi1M11structure},  $ \gr^M \Or(\U^{\Hc}_{1,1})$ is isomorphic to the graded dual of the universal envelopping algebra on 
$\gr^M H_1(\uu^{\Hc}_{1,1})$, which is precisely generated by tensor products of the  $M^{\Hc}_f$ and $\Q(n)$.  Similarly, $\gr^M \Or(\GG^{\Hc}_{\Eq})$ is 
a direct sum of Tate objects $\Q(n)$. 
\end{proof}
 It is convenient to extend scalars to $\overline{\Q}$. Then $\Hc_{\MMM}^{ss}\otimes \overline{\Q}$ is the 
Tannakian category generated by the  objects
$ \Q(-1)$ and $V^{\Hc}_f$,
where $f$ is a Hecke eigenform of weight $w$.  The latter satisfy the relations
\begin{equation} \label{Vfdualrel} (V^{\Hc}_f)^{\vee} = V^{\Hc}_f(w-1) \ .
\end{equation}

 \begin{lem}  The simple objects in $\Hc_{\MMM}^{ss}\otimes \overline{\Q}$ are  factors of  
  \begin{equation} \label{Symftensors}  \mathrm{Sym}^{k_1} V^{\Hc}_{f_1} \otimes \ldots \otimes \mathrm{Sym}^{k_r} V^{\Hc}_{f_r} \otimes \Q(d) 
  \end{equation} 
  for $k_1,\ldots, k_r \geq 0$ and $d \in \Z$. 
  \end{lem}
  \begin{proof}
  The exterior product  $\bigwedge^r V^{\Hc}_f $ vanishes for $r\geq 3$ and is isomorphic to $\Q(1-w)$ if $r=2$ by $(\ref{Vfdualrel})$.  It follows from the theory of Young symmetrizers that an arbitrary tensor product of $V^{\Hc}_{f_i}$ decomposes into a direct sum of objects $(\ref{Symftensors})$. The  dual of an object  $(\ref{Symftensors})$ decomposes into objects of the same type by $(\ref{Vfdualrel})$.    \end{proof}

  It is conjectured, but not known in general,   that the objects $(\ref{Symftensors})$ are simple and  independent.   To get around this issue,
  we can enhance the category $\Hc_{\MMM}^{ss}$ by equipping its objects with an action of the ring of  Hecke operators (replacing 
  $S^{\omega}_{\MMM}$ with its semi-direct product with the Hecke algebra).  
    The category $\Hc_{\MMM}$ itself can then be enhanced by demanding that every object has a semi-simplification which lies in the category $\Hc_{\MMM}^{ss}$  enriched by the Hecke action. In this new category, the simple objects are exactly those  corresponding to $(\ref{Symftensors})$.   Our main objects $(\ref{Thethreeobjects})$ lie in this enriched category.  We shall not discuss this further here, since we are primarily interested  in extensions (i.e., mixed as opposed to pure objects), which are governed by the action of $\U^{\omega}_{\Hc_{\MMM}}$.

\subsubsection{Types}

The cohomology of  $ \uu^{\omega}_{\Hc_{\MMM}}$ is a semi-simple object of $\Hc_{\MMM}$:
$$H_1( \uu^{\omega}_{\Hc_{\MMM}} ;\Q)  =    (\uu^{\omega}_{\Hc_{\MMM}})^{ab} \quad \in \quad \mathrm{Ind}\,  \Hc^{ss}_{\MMM}\ .$$

\begin{defn} \label{defnsigmatypes} \label{defmoddepth} A generator $\sigma \in  H_1( \uu^{\omega}_{\Hc_{\MMM}} ; \overline{\Q} )  $ 
is \emph{of type} $(d, f_1^{(k_1)}\times \ldots  \times f^{(k_r)}_r)$   if the $S_{\MMM}^{\omega}\times \overline{\Q}$-representation it generates is isomorphic to a subquotient of $(\ref{Symftensors})$.

 We shall call the integer $k_1+\ldots +k_r$ the \emph{modular degree}.
\end{defn}

If we choose a splitting of the $M$-filtration in $\omega$,  and hence an isomorphism of 
$ \uu^{\omega}_{\Hc_{\MMM}} \otimes \overline{\Q}$ with its associated $M$-graded, which is an Ind-object of $\Hc_{\MMM}^{ss}$,   we can say that an element $\sigma \in \uu^{\omega}_{\Hc_{\MMM}} \otimes \overline{\Q}$ is of type $(d, f_1^{(k_1)}\times \ldots  \times f^{(k_r)}_r)$ in the same manner. This notion depends of course on the choice of splitting.

\subsubsection{Mixed Tate motives over $\Z$}
Let $\Hc_{\MT(\Z)}$  denote the $\Hc$-realisation of the category of mixed Tate motives over $\Z$.
One advantage of working in the elementary category $\Hc$ is that  $\MT(\Z)$ embeds as a full subcategory. In other words
$$  (\omega_{B}, \omega_{dR}, \comp_{B,dR}) :  \MT(\Z) \To \Hc$$ 
is fully faithful \cite{DG}, so  $\MT(\Z) \rightarrow \Hc_{\MT(\Z)}$ is  an equivalence of categories.

For any fiber functor $\omega$ on $\Hc$, denote by
$$\GG^{\omega}_{\MT(\Z)} = \mathrm{Aut}^{\otimes}_{\MT(\Z)}(\omega) \cong \mathrm{Aut}^{\otimes}_{\Hc_{\MT(\Z)}}(\omega)\ .$$

\begin{thm} The category $\Hc_{\MT(\Z)}$ is a full subcategory of $\Hc_{\MMM}$.
\end{thm} 
\begin{proof} This follows from the fact  \cite{BrSigma}, theorem 3.1,  that the unipotent fundamental group of the infinitesimal Tate curve is mixed Tate over the integers, and  that the group $\GG^{\omega}_{\MT(\Z)}$ acts faithfully upon it.
\end{proof} 

\begin{rem} We believe that the category $\Hc_{\MMM}$ is  generated by  $\Or(\GG^{\Hc}_{1,1})$ alone. 
This would  imply that  $\MT(\Z)$ can be constructed purely from  modular forms \cite{BrICM}, \S6.
\end{rem}

\subsection{Constraints}   \label{sectConstraints}
The action of the group $\GG^{\omega}_{\Hc}$ on the objects $(\ref{Thethreeobjects})$ 
 defines a homomorphism to (right) automorphism groups:
\begin{equation} \GG_{\Hc}^{\omega} \To \mathrm{Aut}_{\U^{\omega}_{1,1}}\big(\GG_{1,1}^{\omega} \big) \times \mathrm{Aut}\big( \GG^{\omega}_{\Eq} \big)\ .
\end{equation}
For notational reasons, we shall write this action on the left. 
Since the maps  $(i), (ii)$  of \S\ref{SectCompat} are morphisms in the category $\Hc$,  this action is 
 constrained by the following conditions, which are not all independent (nor, presumably, exhaustive):

\begin{enumerate}  \setlength\itemsep{0.5em}

\item  (Inertia at the cusp). The action of $\GG_{\Hc}^{\omega}$ on $\GG^{\omega}_{1,1}$ is compatible with the morphism 
$$\kappa^{\omega}: \GG^{\omega}_{\G_m}  \To \GG^{\omega}_{1,1} \   , $$
i.e., $\kappa^{\omega} (g x) = g \kappa^{\omega} (x)$  for any  $x\in   \GG^{\omega}_{\G_m}$ and $g \in \GG^{\omega}_{\Hc}$.
Since $   \GG^{\omega}_{\G_m}  $ is Tate, the action of $\GG^{\omega}_{\Hc}$ on it factors through the quotient $\GG^{\omega}_{\Hc}\overset{\chi}{\rightarrow}\G_m^{\omega}$.  In particular,  
$$
 g \kappa^{\omega} =  \chi(g) \kappa^{\omega} = \kappa^{\omega} \circ \chi(g)    \qquad \hbox{ for all } g \in \GG_{\Hc}^{\omega}  \ .
 $$
 
\item  (Monodromy).  The action of $\GG_{\Hc}^{\omega}$  on $\GG^{\omega}_{1,1}$ and $\GG_{\Eq}^{\omega}$ is compatible with the monodromy action
$$ \GG_{\Eq}^{\omega} \times \GG^{\omega}_{1,1} \To  \GG_{\Eq}^{\omega} \ .$$
 Thus  for all $ x\in \GG_{\Eq}^{\omega}$, $ a \in\GG^{\omega}_{1,1}$,  and $g\in \GG^{\omega}_{\Hc}$,    we have
$ g (x . a) = g(x). g(a)$. 
Equivalently, the morphism 
$$\mu^{\omega}: \GG^{\omega}_{1,1} \To \mathrm{Aut} \,( \GG_{\Eq}^{\omega})$$
commutes with the action of $\GG_{\Hc}^{\omega}$, where $\GG_{\Hc}^{\omega}$  acts on the automorphism group  by conjugation.
In particular, if $R$ is the ideal defined in $(\ref{Rdefn})$  then 
$$
g \, R^{\omega} \subset R^{\omega} \qquad \hbox{ for all } g \in G^{\omega}_{\Hc}\ .
$$

\item (Weight filtrations). The action of $\GG^{\omega}_{\Hc}$ respects the filtrations $W$ and $M$.

\item (Semi-simple objects).   The unipotent radical   $\U^{\omega}_{\Hc}\leq \GG^{\omega}_{\Hc}$ acts trivially upon the associated graded objects with respect to the lower central series:
$$
 \gr_{L}\,  \GG^{\omega}_{1,1} \qquad \hbox{ and } \qquad   \gr_{L}\, \GG^{\omega}_{\Eq} \ .
$$
 This follows since  $ \gr_{L}\,  \g^{\Hc}_{1,1}$ is isomorphic to the semi-direct product of $\mathfrak{sl}_2^{\Hc}$, with the  free Lie algebra on $H_1(\uu^{\Hc}_{1,1})$, and   $ \gr_{L}\,  \g^{\Hc}_{\Eq}$  is isomorphic 
to the  free Lie algebra on  $H_1(\Eq)$, which are both semi-simple.

\item (Mixed Tate quotients). The action of $\GG^{\omega}_{\Hc}$ on $\GG_{\Eq}^{\omega}$ factors through   the quotient
$ \GG^{\omega}_{\Hc} \rightarrow \GG^{\omega}_{\MT(\Z)}.$
 More precisely, its action commutes with the map 
$$\Phi^{\omega}: \pi_1^{\omega} (\Pro^1\backslash \{0,1,\infty\}, -\tone_1) \To \GG^{\omega}_{\Eq}\ .$$
The left-hand fundamental group is in turn related to that of $\mathcal{M}_{0,5}$ \cite{Drinfeld}, and so  the action of $G^{\omega}_{\Hc}$ on the left hand space factors through the $\omega$-realisation of a motivic version of the  Grothendieck-Teichm\"uller group. 
\end{enumerate}

We shall  unpick   some of these constraints in the following sections.

\begin{rem} We have morphisms
\begin{eqnarray}
\GG^{\omega}_{\Hc} & \To & \mathrm{Aut}_{\U^{\omega}_{1,1}}(\GG^{\omega}_{1,1})  \nonumber  \\ 
\GG^{\omega'}_{\Hc} & \To & \mathrm{Aut}_{\pi}(\GG^{\omega}_{1,1})  \nonumber   \\
\U^{\omega}_{\Hc} & \To & \mathrm{Aut}'_{\pi}(\GG^{\omega}_{1,1})  \nonumber   \ ,
\end{eqnarray} 
where we recall that $\GG^{\omega'}_{\Hc} $ is the kernel of $\chi: \GG^{\omega}_{\Hc}\rightarrow \G_m$. 
The second equation follows from the fact that $S$ is pure Tate \S \ref{sectSMHS}. The third equation  follows from $(4)$ which implies that  the unipotent group $\U^{\omega}_{\Hc}$ acts trivially upon $ (\U^{\omega}_{1,1})^{ab}$.

\end{rem}

\subsection{Definition of a group of automorphisms} \label{sectautogroupdefn}
Since we are mainly interested in the non-Tate aspects of the relative completion $\GG^{\Hc}_{1,1}$, about which $\Phi^{\Hc}$ gives no information,
we shall drop $(5)$ and   only consider  the  constraints $(1)-(4)$. 

\begin{defn} Let $\A^{\omega}$ be the subgroup of  (right) automorphisms $\Aut_{\U^{\omega}_{1,1}}(\GG^{\omega}_{1,1})$ with the following properties: for every   $ g\in \A^{\omega}$, 
\begin{eqnarray}
&(\hbox{Inertia}) &   \qquad  g\circ  \kappa^{\omega} = \kappa^{\omega}  \circ \chi(g) \nonumber \\
&(\hbox{Relations}) & \qquad   g \, R^{\omega}  \subset R^{\omega} \nonumber \\
&(\hbox{Weights}) & \qquad   g \hbox{  respects  the  } W, M \hbox{ filtrations.}   \nonumber 
\end{eqnarray}
It is an affine group scheme over $\Q$. 
Since each of these conditions is the $\omega$-realisation of  a condition in the category $\Hc$, the group $\A^{\omega}$ is the $\omega$-realisation
of an affine  group scheme $\A^{\Hc}$ in the category $\Hc$.  
\end{defn}

Denote the pro-unipotent radical of $\A^{\omega}$ by $\A_{\U}^{\omega}$, and its pro-reductive quotient by $\A_S^{\omega}$. 
It follows from the constraints listed earlier that the action of $G^{\omega}_{\Hc}$ on $\GG^{\omega}_{1,1}$ factors through  $\A^{\omega}$.
In particular, there is a commutative diagram
$$
\begin{array}{ccccccccc}
  1 & \To  & \U^{\omega}_{\Hc_{\MMM}}    & \To & \GG^{\omega}_{\Hc_{\MMM}}   & \To & S^{\omega}_{\Hc_{\MMM}}  & \To & 1 \\
  &  &  \downarrow & & \downarrow && \downarrow &&   \\
 1 & \To  &     \A_{\U}^{\omega}  & \To &   \A^{\omega}  & \To & \A_S^{\omega} & \To & 1 
 \end{array}
$$
where all the vertical maps are injective.
The group  $S^{\omega}_{\Hc_{\MMM}}$  acts faithfully on the affine ring of  $(\U^{\omega}_{1,1})^{ab}$, since it
already generates all simple objects of $\Hc_{\MMM}$. 
From now on we focus on the unipotent radical.  

In order to study these groups further, choose an $\omega$-splitting of the $W$-filtration  \S\ref{sectSplittings}, which gives  a decomposition 
$\ref{WsplitGivesGsplit}$. Choose a generator  $\gamma^{\omega} \in \GG_{\G_m}(\Q)^{\omega}$ such that
$$ \kappa^{\omega} (\gamma^{\omega}) = (T , \kappa^{\omega}_+)$$
where $T \in S^{\omega}(\Q)= \mathrm{SL}_2(\Q)$ is the matrix \S\ref{sectGammaST}, and $\kappa^{\omega}_+ \in \U^{\omega}_{1,1}(\Q)$.

\begin{prop} \label{propAU4cond}  The points of the group $\A_{\U}^{\omega}$ are given by equivalence classes 
$$[(B, \phi)]  \quad  \in \quad    \U^{\omega}_{1,1} \rtimes^{( \U^{\omega}_{1,1})^{S^\omega}} \Aut'( \U^{\omega}_{1,1})^{S^{\omega}}$$
viewed as left automorphisms,  where $B,\phi$ satisfy \begin{eqnarray}
&(I) & \qquad       B\big|_{T}  \phi(\kappa^{\omega}_+)  B^{-1} \, = \kappa^{\omega}_+ \nonumber \\
&(R) &\qquad      B  \phi (R^{\omega}) B^{-1}  \subset R^{\omega} \nonumber \\
&(W) &\qquad     B \in W_0 M_0 \,  U^{\omega}_{1,1}     \qquad \hbox{ and } \qquad   \phi \in W_0 M_0  \Aut'( \U^{\omega}_{1,1})^{S^{\omega}} \nonumber  \ .
\end{eqnarray}
Property  $(W)$ is well-defined since $( \U^{\omega}_{1,1})^{S^\omega} \subset W_0 M_0 \,  \U^{\omega}_{1,1}$.
\end{prop} 
\begin{proof} Properties $(I)$, $(R)$   follow from the previous discussion and proposition $\ref{propAutassemidirec}$. 
Via $(\ref{Autsemi-dtoAut})$, the automorphism $ B \phi\,  B^{-1}$ respects the $M$ and $W$-filtrations on $\U^{\omega}_{1,1}$.
Since $B\in \U^{\omega}_{1,1} \leq W_{-1}  \GG^{\omega}_{1,1}  $, we deduce that $\phi \in W_0    \Aut( \U^{\omega}_{1,1} )$. Since $\phi$ is $S^{\omega}$-equivariant,   it lies in $M_0  \,   \Aut( \U^{\omega}_{1,1})$. Now $[(B,\phi)]$ is the exponential of   $[(b,\delta)] \in \Lie \, \A^{\omega}_{\U}$. We have established that $\delta \in M_0 \mathrm{Der} \, \uu^{\omega}_{1,1}$. Since $[(b,\delta)]$ acts on $\uu^{\omega}_{1,1}$ via $\mathrm{ad}(b) + \delta$, it follows 
that $\mathrm{ad}(b) \in M_0 \, \uu^{\omega}_{1,1}$. But because $\uu^{\omega}_{1,1}$ has no center, this implies that $b\in M_0 \uu^{\omega}_{1,1}$, and so $B \in M_0 \U^{\omega}_{1,1}$.  The last statement follows from the fact that $(\uu^{\omega}_{1,1})^{S^{\omega}}$ is contained in the $M=W$  line, and the fact that $\uu^{\omega}_{1,1}$ has negative $W$-degrees. 
\end{proof}

Note that  $\Or(\U^{\omega}_{1,1})$ and $\Or(S^{\omega})$  have both negative and positive $M$-degrees. In particular, the   periods of $\Or(\GG^{\Hc}_{1,1})$ are \emph{not} effective (in the sense of \cite{NotesMot}, 3.4) in general.  

\begin{rem}
Strictly speaking,  we should consider the right automorphism group (\S\ref{sectLtoR}). This has the effect of 
replacing $B$ with $B^{-1}$ in $(I)$ and $(R)$. 
\end{rem} 

\subsection{The Lie algebra of $\A_{\U}^{\omega}$}
The local monodromy gives a morphism  $(\ref{localmonodrN})$ of Lie algebras
$\omega(\Q(1))    \rightarrow \g^{\omega}_{1,1}$. 
Denote the image of a generator by  $N^{\omega}$.  Likewise, let  $(\ref{Rdefn})$
$$\rr^{\omega} = \Lie\, R^{\omega}$$
denote the kernel of the infinitesimal monodromy. The Lie algebra of $\A_{\U}^{\omega}$ consists of the derivations 
$d \in \mathrm{Der}'_{\pi}(\g^{\omega}_{1,1}) $ satisfying 
\begin{eqnarray}
& (I) & \qquad       d (N^{\omega}) = N^{\omega} \nonumber \\
& (R) &  \qquad     d \, \rr^{\omega} \subset \rr^{\omega}  \nonumber  \\
& (W) & \qquad       d \hbox{ respects } W, M   \nonumber 
\end{eqnarray}
The distinction between left and right actions is largely irrelevant (up to a sign), so we shall write the action of derivations on the left. 
After choosing a splitting of the $W$-filtration, we can write  $\g^{\omega}_{1,1} =  \mathfrak{sl}_2 \ltimes \uu^{\omega}_{1,1}$.
Write   $N^{\omega} = (  \varepsilon^{\vee}_0 , N^{\omega}_+)  \in  \mathfrak{sl}_2 \ltimes \uu^{\omega}_{1,1}.$

\begin{prop} \label{propExplicitLiealgofAU} The Lie algebra of $\A_{\U}^{\omega}$ is isomorphic to   equivalence classes
$$ [( b, \delta)]  \quad   \in \quad  \uu^{\omega}_{1,1} \rtimes^{ (\uu^{\omega}_{1,1})^{S^{\omega}}} \mathrm{Der}' ( \uu^{\omega}_{1,1})^{S^{\omega}}$$
of derivations satisfying the following properties:
\begin{eqnarray}
& (I) &   \qquad    [b, \varepsilon^{\vee}_0] + [b, N^{\omega}_+] + \delta(N^{\omega}_+) =0 \nonumber \\
& (R) &   \qquad   [b,  r] + \delta(r)  \qquad \in \qquad  \rr^{\omega}   \qquad \hbox{ for all } \qquad    r \in \rr^{\omega} \nonumber  \\
& (W) &  \qquad      b \in     W_{-1}M_{-1}  \uu^{\omega}_{1,1}  \qquad \hbox{ and} \qquad  \delta  \in  W_{-1} M_{-1}  \mathrm{Der}' ( \uu^{\omega}_{1,1})^{S^{\omega}}   \nonumber  
\end{eqnarray}
More precisely, if  $\sigma \in M_m \Lie \, \A^{\omega}_{\U}$, then it can be represented by an equivalence class  $[(b,\delta)]$ with $b\in M_{m}  \uu^{\omega}_{1,1} $
and $\delta \in M_m \mathrm{Der}'  ( \uu^{\omega}_{1,1})^{S^{\omega}} $.
\end{prop} 
\begin{proof} This is essentially  equivalent  to proposition \ref{propAU4cond}, after observing that $\delta(\varepsilon^{\vee}_0)=0$ by the $S^{\omega}$-equivariance of $\delta$.  For the last part, split $M$ and $W$, and let $\sigma = [(b,\delta)]$ be of $M$-degree $m$.  Recall that $\mathfrak{sl}_2 \cong \Q(1) \oplus \Q(0) \oplus \Q(-1)$ is generated as a Lie algebra by $\varepsilon_0, \varepsilon_0^{\vee}$, where $\varepsilon_0$ lies in $\Q(1)$ and $\varepsilon_0^{\vee}$ in $\Q(-1)$.  For any $\alpha \in \mathfrak{sl}_2$ we have $[\alpha, \sigma] = [\alpha,  b]$ since $[\delta, \alpha]=0$.  In particular, 
$[ b ,\varepsilon_0] $   has $M$-degree $m+2$, and $[b,\varepsilon_0^{\vee}] $ has $M$-degree $m-2$. These operators uniquely determine $b$ up
to an element of $(\uu^{\omega}_{1,1})^{S^{\omega}}$.  By the equivalence relation $(\ref{bdeltaequivrel})$, we can therefore assume that $b$ has $M$-degree $m$. Since $\mathrm{ad}(b) + \delta \in \mathrm{Der}\, \uu^{\omega}_{1,1}$ has $M$-degree $m$, it follows that $\delta $ also has $M$-degree $m$. 
 \end{proof}

\subsection{Modular degree} 

Let us choose a splitting of the $M$-filtration. Then, by the remark following definition \ref{defnsigmatypes}, $\Lie \, \A_{\U}^{\omega}$
decomposes into  modular types. 

\begin{lem} \label{lemModDegandLCS} Let $\sigma \in  \Lie\,  \A_{\U}^{\omega}$ have modular degree $k$. 
Then $\sigma \in L^{k-2} \, \A_{\U}^{\omega}$ and can be represented by an equivalence class $[(b,\delta)]$ with  
\begin{equation} \label{LCSgeometricnongeneric} 
b \in L^{k} \uu_{1,1}^{\omega} \qquad \hbox{ and } \qquad \delta \in L^{k-2} \mathrm{Der}\, (\uu_{1,1}^{\omega})  \ .
\end{equation}
If the $M$-degree of $\sigma$ is sufficiently  negative as a function of the weights of the $f_i$, then 
\begin{equation}\label{LCSgeometricgeneric} 
b \in L^{k+1} \uu_{1,1}^{\omega} \qquad \hbox{ and } \qquad \delta \in L^{k-1} \mathrm{Der}\, (\uu_{1,1}^{\omega})  \ .
\end{equation} 
\end{lem} 
\begin{proof} 
 Suppose that $\sigma  \in   \Lie\, \A_{\U}^{\omega}$ is of modular degree $k$.  
For all $\alpha \in \mathfrak{sl}_2$, the element $[\sigma, \alpha] = [b, \alpha]$ is  of degree at least $k$ in the cuspidal generators $\e_f$. This determines $b$ up to an $S^{\omega}$-invariant  element of $\uu^{\omega}_{1,1}$, so we can assume by modifying $b$ via  $(\ref{bdeltaequivrel})$ that  $b\in L^k \uu^{\omega}_{1,1}$. For every generator $\e$  of $\uu^{\omega}_{1,1}$,  we have $\sigma(\e) = [b,\e] + \delta(\e)$, so $\delta(\e)$ has degree  $\geq k$ in the cuspidal generators if $\e = \e_{2n}$ is an Eisenstein generator, but we can only conclude that it has degree  $\geq k-1$ if $\e$ is cuspidal, by $(\ref{Vfdualrel})$ (see the example $\ref{k=3w=-1example}$ below).   Therefore $\delta(\e) \in L^{k-1} \uu^{\omega}_{1,1}$, which implies that  $\delta \in L^{k-2}$. 
For the second part, notice that   $\e_{f}\otimes V^{dR}_{2n}$ has $M$-degree   $\geq -1-2n$ where $f$ is of modular  weight $2n+2$.  Therefore, if  the $M$-degree of $\sigma$ is sufficiently negative,   then $\mathrm{ad}(b)$ must not only increase the degree in the cuspidal elements $\e_f$ by $k-2$, but also increase the degree in the Eisenstein elements $\e_{2m+2}$ by at least 1 also. The rest of the argument proceeds  as before, on replacing  $k$  with $k+1$. 
\end{proof}
For $k=1,2$ one can obtain a  better estimate, since $\delta$ is necessarily in $L^1 \mathrm{Der}\, (\uu^{\omega}) $.
The modular degree is also related to the geometric weight filtration $W$.
The following corollary will imply that not all expected motivic extensions can occur in $\GG_{1,1}^{\Hc}$. 
\begin{cor} \label{cormoddepthandW} Let  $\sigma \in \Lie\,  \A_{\U}^{dR}$ be of modular degree $k$. 
Then either

 \qquad $(i)$   $\sigma$ lies in  $W_{2-k} \Lie \, \A_{\U}^{dR}$,

or \   $(ii)$ it can be represented in the form $[(b,\delta)]$ where $b \in W_{-k} \uu^{dR}_{1,1}$ and 
$$\delta (\e_f) =0 \hbox{ for all cuspidal generators } \e_f \ .$$
\end{cor}

\begin{proof}  By splitting the $M$ and $W$ filtrations, we can assume that $\sigma$ is of homogeneous $M$-degree. By $(\ref{LCSgeometricnongeneric})$ we can represent $\sigma$ as $[(b,\delta)]$ where 
$b \in L^k \uu^{dR}_{1,1}$ and $\delta \in  L^{k-2} \uu^{dR}_{1,1}$. It follows that $b\in W_{-k} \uu^{dR}_{1,1}$. Now suppose that
there exists a cuspidal generator $\e_f$ such that $\delta(\e_f) \neq 0$. For reasons of type, $\delta(\e_f)$ is of degree at least $k-1$
 in the cuspidal generators $\e_g$, and therefore  $ \delta(\e_f) \in L^{k-1} \uu^{dR}_{1,1} \otimes \overline{\Q} \leq W_{1-k} \uu^{dR}_{1,1} \otimes \overline{\Q} .$ 
By corollary \ref{cordeltaM=W},  $\delta$ is of definite $W$-degree (equal to its $M$-degree) and hence $\delta \in W_{2-k} \mathrm{Der} \, \uu^{dR}_{1,1}$. 
\end{proof} 

\begin{example} \label{k=3w=-1example} An element of modular degree $3$ could potentially have an arithmetic component $\delta$ of the general shape $\delta: \e_f \mapsto [\e_g, \e_h]$, where $\delta$ annihilates all other generators of $\uu^{dR}_{1,1}$ except $\e_g$ and $\e_h$. Such a derivation  would lie in $W_{-1}$. 
This example shows that the corollary is optimal in the case that such derivations exist 
(note that they would have to satisfy the third condition of $(\ref{GeomLowestWeightAndDeltaepsilon})$.)
\end{example}

 \newpage
 
 \begin{center}
 \Large{\bf{ Part  III: $\Hc$-periods, their Galois theory, and applications}}
 \end{center} 

\section{$\Hc$-periods and de Rham periods of $\pi_1^{\rel}(\mathcal{M}_{1,1}, \textstyle{\partial/ \partial q})$}

We define $\Hc$ and de Rham versions of the periods of $\M_{1,1}$, which include the iterated integrals of modular forms.
The results of  Part II enable us to compute the $\GG^{\omega}_{\Hc}$-action on these objects. This should correspond to  the action of the  conjectural motivic Galois group on the  iterated integrals computed in Part I.

\subsection{Reminders on rings of periods} Define the following rings of periods:
$$\Pe^{\mm}_{\Hc} =\Or(\mathrm{Isom}^{\otimes}_{\Hc}(\omega_{dR}, \omega_B )) \qquad \hbox{ and } \qquad \Pe^{\mm}_{\Hc^{ss}} = \Or(\mathrm{Isom}^{\otimes}_{\Hc^{ss}}(\omega_{dR}, \omega_B ))\ .$$
They are equipped with an action of Frobenius $F_{\infty}$,  a period homomorphism 
$$\per : \Pe^{\mm}_{\Hc} \To \C$$
and a left action (resp. right action) by $\GG^{dR}_{\Hc}$ (resp. $\GG^B_{\Hc}$).  They are generated by matrix coefficients 
$[ M, \gamma, \omega]^{\mm}$, where $M$ is an ind-object of $\Hc$, $\gamma \in M_B^{\vee}$, and $\omega \in M_{dR}$. This is defined to be the
function $\phi \mapsto \gamma(\phi(\omega)): \mathrm{Isom}^{\otimes}_{\Hc}(\omega_{dR}, \omega_B ) \rightarrow \A^1$.
Frobenius acts via $F_{\infty} [ M, \gamma, \omega]^{\mm} = [ M, F_{\infty}\gamma, \omega]^{\mm}$ and corresponds, via  $\per$, to complex conjugation. 
There is a universal comparison isomorphism, for every object $M$ of $\Hc$
\begin{equation} \label{univcomp} 
\mathrm{comp}^{\mm}_{B,dR}: M_{dR} \otimes \Pe^{\mm}_{\Hc} \overset{\sim}{\To} M_B \otimes \Pe^{\mm}_{\Hc}
\end{equation}
whose image under the period homomorphism is  $c: M_{dR} \otimes \C \overset{\sim}{\rightarrow} M_B \otimes \C$, which is part of the data of $M$. 
For more information about  these topics, see \cite{NotesMot}. 

\subsubsection{Lefschetz motivic period} \label{sectLef}
The matrix coefficient
$ \Lef^{\mm} = [ \Q(-1) , 1^{\vee} , 1]^{\mm}$,  where $1^{\vee}\in\Q^{\vee}$ is dual to $1\in \Q$, defines
the motivic Lefschetz period.     It is equal to 
$$\Lef^{\mm} = [H^1(\G_m), \gamma_0, [ \textstyle{dz \over z}] ]^{\mm} \quad \in \quad \Pe^{\mm}_{\Hc}$$
where $\gamma_0$ is a small loop in $\C^{\times}$ winding once around $0$ in the positive direction.  Its period is $2 \pi i$.  It is semi-simple, 
and the action of $\GG^{dR}_{\Hc}$ is via $\chi:  \GG^{dR}_{\Hc}\rightarrow \G_m$.  Indeed, by  definition of the character $\chi$, 
the element $\lambda \in \G_m(\Q) =\Q^{\times}$ maps $\Lef^{\mm}$ to $\lambda \Lef^{\mm}$.

\subsubsection{de Rham periods} The ring of de Rham periods   is the ring $\Pe^{\dR}_{\Hc} = \Or (\GG^{dR}_{\Hc})$. 
It is generated by matrix coefficients $[M, v, \omega]^{\dR}$, were $\omega \in M_{dR}$ and $v \in M_{dR}^{\vee}$. This ring carries, in particular, 
a left action by $\GG^{dR}_{\Hc}$ (via its action on $\omega$), and  comes equipped with a `single-valued' period homomorphism $\sv : \Pe^{dR}_{\Hc} \rightarrow \C$
which we shall study in    \S\ref{sectSV}. This homomorphism is not injective.

\subsection{Universal $\Hc$-periods of $\G^{\Hc}_{1,1}$.}
Let $\gamma \in \SL_2(\Z)= \pi_1^{\mathrm{top}}(\mathcal{M}_{1,1}, { \tone_{\infty}})$.
Denote its image in its  relative completion  under the map $(\ref{pitptoBetti})$  by
$$\gamma^B \quad \in \quad \GG^{B}_{1,1}(\Q)\ .$$ 
For every $\phi \in \mathrm{Isom}_{\Hc}(\omega_{dR}, \omega_B)(R)$ we obtain an element 
$\phi(\gamma^B) \in \GG^{dR}_{1,1}(R)$, where $R$ is any commutative $\Q$-algebra. Take $R =  \Pe^{\mm}_{\Hc}$ and $\phi$ to be the identity.  
\begin{defn}  The  universal element $\gamma^{\mm} \in \GG^{dR}_{1,1}( \Pe^{\mm}_{\Hc})$ is defined by $\id(\gamma^{B})$. 
\end{defn} 
Alternatively, it can be defined to be the family of matrix coefficients:
\begin{eqnarray} \label{gammamasmotper}  \gamma^{\mm}: \Or(\GG^{dR}_{1,1}) & \To &  \Pe^{\mm}_{\Hc}     \\
\omega & \mapsto &   [ \Or(\GG^{\Hc}_{1,1}) , \gamma, \omega]^{\mm}\ . \nonumber 
\end{eqnarray} 
Since composition of paths is a  morphism in the category $\Hc$, we have
$$(\gamma_1 \gamma_2)^{\mm} = \gamma_1^{\mm} \gamma_2^{\mm} \qquad \hbox{ for all } \  \gamma_1,\gamma_2 \in \SL_2(\Z)\ .$$
  Put another way, we have a canonical homomorphism
\begin{equation}\label{gammatogammaMot} \gamma \mapsto \gamma^{\mm}\quad :\quad  \SL_2(\Z) \To  \GG^{dR}_{1,1}(\Pe_{\Hc}^{\mm})\ .
\end{equation}
A third way to define this is via the universal comparison isomorphism $(\ref{univcomp})$
\begin{equation} \label{univGGcomp}  \GG^{B}_{1,1} \times \Pe_{\Hc}^{\mm} \overset{\sim}{\To} \GG^{dR}_{1,1} \times \Pe_{\Hc}^{\mm}\ .
\end{equation} 
Taking $\Pe_{\Hc}^{\mm}$-points gives an isomorphism  $\GG^B_{1,1}(\Pe_{\Hc}^{\mm}) \overset{\sim}{\rightarrow} \GG^{dR}_{1,1}(\Pe_{\Hc}^{\mm})$. Restricting this to $\SL_2(\Z) \rightarrow \GG_{1,1}^B(\Q) \leq  \GG^B_{1,1}(\Pe_{\Hc}^{\mm})$ gives back the homomorphism   $(\ref{gammatogammaMot})$.

\begin{lem}   \label{lemGaloisactiononuniversal} The left action of $\GG^{dR}_{\Hc}$ (resp. right action of  $\GG^B_{\Hc}$) on the coefficients of $\gamma^{\mm}$ is given by the right action of 
$\GG^{dR}_{\Hc}$ on $\GG^{dR}_{1,1}$  (resp.  right action of $\GG^{B}_{\Hc}$ on $\GG^B_{1,1}$).
\end{lem} 
\begin{proof} The homomorphism $(\ref{gammamasmotper})$ is equivariant for the action of $\GG^{dR}_{\Hc}$.  The action of $\GG^B_{\Hc}$ 
on the motivic period $  [ \Or(\GG_{1,1}) , \gamma, \omega]^{\mm}$ is given by its right action on $\gamma$ \cite{NotesMot}, \S2.  \end{proof}

 \subsubsection{Periods of $S$} Composing  $(\ref{gammatogammaMot})$ with the projection $\pi : \GG^{dR}_{1,1} \rightarrow S^{dR}$ gives a homomorphism
 \begin{equation}\label{motdRaction}  \gamma \mapsto \pi \gamma^{\mm} :  \SL_2(\Z) \To  S^{dR} (\Pe^{\mm}_{\Hc}) \ .
 \end{equation} 
 Since $\Or(S^{\Hc})$ is pure Tate \S\ref{sectSMHS}, its coefficients lie in $\Q [ \Lef^{\mm}, ( \Lef^{\mm})^{-1}]$. In particular, using the fact that elements of $S^{dR}$ 
 are endomorphisms of $\Q(0) \oplus \Q(1)$, we check that 
 \begin{equation} \label{piSm}  \pi S^{\mm} =      \begin{pmatrix} 0 & \Lef^{\mm} \\ (\Lef^{\mm})^{-1} & 0  \end{pmatrix}   \qquad \hbox{ and } \qquad 
 \pi T^{\mm} =      \begin{pmatrix} 1 & \Lef^{\mm} \\  0 & 1  \end{pmatrix}\ .   \
       \end{equation}

 \begin{rem}
The coefficients of $\gamma^{\mm}$ lie, by definition, in  $\Pe^{\mm}_{\Hc_{\MMM}} \subset \Pe^{\mm}_{\Hc}.$ 
 Everything that follows takes place in $ \Pe^{\mm}_{\Hc_{\MMM}}$, but  we shall write $\Pe^{\mm}_{\Hc}$ for simplicity of notation. 
 \end{rem}

\subsection{$\Hc$-cocycle} From now on let us write $\Gamma$ for $\SL_2(\Z)$, as in the first part of this paper. 
Fix a dR-splitting of the $W$-filtration on $\GG_{1,1}^{dR}$, which provides a splitting 
$(\ref{WsplitGivesGsplit})$, and in particular a homomorphism $\GG_{1,1}^{dR} \rightarrow \U^{dR}_{1,1}$.

\begin{defn} Composing the map $(\ref{gammatogammaMot})$ with  $\GG_{1,1}^{dR} \rightarrow \U^{dR}_{1,1}$  defines a cocycle
$$\CC^{\mm} \in Z^1(\Gamma; \U^{dR}_{1,1}) (\Pe_{\Hc}^{\mm})\ . $$
The action of $\Gamma$ on $\U^{dR}_{1,1}( \Pe_{\Hc}^{\mm})$ is via  $(\ref{motdRaction})$, i.e., 
$\CC^{\mm}_{gh} = \CC^{\mm}_g\big|_{\pi h^{\mm}} \CC^{\mm}_h $ for all $g, h\in \Gamma$. 
\end{defn} 

This cocycle is induced by the isomorphism $(\ref{univGGcomp})$, which restricts to 
 $$ \mathrm{comp^{\mm,ss}}:  S^B\times  (\U^{B}_{1,1})^{ab}(\Pe^{\mm}_{\Hc^{ss}}) \overset{\sim}{\To}  S^{dR} \times  (\U^{dR}_{1,1})^{ab} (\Pe^{\mm}_{\Hc^{ss}}) \ .$$
 The fact that the coefficients  lie in the subring $\Pe^{\mm}_{\Hc^{ss}} \subset \Pe^{\mm}_{\Hc}$ of semi-simple periods follows from the fact  that $S^{\Hc}$ and  $(\U_{1,1}^{\Hc})^{ab}$ are semi-simple\footnote{We shall show that the quotient  $(\GG^{\Hc}_{1,1})_1 $, extension of $S^{\Hc}$ by $(\U_{1,1}^{\Hc})^{ab}$,  \S\ref{sectLCS}, is not semi-simple.}  pro-objects of $\Hc$.   Taking the period of  the previous isomorphism and passing to its dual gives a map 
 $$\Or((\U^{dR}_{1,1})^{ab})  \otimes \C \To \Or(\U^B_{1,1})^{ab} \otimes \C\ .$$
 Its restriction to $\Or((\U^{dR,\hol}_{1,1})^{ab})  $ is equivalent to 
 the Eichler-Shimura isomorphism \S\ref{sectESisom}.
 By definition $\ref{defnCocycleTangency}$,   $\CC^{\mm}$ lies in the space of cocycles with a tangency condition:
 $$
\CC^{\mm}  \quad \in \quad Z_ {\mathrm{comp^{\mm,ss}}}^1 (\Gamma;   \U^{dR}_{1,1}) (\Pe_{\Hc}^{\mm})\ .
$$
The general properties of cocycles apply to $\CC^{\mm}$.  In particular,   since $\Gamma$ is generated by $S$ and $T$, the cocycle $\CC^{\mm}$    is determined by $\CC^{\mm}_S$ and $\CC^{\mm}_T$, and satisfies three equations identical  to those stated in lemma  \ref{lem3eq},   on replacing $C$ by $\CC^{\mm}$.

\subsubsection{Periods and canonical holomorphic cocycle} 
We can obtain information about  $\CC^{\mm}$  via the period homomorphism. A portion of its coefficients are given by the totally holomorphic
iterated integrals of modular forms studied in Part I.  Let 
$$\CC^{\mm, \hol } \in Z^1  (\Gamma ;   \U^{dR,\hol}_{1,1}) (\Pe_{\Hc}^{\mm})$$ 
denote the image of $\CC^{\mm}$ under the map  $\U^{dR} \rightarrow \U^{dR, \hol}$.

\begin{lem}  \label{lemCanFromHol}
The  canonical cocycle $\CC  \in Z^1  (\Gamma ;   \U^{dR,\hol}_{1,1}) (\C)$ satisfies
$$  \CC=  \alpha\, \mathrm{per} \, ( \CC^{\mm, \hol}) \ , $$
where $\alpha$ is the map  which scales the Betti generators $(X,Y) \mapsto (2 \pi i  X, 2  \pi i Y)$. 
 \end{lem}
\begin{proof}
This follows from lemma \ref{lemTotHol} and the normalization $(\ref{QdR})$.
\end{proof} 

The reader is warned that the canonical cocycle was written in terms of Betti elements $X,Y$ as opposed to their de Rham counterparts.

\subsubsection{Value at $T$}
Local monodromy at the cusp gives  a
 commutative diagram 
$$\begin{array}{ccc}
  \Z = \pi^{\tp}_1(\G_m,1) & \To   & \pi^{\tp}_1(\M_{1,1}, \tbp) = \SL_2(\Z)    \\
  \downarrow &   &  \downarrow  \\
  \GG_{\G_m}^{dR} (\Pe^{\mm}_{\Hc})  & \To   &   \GG_{1,1}^{dR} (\Pe^{\mm}_{\Hc}) 
\end{array}
$$
The vertical maps are the natural maps $ \gamma \mapsto \gamma^{\mm}$ from fundamental groups to the $\Q$-points of their Betti relative  completion, followed by the universal comparison isomorphism.
The generator in the top left group is the path $\gamma_0$ which winds once around $0$ in the positive direction. It maps to $T$ in the top right group.  We deduce that the map along the bottom sends $\gamma_0^{\mm}$ to $T^{\mm}$, or equivalently:
$$\exp( \Lef^{\mm} \x_0 ) \quad  \mapsto \qquad  (\pi(T^{\mm}), \CC^{\mm}_T) \quad \in \quad  S^{dR} \rtimes \U^{dR}_{1,1}(\Pe^{\mm}_{\Hc})\ .$$
The following lemma provides a means to compute $N^{dR}_+$ via periods, by  applying  the  homomorphism $\Q[\Lef^{\mm}] \rightarrow \Q$ which sends $\Lef^{\mm} $ to $1$ in the previous expression, or equivalently, by taking the period and scaling by the appropriate powers of $2 \pi i$.

\begin{lem} \label{lemNfromCmT} The coefficients of $\CC^{\mm}_T$ are powers of $\Lef^{\mm}$. 
Furthermore
$$\exp ( \Lef^{\mm} \varepsilon^{\vee}_0, \Lef^{\mm} N^{dR}_{+}) =   (\pi(T^{\mm}), \CC^{\mm}_T) \qquad  \in  \quad  S^{dR} \ltimes \U^{dR}_{1,1}( \Q[\Lef^{\mm}])\ . $$

\end{lem}

\begin{proof} Follows from the commutative diagram above, and the definition of $N^{dR}_+$ \S\ref{sectOperatorN}.  \end{proof}

\subsubsection{Frobenius and Galois action}
The real Frobenius $F_{\infty}$ acts upon $C^{\mm}$ as follows: 
$$ F_{\infty} C_{\gamma}^{\mm} = C^{\mm}_{\epsilon \gamma \epsilon^{-1}}  \ ,$$
since we have already established in \S\ref{sectRealStructure} that $F_{\infty}$ acts on $\Gamma$ by conjugating by $\epsilon$. 
This differs from the formula in  \S\ref{sectRealStructure}, since the latter was expressed using the Betti versions $X, Y$, and the action of $F_{\infty}$
on $V_{2n}$ is by right-action via $\epsilon$.

The group  $\GG^{dR}_{\Hc}$  acts on the coefficients of $C^{\mm}$ on the left. By lemma \ref{lemGaloisactiononuniversal}, this action factors through  the homomorphism 
$$\GG^{dR}_{\Hc} \To \A^{dR}\ .$$
The latter acts on cocycles on the right. 
In formulae, if the image of $g\in \GG^{dR}_{\Hc}(\Q)$  is $[(b,\phi)] \in  \A^{dR}(\Q)$, where $\phi \in \Aut(\U^{dR}_{1,1})^{S,\chi(g)}$ then for all $\gamma \in \Gamma$, 
\begin{equation} \label{gactiononCmotivic} 
g(C_{\gamma}^{\mm})  =  ( b^{-1})\big|_{\gamma^g}  \phi(C_{\gamma}^{\mm})\,  b\ . 
\end{equation} 
where $\gamma^g $ is the image of $\gamma$ under $\chi(g) \in \G_m$, as  computed in   $(\ref{EqnActiononSL2})$.

\subsubsection{Transference for general cocycles}  Let $\omega$ be a fiber functor on $\Hc$.  Fix a $W$-splitting on $\G^{\omega}_{1,1}$ and hence a decomposition $(\ref{WsplitGivesGsplit})$. 
For every  $n\geq 0$, fix $S$-invariant injections  $V^{\omega}_{2n} \hookrightarrow V^{\omega}_{2a} \otimes V^{\omega}_{2b}$ for every $|a-b|\leq n\leq a+b$. For example, one can take the dual of the $\partial^k$ of \S\ref{sectdeltakdef}.
Using the fact that $\Or(S^{\omega}) \cong \bigoplus_n (V^{\omega}_{2n})^{\vee} \otimes V^{\omega}_{2n}$, the coproduct on  $\Or(\GG^{\omega}_{1,1}) \cong \Or(S^{\omega}) \otimes \Or(\U^{\omega}_{1,1}) $  dual to multiplication in $S^{\omega} \ltimes \U^{\omega}_{1,1}$ is a map 
$$ \bigoplus_n (V^{\omega}_{2n})^{\vee} \otimes V^{\omega}_{2n} \otimes \Or(\U^{\omega}_{1,1}) \To  \bigoplus_{a,b}  (V^{\omega}_{2a}\otimes V^{\omega}_{2b})^{\vee} \otimes V^{\omega}_{2a}\otimes V^{\omega}_{2b} \otimes \Or(\U^{\omega}_{1,1})  \otimes \Or(\U^{\omega}_{1,1})\   $$
in the category of $S^{\omega} \times S^{\omega}$-representations (corresponding to the left and right actions of $S^{\omega}$ on $S^{\omega} \ltimes \U^{\omega}_{1,1}$). 
Taking $V^{\omega}_{2n}$-isotypical components with respect to the right $S^{\omega}$-action and then taking $S^{\omega}$-invariants with respect to the other,  gives  a map
$$\Delta: \Hom_{S^{\omega}}(V^{\omega}_{2n}, \Or(\U^{\omega}_{1,1})) \To \bigoplus_{a,b} \Hom_{S^{\omega}}(V^{\omega}_{2a},  \Or(\U^{\omega}_{1,1})) \otimes \Hom_{S^{\omega}}(V^{\omega}_{2b},  \Or(\U^{\omega}_{1,1}))\ .$$
where $\Hom_{S^{\omega}}$ denotes  $S^{\omega}$-equivariance. The unit of $\Or(\U^{\omega}_{1,1})$ defines a map $1: \Q = V^{\omega}_0 \rightarrow \Or(\U^{\omega}_{1,1})$. For any $w \in \Hom(V^{\omega}_{2n}, \Or(\U^{\omega}_{1,1}))$ we shall use the Sweedler notation 
\begin{equation} \label{FormulaforDeltaw} \Delta(w) = w\otimes 1 + 1 \otimes w   + \sum w' \otimes w'' \ .
\end{equation} 
Given a homomorphism $f : \Or(\U^{\omega}_{1,1}) \rightarrow R$, where $R$ is a commutative $\Q$-algebra,  write $f(w) \in V^{\omega}_{2n} \otimes R$ for composition $ f \circ w: V^{\omega}_{2n} \rightarrow R$. 
\begin{prop} Let $C : \Gamma \rightarrow  \U^{\omega}_{1,1}(R)$ be map such that $C(1)=1$.   Then $C$ is a cocycle  if and only if the following Maurer-Cartan equation holds:  for every   $$w \in \Hom(V^{\omega}_{2n}, \Or(\U^{\omega}_{1,1}))$$
we have 
\begin{equation} \label{MC} \delta C(w)  = \sum C(w') \cup C(w'')\  
\end{equation} 
where $\delta C(w) \in Z^2(\Gamma, (V^{\omega}_{2n})^{\vee}\otimes R)$ is an abelian two-cochain, and 
$$ \cup: C^1(\Gamma; (V^{\omega}_{2a})^{\vee} \otimes R) \otimes C^1(\Gamma; (V^{\omega}_{2a})^{\vee} \otimes R) \To C^2(\Gamma; (V^{\omega}_{2n})^{\vee}\otimes R)$$
is the cup product followed by $C^2 (\Gamma; (V^{\omega}_{2a} )^{\vee}\otimes (V^{\omega}_{2b} )^{\vee} \otimes R) \rightarrow C^2(\Gamma; (V^{\omega}_{2n})^{\vee}\otimes R)$ induced by the  dual of the chosen morphisms $V^{\omega}_{2n} \rightarrow V^{\omega}_{2a} \otimes V^{\omega}_{2b}$.
\end{prop} 
\begin{proof} The map $C$ is a cochain if and only if 
$$(gh, C_{gh}) = (g, C_g) .(h ,C_h) \qquad \in \quad  S^{\omega}\ltimes \U^{\omega}_{1,1}$$
holds for all $g,h\in \Gamma$.  Since multiplication in $\GG^{\omega}_{1,1} = S^{\omega} \ltimes \U^{\omega}_{1,1}$ is dual to the comultiplication, applying 
$(\ref{FormulaforDeltaw})$ gives $$ (gh, C(w)_{gh})  = (g, C(w)_g).(h, 1)  + (g, 1) .(h, C(w)_h) + \sum (g, C(w')_g). (h, C(w'')_h)$$
in  $S^{\omega} \ltimes (V^{\omega}_{2n})^{\vee} \otimes R$. Projecting onto $(V^{\omega}_{2n})^{\vee} \otimes R$ gives
$$C(w)_{gh} - C(w)_g\big|_h - C(w)_h = \sum C(w')_g \big|_h C(w'')_h$$
which is equivalent to $(\ref{MC})$ via the formulae of \S \ref{sectCupproducts}.
\end{proof} 
Equation $(\ref{MC})$, when restricted to the Hopf subalgebra $\Or(\PiU)$ implies $(\ref{phiMassey})$. We deduce the following transference principle for general cocycles.
\begin{thm} \label{thmTransferGeneral} Let $w \in \Hom_{S^{\omega}} (V^{\omega}_0, \Or(\U^{\omega}_{1,1}))$, and $C \in Z^1(\Gamma, \U^{\omega}_{1,1})$ a cocycle. Then 
\begin{equation} \label{TransferenceACME}   \pi \big( \, C(w)_T + \sum  \h( C(w'), C(w''))  \big)=0 \ 
\end{equation} 
where $\pi: (V^{\omega}_{2a})^{\vee} \otimes (V^{\omega}_{2b})^{\vee} \rightarrow (V_0^{\omega})^{\vee}$ is dual to the maps $V^{\omega}_0 \hookrightarrow V^{\omega}_{2a} \otimes V^{\omega}_{2b}$ which were used  in the definition of the coproduct $(\ref{FormulaforDeltaw})$.
\end{thm}
\begin{proof}  The proof is the same as for theorem $\ref{proptransfer}$. 
\end{proof} 
In particular, equation $(\ref{TransferenceACME})$ applies both to  the universal  cocycle $\CC^{\mm}$, and to its images $g\, \CC^{\mm}$
under the group $\A^{\omega}$. This can be interpreted as saying that the group of automorphisms $\A^{\omega}$  preserves a non-abelian version of the  Petersson inner product.

\subsubsection{Notations for $\Hc$-periods} 
 We can use these constructions to single out particular $\Hc$ periods of relative completion of $\mathcal{M}_{1,1}$. Let $w\in \Or(\U^{dR}_{1,1})$, and let $\gamma \in \SL_2(\Z)$. Then we can consider the matrix coefficient
$$ [ \Or(\GG^{dR}_{1,1}), \gamma^B, w]^{\mm} \in \Pe^{\mm}_{\Hc}\ ,$$
where $w$ is viewed in $\Or(\GG^{dR}_{1,1})$ by our choice of de Rham $W$-splitting.  This is nothing other than the coefficient of $w$ in the cocycle $C^{\mm}_{\gamma}$. We can write this 
$$\int^{\mm}_{\gamma} w:  = w (\CC^{\mm}_{\gamma})  \ .$$
Equation $(\ref{gactiononCmotivic})$ gives the formula for the action of $\GG^{dR}_{\Hc}$ upon it.  Furthermore, if
$w\in \Or(\U^{dR, \hol}_{1,1})$, then  by lemma \ref{lemCanFromHol} its period is given by  the iterated integral
$$\per \int^{\mm}_{\gamma} w= \int_{\gamma} w $$
which is the normalised coefficient of $w$ in the canonical holomorphic cocycle $\CC$.

\subsubsection{Example: modular construction of a motivic zeta value}  \label{sectModularZeta2n+1}
Let 
$$w = \e_{2n+2} \Ys^{2n} \quad \in \quad  \Or(\U^{dR,\hol}_{1,1}) \leq \Or(\U^{dR}_{1,1})$$ for all  $n\geq 1$ be the function `coefficient of $\e_{2n+2} \Ys^{2n}$'.   Write
$$\xi^{\mm}_{2n+1} = \int^{\mm}_S w = w(\CC^{\mm}_S)\ .$$
  Let $g\in G^{dR}_{\Hc}(\Q)$, and let $\lambda = \chi(g) \in \G_m(\Q) = \Q^{\times}$. Denote the image of $g$
in $\A^{dR}(\Q)$ by $[(b,\phi)]$. Since $(\U^{dR}_{1,1} )^{ab}$ is semi-simple in $\Hc$, it follows that 
$$\phi (w  )   \equiv  \lambda^{2n+1} w  \pmod{L^2}$$
where $L$ denotes terms of length $\geq 2$. This holds because $\e_{2n+2}\Ys^{2n}$ spans a copy of $\Q(2n+1)$,  upon which $g$ acts by $\lambda^{-2n-1}$. Since $w$ is a word of length one,  
$$g(\xi^{\mm}_{2n+1}) = w (   b^{-1} \big|_{S^g} \, \phi(\CC^{\mm}_S) b ) =  w(b) - w(b\big|_{S^g}) + \lambda^{2n+1} w(\CC^{\mm}_S)\ .  $$
Since the coefficients of $b$ are rational, we deduce that  $g$ acts by 
$$g(\xi^{\mm}_{2n+1}) = \chi(g)^{2n+1} \xi^{\mm} +   \nu_g$$
for some $\nu_g \in \Q$. Therefore  $\xi^{\mm}_{2n+1}$ defines a two-dimensional  representation 
$$ g\mapsto ( \chi(g)^{2n+1}, \nu_g)  : \GG^{dR}_{\Hc} \rightarrow \G_m \ltimes \G_a \ . $$

\begin{prop} \label{Corzetam2n+1} Let $\zetam(2n+1) \in \Pe^{\mm}_{\Hc}$ denote the image of the motivic zeta value $\zetam(2n+1) \in \Pe^{\mm}_{\MT(\Z)}$, as defined in \cite{BrICM}.   Then 
$$\xi^{\mm}_{2n+1} = - {(2n)!\over 2}  \zetam(2n+1)\ .$$
\end{prop} 
\begin{proof} By the  previous discussion, the minimal object  \cite{NotesMot} \S2.4 of $\Hc$ generated by $\xi^{\mm}_{2n+1}$ is of rank 2, and its semi-simplification is 
$\Q(0) \oplus \Q(-2n-1)$.  The same holds for $\zetam(2n+1)$.  Any such  an element in $\Pe^{\mm}_{\Hc}$ is uniquely determined by its period $\alpha$, since its period matrix is of the form 
$$ \begin{pmatrix} 1 & \alpha \\  0 &  (2 i \pi)^{2n+1}
  \end{pmatrix}$$
Since $w \in \Or(\U^{dR,\hol}_{1,1})$, we have 
 $\per \, \xi^{\mm}_{2n+1} =  w(\CC_S)= -\textstyle{(2n)!\over 2}  \zeta(2n+1)$ by lemma $\ref{propE2k}$, since   two factors of $(2 \pi i)^{2n}$ cancel out via the
 the normalisations $(\ref{QdR})$. 
\end{proof}

\begin{rem}  This provides a modular construction of the motivic zeta elements as suggested in \cite{BrICM}. 
This interpretation is fundamentally different from the usual one, and indeed its period
gives rise to a rapidly-converging Lambert series for $\Lambda(E_{2n+2}, 2n+1)$ as opposed to the definition of $\zeta(2n+1)$ as a sum of reciprocals of integer powers.

One can trace a geometric route between the two definitions of motivic zeta values as follows. Via the local monodromy
morphism, the element $\xi^{\mm}_{2n+1}$ can be viewed as a motivic period of the punctured infinitesimal Tate curve. Via the Hain morphism $\Phi^{\Hc}$ of \S\ref{sectHainmorphism}, it can in turn be pushed down to the motivic fundamental group of $\Pro^1\backslash \{0,1,\infty\}$ with tangential base point $1$. Finally, by the action of the latter by conjugation on the motivic fundamental torsor of paths of $\Pro^1\backslash \{0,1,\infty\}$ along the straight line path from $0$ to $1$,
one can make the connection with the definition of $\zetam(2n+1)$ in \cite{BrICM}.   Note that in this process,  our simple iterated integral of a Eisenstein series (length 1), becomes a more complicated iterated integral of length $2n+1$.
\end{rem}

\subsection{First coefficients of $C^{\mm}$}  \label{sectFirstcoeffs} We give a formula for the image of $C_S^{\mm}$ in $\U_{1,1}^{dR}(\Pe^{\mm}_{\Hc}\otimes\overline{\Q})$ modulo $L^2$. 
 By the above remarks it can be written
\begin{multline} C_S^{\mm} \equiv 1 + \sum_{n\geq 1 } \e_{2n+2} \Big( {(2n)!\over 2}  { \zetam(2n+1) \over( \Lef^{\mm})^{2n}   } (X^{2n} - Y^{2n}) +  \Lef^{\mm} e^0_{2n+2}(X,Y)\Big)  \\
+    \Lef^{\mm} \sum_f \big(\e'_f \omega^{\mm,+}_f + \e''_f \eta_f^{\mm,+}) P^+_f (X,Y)+    \big(\e'_f \omega^{\mm,-}_f + \e''_f \eta_f^{\mm,-})P^-_f (X,Y) \pmod{L^2} \end{multline}
where $f$ ranges over  Hecke eigenforms, and $P_f^{\pm}\in V_{2n}^{dR}\otimes \overline{\Q}$ are its  Hecke-invariant odd and even period polynomials (which are only defined up to multiplication in $\overline{\Q}^{\times}$).    It is written in terms of the Betti space $V_{2n}$. To pass to $V_{2n}^{dR}$, replace  $(X,Y)$ with $(\Xs, \Lef^{\mm} \Ys)$.  In the above, $\e'_f, \e''_f$ is a choice of basis of $\e_f$ as in remark $\ref{reme'fe''f}$ and
$$ \begin{pmatrix} \omega^{\mm,+}_f  & \eta_f^{\mm,+} \\ \omega^{\mm,-}_f &  \eta_f^{\mm,-} 
  \end{pmatrix}$$ is the matrix of the universal comparison isomorphism $V_f^{dR} \otimes \Pe^{\mm}_{\Hc^{ss}\otimes \overline{\Q}} \overset{\sim}{\rightarrow}  V_f^{B}\otimes \Pe^{\mm}_{\Hc^{ss}\otimes \overline{\Q}}$ with respect to $\e'_f,\e''_f$ and a suitable basis of $V_B^{\pm}$. Here plus (resp. minus) denotes invariance (anti-invariance) under the real Frobenius.  The usual (holomorphic) periods of the cusp form $f$ are  $\mathrm{per} \,\omega^{\mm, \pm}_f = \omega^{\pm}_f$.  The $\eta^{\mm,\pm}_f$ could be called its `quasi-periods' and depend on the choice of element $e'_f$. 
 \vspace{0.1in}

In general,  the coefficients of $C^{\mm}$  are   complicated  linear combinations of periods of different types.  In order to tease the constituent pieces apart, we need to exploit the structure of the automorphism group $\A^{dR}$.

\subsection{Periods of the automorphism group} \label{sectPeriodsofAuto}  Choose an element
$$ s \quad \in \quad \mathrm{Isom}(\omega_{dR}, \omega_B)( \Pe^{\mm}_{\Hc^{ss}})\ .$$
This can be done as follows: choose two splittings of the $M$-filtration \S\ref{sectSplittings}: one in the Betti, the other in the  de Rham realisation, to obtain 
 a functorial morphism 
$$V_{dR}  \cong \gr^M V_{dR} \overset{\comp^{\mm}_{B,dR}}{\To} \gr^M V_B\otimes_{\Q} \Pe^{\mm}_{\Hc^{ss}}  {\cong} V_B\otimes_{\Q} \Pe^{\mm}_{\Hc^{ss}}$$
which becomes the required  isomorphism  $s$ after tensoring with $ \Pe^{\mm}_{\Hc^{ss}}$.  Modifying these choices splittings  correspond to multiplying the element $s$ on the  right (resp. left) by an element of $\U_{\Hc}^{dR}(\Q)$ (resp. $\U_{\Hc}^B(\Q)$).

Applying this to the relative completion gives an isomorphism of schemes
$$\GG_{1,1}^B \times \Pe^{\mm}_{\Hc^{ss}} \overset{\sim}{\To} \GG_{1,1}^{dR} \times \Pe^{\mm}_{\Hc^{ss}}$$
and hence, via our choice of de Rham $W$-splitting and \S\ref{sectCocyclesWithTangency} a cocycle $$ s \quad  \in \quad Z^1_{ \mathrm{comp^{\mm,ss}}} (\Gamma; \U_{1,1}^{dR})(\Pe^{\mm}_{\Hc^{ss}})\ . $$
By corollary   \ref{corZ1torsor}, the cocycles $\CC^{\mm}$ and $s$ differ by a unique element of the automorphism group $\mathrm{Aut}'_{\U^{dR}_{1,1}}(\GG^{dR}_{1,1})(\Pe^{\mm}_{\Hc})$ (or more precisely,  the subgroup $\A_{\U}^{dR}(\Pe^{\mm}_{\Hc})$.) 

\begin{defn} There exists a unique (depending on the choice of $s$) equivalence class $[(b^{\mm} , \phi^{\mm} )] $ in $\A_{\U}^{dR}(\Pe_{\Hc}^{\mm})$, viewed as a group of right automorphism such that 
$$\CC^{\mm} =  s\circ  [(b^{\mm}, \phi^{\mm})]\ .  $$
\end{defn} 
Having fixed  $s$, any element   $w \in \Or(\A_{\U}^{dR})$ defines an $\Hc$-period by 
$$w( [(b^{\mm}, \phi^{\mm})]  ) \ \in \  \Pe^{\mm}_{\Hc}  \ .$$
Elements $w$ in the coordinate ring of $\A_{\U}^{dR}$ can easily be written down  using the duals of the three maps $(\ref{Autsemi-dtoAut}), (\ref{Autpitoquotient}), (\ref{AuttoSmap}).$
The coefficients of  $[(b^{\mm}, \phi^{\mm})]$  are simpler than those of $\CC^{\mm}$, which are superimposed in a complex way. There is no analogue of this phenomenon in genus $0$. A disadvantage  of this approach is the dependence on the choice of element $s$. 
We can remove this dependence by working with  de Rham periods. 

\subsection{Example: length one}
A choice of $s$ defines a cocycle whose value on $S$ is
\begin{multline} s_S \equiv 1 + \sum_{n\geq 1 } \e_{2n+2}   \Lef^{\mm}  e^0_{2n+2}(X,Y)  \\
+ \Lef^{\mm} \sum_f \big(\e'_f \omega^{\mm,+}_f + \e''_f \eta_f^{\mm,+}) P^+_f (X,Y)+    \big(\e'_f \omega^{\mm,-}_f + \e''_f \eta_f^{\mm,-})P^-_f (X,Y) \pmod{L^2} \ ,\end{multline}
using the notations of \S\ref{sectFirstcoeffs}. Its coefficients lie in $\Pe^{\mm}_{\Hc^{ss}}$. The element $b^{\mm}$ satisfies
$$ b^{\mm} = 1 - \sum_{n\geq 1}  \e_{2n+2} {(2n)!\over 2}  \zetam(2n+1)    \Ys^{2n}  \pmod{L^2}  \  .$$
Note that in the first formula we used Betti generators $(X,Y)$, but in the second the formula is simpler using the  de Rham generator $\Ys^{2n} = (\Lef^{\mm})^{-1} Y$. 
\begin{rem} Taking the image in $\PiU$ and applying the period homomorphism  via lemma $\ref{lemCanFromHol}$ gives an expression for the canonical holomorphic cocycle of the form  $\CC_{g} =    b^{-1}\big|_g \phi(s_g) b $ for all $g\in \Gamma$.  It implies in particular that 
$ \CC_{e_{2m}} = - \delta b_{\e_{2m}} + s_{2m}$
where $s_{2m}$ is proportional to the rational cocycle  $e^0_{2m}$ of $(\ref{e^0defn})$.   The  higher length coefficients of $\CC$
can be interpreted in terms of cup products via \S\ref{sectCupproducts}.  
\end{rem} 

\subsection{de Rham  periods} \label{sectDRperiods} 
The homomorphism $\GG^{dR}_{1,1} \rightarrow \A^{dR}$  gives rise to a homomorphism of Hopf algebras on their affine rings 
 $$\Or(\A^{dR}) \To \Pe^{\dR}_{\Hc} $$ 
 and enables us to construct de Rham periods directly out of elements of the coordinate ring of $\A^{dR}$ or  via  $(\ref{Autsemi-dtoAut}), (\ref{Autpitoquotient}), (\ref{AuttoSmap}).$ We can spell out the first and third constructions more directly as follows. 
 Fix a dR-splitting of the $W$-filtration on $\GG^{dR}_{1,1}$, so we write $\GG^{dR}_{1,1} \cong S^{dR} \ltimes \U^{dR}_{1,1}$.  Let $g\in \Or(\GG^{dR}_{1,1})^{\vee}$, and let 
 $w \in \Or(\U^{dR}_{1,1})$.  We can view $w$ as an element of $\Or(\GG^{dR}_{1,1})$. Then the data of $g$ and $w$ provides a de Rham period
 \begin{equation} \label{(g,w)defn}  
 (g,w)^{\dR}:= [\Or(\GG^{\Hc}_{1,1}) , g, w]^{\dR}  \quad \in \quad \Pe^{\dR}_{\Hc} \  \end{equation}
   via the matrix coefficient construction. It assigns to $\phi \in \GG^{dR}_{\Hc}$  the element $w(\phi(g)) \in \A^1$. Two cases are of  particular interest: when 
   $g\in \U^{dR}_{1,1}(\Q)$, which corresponds to $ (\ref{Autsemi-dtoAut})$, 
      or when    $g\in S^{dR}(\Q)$,  which corresponds to $(\ref{AuttoSmap})$.

\section{Structure of derivations} 

 We  analyse some consequences of the inertial condition on the  structure of $\Lie\, \A^{dR}_{\U}$.

\subsection{Computation of   $N^{dR}$} \label{sectNdR}  
Let us choose dR-splittings of the $M$ and $W$ filtrations.  

\begin{lem} \label{lemNdRto1storder} The element $N^{dR}_+ \in \uu^{dR}_{1,1}$ defined in \S\ref{sectOperatorN} satisfies
$$N^{dR}_+ \equiv   \sum_{n\geq 1}   {\Be_{2n+2} \over 4n+4 }  \e_{2n+2}   \Xs^{2n}  \pmod{L^2}$$
\end{lem} 
\begin{proof} The element $N^{dR}_+$ is of type $\Q(1)$, and in particular has $M$-degree $-2$.  The only  elements of length one in $\uu^{dR}_{1,1}$ of type $\Q(1)$ 
are of the form $c_{2n} \e_{2n+2} \Xs^{2n}$ for some $c_{2n} \in \Q$. It remains to determine the coefficients. Using lemma $\ref{lemNfromCmT}$, they can be obtained from $\CC^{\mm}_T$.  Since   $\e_{2n+2}$ is holomorphic, these can be read off the canonical cocycle $\CC_T= \mathrm{per} \, \CC^{\mm, \hol}_T $. One  checks that  indeed
$$\e_{2n+2} (\CC_T)  = (2 i \pi)^{2n+1} \exp  \Big( { 1\over 2\pi i} \Ys{\partial \over \partial \Xs} \Big)   \,  c_{2n}  \e_{2n+2}   \Xs^{2n} $$
has the unique solution $c_{2n} = \textstyle{\Be_{2n+2} \over 4n+4}$.
 \end{proof}

This is consistent with the results of \cite{HaKZB}, where the image of $N^{dR}_+$  under the monodromy representation was computed using the KZB-connection.

\subsubsection{Length two terms in $N^{dR}_+$.} We can compute $N^{dR}_+$ modulo $L^3$ as follows.   In length two, the only  elements of Tate type in $\uu^{dR}_{1,1}$ are of the form $[\e_{2a+2} V^{dR}_{2a} ,\e_{2b+2} V^{dR}_{2b}]$  or a sub-object of  $[\ms_fV^{dR}_{2n} , \ms_f V^{dR}_{2n}]$, where $f$ is a generalised eigenspace of cusp forms over $\Q$ of weight $2n+2$. The first case is ruled out  since $\e_{2a+2}V^{dR}_{2a}$ has $M$-degrees $\leq -2$. The  only component of $M$-degree $-2$  in the second case is  the $S^{dR}$ -invariant  term:  
$\ms_f \ms_f (\Xs_1 \Ys_2 - \Xs_2 \Ys_1)^{2n}$. 
It contains a Tate sub-object coming from  the polarisation on $\M^{\Hc}_f$. Let us extend scalars to $\overline{\Q}$, and for any Hecke eigenform $f$, 
let us denote by 
$$\PP_f : \Q(1-2n)=V^{dR}_0(1-2n) \To  \e_f V^{dR}_{2n} \otimes \e_f   V^{dR}_{2n} \qquad \hbox{ (sub-object in } \Hc\otimes \overline{\Q} \hbox{)} \, $$
 a copy of $\Q(1-2n)$ corresponding to the duality relation $(\ref{Vfdualrel})$. We conclude that
 \begin{equation} 
 N^{dR}_+ \equiv   \sum_{n\geq 1}   {b_{2n+2} \over 4n+4 }  \e_{2n+2}   \Xs^{2n}    
 +       \sum_f   c_f \PP_f   (\Xs_1 \Ys_2 - \Xs_2 \Ys_1)^{2n}\   \pmod{L^3} \ ,
  \end{equation} 
 for some coefficients $c_f \in \overline{\Q}$, where the sum is over Hecke eigenforms of weight $2n+2$. 
These coefficients can be determined from the transference principle via lemma $\ref{lemNfromCmT}$. Indeed,  by theorem 
$\ref{thmTransferGeneral}$ 
 $$C^{\mm}(\PP_f)_T   +  \h  ( C^{\mm}({\e_f}) , C^{\mm}({\e_f})) = 0\ .$$
 After applying the period, the right-hand factor reduces by  lemma \ref{lemhiscurly} to  a non-zero multiple of the Petersson norm of the cocycles (\S\ref{sectPeriodsofscusp})   of the cusp form $f$.  The latter are non-zero, 
 and it follows that $C^{\mm}(\PP_f)_T$ is non-zero, since its period is non-zero. 
   We can  therefore normalise the elements $\PP_f$ so that all coefficients are one.
   \begin{prop}  With the above normalisation of the $\PP_f$, we have 
\begin{equation} \label{Ndrfull}
 N^{dR}_+ \equiv   \sum_f    \PP_f   (\Xs_1 \Ys_2 - \Xs_2 \Ys_1)^{2n} +  \sum_{n\geq 1}   {b_{2n+2} \over 4n+4 }  \e_{2n+2}   \Xs^{2n}    
      \   \pmod{L^3} \ .
  \end{equation} 
 \end{prop}

\subsection{Heads and tails}
It follows from  corollary $\ref{cordeltaM=W}$  that
for every $m$, there is a  map 
\begin{eqnarray} \label{geometricneck} 
M_m  \Big( \uu_{1,1}^{\omega} \rtimes^{(\uu_{1,1}^{\omega})^{S^{\omega}}} \mathrm{Der}\,(\uu_{1,1}^{\omega})^{S^{\omega}}   \Big)&  \To & M_m \uu^{\omega}_{1,1} \pmod{W_m}   \\ 
{[} (b,\delta)]  & \mapsto &  b \pmod{ W_m \uu^{\omega}_{1,1} }   \ ,  \nonumber
\end{eqnarray} where   $b \in M_m \uu_{1,1}^{\omega}$ and $
\delta \in M_m  \mathrm{Der}(\uu^{\omega}_{1,1})$, by   the last part of 
 proposition $\ref{propExplicitLiealgofAU}$. 
 This map is well-defined because $[(b,\delta)]$ is only ambiguous up to  modification via $(\ref{bdeltaequivrel})$ by an element in $M_m (\uu^{\omega}_{1,1})^{S^{\omega}}$, which is contained in  $W_m \uu^{\omega}_{1,1}$.

\subsubsection{Anatomy of a derivation} From now on, choose splittings of $M$ and $W$ as in \S\ref{sectSplittings}.
Consider a derivation  $$d \quad \in \quad \uu_{1,1}^{\omega} \rtimes^{(\uu_{1,1}^{\omega})^{S^{\omega}}} \mathrm{Der}\,(\uu_{1,1}^{\omega})^{S^{\omega}}  $$ of $M$-degree $m$.  Decompose $d = \sum_w d_w$ according to $W$-degree. 

\begin{itemize}
\item  Define the  \emph{neck}  of $d$ to be the part which lies above the line $W=M$,
$$neck(d) = \sum_{w<m } d_w\ .$$
By $(\ref{geometricneck})$, each $d_w$ is inner:  $d_w = \mathrm{ad}(b_w)$ for some $b_w \in \uu^{\omega}_{1,1}$, for $w<m$.

\item Say that $d$ has a \emph{geometric head} if $neck(d)$ is non-zero. Let
$$head(d) = d_w \qquad  \in \quad  gr^M_m \gr^W_w \uu_{1,1}^{\omega}  $$
where $w$ is minimal $<m$ such that $d_w \neq 0$. 
\end{itemize}

One could furthermore define 
the  \emph{invariant part} of $d$ to be the component $d_m$, and the 
 \emph{tail} of $d$ to be the part lying below the line $W=M$, i.e., 
$tail(d) = \sum_{w>m} d_m.$
These notions (except the  head) depend  on the choice of splittings.  

\subsubsection{The inertial condition}
The following lemma implies that the neck, to the lowest order in the lower central series filtration, is always a lowest weight vector.

\begin{prop} \label{lemsimpleKerNargument} Let $\sigma \in \Lie\,  \A_{\U}^{dR}$. If  it has a geometric head $\mathrm{ad}(b)$,  then $b$ is a lowest weight vector.
Now suppose  that $\sigma$ can be represented by $[(b,\delta)]$ with $b \in L^k \uu_{1,1}^{dR}$ and $\delta \in L^{k-2} \mathrm{Der} \, \uu_{1,1}^{dR}$ (see lemma $\ref{lemModDegandLCS}$).
Then 
\begin{eqnarray}  \label{GeomLowestWeightAndDeltaepsilon}
 b^T -b & \equiv&  0   \pmod{ L^{k+1}} \\
 \delta (\e_{2n+2} ) & \equiv & 0 \pmod{L^{k+1}} \qquad \hbox{ for all } n\geq 1\ ,  \nonumber  \\
  \sum_f \delta( \PP_f   (\Xs_1 \Ys_2 - \Xs_2 \Ys_1)^{2n} ) ) & \equiv & 0 \pmod{L^{k+1}} \ . \nonumber 
\end{eqnarray} 
If, furthermore, it has a geometric head $ \mathrm{ad}\, (b)$, then 
$\deg_M b \leq \min\{-3,-k-1\}$.
\end{prop}

\begin{proof}
Recall the inertial condition $(I)$: 
\begin{equation} \label{inproofkerN} [b,\varepsilon_0^{\vee}] + [b, N^{dR}_+] + \delta(N^{dR}_+)=0\ .\end{equation}
For the first statement,  apply $\gr^W$  to this formula and use the fact that $\gr^W_0 N^{dR}_+ = 0$. This implies that 
 a geometric head $\mathrm{ad}(b)$ satisfies $[b,\varepsilon_0^{\vee}]=0$.

For the second part, the  assumption on $b$ implies that  $[b, N^{dR}_+] \in L^{k+1} \uu_{1,1}^{dR}$.  Therefore 
$$[b,\varepsilon_0^{\vee}]  + \delta(N^{dR}_+)\equiv 0 \pmod{ L^{k+1}} .$$
By  $(\ref{Ndrfull})$ it follows that
$$[b,\varepsilon_0^{\vee}]  + \sum_{n \geq 1}  {\Be_{2n+2} \over 4n+4}  \delta( \e_{2n+2} \Xs^{2n}) + \sum_f \delta( \PP_f   (\Xs_1 \Ys_2 - \Xs_2 \Ys_1)^{2n} )
    \equiv 0 \pmod{ L^{k+1}} .$$
Since $\delta$ is $S^{dR}$-equivariant, the image of $\e_{2n+2}\Xs^{2n}$ and $\PP_f   (\Xs_1 \Ys_2 - \Xs_2 \Ys_1)^{2n} $ under $\delta$ are highest-weight vectors for the action of $S^{dR}$, and  cannot be in the image of $\varepsilon_0^{\vee}= \Ys { \partial \over \partial \Xs}$. 
Therefore $[b,\varepsilon_0^{\vee}]$ vanishes modulo $L^{k+1}$. Since $\delta (  \e_{2n+2} \Xs^{2n}) $ is  never a highest weight vector, 
the second and third terms in the previous expression also vanishe modulo $L^{k+1}$. This proves $(\ref{GeomLowestWeightAndDeltaepsilon})$.

For the last  part,  a lowest weight vector $v\in \uu^{dR}_{1,1}$   satisfies $\deg_M v \leq \deg_W v$ by corollary \ref{corHwandLw}.   The case $\deg_M v = \deg_W v$ is ruled out since  a geometric head does not lie on the $M=W$ line by definition. Since $b\in L^k$, it lies in $W_{-k}$, and hence 
$\deg_M(b) \leq -k-1$. It suffices to consider  the case when $k=1$. If $b$ has $W$-degree $-1$  it is  of the form $b= \e_f \Ys^{2n}$, where
$f \in \B_{2n+2}$, and  has  $M$-degree $-1-2n$.  This is $\leq -3$ since there are no cusp forms of weight two.
\end{proof}

Note that  the  type of  derivation  can be read off from its geometric head.

\begin{example}
We have the following possible geometric heads  in $\uu^{dR}_{1,1}\otimes \overline{\Q}$ in length one, where $f$ denotes a Hecke eigenform  of weight $2n+2$:
$$
\begin{array}{|c|c|c|c|} \hline
 \hbox{Head}  &  \hbox{Type}   &  \hbox{W - degree} &   \hbox{Manifestation} \\  \hline
\e_{2n+2} \Ys^{2n} & \Q(1+2n)   &   -2n-2        &  \sigma_{2n+1}\\
\e_f \Ys^{2n}  & V_f (1+2n)          &      -1      &   \hbox{N/A}  \\ \hline
\end{array}
$$
The first element corresponds to the `zeta elements' $\sigma_{2n+1}$ which we shall discuss below.  The second does not occur in $\Lie\, \A_{\U}^{\omega}$, as we presently show.
\end{example}

\subsubsection{Uniqueness of tails} 
\begin{thm}  \label{thmMWabsent} There are no elements of $\Lie \, \A_{\U}^{dR}$ in the region $W<M$, i.e., 
\begin{equation}  {  W_n \over  W_n \cap M_n } \Lie \, \A_{\U}^{dR} = 0 \ .
\end{equation}
\end{thm} 
\begin{proof} Let $\sigma = [(b,\delta)]$ lie in $W_n$. If it has a geometric head, it is a lowest weight vector by the previous proposition,
and must lie in the region $W\geq M$   by corollary $\ref{corHwandLw}$. Therefore $b \in M_n$. Otherwise, we can take $b=0$.  Since $\delta$ is $S$-equivariant, it lies in the same region  by corollary \ref{cordeltaM=W}. Therefore $\sigma \in M_n \Lie\, \A_{\U}^{dR}$. 
 \end{proof} 
 We can replace $dR$ with any other fiber functor in the statement of the theorem.  The theorem implies the neck and invariant part of an element in $\Lie \, \A_{\U}^{dR}$ uniquely determine its tail by the inertial condition $(I)$. From this point of view we see that it is quite non-trivial to construct non-zero elements in $\Lie\, \A_{\U}^{dR}$. 
 
 \begin{cor} \label{corModDegreeandM} Let $\sigma \in \Lie\, \A_{\U}^{dR}$ be of modular degree $k$. Then 
 $$\sigma \in M_{2-k} \Lie \, \A_{\U}^{dR}\ .$$
  \end{cor} 
 \begin{proof} By splitting the $M$ and $W$-filtrations, we can assume that $\sigma$ is of definite $M$-degree $m$. 
By  corollary \ref{cormoddepthandW} we can represent $\sigma$ in the form $[(b,\delta)]$ where $b\in L^k \uu^{dR}_{1,1}$. Furthermore, if $\delta(\e_f) \neq 0$ for some
cuspidal generator $\e_f$, then we showed that $\delta \in W_{2-k}$, and hence $m= 2-k$ by corollary \ref{cordeltaM=W}.
Otherwise, suppose that $\delta$ vanishes on all cuspidal generators.  The inertial condition $(I)$  implies that 
$$[b, \varepsilon_0^{\vee}] + \delta(N^{dR}_+)  \equiv  0 \pmod{L^{k+1}}\ .$$
Since $\delta \in L^{k-2}$, this implies by  $(\ref{Ndrfull})$ that
$$[b, \varepsilon_0^{\vee}] + \sum_{n \geq 1} { \Be_{2n+2} \over 4n+4}  \delta(\e_{2n+2} \Xs^{2n})  \equiv  0 \pmod{L^{k+1}}\ .$$
Each term on the left-hand side vanishes modulo $L^{k+1}$, since the $ \delta(\e_{2n+2} \Xs^{2n})$ are lowest weight vectors and $[b,\varepsilon_0^{\vee}]$ cannot be a lowest weight vector.   If $b \in L^k$ is non-vanishing, then $[b, \varepsilon_0^{\vee}]  \equiv 0 \pmod{L^{k+1}}$ implies that the class of $b$ in $\gr_L^k$ would be  a lowest-weight vector and would lie in the region $M\leq W$. This would imply $m\leq -k$ and the conclusion of the corollary. Otherwise, suppose that $b \in L^{k+1}$. We have shown that  
$\delta(\e_{2n+2}) \in L^{k+1}$ for all $n$. This implies that $\delta \in L^{k}$ (and in particular $\delta \in L^{k-1}$)  since $\delta$ is uniquely determined by its action on Eisenstein generators $\e_{2n+2}$. More generally,   if $b, \delta $ lie in $L^a, L^{a-2}$ respectively, then replacing $k$ with $a\geq k $ in the above argument shows  that $b,\delta $ in fact   lie in $L^{a+1}, L^{a-1}$.  This imples that $\sigma$ vanishes. 
 \end{proof} 
 
\subsection{Cuspidal heads}
\begin{prop} \label{lemcusphead} Let  $\sigma \in \Lie \, \A_{\U}^{dR}\otimes \overline{\Q}$ be represented in  the form $([b,\delta])$. Then the coefficient of $\e_f$ in $b$ is zero, for every   Hecke eigenform $f$.
\end{prop}

\begin{proof} 
 Consider the inertial condition $[\sigma, N^{dR}]=0$. 
 We can assume that the geometric head of $b$ is $ \ms_f\Ys^{2n}$,  by lemma $\ref{lemsimpleKerNargument}$.  Since $\delta \in L^1$, we obtain
$$ \varepsilon_0^{\vee}(b) + \big( \mathrm{ad}( \e_f \Ys^{2n}) + \delta \big)   \sum_{n\geq 1} {\Be_{2n+2} \over 4n+4} \e_{2n+2} \Xs^{2n}   \equiv 0 \mod L^3$$
Project onto the $S^{dR}$-invariant component, by first projecting onto  highest weight vectors (this kills all terms $\delta ( \e_{2n+2}\Xs^{2n})$, for $n\geq 1$),
and then projecting onto  lowest weight vectors (this kills $ \varepsilon_0^{\vee}(b)$). All that remains is 
$$\alpha\,  [\ms_f, \e_{2n+2}] (\Xs_1 \Ys_2 - \Xs_2 \Ys_1)^{2n} \equiv 0 \mod L^3\ ,$$
for some non zero $\alpha \in \Q$, which is a contradiction. 
\end{proof}

\begin{lem} \label{lemMovedownLCS} Let $\sigma \in \Lie \, \A_{\U}^{dR}$ be of the form $[(b,\delta)]$ with $b \in L^k$ and $\delta \in L^{k-1}$.  
For any $\tau \in \Lie \, \A_{\U}^{dR}$, the commutator $[\sigma, \tau]$ is further down the lower central series: it is of the form  $[(b',\delta')]$ with $b' \in L^{k+1}$ and $\delta' \in L^k$. 
\end{lem} 
\begin{proof} The fact that $\delta' \in L^k$ is automatic by $(\ref{semidLieformula})$. It suffices to check that $b' \in L^{k+1}$. The only potentially problematic case 
is when the geometric component of $\tau$ is of length one. By the previous lemma, we can assume that $\tau$ is equal to $[\e_{2n+2} \Ys^{2n}, \delta_{2n+1}]$, for some $\delta_{2n+1}$, plus higher order terms in the lower central series. Therefore
\begin{eqnarray}  b' & = & [b, \e_{2n+2} \Ys^{2n}] + \delta( \e_{2n+2} \Ys^{2n}) + \delta_{2n+1} (b) \nonumber \\ 
& \equiv & \delta( \e_{2n+2} \Ys^{2n}) \qquad \pmod{L^{k+1}} \nonumber 
\end{eqnarray} 
By proposition $\ref{lemsimpleKerNargument}$ this vanishes.
\end{proof}

\subsection{Arithmetic component of derivations of Tate type} \label{sectArithmeticdelta}
For every $|m-n| \leq k-1 \leq m+n-2 $, let  
$$\iota_{k}^{m,n} : V_{2k-2} \rightarrow V_{2m-2} \otimes V_{2n-2}$$ be  the $\mathrm{SL}_2$-equivariant map which is the inverse of  $\partial^{m+n-k}$ described in \S\ref{sectdeltakdef}. Choose $M$ and $W$ splittings as in \S\ref{sectSplittings}.

\begin{thm}  \label{thmphieis} Consider any element $\sigma \in \Lie \, \A^{dR}_{\U}$ of type $\Q(1-2m)$, for $m\geq 2$. Then there exists an $\alpha_m \in \Q$ such that it is of  the form $\sigma = [(b, \delta)]$ where
$$b \equiv \alpha_m \e_{2m} \Ys^{2m-2} \pmod{L^2}$$
where $L$ denotes the lower central series, and $\delta$ satisfies
$$\delta ( \e_{2k} v) \equiv \sum_{m<n} \lambda^{m,n}_k     [\e_{2m}, \e_{2n}] \iota_k^{m,n} (v) \pmod{L^3} \qquad \hbox{ for all } v \in V^{dR}_{2k-2}$$
where  the coefficients $\lambda^{m,n}_k$ vanish if $k\neq n-m+1$ and are otherwise given by 
\begin{equation} \label{lambdaformula}
\lambda^{m,n}_{n-m+1} = - {(n-m+1) \over n} {(2n-2)!(2m-2)!\over (2n-2m)!} {\Be_{2n} \over \Be_{2n-2m+2}} \alpha_{2m} \ .
\end{equation}
\end{thm} 

\begin{proof}
We know from proposition \ref{lemsimpleKerNargument} that $b$ is a lowest weight vector,  and furthermore that
it has no cuspidal components in length one (either for reasons of type, or by proposition \ref{lemcusphead}). 
It is therefore of the specified form.  The theorem follows from the inertial condition $[\sigma, N^{dR}] = 0$, which implies that
$$[b,\varepsilon_0^{\vee}] + [b, N^{dR}_+] + \delta(N^{dR}_+) \equiv 0 \pmod{L^3}\ .$$
 Writing this out via lemma \ref{lemNdRto1storder} gives the following equation $\mod L^3$:
 $$ [b,\varepsilon_0^{\vee}] + \sum_{ n \geq 2} \alpha_{2m} {\Be_{2n} \over 4n}  \big(\e_{2m} \e_{2n} \Ys_1^{2m-2} \Xs_2^{2n-2}- \e_{2n} \e_{2m} \Xs_1^{2n-2} \Ys_2^{2m-2}) +  \sum_{r \geq 2} {\Be_{2r} \over 2r} \delta( \e_{2r} \Xs^{2r-2}) \equiv 0 \ .
 $$
 Let $c_X$ be the fourth map of $(\ref{Tlongexact})$, which sends $\Ys$ to zero. We check that
 $$   c_X \partial^r \,  \Ys_1^{2b} \Xs_2^{2a} = c_X \partial^r  \, \Xs_1^{2a} \Ys_2^{2b} =\begin{cases} {(2a)! (2b)! \over  (2a-2b)!   } \Xs^{2a-2b}  \qquad \hbox{ if  }  r= 2b \hbox{ and } a>b  \\ 0  \qquad \qquad \qquad  \qquad   \hbox{ otherwise  }  
  \end{cases} $$ 
Now apply $c_X \partial^{2m-2}$ to the previous expression, which kills the $[b,\varepsilon_0^{\vee}]$ term, and take the coefficient of $\Xs^{2n-2m}$. This yields the equation 
$$ \alpha_{2m} {\Be_{2n} \over 4n}   [ \e_{2m} ,\e_{2n}] {(2n-2)! (2m-2)! \over (2n-2m)!} + {\Be_{2n-2m+2} \over 4n-4m+4} \partial^{2m-2} \delta(\e_{2n-2m+2}) =0   \ .$$
Rearranging terms gives equation $(\ref{lambdaformula})$. The derivation $\delta$ has no other components in length two for reasons of type. 
\end{proof} 

Theorem \ref{thmphieis} and  proposition  \ref{lemcusphead} were proved in an earlier version of this paper by a different, but essentially equivalent, method.

\section{Outline of a programme} \label{sectOutline}
The affine ring  $\Or(\GG^{\Hc}_{1,1})$ is an Ind-object of $\Hc$, and in fact  lies in  the  sub-category $\Hc_{\MMM}\subset \Hc$. 
One expects there to exist an abelian category of mixed motives $ \mathcal{MM}_{\Q}$ with a fully-faithful functor $h$ to $\Hc$. 
The category $\Hc_{\MMM}$  should correspond to a Tannakian subcategory of mixed modular motives, whose simple objects are generated by the motives of modular forms for $\mathrm{SL}_2(\Z)$.   Beilinson's conjecture predicts which extensions should exist in 
$\mathcal{MM}_{\Q}$. We show that not all  the predicted extensions can actually occur in the category $\Hc_{\MMM}$, but for the others,  we indicate where we expect them to occur in   $\GG^{\Hc}_{1,1}$.

\subsection{Motivic extensions as predicted by Beilinson}
The weight filtration  on $\mathcal{MM}_{\Q}$ will be denoted by $M$ here. 
 Beilinson's conjecture  predicts in particular that for $V$  a simple object of $\mathcal{MM}_{\Q}$  satisfying
 $M_{-3} V = V$, the group  of extensions $\mathrm{Ext}^1_{\mathcal{MM}_{\Q}}(\Q, V)$  should be a lattice in $\mathrm{Ext}^1_{\Hc \otimes \R}(\R, h(V)\otimes \R)$.  
 
 \begin{lem} The dimension of this space is given by 
\begin{equation} e(V): = \dim_{\R} \mathrm{Ext}^1_{\Hc \otimes \R}(\R, h(V)) = \dim_{\Q}  V_B^- -  \dim_{\Q} F^0 V_{dR} \ .
\end{equation}
\end{lem}
\begin{proof} This is well-known. See for example \cite{NotesMot}, corollary 6.7.
\end{proof} 

Let us translate this into Tannakian terms. Consider the Lie algebra $\uu_{\mathcal{MM}_{\Q}}^{\omega}$ of the unipotent radical  
of the Tannaka  group 
of the category $\mathcal{MM}_{\Q}$ with respect to any fiber functor $\omega: \mathcal{MM} \rightarrow \mathrm{Vec}_{\Q}$. 
Then one can show (e.g.  \cite{NotesMot}, 6.1) that
\begin{equation} \label{H1uudescription}  H_1 (\uu^{\omega};  \overline{\Q} ) \cong  \prod_{V_{\lambda}}  \mathrm{Ext}^1_{\mathcal{MM}_{\Q} \otimes{\overline{\Q}}}(\Q, V_{\lambda})^{\vee} \otimes_{\overline{\Q}} \omega(V_{\lambda}) \end{equation} 
where the product is over a representative $V_{\lambda}$ of every isomorphism class of simple objects in $\mathcal{MM}_{\Q} \otimes \overline{\Q}$.

\subsubsection{Lie algebra elements for mixed modular motives} 
Let us apply this to the simple objects in the category  $\Hc_{\MMM}\otimes{\overline{\Q}}$ of $\S\ref{SectHeckeSS}$. Consider an object  
\begin{equation} \label{Mgenerator} 
V= \mathrm{Sym}^{i_1} V_{f_1} \otimes_{\overline{\Q}}  \ldots \otimes_{\overline{\Q}}    \mathrm{Sym}^{i_r} V_{f_r} \otimes_{\overline{\Q}} \overline{\Q}(d)    
\end{equation} 
where the $f_i$  are distinct  normalised Hecke eigenforms of weight $2n_i+2$.

\begin{defn} Define the \emph{elevation} of the object $V$ in $(\ref{Mgenerator})$ to be the negative of the sum of its $M$-degree and its modular degree:
\begin{eqnarray} 
\ell(V)  & = &   - \deg_{M} (V)   - (i_1+ \ldots + i_r)  \nonumber \\ 
& = & 2d -  2 \sum_{k=1}^r  i_k (n_k+1)  \ . \nonumber 
\end{eqnarray} 
\end{defn} 
The elevation provides a measure of how far above the $M=W$ line a derivation of type $V$ could occur in $\Lie\,\A^{dR}_{\U}$. 
An immediate consequence of corollary  \ref{corModDegreeandM} is the:
\begin{cor}  A derivation  of type $(\ref{Mgenerator})$ can only occur in $\Lie \, \A_{\U}^{dR}$ if
\begin{equation} \label{ellVminus2} \ell(V) \geq -2\ .\end{equation}
\end{cor} 

\begin{cor}
Consider an element $\sigma \in \Lie \, \A_{\U}^{dR}$ of type $(\ref{Mgenerator})$, which has a geometric head. Then  the elevation is positive:
$\ell(V) > 0. $
\end{cor} 
\begin{proof}  By lemma $\ref{lemModDegandLCS}$, $\sigma$ lies in $W_{-k}$, where $k$ is the modular degree. A geometric head
lies in $M<W$ so $\deg_M \sigma < -k$, which implies that $\ell(V)>0$. 
\end{proof} 
We shall mainly consider elements satisfying $\ell(V)>0$. Note that some of the conjectural  examples which follow satisfy $\ell(V)=0$, and lie on the $M=W$ line.  We have nothing to say about the case $\ell(V) =-2$. 

\subsubsection{Conjectures}  \label{SectConjels}   Inspired by Beilinson's conjectures  we  expect: 
\begin{enumerate} 
\item (Generators)  For any  $V$ of the form  $(\ref{Mgenerator})$, satisfying $\ell(V)>0$, 
 there should  be  $e(V)$   non-canonical subspaces
$$\sigma_V^1, \ldots, \sigma_V^{e(V)} \quad  \subset \quad  \uu^{dR}_{\Hc_{\MMM}\otimes \overline{\Q}}\ ,$$ each of which is isomorphic to a copy of  $V_{dR}$, i.e., of type $V$,  and whose extension classes are linearly independent in $ \mathrm{Ext}^1_{\Hc \otimes \R}(\R, h(V))^{\vee}$.  

\item (Freeness) The $\sigma^{i}_V$, for $1\leq i\leq e(V)$,  as $V$ ranges over  objects of the form $(\ref{Mgenerator})$ satisfying $\ell(V)>0$,  freely generate a  Lie sub-algebra of $\Lie\, \A_{\U}^{dR}$.
\end{enumerate}

Furthermore, we expect that the $\sigma_V^{(i)}$ have geometric heads, i.e., the condition $\ell(V)>0$ is not only necessary but  sufficient for a derivation to appear in $\Lie\,A^{dR}_{\U}$ with a geometric head.

In order to prove $(2)$ it suffices to exhibit sufficiently many primitive elements in $\Or(\A^{dR})$ and prove that a certain period dual to the regulator  is non-zero.  In this paper we  prove that the above predictions are correct  for any $V$ of modular depths $0,1$.

\subsection{Examples proved in this paper} \label{Examplesprovedinpaper}
\begin{example} Take $V= \Q(d)$, with $d\geq 2$.  Then $F^0 V_{dR}=0$ and  $\dim V_B^-$ is $1$ if $d$ is odd, and $0$ otherwise.  Therefore
$\ell(V) = 2d >0$ and 
$$e(V) = \begin{cases}    1  \quad \hbox{ if } d  \hbox{ is odd} \\ 0 \quad \hbox{ otherwise }\ .   \end{cases}$$
We therefore expect a  sequence of `Tate' generators   in $\Lie \A_{\U}^{dR}$    denoted 
$\sigma_3, \sigma_5, \ldots $ for every odd integer $\geq 3$, 
where $\sigma_{2n+1}$ spans a copy of $\Q(-1-2n). $  These indeed exist,  and
their geometric heads are given by a rational multiple of  $\e_{2n+2} \Ys^{2n} \in \uu_{1,1}^{dR}$.  These elements are dual to the 
elements $\zetam(2n+1)$ constructed in \S\ref{sectModularZeta2n+1}. \end{example}

\begin{example} Let $f$ be a Hecke eigenform of weight $2n+2$, and  $V= V_f(d)$, $d \geq n+2$. Then $\dim V_B^+ = \dim V_B^-=1$. Since $V$ is of Hodge type
$(2n+1-d, -d)$ and $(-d, 2n+1-d)$, 
$\dim F^0 V_{dR}$ is equal to $1 $ if $d\leq  2n+1$ and $0$ if $d\geq 2n+2$. Thus
$$e(V) = \begin{cases}    0  \quad \hbox{ if }  \, d  \leq  2n+1  \\ 1  \quad \hbox{ if  }  \, d \geq 2n+2 \ .   \end{cases}$$
Moreover, $\ell(V) = 2d-2n-2 \geq 2n+2 >0$ whenever $e(V) \neq 0$. 
We therefore expect a  sequence of `modular' generators 
$\sigma_f(d)$ of rank $2$ for every  integer $d \geq 2n+2$, of type $V_f(d)$. 
We shall show that their images in $\Lie \A_{\U}^{dR}\otimes \overline{\Q}$ are  given by a certain linear combination of lowest weight vectors  in $[\e_f \otimes V^{dR}_{2n}, \e_{2k+2} \otimes V^{dR}_{2k}]$, which are of  the form 
$$ [ \  \e_{f}\Ys^b  \  , \ \e_{2k+2} \Ys^c \  ]\, (\Xs_1\Ys_2-\Xs_2\Ys_1)^a \ ,$$
which is shorthand for
$$ \e_f \e_{2k+2}(\Xs_1\Ys_2-\Xs_2\Ys_1)^a \Ys_1^b \Ys_2^c -    \e_{2k+2} \e_{f}(\Xs_2\Ys_1-\Xs_1\Ys_2)^a \Ys_1^c \Ys_2^b \ , $$ 
where $a+b=2n, a+c=2k$, and for the $M$-degrees to match, $a+b+c=d-2$. \end{example}

\begin{rem} It is interesting to note that the vanishing of  $e(M)$  for $d < 2n+2$  is exactly consistent   with the conclusion of proposition \ref{lemcusphead}. 
\end{rem}

\subsection{Some conjectural examples} We explore some consequences of the generation conjecture of \S\ref{SectConjels}, and provide some further evidence for it.  

\begin{example} \label{Exftensg} Let $f , g $ be two distinct normalised Hecke eigenforms of respective weights $2m+2 \geq 2n+2$.   Let  $V= V_f\otimes V_g (d)$. Assume $\deg_M V = -2d-2m-2n-2\leq 0$.    Then $\dim V_B^+ = \dim V_B^- =2$, and one verifies that 
$$\dim  \, F^0 V_{dR} = \begin{cases}  2 \quad \hbox{ if } \,  2n+2 \leq d \leq 2m+1\ , \\ 
                                                                1 \quad \hbox{ if }\,  2m+2 \leq d \leq 2m+2n+2\ , \\
                                                                0 \quad \hbox{ if }\,  2m+2n+2 < d \ . \end{cases}$$
It follows that 
$$ e(V) =        \begin{cases}  0 \quad \hbox{ if } \, d \leq 2m+1\ , \\ 
                                                                1 \quad \hbox{ if }\,  2m+2 \leq d \leq 2m+2n+2\ , \\
                                                                2  \quad \hbox{ if }\,  2m+2n+2 < d \ . \end{cases}$$                                                         
We have $\ell(V) = 2d - 2m-2n-4$, and so $\ell(V) \geq 0$ whenever $e(V) \neq 0$.  
The (unstable) range for which $e(V)=1$ should produce  rank $4$ generators 
$$\sigma_{f\otimes g} (d) \qquad \hbox{ for } \quad  2m+2 \leq d \leq 2m+2n+2$$
which, by $(\ref{LCSgeometricnongeneric})$ we expect to appear with necks 
\begin{equation} \label{fgrank1heads}  [  \e_f \Ys^{2n-k}    \ , \  \e_g    \Ys^{2m-k} ]  \, (\Xs_1 \Ys_2 - \Xs_2 \Ys_1)^{k}  \qquad \hbox{ for } 0\leq k \leq 2n\ .
\end{equation} 
The integers $k,d$ are related by $d=2+2n+2m-k$. 
The Rankin-Selberg method should allow one to show that these generators indeed occur in $\Lie\, \A^{dR}_{\U}\otimes \overline{\Q}$, and that their  periods are  related to $L(f\otimes g, d)$.

The stable range  $e(M)=2$ should produce pairs of rank $4$  generators 
$$ \sigma^1_{f\otimes g}(d) , \sigma^2_{f\otimes g}(d)  \qquad \hbox{ for } \quad   d \geq 2m+2n+2\ ,$$
which by  $(\ref{LCSgeometricgeneric})$ we expect to appear with necks  of the form 
$$\hbox{lowest weight vector in }     \left\lbrace \begin{aligned}   {[}\e_f  V^{dR}_{2n}, [\e_g  V^{dR}_{2m}, \e_{2k+2} V^{dR}_{2k}]]  
  \\  {[}\e_{g} V^{dR}_{2m}, [\e_f  V^{dR}_{2n}, \e_{2k+2} V^{dR}_{2k}]]   \end{aligned} \, \right.  $$
  The weak version of Beilinson's conjecture (relating the special value of $L$-functions to a regulator) is not presently known in this case, but the above procedure suggests that the $L$-values
$L(f\otimes g, d)$ for $d\geq 2m+2n+2$ occur as  a  \emph{triple} iterated integral, i.e., 
in the cohomology of  a product of  three modular curves.  A proof would seem to require a `higher' Rankin-Selberg method involving three modular forms.
\end{example}

\subsubsection{Generic stable case}
 For $r$ modular forms $V= V_{f_1} \otimes \ldots \otimes  V_{f_r}(d)$ and $d$ in the stable range, i.e., sufficiently large, we have 
$e(V) = r$ and $\ell(V)\gg 0$, and so we expect to find $r$ generators $\sigma^i_{f_1\otimes \ldots \otimes f_r}(d)$ for $1 \leq i \leq r$, appearing in $\Lie\, \A_{\U}^{dR} \otimes \overline{\Q}$
as  lowest weight vectors in an $r+1$-fold bracket
$$ [ \e_{f_{\pi(1)}}  V^{dR}_{2m_{\pi(1)}}, [  \e_{f_{\pi(2)}}   V^{dR}_{2m_{\pi(2)}}, \ldots, [ \e_{f_{\pi(r)}}   V^{dR}_{2m_{\pi(r)}}, \e_{2k+2}  V^{dR}_{2k} ]] \cdots ]   $$
 where $\pi$ is a permutation of $(1,\ldots, r)$, and $f_i$ is of weight $2m_i+2$.  This suggests that there is indeed enough `room' in $\Lie \, \A^{dR}_{\U}$  for the conjectured generators to appear. 

\subsubsection{A degenerate example}  Let $V= \mathrm{Sym}^2 V_f(d)$, which corresponds to the degenerate case $f=g$  and $m=n$ in   example
\ref{Exftensg}.  In this situation we have 
$$ \dim V_B^- = \begin{cases}  2  \quad \hbox{ if } d \hbox{ odd} \\ 1 \quad \hbox{ if } d \hbox{ even} \ . \end{cases} $$
Let us look only at the range $2n+2 \leq d \leq 4n+2$, for which $\dim F^1 V_{dR} = 1$. Then 
$$e(V) = \begin{cases}  1  \quad \hbox{ if } d \hbox{ odd} \\ 0 \quad \hbox{ if } d \hbox{ even} \ , \end{cases} $$
suggesting the existence of Lie algebra generators $\sigma_{f^{(2)}}(d)$  for $d$ odd in this range. 
Their heads  should appear, as in $(\ref{fgrank1heads})$, as  terms of the form
$$  [  \e_f \Ys^{2n-k}    \ , \  \e_f    \Ys^{2n-k} ]  \, (\Xs_1 \Ys_2 - \Xs_2 \Ys_1)^{k}  \qquad \hbox{ for } 0\leq k \leq 2n \hbox{ odd} \ .$$
Recall that this notation  means 
$$    \e_f\e_f  \Ys_1^{2n-2r-1} \Ys_2^{2n-2r-1}    (\Xs_1 \Ys_2 - \Xs_2 \Ys_1)^{2r-1}  $$
if $k= 2r-1$ is odd, and  $1\leq r\leq n$. 
The terms in $(\ref{fgrank1heads})$ for even $k$ vanish when $f=g$, which is  consistent with the vanishing of $e(V)$ for even $d$ in this range. 
The  existence of the $\sigma_{f^{(2)}}(d)$ for odd $2n+2 \leq d \leq 4n+2$ should be provable using the methods of this paper via the Rankin-Selberg method. The corresponding periods should be proportional to $L(\mathrm{Sym}^2 f, d)$.

\section{Single-valued Periods} \label{sectSV}
We recall the definition of the single-valued period $\sv : \Pe^{\dR}_{\Hc} \rightarrow \C$ and apply it to the periods of relative completion of the fundamental group of $\M_{1,1}$. 
Applying this construction to families leads to a new class of non-holomorphic modular forms.

\subsection{The single-valued period homomorphism} 
\label{sectSVPreminders}  We recall the construction from  \cite{NotesMot}, \S4.1. 
The scheme $\mathrm{Isom}_{\Hc}^{\otimes}( \omega_{dR}, \omega_B)$ is a right $\GG^{dR}_{\Hc}$-torsor, hence
$$ \mathrm{Isom}_{\Hc}^{\otimes}( \omega_{dR}, \omega_B) (\Pe^{\mm}_{\Hc} )  \times    \GG^{dR}_{\Hc}(\Pe^{\mm}_{\Hc} )   \To \mathrm{Isom}_{\Hc}^{\otimes}( \omega_{dR}, \omega_B)(\Pe^{\mm}_{\Hc} ) \ . $$
The space $ \mathrm{Isom}_{\Hc}^{\otimes}( \omega_{dR}, \omega_B) (\Pe^{\mm}_{\Hc} ) = \mathrm{Hom}(\Pe^{\mm}_{\Hc},\Pe^{\mm}_{\Hc})$
contains two natural elements; the identity $\id$ and the real Frobenius $F_{\infty}: \Pe^{\mm}_{\Hc} \rightarrow \Pe^{\mm}_{\Hc}$.
Therefore there exists a unique point $\sv^{\mm} \in \GG^{dR}_{\Hc}(\Pe^{\mm}_{\Hc} )  $ satisfying 
\begin{equation} \label{svmdef}  F_{\infty} \circ \sv^{\mm} = \id \ .
\end{equation}

 The element $\sv^{\mm}$ assigns to any de Rham period  an element of $\Pe^{\mm}_{\Hc}$, i.e., it is an algebra homomorphism: $\Pe^{\dR}_{\Hc} \rightarrow \Pe^{\mm}_{\Hc}$.
  For any object $M$ in $\Hc$, it is the map
$$ \sv^{\mm} :  M_{dR} \otimes  \Pe^{\mm}_{\Hc}  \To M_B \otimes \Pe^{\mm}_{\Hc} \overset{F_{\infty} \otimes \id}{\To} M_B \otimes \Pe^{\mm}_{\Hc} \To M_{dR} \otimes \Pe^{\mm}_{\Hc}$$
where the first arrow is the universal comparison $(\ref{univcomp})$, and the last arrow  its inverse. To any de Rham matrix coefficient $[M, v, f]^{\dR} \in \Pe^{\dR}_{\Hc}$, where $ v\in M_{dR}$ and $f\in M_{dR}^{\vee}$,  we can associate its single-valued period $f(\sv^{\mm}(v)) \in \Pe^{\mm}_{\Hc}$.  

The image of $\sv^{\mm}$ under the period homomorphism defines an element 
$$\sv = \per ( \sv^{\mm} )\in G^{dR}_{\Hc}(\C)\ .$$
It is a homomorphism $\sv: \Pe^{\dR}_{\Hc} \rightarrow \C$ and assigns a number to any de Rham period.

\begin{lem}  \label{lemsvchi} The image of $\sv^{\mm}$ under $\chi: \GG^{dR}_{\Hc} \rightarrow \G_m$ is the element 
$$\chi(\sv^{\mm}) = -1  \ \in\  \G_m(\Q)$$
where $-1$ acts by multiplication by $-1$ on $\Q(-1)_{dR}$.  
\end{lem} 
\begin{proof}  Recall $\Lef^{\mm}$ from \S\ref{sectLef}.  From the definitions, $\sv^{\mm}$ is
$$\Q(-1)_{dR}\otimes \Lef^{\mm} \Q  \overset{\sim}{\To}  \Q(-1)_B \otimes \Lef^{\mm} \Q \overset{F_{\infty}}{\To}  \Q(-1)_B \otimes \Lef^{\mm} \Q \overset{\sim}{\longleftarrow} \Q(-1)_{dR} \otimes \Lef^{\mm} \Q \ .$$
Frobenius  $F_{\infty}$ acts by $-1$ on $\Q(-1)_B$, so $\sv^{\mm}$ is induced by $-1 \in \mathrm{Aut}\,  \Q(-1)_{dR}$. 
\end{proof} 
The de Rham Lefschetz period  is $\Lef^{\dR} = [\Q(-1), v, v^{\vee}]^{\dR}$, 
where $v \in \Q(-1)_{dR}$ is any element. An identical computation  implies  that 
\begin{equation} \label{svofLefdR} \sv^{\mm}(\Lef^{\dR})=-1\ .
\end{equation}

\subsection{Single-valued multiple modular values}

\begin{defn} The ring of \emph{single-valued motivic multiple modular values} is the subring of $\Pe^{\mm}_{\Hc}$ generated by the single-valued motivic periods of $\Or(\GG^{dR}_{1,1})$:
$$ \langle \sv^{\mm} [ \Or(\GG^{dR}_{1,1}), v, f]^{\dR} \qquad \hbox{ for } \qquad v\in \Or(\GG^{dR}_{1,1}), f \in \Or(\GG^{dR}_{1,1})^{\vee} \rangle_{\Q}\ . $$
  The \emph{single-valued multiple modular values} are their images under the period map. 
\end{defn} 

We shall compute these objects by fixing a dR-splitting of the $W$ filtration \S\ref{sectSplittings}. 
 This provides via $(\ref{WsplitGivesGsplit})$  a morphism
 \begin{eqnarray} \mathrm{SL}_2(\Q) = S^{dR}(\Q)  &\To & \GG_{1,1}^{dR}(\Q)     \\
    \gamma & \mapsto & \gamma^{dR} \ .\nonumber
    \end{eqnarray} 
 \begin{defn} Define $\sv^{\mm}_{\gamma} \in \GG_{1,1}^{dR}(\Pe^{\mm}_{\Hc})$ to be the image of $\gamma^{dR}\in        \GG_{1,1}^{dR}(\Q)     \leq  \GG_{1,1}^{dR}(\Pe^{\mm}_{\Hc})$ under the action of $\sv^{\mm} \in \GG_{\Hc}^{dR}(\Pe^{\mm}_{\Hc})$, which acts upon $\GG_{1,1}^{dR}(\Pe^{\mm}_{\Hc})$ on the right.  Equivalently, it can be viewed as the homomorphism
  \begin{eqnarray} \sv^{\mm}_{\gamma} : \Or(\GG_{1,1}^{dR}) & \To &  \Pe^{\mm}_{\Hc} \nonumber \\
  w & \mapsto & \sv^{\mm} \big(  [  \Or(\GG_{1,1}^{dR}), \gamma ,w]^{\dR}\big)\ . \nonumber 
    \end{eqnarray} 
 Write   $\sv_{\gamma}: = \per (\sv^{\mm}_{\gamma}) \in \GG_{1,1}^{dR}(\C)$. It is   a homomorphism  $\sv_{\gamma} : \Or(\GG_{1,1}^{dR}) \rightarrow \C$.  
Via the $W$-splitting, the image of $\sv^{\mm}_{\gamma}$ in $\U_{1,1}^{dR}(\Pe^{\mm}_{\Hc})$, and its period $\sv_{\gamma}$, defines a $\Gamma$-cocycle 
which we call the \emph{single-valued (motivic) cocycle}:
$$\CC^{\mm,\sv} \in Z^1(\Gamma, \U_{1,1}^{dR})( \Pe^{\mm}_{\Hc}) \quad \hbox{ and } \quad  \CC^{\sv} = \per (\CC^{\mm,\sv}) \in Z^1(\Gamma, \U_{1,1}^{dR})( \C)\ .$$
Its coefficients lie in the ring of single-valued (motivic) multiple modular values. 
\end{defn} 
 These definitions implicitly depend on our choice of $W$-splitting. 
 The series $\CC^{\sv}_{\gamma}$ could be thought of  as the generating series of  `single-valued iterated integrals' along the path $\gamma$. 
 Note that the cocycles $\CC^{\mm, \sv} , \CC^{\sv}$ are cocycles with respect to the action of $\Gamma$ on $\U_{1,1}^{dR}$ twisted by $\chi$. Indeed, by lemma $\ref{lemsvchi}$, the latter  satisfies 
 $$\CC^{\sv}_{g h} = \CC^{\sv}_g\big|_{\overline{h}}\, \CC^{\sv}_h\qquad \hbox{ for all } g, h \in \Gamma$$
 where $\overline{h}$ denotes the image of $h \in S^{dR}(\C)$ after  applying $-1 \in \G_m(\Q)$, or, equivalently, by taking the complex conjugate of its entries.  By $(\ref{EqnActiononSL2})$,
 $$ \overline{S} = - S \qquad \hbox{ and } \qquad \overline{T} = T^{-1}\ . $$
 The analogous statement holds for $\CC^{\mm}$ as the action of $F_{\infty}$ on the coefficients of $(\ref{piSm})$ 
 is equivalent, by lemma \ref{lemsvchi},  to multiplication by $-1 \in \G_m(\Q)$. 

\subsection{Formulae for $\sv$} 
Let us also denote by 
$$\sv^{\mm} \ \in \ \A^{dR}(\Pe_{\Hc}^{\mm}) \qquad \hbox{ and } \qquad \sv \ \in \ \A^{dR}(\C)$$
the images of $\sv^{\mm}, \sv$ under the morphism $\GG^{dR}_{\Hc} \rightarrow \A^{dR}$. By the discussion in \S\ref{sectAutUSltimesU}, there exists a unique equivalence class  representing $\sv^{\mm}$
 $$[(b_{\sv}^{\mm},  \phi_{\sv}^{\mm} )]   \quad \in \quad \U^{dR}_{1,1} \rtimes^{(\U^{dR}_{1,1})^{S^{dR}}} \mathrm{Aut}(\U_{1,1}^{dR})^{S^{dR}, -1} (\Pe^{\mm}_{\Hc})\ . $$
 Denote  the equivalence class  representing $\sv$ by  $[(b_{\sv},  \phi_{\sv} )] $ where $b_{\sv} = \per \, b^{\mm}_{\sv}$ and $\phi_{\sv} = \per  \, \phi^{\mm}_{\sv}$. We wish to compute $\sv^{\mm}$ via the formula $(\ref{svmdef})$, applied to the elements $\gamma^{\mm} \in  \GG^{dR}_{1,1}(\Pe^{\mm}_{\Hc})$. We obtain $F_{\infty}(\gamma^{\mm}) \circ \sv^{\mm} = \gamma^{\mm}$ for all $\gamma \in \Gamma$, 
which is equivalent to 
$$( F_{\infty} \pi  \gamma^{\mm} , F_{\infty}  \CC_{\gamma}^{\mm}) \circ [(b, \phi)]  = ( \pi \gamma^{\mm},   \CC^{\mm}_{\gamma} ) \qquad \hbox{ for all } \gamma \in \Gamma \ .$$
Since $\phi$ is an automorphism with character $-1 \in \G_m(\Q)$, this is equivalent to 
 \begin{equation} \label{svtosolveequation}   b^{-1}\big|_{\gamma}   \,  \phi( F_{\infty} \CC^{\mm}_{\gamma})  \,   b   = \CC_{\gamma}^{\mm}\   \qquad \hbox{ for all } \gamma\in \Gamma\ .
 \end{equation} 
 where the $\Gamma$-action on $\U^{dR}_{1,1} (\Pe^{\mm}_{\Hc})$ is via $\gamma \mapsto \pi \gamma^{\mm}: 
\Gamma \rightarrow S^{dR}(\Pe^{\mm}_{\Hc})$ (see $(\ref{piSm})$). This formula holds because 
$\chi(   F_{\infty} \pi  \gamma^{\mm}) =  \pi  \gamma^{\mm}$.
 We know that $[(b,\phi)]=[(b^{\mm}_{\sv}, \phi^{\mm}_{\sv})]$ is a solution to equation $(\ref{svtosolveequation})$.
 It is equivalent  to the equations: 
 \begin{eqnarray} \label{bphibsvequations} 
  b^{-1}\big|_{T} \phi(F_{\infty} \CC^{\mm}_T ) b   &=&   \CC^{\mm} _T   \\ 
  b^{-1}\big|_{S}  \phi(F_{\infty} \CC^{\mm}_S ) b  &=&   \CC^{\mm} _S \ . \nonumber  
 \end{eqnarray} 
 Taking the period defines a pair of defining equations for $[(b_{\sv}, \phi_{\sv})]$, viz.
 $$  b^{-1}\big|_{T} \phi( \overline{\CC}_T ) b   =   \CC _T \quad \hbox{ and } \quad  b^{-1}\big|_{S} \phi( \overline{\CC}_S ) b   =   \CC _S $$
where $\CC_g = \per \, \CC^{\mm}_g \in \GG^{dR}_{1,1}(\C)$ and the bar denotes complex conjugation.  Note
that by lemma \ref{lemCanFromHol}, the comparison of $\CC$ with the canonical cocycle in part 1 involves changing variables $(X,Y)\mapsto (2 \pi i X, 2 \pi i Y)$, which we have suppressed from the notation for simiplicity.   Since all coefficients of $\CC$ are  homogeneous in $X,Y$ of even degrees, this change of variables involves scaling by  powers of $(2 \pi i)^2$ which is real, and therefore this operation   commutes with complex conjugation. 
 Recall that the $\Gamma$-action on $\U^{dR}_{1,1}(\C)$ in this case is via $\mathrm{comp}_{B,dR}: S^B(\Q) \rightarrow S^{dR}(\C)$. 
 
 \begin{cor} The equations $(\ref{bphibsvequations})$ have a unique solution $[(b_{\sv}^{\mm}, \phi^{\mm}_{\sv})] \in \A^{dR}(\Pe^{\mm}_{\Hc})$. 
  The single-valued cocycles satisfy, for all $g \in \Gamma$,
 \begin{eqnarray}   \label{CsvasBoundary}
  \CC^{\sv, \mm}_{g} & = & (b^{\mm}_{\sv})^{-1}\big|_{ \pi g^{\mm}}    b^{\mm}_{\sv}  \quad \in \quad \U^{dR}_{1,1}(\Pe^{\mm}_{\Hc})  \\
  \CC^{\sv}_{g}  &= &  b_{\sv}^{-1}\big|_{\overline{g}}  b_{\sv}  \quad \qquad \in \quad  \U^{dR}_{1,1}(\C) \ . \nonumber
  \end{eqnarray}  \end{cor}   
 
 \begin{proof}
 We have already seen that $\sv^{\mm}$ exists and solves the equations  $(\ref{bphibsvequations})$.  Uniqueness holds
 by   corollary \ref{corZ1torsor}.  The single-valued cocycles are defined by the action of $\sv^{\mm}$ and $\sv$ upon the elements $\gamma^{dR} = (\gamma, 1) \in S^{dR} \rtimes \U^{dR}_{1,1}(\Q)$. 
 Equations $(\ref{CsvasBoundary})$ follow from $(\gamma, 1) \circ (b, \phi) = (b^{-1}\big|_{\chi(\gamma)} b, 1)$ where $\chi\in \Aut(S)$ is the image of $\phi$. 
 \end{proof}

  \subsection{Single-valued iterated integrals of modular forms}
    We apply the single-valued construction to indefinite iterated integrals on $\M_{1,1}$.

  \subsubsection{Relative completion of the fundamental groupoid}   
  Let $\tau $ be a point in the extended upper half plane $\HH \cup_{\Q \cup \{\infty\}} \C$ of remark \ref{remUnivCov}.  Write
    $$\GG^{B/dR}_{1,1}(\tau) = \pi_1^{B/dR,S}(\M_{1,1}, \tbp, \tau)  \  .$$
 These form  a right torsor over $\GG^{\bullet}_{1,1}$:
 \begin{equation} \label{GGtauTorsor} \GG^{\bullet}_{1,1}(\tau) \times \GG^{\bullet}_{1,1} \To \GG_{1,1}^{\bullet}(\tau)\ , \qquad \hbox{ where } \bullet = B, dR \ .
 \end{equation} 
Write $\Gamma = \pi_1^{\tp}(\M_{1,1}, \tbp)$, and denote the natural map $\gamma\mapsto \gamma^B: \Gamma \rightarrow \GG^B_{1,1}(\tau)(\Q)$. The group  $\Gamma$ acts on the extended upper half plane on the left.  Let $p_{\tau}$ denote the unique homotopy class of paths from $\tbp$ to $\tau$.  Then for any $\gamma \in \Gamma$,
 $$p_{\gamma\tau} = p_{\gamma} . \gamma_*(p_{\tau})$$
 where $p_{\gamma}$ denotes the path from $\tbp$ to $\gamma \tbp$, and $\gamma_*(p_{\tau})$ goes from $\gamma \tbp$ to $\gamma \tau$.
  Throughout this paper, our convention has been that the fundamental group acts on the right on local systems, and so  the image of this relation in $\GG^B_{1,1}(\tau)$ is 
 \begin{equation}  \label{pBgammatau}
 p^B_{\gamma\tau} = \gamma_*( p^B_{\tau}) \circ (\gamma^B)^{-1} 
 \end{equation} 
  in accordance with $(\ref{GGtauTorsor})$, where $\gamma^B \in \GG^B_{1,1}(\Q)$. The image of the previous equation under the comparison isomorphism yields an analogous relation in $\GG_{1,1}^{dR}(\tau)(\C)$.

  \subsubsection{Trivialisation} 
Now let us fix  a splitting of the $W$-filtration in the de Rham realisation, which is possible by a version of \S\ref{sectSplittings},
 or via the method of power series connections. This  provides a map 
$S^{dR}(\Q) \rightarrow \GG^{dR}_{1,1}(\tau)(\Q)$.  Denote the  image of $1$ by 
 $1_{\tau}  \in   \GG^{dR}_{1,1}(\tau)(\Q)$.   
   The torsor structure $(\ref{GGtauTorsor})$ produces  an isomorphism 
  \begin{equation} \label{GGdrtauTrivialisation} g\mapsto 1_{\tau}. g \quad :  \quad \GG^{dR}_{1,1} \overset{\sim}{\To} \GG^{dR}_{1,1}(\tau)
  \end{equation}
  of schemes, 
  via which we shall sometimes view $\GG^{dR}_{1,1}(\tau)$ as an affine group scheme, by transport of structure. Via the isomorphism of schemes 
  $\GG^{dR}_{1,1}(\tau) \cong S^{dR} \times \U^{dR}_{1,1}$, we shall denote the image of the  Betti path $p^B_{\tau}$ by 
  $$\comp_{B, dR} (p^B_{\tau}) = ( 1,   I_{\tau})   \qquad \in \quad  S^{dR} \times \U^{dR}_{1,1}(\C)\ .$$
  Thus the image of equation $(\ref{pBgammatau})$ is the equation 
  $$ (1, I_{\gamma \tau}) = (\gamma, \gamma_* I_{\tau}) \circ (\gamma, \CC_{\gamma})^{-1}$$
  since the image of $\gamma^B$ in $\GG^{dR}_{1,1}(\C)= S^{dR} \ltimes \U^{dR}_{1,1}(\C)$ is  $(\gamma, \CC_{\gamma})$,  and hence 
    \begin{equation} \label{Igammatautrans}  I_{\gamma \tau} \big|_{\gamma} \CC_{\gamma} = \gamma_* I_{\tau} \ .
    \end{equation}
  The image of this equation in $\PiU(\C)$ is equivalent to equation $(\ref{Cgamdef})$, since $I(\tau) $ is the rescaled image of $I_{\tau}$ under the holomorphic projection \S\ref{sectTotHolQuotient}, 
  by lemma \ref{lemCanFromHol}.

  \subsubsection{}   We now construct a complex point $\widetilde{\sv}$ in a certain group of automorphisms of $\GG^{dR}_{1,1}(\tau)$. It is defined by the following  equation, which is the analogue of $(\ref{svmdef})$:  
  \begin{equation} \label{svtildedefn}  \overline{I_{\tau}} \circ \widetilde{\sv} = I_{\tau} \ ,
  \end{equation} 
  where the bar denotes complex conjugation. The element $\widetilde{\sv}$  defines  an isomorphism
   $$ \Or(\GG^{dR}_{1,1}(\tau)) \otimes \C  \overset{\sim}{\To} \Or(\GG^{dR}_{1,1}(\tau))\otimes \C\ . $$
 By composing with  $1_{\tau}$  we obtain  a homomorphism we denote by
  $$\sv_{\tau} : \Or(\GG^{dR}_{1,1}(\tau)) \To \C\ .$$
    Let $\A^{dR}_{\tau}$ denote the group of (right) automorphisms   of    $\GG^{dR}_{1,1}(\tau)  \times \GG^{dR}_{1,1}$ which  preserve the torsor structure, and act on $\GG_{1,1}^{dR}$ via $\A^{dR}$.
  
  \begin{lem}  There is an isomorphism
  \begin{eqnarray}
  \A^{dR}_{\tau} & \overset{\sim}{\To} & \A^{dR} \ltimes \GG^{dR}_{1,1}(\tau) \nonumber \\
  \alpha  &  \mapsto & ( \alpha\big|_{\GG^{dR}_{1,1}} , \alpha(1_{\tau}))\ .  \nonumber 
  \end{eqnarray} 
      \end{lem}

  \begin{proof} Let $\alpha \in \A^{dR}_{\tau}$. Since every element of $\GG^{dR}_{1,1}(\tau) $ is uniquely of the form $1_{\tau}.g$ and 
  $(1_{\tau}. g) \circ \alpha  = \alpha(1_{\tau}) . \alpha(g)$, it follows that $\alpha$ is uniquely determined by $\alpha(1_{\tau})$ and the restriction of $\alpha$ to $\GG^{dR}_{1,1}$, which is by definition in  $\A^{dR}$. Conversely, an element  $(\beta, w) \in \A^{dR} \times \GG^{dR}_{1,1}(\tau) $  acts upon $\GG^{dR}_{1,1}(\tau)$ by sending  the element 
  $1_{\tau}. g$ to $w. \beta(g)$.  
  \end{proof}
  A general construction \cite{NotesMot}, \S8.3 provides an element  $\widetilde{\sv}\in \A^{dR}_{\tau}$ which satisfies  $(\ref{svtildedefn})$, 
and whose image  in   $\A^{dR}(\C)$ is  $\sv \in \A^{dR}(\C)$. Therefore we can write
  $$\widetilde{\sv}    =  ( \sv, \sv_{\tau}    ) \quad \in \quad   \A^{dR}\times  \GG^{dR}_{1,1}(\tau)  (\C)  \ .$$
  where $\sv_{\tau} = 1_{\tau} \circ \widetilde{\sv}$.

  \begin{lem} Let  $\sv= [(b_{\sv},\phi_{\sv})] \in \A^{dR}(\C)$ be defined as earlier.  Then  
  $$\sv_{\tau} = I_{\tau}  \ . \ b_{\sv}^{-1}  \phi_{\sv}(  \overline{1}_{\tau} )^{-1} b_{\sv}  \ .$$
   \end{lem}   
  \begin{proof}  By $(\ref{GGtauTorsor})$, there exists a unique    $\alpha_{\tau} \in \GG^{dR}_{1,1}(\C)$ such that $\overline{I_{\tau}}= 1_{\tau}. \alpha_{\tau}$. Then 
$$ I_{\tau} \  \overset{(\ref{svtildedefn})}{=} \  \overline{I_{\tau}} \circ \widetilde{\sv} =  (1_{\tau} . \alpha_{\tau}) \circ \widetilde{\sv}    = \sv_{\tau} . \sv(\alpha_{\tau})$$
It follows that $\sv_{\tau}  = I_{\tau}.  \sv(\alpha_{\tau})^{-1}.$ On the other hand, $\alpha_{\tau} \in \GG^{dR}_{1,1}(\C) = S^{dR} \ltimes \U^{dR}_{1,1}(\C)$ 
maps to the element $1$ in $S^{dR}(\C)$. Let us (abusively) denote its unipotent component  by $\overline{I}_{\tau} \in \U^{dR}_{1,1}(\C)$, since it is the image
of      $\overline{I_{\tau}}$  under the trivialisation $(\ref{GGdrtauTrivialisation})$.  The action of $\sv$ on $\U^{dR}_{1,1}$ is via $b^{-1}_{\sv} \phi_{\sv} b_{\sv}$, which leads to the stated formula. 
  \end{proof}

    Let us  use the more suggestive notation $I^{\sv}_{\tau} \in  \U^{dR}_{1,1}(\tau) (\C) $ for the unipotent component of $\sv_{\tau}=(1,I^{\sv}_{\tau}) $.  It is the `single-valued' version of $I_{\tau}$ and satisfies
    \begin{equation}  \label{Isvbartau} I^{\sv}_{\tau} = I_{\tau}   b^{-1}_{\sv} \phi_{\sv}( \overline{I_{\tau}} )^{-1} b_{\sv}\ .
    \end{equation} 
 It is instructive to compute the cocycle associated to $I^{\sv}_{\tau}$ directly.   Applying $\gamma_*$ to the previous equation and substituting  $(\ref{Igammatautrans})$ and its complex conjugate, we obtain  
 \begin{eqnarray} \gamma_* I^{\sv}_{\tau}  &=  &  \big(  I_{\gamma \tau} \big|_{\gamma}   \CC_{\gamma} \big) b^{-1}_{\sv} \phi_{\sv} ( \overline{I_{\gamma \tau} }\big|_{\overline{\gamma}} \overline{\CC}_{\gamma} )^{-1} b_{\sv}        \nonumber \\
 & = &   I_{\gamma \tau} \big|_{\gamma}   \CC_{\gamma}  b^{-1}_{\sv} \phi_{\sv}(\overline{\CC}_{\gamma}^{-1})  \phi_{\sv} ( \overline{I_{\gamma \tau} }\big|_{\overline{\gamma}} )^{-1} b_{\sv}  \nonumber \\ 
 & = &   I_{\gamma \tau} \big|_{\gamma}    b_{\sv}^{-1} \big|_{\gamma}     b_{\sv} \big|_{\gamma}     \CC_{\gamma}  b^{-1}_{\sv} \phi_{\sv}(\overline{\CC}_{\gamma})^{-1}  \phi_{\sv} ( \overline{I_{\gamma \tau} })^{-1}  \big|_{\gamma} b_{\sv} \ . \nonumber 
    \end{eqnarray} 
   On the other hand, by definition of $[(b_{\sv}, \phi_{\sv})]$, we have $ \phi_{\sv}( \overline{\CC}_{\gamma})  =b_{\sv}\big|_{\gamma} \CC_{\gamma} b^{-1} _{\sv}$, so the terms in the middle of the  previous expression cancel to give
   \begin{eqnarray}   \gamma_* I^{\sv}_{\tau}  & =&  I_{\gamma \tau} \big|_{\gamma} b_{\sv}^{-1} \big|_{\gamma}  \phi_{\sv} ( \overline{I_{\gamma \tau} })^{-1}  \big|_{\gamma} b_{\sv}\big|_{\gamma}   b^{-1}_{\sv}\big|_{\gamma} b_{\sv}     \nonumber \\
   & = &   I^{\sv}(\gamma \tau)\big|_{\gamma}  b^{-1}_{\sv}\big|_{\gamma} b_{\sv}  \nonumber \ .
   \end{eqnarray} 
 We conclude  using $(\ref{CsvasBoundary})$ that the cocycle associated to $I^{\sv}$  is the  cocycle  $\gamma \mapsto \CC^{\sv}_{\overline{\gamma}}$,  where $\CC^{\sv}$ is the single-valued cocycle: 
    \begin{equation} \label{CsvasacocycleofIsv}  \gamma_*   I^{\sv}_{\tau} = 
 I^{\sv}_{\gamma \tau}\big|_{\gamma} \, \CC^{\sv}_{\overline{\gamma}}\ .          
 \end{equation}

    \subsection{A class of real-analytic modular forms} 
The single-valued iterated integrals $(\ref{Isvbartau})$ are not quite $\Gamma$-equivariant.  However, since their associated cocycle $\gamma \mapsto \CC^{\sv}_{\overline{\gamma}}$ is a
coboundary, we can modify them to produce a $\Gamma$-equivariant element.

\begin{defn} Choose a representative $(b_{\sv}, \phi_{\sv})$ for $[(b_{\sv}, \phi_{\sv})]$. Define a generating series of  \emph{equivariant iterated modular integrals} to be
\begin{eqnarray} \label{Ieqvdef}  I^{\eqv}_{\tau}  &  =&  I^{\sv}_{\tau}  b^{-1}_{\sv}  \\
& = &    I_{\tau}   b^{-1}_{\sv} \phi_{\sv}( \overline{I_{\tau}} )^{-1}  \ .  \nonumber
\end{eqnarray} 
Note that it is well-defined up to right multiplication by an element of $(\U^{dR}_{1,1})^{\Gamma}$. 
Define the \emph{ring of equivariant iterated modular integrals} to be the ring over  $(\U^{dR}_{1,1})^{\Gamma}$ generated by the  coefficients of $I^{\eqv}_{\tau}$. 
It is well-defined (independent of choices). 
\end{defn}

\begin{cor}The element $I^{\eqv}_{\tau}$ is  modular invariant
$$\gamma_* I^{\eqv}_{\tau} =  I^{\eqv}_{\gamma \tau} \big|_{\gamma} \ .$$
The value of $I_{\tau}$ at $\tau = \tone_{\infty}$ is $b^{-1}_{\sv}$. 
\end{cor}
\begin{proof}
This follows from  equations  $(\ref{CsvasacocycleofIsv})$ and $(\ref{CsvasBoundary})$, or by a similar computation to that which precedes $(\ref{CsvasacocycleofIsv})$. The second statement follows from  $I^{\sv}(\tone_{\infty})=1$. 
\end{proof}

\begin{defn} Given an $\SL_2$-equivariant map $w:(V^{dR}_{2n})^{\vee} \rightarrow  \Or(\U^{dR}_{1,1}) $, define 
\begin{equation}\label{fwdef} f^{\eqv}_w(\tau) =  w( I^{\eqv}_{\tau})   \ .
\end{equation} 
\end{defn} 
It defines a section 
\begin{eqnarray} f^{\eqv}_{w} : \HH &\To& V^{dR}_{2n}(\C) \quad \hbox{ such that } \nonumber \\
  f^{\eqv}_w(\gamma(\tau))\big|_{\gamma} & =&  f^{\eqv}_w(\tau)  \, \, \qquad \hbox{ for } \gamma \in \Gamma \ . 
\end{eqnarray} 
It therefore transforms like a modular form of weight $2n+2$.  It is a linear combination of products of  iterated integrals of modular forms and their 
complex conjugates, and therefore defines a real analytic function on $\HH$.  

The  ring of equivariant iterated modular integrals  is generated by the $f^{\eqv}_w$. It therefore defines a class of  non-holomorphic modular forms, and merits further study. 

\begin{rem}  The previous equation is clearly stable under right-multiplication by any element of $(\U^{dR}_{1,1})^{\Gamma}$. Thus, for example,  the $\Gamma$-invariant element $(\Xs_1 \Ys_2 - \Ys_2 \Xs_1)^{2n}$  is viewed  in this context as a  `modular form' of weight $2n+2$.  
\end{rem}
  
Consider  the image $I^{\eqv,\hol}(\tau)$  of $I^{\eqv}_{\tau} \in \U_{1,1}^{dR, \hol}$.   It  satisfies
$$ I^{\eqv,\hol}(\tau) = I(\tau)   (b^{\hol}_{\sv})^{-1} \pi^{\hol}(  \phi_{\sv}( \overline{I_{\tau}} )^{-1}) \ ,$$
  where $\pi^{\hol}: \U_{1,1}^{dR} \rightarrow \U_{1,1}^{\hol, dR}$ is the quotient map, and $b^{\hol}_{\sv} = \pi^{\hol} b_{\sv}$ and $I(\tau) = \pi^{\hol} I_{\tau}$ was studied in  \S\ref{sectIteratedEichler}. By lemma \ref{lemCanFromHol}, the comparison involves implicitly scaling  $(X,Y) \mapsto ( 2  \pi i X, 2  \pi i  Y)$.  Note that since $\phi_{\sv}$
  does not necessarily commute with $\pi^{\hol}$, it may involve iterated integrals which are not  totally holomorphic. 
  
   \begin{prop} The function $ I^{\eqv,\hol}(\tau)$ is well-defined up to right multiplication by an element of 
   $(\U^{dR, \hol}_{1,1})^{\Gamma}$ and satisfies 
   \begin{eqnarray}
   I^{\eqv,\hol}(\gamma(\tau))\big|_{\gamma} &  = &  I^{\eqv,\hol}(\tau)  \nonumber \\
 { \partial \over \partial \tau}  I^{\eqv,\hol}( \tau ) &  = & - \Omega(\tau) \, I^{\eqv,\hol}(\tau)  \nonumber  \\
   I^{\eqv,\hol}( \tone_{\infty}) & =   & (b^{\hol}_{\sv})^{-1} \ . \nonumber 
       \end{eqnarray} 
  Its coefficients satisfy the shuffle equations. 
      \end{prop} 
   \begin{proof} It only remains to prove the differential equation. Since $I^{\eqv,\hol}$ is obtained from $I(\tau)$
   by right-multiplication by an anti-holomorphic function of $\tau$, it satisfies the same differential equation  (proposition \ref{propEichler}) with respect to $\tau$.
   \end{proof}

 \subsection{Examples} 
 Let $\e_{2n+2} : ( V_{2n}^{dR} )^{\vee}\rightarrow  \Or(\GG^{dR}_{1,1})  $ denote the coefficient of $\e_{2n+2}$. We work with the rescaled Betti generators $ 2 \pi i  X, 2 \pi i Y$ in accordinace with lemma \ref{lemCanFromHol}.
 
 \begin{lem}  \label{lemsvzetavalue} The coefficient of $\e_{2n+2}$ in  $b_{\sv}$ is 
 $(2n)! \zeta(2n+1) Y^{2n}.$
The equivariant integral of an  Eisenstein series is  the coefficient of $\e_{2n+2}$ in $I^{\eqv, \hol}(\tau)$. It is:
$$ f^{\eqv}_{\e_{2n+2}} (\tau) = 2  \mathrm{Re}\, \displaystyle{\int^{\tbp}_{\tau}} \underline{E}_{2n+2}(\tau)   +  (2n)! \zeta(2n+1) Y^{2n}\ .$$
 \end{lem} 
 \begin{proof} Take the coefficient of $\e_{2n+2}$ in the two defining equations
 \begin{eqnarray}   \label{inproofdefinequations} b_{\sv}^{-1}\big|_{T} \phi_{\sv} (\overline{C}_T) b_{\sv}  & = & C_T    \\
   b_{\sv}^{-1}\big|_{S} \phi_{\sv} (\overline{C}_S) b_{\sv}  &= & C_S\ . \nonumber 
   \end{eqnarray} 
   Since $\e_{2n+2}$ is a copy of the  object $\Q_{dR}(1)$, we know by proposition \ref{lemsvchi} that $\phi_{\sv}$ acts upon it by $-1$. 
  The first equation is almost exactly the inertial condition $(I)$, and so  by a similar argument to lemma $\ref{lemsimpleKerNargument}$, we know
  that the coefficient of $\e_{2n+2}$ in $b_{\sv}$ is a lowest weight vector $f  Y^{2n}$, for some $f\in \C$. 
 The second  yields the  equation 
 \begin{multline} f (Y^{2n}- X^{2n})   -  \big({(2n)! \over 2} \zeta(2n+1) (X^{2n}-Y^{2n}) -  (2\pi i)^{2n+1} e^0_{2n}(X,Y) \big) \\ =  {(2n)! \over 2} \zeta(2n+1) (X^{2n}-Y^{2n}) + (2\pi i)^{2n+1} e^0_{2n}(X,Y) \ .
\end{multline} 
The minus sign in the first line is once again due to the fact that $\phi_{sv}$ acts on $\e_{2n+2}$ by $-1$. 
 It follows  that $f =- (2n)! \zeta(2n+1) $. This is as  expected since $2\zeta(2n+1)$ is indeed the single-valued version of $\zeta(2n+1)$ (\cite{NotesMot}, \S4). Now 
 take the coefficient of $\e_{2n+2}$ in  the equation $I_{\tau}   b^{-1}_{\sv} \phi_{\sv}( \overline{I_{\tau}} )^{-1}$ to obtain the formula for $f^{\eqv}_{\e_{2n+2}}$. 
   \end{proof}
 The function $f^{\eqv}_{\e_{2n+2}}$ is  expressible as a real analytic Eisenstein series (lemma $\ref{lemImpart}$).

 \subsubsection{Equivariant double Eisenstein integrals} \label{SectEquivDoubleEis}
 Let $w^k_{a,b} :(V_{2a+2b-2k}^{dR})^{\vee} \rightarrow  \Or(\GG^{dR}_{1,1})  $ denote the  dual of the projection of $S^{dR}$-representations, 
 $\partial^k: V^{dR}_{2a} \otimes V^{dR}_{2b} \rightarrow V_{2a+2b-2k}^{dR}$ followed by the map `coefficient of $\e_{2a+2}\e_{2b+2}$'.

 We shall work in the completed universal envelopping algebra of $\U^{dR}_{1,1}$, which allows us to write $\phi_{\sv} = \id + \phi'_{\sv}$. Then 
  \begin{equation} 
  \label{Ieqvshape} 
  I^{\eqv}_{\tau} = I_{\tau} b_{\sv}^{-1} \overline{I_{\tau}}^{-1} + I_{\tau} b^{-1}_{\sv} \phi'_{\sv}( \overline{I_{\tau}})^{-1}\ .
  \end{equation} 
 Take the coefficient of $\e_{2a+2}\e_{2b+2}$ in each expression on the right-hand side of $(\ref{Ieqvshape})$.  The   term on the left gives 
 \begin{multline} \label{LHterm}
  I_{\e_{2a+2} \e_{2b+2}}(\tau) -  I_{\e_{2a+2}}(\tau)  \overline{I_{\e_{2b+2}}(\tau)} +  \overline{ I_{\e_{2b+2} \e_{2a+2}}(\tau)}   \\
  - I_{\e_{2a+2}}(\tau) (b_{\sv})_{\e_{2b+2}} + (b_{\sv})_{\e_{2a+2}} \overline{I(\tau)}_{\e_{2b+2}} + (b_{\sv})^{-1}
 _{\e_{2a+2}\e_{2b+2}}\ .
 \end{multline} 
 The last term is constant (does not depend on $\tau$), and we shall drop it. The coefficients of $\e_{2n+2}$ in $b_{\sv}$ were determined in the previous lemma.   The first three terms  of $(\ref{LHterm})$ can be simplified using  the shuffle product relation $(\ref{phishuffle})$:
 $$I_{\e_{2a+2} } (\tau) I_{\e_{2b+2}}(\tau)  = I_{\e_{2a+2}, \e_{2b+2}}(\tau)  + I_{\e_{2b+2},\e_{2a+2}}(\tau) $$
  which implies that the  leading term  of $(\ref{LHterm})$ is $2 i \mathrm{Im} ( I_{\e_{2a+2} \e_{2b+2}}(\tau) )$, and  its  differential is equal to the function   $2  i \mathcal{F}_{2a+2,2b+2}(\tau)$ defined in \S\ref{sectDoubleEisensteinFF},  modulo products of real-analytic Eisenstein series $f^{\eqv}_{\e_{2n+2}}$.
 Now take the coefficient of $\e_{2a+2}\e_{2b+2}$ in the right-hand term on the right-hand side of $(\ref{Ieqvshape})$.   It is equal to the coefficient of $\e_{2a+2}\e_{2b+2}$ in $\phi'_{\sv}( \overline{I_{\tau}})^{-1}$, which is 
 $$ - \sum_f {}_fc^k_{a+1,b+1} I_{\e_f}(\overline{\tau})   -   \lambda_{k+1}^{a+1,b+1} I_{\e_{2k+2}}(\tau) \ ,$$
where the  sum is over  a basis of normalised Hecke eigencusp forms $f$ of weight $2a+2b-2k$. The number ${}_fc^k_{a+1,b+1} \in \C$ is  the coefficient of $\partial^k \e_{2a+2} \e_{2b+2}$
in $\phi_{\sv}(\e_f)$, and $\lambda_{k+1}^{a+1,b+1} \in \C$ is the same coefficient in $\phi_{\sv}(\e_{2k+2})$.  We conclude that
\begin{equation} \label{wkcongtoF}
w^k_{a,b}  \big( I^{\eqv}(\tau) \big) \equiv  2  i \,\partial^k [\mathcal{F}_{2a+2,2b+2}(\tau)]  - \sum_f {}_fc^k_{a+1,b+1} I_{\e_f}(\overline{\tau})   -   \lambda_{k+1}^{a+1,b+1} I_{\e_{2k+2}}(\tau)   
\end{equation}
modulo  constants and products of real-analytic Eisenstein series $f^{\eqv}_{\e_{2a+2}}$. 
The undetermined coefficients can be computed as follows.  Since $w^k_{a,b}  \big( I^{\eqv}(\tau) \big)$ is modular, its associated cocycle is trivial.  The cocycle associated to a constant function of $\tau$ is a coboundary. Therefore, by
taking the cocycle of  $(\ref{wkcongtoF})$, and taking the Petersson inner product with the cocycle $C_f$ of $f$, we obtain the equation
\begin{equation}   {}_fc^k_{a+1,b+1}  \{\overline{C_f}, C_f\}     =  2i  \{ I^k_{2a+2,2b+2}, C_f\} \ .
\end{equation}
The right-hand side of this expression was computed in theorem \ref{thmImEab}. In particular,   the coefficients ${}_fc^k_{a,b}$ are not all zero.   
 The coefficients $\lambda_k^{a,b}$ were computed in \S\ref{sectArithmeticdelta}. 

\subsection{A class of modular forms arising from multiple elliptic polylogarithms}
We can restrict the above construction along  the geometric monodromy map
\begin{equation} \label{geommonod} \mu:  \U^{dR}_{1,1} \To \mathrm{Aut}\, \U^{dR}_{\Eq}\ .
\end{equation} 
The element $\sv^{\mm} \in \GG^{dR}_{\Hc}( \Pe^{\mm}_{\Hc})$ restricts to 
an element in $\mathrm{Aut} \, (\U^{dR}_{\Eq}) (\Pe^{\mm}_{\Hc})$.
Since the object $\uu^{dR}_{\Eq}$ is  a mixed Tate motive over $\Z$, the image of $b_{\sv}$  (resp. $b_{\sv}^{\mm}$) under the map $(\ref{geommonod})$  has 
coefficients given by single-valued (motivic) multiple zeta values. The same is true for the restriction of  the automorphism $\phi_{\sv}$ (resp. $\phi_{\sv}^{\mm}$). In conclusion, the image of $I^{\eqv}(\tau) \in \U^{dR}_{1,1}(\C)$ under the geometric monodromy $(\ref{geommonod})$ defines a generating
series of iterated integrals of Eisenstein series
$$\mu I^{\eqv}(\tau) = \mu(I(\tau)) \, \mu(b) \,  \mu \phi(  \overline{ I(\tau)})^{-1}$$
which are modular invariant. They have the property that their values at $\tbp$ are linear combinations of single-valued multiple zeta values.

\section{Zeta and modular elements}

We define non-trivial zeta elements $\sigma_{2n+1}$ and modular elements $\sigma'_f(d), \sigma''_f(d)$ which lie in the image
of  $(\uu^{dR}_{\Hc}\otimes \overline{\Q})^{ab}$ in $(\mathrm{Lie} \, \A^{dR}_{\U}\otimes \overline{\Q})^{ab}$.

\subsection{Primitive elements}  Split the de Rham  $W$ and $M$ filtrations as in \S \ref{sectSplittings}. Continuing  \S\ref{sectDRperiods}, 
for any $v \in  \Or(\GG_{1,1}^{dR})^{\vee}$ and $w\in  \Or(\GG^{dR}_{1,1})$ we have an element $(\ref{(g,w)defn})$
$$(v, w )^{\dR}    \in     \Or(\A^{dR})  $$ 
Via the morphism $\GG^{dR}_{\Hc} \rightarrow \A^{dR}$, the image of $(v, w)^{\dR} $ in  $\Pe_{\Hc}^{dR}= \Or(\GG^{dR}_{\Hc})$ is $$      [ \Or(\GG^{\Hc}_{1,1}), v, \omega ]^{\dR} \quad \in \quad \Pe^{\dR}_{\Hc_{\MMM}}\ .$$
Consider the special case when $v$ is the image of $\gamma \in S^{dR}(\Q)$. Its value on an element $[(b,\phi)] \in \A^{dR}$, where $\phi \in \Aut(\U^{dR}_{1,1})^{S^{dR}, \chi}$  is 
\begin{equation} \label{wbgammabinv} w  \big(  (\gamma, 1) \circ [(b,\phi)] \big) =  w \big(   b^{-1}\big|_{\chi(\gamma)}-b\big)\ .
\end{equation} 
\subsubsection{Definition} \label{defnofprimitiveelements}  Consider the following  elements in $\Or(\A^{dR})\otimes \overline{\Q}$. Their images in $\Pe^{\dR}_{\Hc_{\MMM}}\otimes \overline{\Q}$ will be denoted by the same symbol without
ambiguity.
\begin{enumerate}
\item For all $n\geq 1$, define an element $f_{2n+1} \in \Or(\A^{dR})$  by
$$ f_{2n+1} = (S, {2 \over (2n)!} \, \e_{2n+2} \, \Ys^{2n})^{\dR}  \ , $$
where $S \in S^{dR}(\Q)$ is given by \S\ref{sectGammaST}.
\vspace{0.1in}

\item  Let $f$ be a Hecke eigenform of weight $2n+2$, and let $m\geq 1$. For every $0 \leq k  \leq 2 \min \{m,n\}$ define
$$ g_{f,2m+2}^{(k)} = (S,    [ \e_{f}\Ys^{2n-k}  \  , \ \e_{2m+2} \Ys^{2m-k} \  ]\, (\Xs_1\Ys_2-\Xs_2\Ys_1)^k  )^{\dR} $$
If $k = 2m=2n$ then $g_{f,2m+2}^{(k)}=0$, since the right-hand term is $S^{dR}$-invariant in this case.  This follows from $(\ref{wbgammabinv})$.
\vspace{0.1in}

\item With the same notations as $(2)$, define 
 $$  {}_fc_{2a+2,2b+2}^{(k)} =  ( \e_f \Ys^{2n} \ ,  \    [\e_{2a+2} \Ys^{2a-k}, \e_{2b+2}  \Ys^{2b-k}] (\Xs_1 \Ys_2 - \Xs_2 \Ys_1)^{k})^{\dR}$$ 
where $a+b = n+k$,  $a,b\geq 1$, and $0\leq k \leq 2 \min\{a,b\}$.    They  satisfy
$${}_fc_{2a+2,2b+2}^{(k)} = - {}_fc_{2b+2,2a+2}^{(k)}\ .$$

\vspace{0.1in}
\end{enumerate}
It will turn out that the families of elements $(2)$ and $(3)$ are related to each other.  

\begin{rem} Note that the families $(2)$ and $(3)$ in fact define a rank 2 submodule of $\Or(\A^{dR})\otimes \overline{\Q}$. Choosing
a basis $\e'_f, \e''_f$ of $\e_f$ as in remark \ref{reme'fe''f} gives rise to pairs of elements      $ g_{f',2m+2}^{(k)},  g_{f'',2m+2}^{(k)} \in \Or(\A^{dR})\otimes \overline{\Q}$ and similarly for the $ {}_fc_{2a+2,2b+2}^{(k)}$.  With this in mind, we shall refer to the families $(2)$ and $(3)$ simply as elements of 
$\Or(\A^{dR})\otimes \overline{\Q}$. 
\end{rem}

The elements $(1)$ and $(2)$ can also  be defined via the construction $(\ref{geometricneck})$, and    $(3)$ can also be defined via the  representation $\A^{dR} \rightarrow \mathrm{Aut}(\U^{dR}_{1,1})/(\U^{dR})^S$, as the following lemma shows.

\begin{lem} Every element $(1)$ or $(2)$ of the form  $(S, w)^{\dR}$ satisfies
\begin{equation}  \label{gammawoftype(1,2)action}
(S, w)^{\dR} [(b, \phi)]   =  w(b)  \qquad \hbox{ for all } [(b,\phi)] \in \A^{dR}_{\U}\ .
\end{equation} 
Every element  $(3)$ of the form $(v, w)^{\dR}$ satisfies:
\begin{equation}    \label{gammawoftype(3)action} (v,w)^{\dR}  [(b,\phi)] = w(\phi(v))   \qquad \hbox{ for all } [(b,\phi)] \in \A^{dR}\ . \end{equation} 
\end{lem} 

\begin{proof} It follows from the formulae \S\ref{sectLtoR} for the action that 
    $$(S, w)^{\dR} [(b, \phi)]  = w(b^{-1}\big|_S) +w(b)\ ,$$
    and $\phi$ plays no role. It suffices to show that the first term vanishes.    Consider the case $(2$).    If $b$ is the exponential of an element in $\uu^{dR}_{1,1}$ of Tate type, then for reasons of type, both $w(b)$  and $w(b^{-1}\big|_S)$ vanish and $(\ref{gammawoftype(1,2)action})$ holds.  
If $b$  is not of Tate type, it vanishes in length 1, by  lemma \ref{lemcusphead}, i.e. $b \equiv 1 \pmod{L^2}$ . Therefore by lemma $\ref{lemsimpleKerNargument}$
its  coefficients in length two are lowest weight vectors. Hence the coefficients of length two in  $b^{-1}\big|_S$ are highest weight vectors, and $w(b^{-1}\big|_S)=0$, since $w$ is of length 2 and  never a highest-weight vector. 
The case $(1)$ is easier, and follows directly from lemma $\ref{lemsimpleKerNargument}$;  since $w$ is a lowest weight vector it must vanish on $b^{-1} \big|_S$. 

Now consider an element $(v,w)^{\dR}$ of type $(3)$. By  \S\ref{sectLtoR},   it satisfies
 $$ (v,w)^{\dR} [(b, \phi)] = w ( b^{-1} \phi(v) b )\ .$$ 
Since $v$ is of length one, we have  $\phi(v) \equiv \lambda (v) \mod L^2$, where $\lambda \in S^{dR}_{\Hc}$.
Since $\lambda(v)$ is of cuspidal type  it can never occur as a subword of $w$, which consists of two Eisenstein elements. 
It follows that
$$w(b^{-1}\phi(v) b) = w (\phi(v))\ .$$ 
\end{proof}

\subsubsection{Primitives} 
The left action of  $\U^{dR}_{\Hc}$ on $\Pe^{\dR}_{\Hc}$ defines a right coaction 
\begin{equation} \label{DeltaPUcoaction} \Delta: \Pe^{\dR}_{\Hc} \To \Pe^{\dR}_{\Hc} \otimes \Or(\U^{dR}_{\Hc})\ .
\end{equation}
For any $p \in \Pe^{\dR}_{\Hc}$, $p(ug) = m \Delta p (g \otimes u)$ where $m$ is multiplication, for all $u\in \U^{dR}_{\Hc}$, $g \in \GG^{dR}_{\Hc}$. 
An element $p \in \Pe^{\dR}_{\Hc}$ is \emph{primitive} if it lies in  $C_1 \Pe^{\dR}_{\Hc}$, where $C_i$ denotes the coradical filtration \cite{NotesMot}, \S2.5. Equivalently, if  $p^{\uu}$ is the image
of $p $ under the natural map 
$\Pe^{\dR}_{\Hc} = \Or(\GG^{dR}_{\Hc}) \rightarrow \Or(\U^{dR}_{\Hc})$, and if  $p(1)=0$, then $p$   is primitive if and only if 
\begin{equation} \label{Deltauforpu} \Delta^{\uu}(p^{\uu}) = p^{\uu} \otimes 1 + 1 \otimes p^{\uu}
\end{equation} 
where $\Delta^{\uu}$ is the coproduct on $ \Or(\U^{dR}_{\Hc})$ dual to  multiplication on $\U_{\Hc}^{dR}$.
By $(\ref{Deltauforpu})$,    $p \in \Or(\GG^{dR}_{\Hc})$  satisfying $p(1)=0$ is primitive if and only if, for all $u_1,u_2 \in \U^{dR}_{\Hc}$, 
\begin{equation} \label{padditiveequation} p(u_1u_2) = p(u_1) + p(u_2) \ . \end{equation}
If $P$ denotes the set of primitive elements, then  $P  \Or(\U^{dR}_{\Hc})  = P \Or((\U^{dR}_{\Hc})^{ab})$.

\begin{prop}
The elements $(1)$-$(3)$ defined above are primitive.
\end{prop} 
\begin{proof}  We shall prove a slightly stronger version of  $(\ref{padditiveequation})$. For the elements $(1)$ and $(2)$ we shall in fact show that for all $u \in \U^{dR}_{\Hc}$ and $g \in \GG^{dR}_{\Hc}$, we have
\begin{equation}\label{versionpadditive12} 
p( u g) =   p(u)  + p(g) \ .
\end{equation} 
For this, let the image of 
 $g$ in $\A^{dR}$ be $[(b_1,\phi_1)] $  and the image of $u$ in $\A^{dR}_{\U}$ be $[(b_2,\phi_2)]$.  The image of  $u  g $ is $[(b_2  \phi_2(b_1)  , \phi_2 \phi_1)]$. Let $p= (\gamma, w)^{\dR}$ where $w \in \Or(\U^{dR}_{1,1})$ and $\gamma \in S^{dR}(\Q)$. 
  By $(\ref{gammawoftype(1,2)action})$,  equation $(\ref{versionpadditive12})$ is equivalent to: 
 \begin{equation}\label{inproofwprimequation}  w(b_1) + w(b_2)   = w(  b_2 \phi_2(b_1))\ .
\end{equation} 
 Let us work in the envelopping algebra of $\GG^{dR}_{1,1}$ and   write  $\phi_2 = \id + \phi$.
If $w$ is a word of length one, then $w(u u') = w(u) + w(u')$ for any $u,u' \in \U^{dR}_{1,1}$. Since $\phi$ strictly increases the length of words, it follows almost immediately that $(\ref{inproofwprimequation})$ holds for the element $(1)$. 
Now consider  $(2)$. 
For a word $w= ef$ of length two in $\Or(\U^{dR}_{1,1})$, we have 
 $$w(u u') = w(u )  + e(u) f(u') + w(u') $$
 for any $u,u' \in \U^{dR}_{1,1}$, since the multiplication in $\U^{dR}_{1,1}$ is dual to the deconcatenation coproduct in $\Or(\U^{dR}_{1,1})$. 
   By lemma \ref{lemcusphead}, the coefficient of $\e_f \Xs^i \Ys^j$ in  $b_1,b_2$ vanishes. 
   Therefore  the right-hand side of   $(\ref{inproofwprimequation})$ is 
   $$w(b_2) + w (\phi_2(b_1)) = w(b_1) + w(b_2) + w(\phi(b_1))\ .$$
   By lemma    \ref{lemcusphead}, the only non-zero coefficient of length one in $b_1, b_2$  is of the form $\e_{2n+2} \Ys^{2n}$. By the second equation of $(\ref{GeomLowestWeightAndDeltaepsilon})$, $\phi(  \e_{2n+2} \Ys^{2n}) \in L^3$, and so   $\phi(b_i) \in L^3$ for $i=1,2$. Therefore 
   $w ( \phi(b_1))=0$ since $w$ is of length two, and $\phi$ strictly increases the length. 
   This proves that  $(\ref{versionpadditive12})$ and in particular  $(\ref{padditiveequation})$ is satisfied, so the elements of the form $(2)$ are indeed primitive.

For the element $(3)$ we shall prove the following stronger version of $(\ref{padditiveequation})$:
\begin{equation}\label{versionpadditive3} 
p( u g) =   p(g)  + p_{g} (u) \ ,
\end{equation} 
for all $u \in \U^{dR}_{\Hc} \times \overline{\Q}$ and $ g\in \GG^{dR}_{\Hc} \times \overline{\Q}$, 
where $p_{g}= ( g v, w)^{\dR}$, and $g$  acts on  $v=\e_f \Ys^{2n} \in \gr^M \uu^{dR}_{1,1}$ through $S^{dR}_{\Hc}$, since $v$ is isomorphic to a copy of $V_f^{dR} (1+2n)$.

Let $[(b_i, \phi_i)]$ for $i=1,2$ be as before.   By $(\ref{gammawoftype(3)action})$, equation $(\ref{versionpadditive3})$ is equivalent to 
$$w(\phi_1(v)) + w(\phi_2( \lambda v)) = w (\phi_2 \phi_1(v))$$ 
where $\phi_1 (v) \equiv \lambda( v)  \mod L^2$. The equation follows by writing  $\phi_1(v) = \lambda(v) + f_1 \mod L^3$, where $f_1 \in L^2$, and checking that  $\phi_2\phi_1(v) \equiv \phi_2(\lambda(v))+ f_1 \mod L^3$.   Since $w$ has length 2, 
$w(\phi_2\phi_1(v)) = w( \phi_2(\lambda(v)) + w(f_1)$, and $w(f_1) = w(\phi_1(v))$. 
\end{proof}

These computations could be much simplified using an explicit formula for the coaction of $\Or(\A_{\U}^{dR})$ upon 
$\Or(\GG^{dR}_{1,1})$ which is dual to our formulae for  the action of $\A^{dR}$ on $\GG^{dR}_{1,1}$.  
This would have increased the length of the paper substantially, but would be required in order to proceed further with the generation conjecture of 
\S\ref{SectConjels}.

\subsection{Types} The group $S^{dR}_{\Hc}$ acts, via the $M$-splitting, on $\Or(\U^{dR}_{\Hc})$ by conjugation (the coaction  $(\ref{DeltaPUcoaction})$ is  $S^{dR}_{\Hc}$-equivariant). Its action on $ \Or((\U^{dR}_{\Hc})^{ab})$ is independent of the splitting. An element
in $ P \Or(\U^{dR}_{\Hc_{\MMM}})\otimes \overline{\Q}=P \Or((\U^{dR}_{\Hc_{\MMM}})^{ab})\otimes \overline{\Q}$ generates a representation of $S^{dR}_{\Hc_{\MMM}}(\overline{\Q})$, and hence an object in $\Hc_{\MMM}^{ss}\otimes \overline{\Q}$. 

\begin{lem} The images of the elements $(1)$-$(3)$ in $\Or(\U^{dR}_{\Hc})$ satisfy:
\begin{eqnarray}
& (1)& f^{\uu}_{2n+1}   \qquad  \qquad \,\,\,\quad \hbox{ is of type } \Q(-2n-1) \nonumber \\
& (2)& \big(g^{(k)}_{f,2m+2}\big)^{\uu}    \quad \qquad \hbox{ is of type }  V_f^{\Hc}(-1- 2m+k) \nonumber \\
& (3)& \big({}_fc^{(k)}_{2a+2,2b+2}\big)^{\uu}    \,\,\,\quad \hbox{ is of type }   V_f^{\Hc}(-1- k)\nonumber 
\end{eqnarray} 
Note that there are only finitely many elements $(2)$, $(3)$ of any fixed type.
\end{lem} 

\begin{proof} Since $S^{dR}_{\Hc}$ acts on $\Or(\U^{dR}_{1,1})$  by conjugation, it follows from $(\ref{gammawoftype(1,2)action})$ that  the $S^{dR}_{\Hc}$-action on the image of $(\gamma, w)^{\dR}$  of type $(1)$ or $(2)$
 in $\Or(\U^{dR}_{\Hc})$ is given by the $S^{dR}_{\Hc}$-action on $w$. For $(1)$, $w$ is dual to $\e_{2n+2}\Ys^{2n} \cong \Q(1)(2n)= \Q(2n+1)$. For $(2)$, to 
 $$\big(V^{dR}_f(1)(2n-k) \otimes \Q(1)(2m-k) \big) (k) = V_f^{dR}( 2n+ 2+ 2m-k)\ .$$
 Its dual is computed using  $(V_f^{dR})^{\vee}= V_f^{dR}(2n+1)$. 
 By $(\ref{gammawoftype(3)action})$,  the  $S^{dR}_{\Hc}$-action on the image of $(v, w)^{\dR}$  of type $(3)$  is given by the $S^{dR}_{\Hc}$ action on  $v^{\vee} \otimes w$, which is  dual to 
 $$\big(V^{dR}_f(1)(2n)\big)^{\vee} \otimes \big(\Q(1)(2a-k) \otimes \Q(1)(2b-k)\big) (k) =V^{dR}_f(2+2a+2b-k)\ .$$
 This is equal to  $V^{dR}_f(2n+2+k)$,   using $a+b =n+k$.
\end{proof}

\begin{example} Let $f$  be the normalised cuspidal Hecke eigenform of weight 12.
Then, of type $V_f^{\Hc}(-1)=(V_f^{\Hc}(12))^{\vee}$, of $M$ degree -$13$,  we have constructed elements
$$
\begin{array}{|c|c|c|c|c|c|c|c|} \hline
W   = & -5 & -7 & -9 & -11  &          -13 \\ \hline 
     & g^{(2)}_{f,4}  & g^{(4)}_{f,6}  &g^{(6)}_{f,8}  & g^{(8)}_{f,10} &       {}_fc^{(0)}_{4,10} \ , \    {}_fc^{(0)}_{6,8}      \ , \   {}_fc^{(0)}_{ 8,6} \ , \  {}_fc^{(0)}_{10,4}  \\
     \hline  \end{array}
$$
Of type $V_f^{\Hc}(-2) = (V_f^{\Hc}(13))^{\vee}$, of $M$-degree -$15$ we have  two sets of five  elements:
$$
\begin{array}{|c|c|c|c|c|c|c|c|c|} \hline
W   = & -5 & -7 & -9 & -11  &          -13  & -15 \\ \hline 
     & g^{(1)}_{f,4}  & g^{(3)}_{f,6}  &g^{(5)}_{f,8}  & g^{(7)}_{f,10} &  g^{(9)}_{f,10} &     {}_fc^{(1)}_{4,12} \ , \    {}_fc^{(1)}_{6,10}      \ , \   {}_fc^{(1)}_{ 8,8} \ , \  {}_fc^{(1)}_{10,6} \ , \ {}_fc^{(1)}_{12,4}   \\
     \hline  \end{array}
$$
In general, the number of terms of type $V_f^{\Hc}(2n+2+r)$ are equal for the families $(2)$ and $(3)$ and depend only on the weight  $2n$ of the modular form $f$  and the parity  of $r$. Indeed, as we shall see, they are in one-to-one correspondence with the coefficients of the even or odd period polynomials of $f$. 
\end{example} 

\subsection{Extensions} 
Equation $(\ref{versionpadditive12})$  implies that  elements $f_{2n+1} \in \Pe^{\dR}_{\Hc}$  satisfy  the equation 
$\Delta f_{2n+1} = f_{2n+1} \otimes 1 + 1 \otimes f_{2n+1}^{\uu}$. It follows that  the object generated by $f_{2n+1}$, which is the object of $\Hc$ whose de Rham realisation is the $\GG^{dR}_{\Hc}$-representation generated by $f_{2n+1}$,  
is an extension $\mathcal{E}$ in the category $\Hc$ of the form:
\begin{equation} \label{Ext(1)}  0 \To \Q \To \mathcal{E} \To \Q(-2n-1) \To 0\ .
\end{equation}
Therefore $f_{2n+1}$ is equivalent to a de Rham period of $\mathcal{E}$.
Similarly, the elements  $g^{(2m-k)}_{f,2m+2}$ are equivalent to de Rham periods of extensions
\begin{equation} \label{Ext(2)}  0 \To \Q \To \mathcal{E} \To V_f^{\Hc}(-1 -2m +   k)  \To 0   
\end{equation}
in $\Hc \otimes K_f$. 
On the other hand,  equation $(\ref{versionpadditive3})$ implies that $ {}_fc^{(k)}_{2a+2,2b+2}$ is equivalent to a de Rham period of
an extension in  $\Hc \otimes K_f$ of the form
\begin{equation} \label{Ext(3)} 
 0 \To V_f^{\Hc} (2n+1)^{\vee} \To \mathcal{E} \To \Q(-2-2n -  k)  \To 0 \ . 
 \end{equation}
In each  case, the corresponding group of extensions in $\Hc\otimes \R$ are one-dimensional:
$$\dim\, \mathrm{Ext}^1_{\Hc\otimes \R}(\Q, \Q(2n+1)) =    1 \quad \hbox{ and } \quad \dim\, \mathrm{Ext}^1_{\Hc\otimes \R}(\Q, V_f^{\Hc}(2n+2 +k)) =    1\ .$$ 
This  follows from \S\ref{Examplesprovedinpaper}.  Beilinson's conjecture therefore  predicts that there exist  relations between the elements $(2)$ and $(3)$ of  similar types,  as indeed we shall show.

\subsubsection{Non-triviality}  The first task is to prove that the extensions $(\ref{Ext(1)}) -(\ref{Ext(3)})$ do not split. We shall do this as follows.  Suppose that we have an extension $M$ in $\Hc\otimes \overline{\Q}$ of the form 
  \begin{equation} \label{Mextension} 0 \To A \To M \To B \To 0
  \end{equation} 
  where $A, B$ are simple objects, and such that the Hodge filtration on $M_{dR}$ splits the weight filtration. This means that there exists an $m$ such that $F^m A_{dR}=0$ and $F^m B_{dR}=B_{dR}$, which implies that $M_{dR}/F^m M_{dR} \cong A_{dR}$ and hence
  $$M_{dR} = A_{dR} \oplus B_{dR}\ .$$
\begin{lem} Suppose that $M$ is such an extension. Let $b =(0,b) \in M_{dR}$ and $a=(a,0) \in M_{dR}^{\vee}$ supported on $B_{dR}$ and $A_{dR}^{\vee}$ respectively.   If $(\ref{Mextension})$ splits, its single-valued period vanishes:
$$\sv [ M, a, v]^{\dR} = 0 \ .$$
 \end{lem} 
  \begin{proof} Given an  isomorphism $M \cong A\oplus B$ in the category $\Hc\otimes \R$, we deduce the following identity of  matrix coefficients
  $$[ M, a, b]^{\dR} = [ A \oplus B ,  (a,0), (0, b) ]^{\dR} = [A, a, 0]^{\dR} + [B, 0, b]^{\dR}\ .$$
  The two objects on the right-hand side vanish, so $[M,a,b]^{\dR}$ is itself zero. In particular, its single-valued period vanishes. 
  \end{proof} 
   
   It follows from the description of the Hodge filtrations in \S \ref{sectMainOb}, together with $(\ref{gammawoftype(1,2)action})$,  that the conditions of the lemma are satisfied for each of the three
   families of extensions associated to the elements $(1)$-$(3)$.  The next task is to compute their images under $\sv$.

 \begin{rem} The single-valued period will determine the  extension class in $\Hc\otimes \R$, but not in $\Hc$, for it could happen that a non-trivial extension $\Hc$ has vanishing single-valued period. However, in the case $(1)$,  
 $ \Q(2n+1)$ has rank one, so the single-valued period determines all the matrix coefficients of the period matrix, and 
 this uniquely determines the extension (see the proof of corollary \ref{Corzetam2n+1}).   In the cases $(2)$ and $(3)$ the extension class in $\Hc\otimes K_f$ is determined by \emph{two} periods since $V^{\Hc}_f$ has rank $2$.
 \end{rem}

\subsection{Single-valued periods}
We defined $\sv \in \A^{dR}(\C)$ to be the image of $\sv \in \GG^{dR}_{\Hc}$. 
It  extends by linearity to a homomorphism $\sv :  \Or(\A^{dR})\otimes \overline{\Q} \rightarrow \C$ (recall  $\overline{\Q} \subset \C$).

\begin{thm}   \label{thmcomputesvperiods} The single-valued periods of the elements $(1)$  are given by 
$$\sv (f_{2n+1} )=  - (2n)! \zeta(2n+1) \ .$$
For fixed $f,k$, the $K_f$-vector space generated by the   numbers  $\sv( {}_f c^{(k)}_{2a+2,2b+2} ) $
for varying $a,b$ is one-dimensional, and equal to 
$$      (2\pi i )^{k}  \Lambda(f,2n+2+k) \,K_f  , $$
where $K_f$ is   the field generated by the Fourier coefficients of $f$.
\end{thm} 

\begin{proof}  For an element of the form $(g, w)^{\dR}$, where $g \in S^{dR}(\Q)$, we have
$$\sv (g,w)^{\dR} = w( (g,1) \circ \sv) = w( (g,1) \circ [(b_{\sv}, \phi_{\sv})] = w( b_{\sv}^{-1}\big|_{\overline{g}}  b_{\sv})\ .$$
For the elements $(1)$, this amounts to the coefficient of $\e_{2n+2}\Ys^{2n}$ in $  b_{\sv} $ by $(\ref{gammawoftype(1,2)action})$.
By lemma   \ref{lemsvzetavalue},  this is $-(2n)! \zeta(2n+1)$, which is, in particular,  non-zero.

For the element $(3)$ of the form $(v, w)^{\dR}$, it follows from $(\ref{gammawoftype(3)action})$ that it is 
$$\sv (v,w)^{\dR} =   w( \phi_{\sv}(v))$$
so we want the coefficient of $\partial^k     \e_{2a+2} \e_{2b+2}$ in $\phi_{\sv}(\e_f)$.  These were computed in  
\S \ref{SectEquivDoubleEis}. Combining this with  theorem \ref{thmImEab} gives the result.
\end{proof}

\begin{rem}  \label{remnonzercab} In fact, the proof shows more, namely that the  $\sv( {}_f c^{(k)}_{2a+2,2b+2} )$, for fixed $f$ and $k$, are proportional, by  some non-zero and  explicit constant of proportionality,  to the coefficients 
in the (odd if $k$ odd, even if $k$ even) period polynomial of $f$:
$$\sv( {}_f c^{(k)}_{2a+2,2b+2} ) \quad \hbox{ is proportional to }   \quad \Lambda(f,2a+1-k) \Lambda(f, 2n+2+k)\ .$$
 Since $\Lambda(f,r)$ is non-zero for all $r > n+2$ (see \S\ref{sectLvalues}), 
we deduce that the $c^{(k)}_{2a+2,2b+2}$ are non-zero for $2a >n+k+1$. Since they are anti-symmetric in $a$ and $b$, they are also non-zero  for $2b > n+k+1$ also. Since $a+b=n+k$, 
this means that the $c^{(k)}_{2a+2,2b+2}$ are non-zero whenever $|a-b|\geq 2$. 
\end{rem}

\subsection{Definition of zeta and modular  elements}

\subsubsection{Zeta elements} By $(\ref{H1uudescription})$, define elements 
$$\sigma_{2n+1} \in H_1(\uu^{dR}_{\Hc};\Q) =  \big( \uu_{\Hc}^{dR}\big)^{ab}$$
which are dual to the elements $f_{2n+1} \in \Or(\U^{ab}_{\Hc})$  for all $n\geq 1$. This means that they satisfy 
$\sigma_{2n+1}(f_{2n+1})=1$ and that their image in all other $S^{dR}$-isotypical components of  $(\ref{H1uudescription})$ are zero. 
Thus 
$$\sigma_{2n+1} \in \mathrm{Ext}^1_{\Hc}(\Q, \Q(2n+1))^{\vee} \otimes \Q_{dR}(2n+1)$$
By theorem \ref{thmcomputesvperiods}, and the discussion which precedes it,  $\sigma_{2n+1} \neq 0$.
We shall call these  `zeta elements'.  By abuse of notation, we shall also refer to any choice of element  in $ \uu^{dR}_{\Hc}$ of type $\Q(2n+1)$ whose image in $H_1$ is $\sigma_{2n+1}$ as a `zeta element' also. 
Any such element has  a geometric head of the form
\begin{equation} \label{headsigma2n+1} h(\sigma_{2n+1})   =  - {2 \over (2n)!}  \e_{2n+2} \Ys^{2n} \ .
\end{equation}

\subsubsection{Modular elements} Let   $f$ be a Hecke  eigencusp form of weight $2n+2$, and $ k\geq 0$, 
and  $K_f$ the field generated by the Fourier coefficients of $f$. Write $d= 2n+2+k$. 
 Let us choose 
indices $a,b$  in the allowed range  (\S\ref{defnofprimitiveelements} $(2)$)  and an $\alpha \in K_f$  such that 
$$\sv(\alpha\, {}_f c^{(k)}_{2a+2,2b+2}) = (2\pi i)^{2n+2} \Lambda(d) \ ,$$
by theorem \ref{thmcomputesvperiods}.
This equation makes sense for  any $a,b$ for which $\sv  ({}_f c^{(k)}_{2a+2,2b+2})$
 is non-zero.  For any such $a,b$, let us define
$$\sigma^{a,b}_f(d) \in \mathrm{Hom}_{S^{dR}}(V^{dR}_f (d),  H_1(\uu^{dR}_{\Hc};K_f))$$
to be the element dual to $\alpha\, {}_f c^{(k)}_{2a+2,2b+2}$.  It defines a pair of  elements 
$$\sigma^{a,b}_{f'} \ , \sigma^{a,b}_{f''} \   \qquad \in \qquad   \mathrm{Ext}^1_{\Hc\otimes K_f}(K_f, V^{\Hc}_{f}(d)) \otimes_{K_f} V^{dR}_{f}(d)\  , $$
images of a choice of basis $\e'_f, \e''_f$ of $\e_f$ as in remark \ref{reme'fe''f}. 
The element $\sigma^{a,b}_f$  is of type $V_f^{\Hc}(d)$ and is non-zero by theorem \ref{thmcomputesvperiods} and the discussion preceding it.

\begin{rem} If Beilinson's conjecture holds for the motive $V_f(d)$ of $f$, then the elements $\sigma_f^{a,b}(d)$ are independent of $a,b$. Nevertheless  we can prove the

\begin{prop} \label{thmsigmasequal}
Let $2n+2 \in \{12,16,18,20,22,26\}$. Then $\dim S_{2n+2}(\Gamma)=1$, and  the elements $\sigma_f^{a,b}(d)$ are independent of the choice of indices $a,b$.
\end{prop}

\begin{proof}
The proof is postponed to \S\ref{SectCompat}. 
\end{proof}
\end{rem} 

\begin{defn} Define `modular elements'
$$\sigma_f(d) \in \mathrm{Hom}_{S^{dR}}(V_f^{dR}(d),  H_1(\uu^{dR}_{\Hc};K_f)) $$
to be equal to $\sigma^{a,b}_f(d)$, were  
$b$ is minimal such that $\sv({}_f c^{(k)}_{2a+2,2b+2}) \neq 0 $.    By  remark  \ref{remnonzercab}, we can take $b$ to be the smallest integer such that $2b \geq \min\{ k,2\}$. \end{defn}

 By the usual abuse of terminology, a modular element will  sometimes  refer to a choice of lift  to $\uu^{dR}_{\Hc}\otimes K_f$ of type $V_f^{dR}(d)$.  Choosing a basis $\e'_f, \e''_f$
for $V_f^{dR}$ as in remark \ref{reme'fe''f}, we write $\sigma_{f'}(d)$ and $\sigma_{f''}(d)$ to be their images in 
  $H_1(\uu^{dR}_{\Hc};K_f)$ ( or $\uu^{dR}_{\Hc} \otimes K_f$).

\subsection{Relation between the families $(2)$ and $(3)$}
 Consider an element in the  image of  $\sigma= \sigma_f(d)$ in $\Lie \, \A^{dR}_{\U} \otimes \overline{\Q}$.  Since it is not of Tate type,
it can be written in the form $[(b,\delta)]$, where $b\in L^2 \uu^{dR}_{1,1} \otimes \overline{\Q}$  by lemma \ref{lemcusphead}, and $\delta \in L^1 \mathrm{Der} \, (\uu^{dR}_{1,1})^{S^{dR}} \otimes \overline{\Q}$. 
Write $\exp \sigma = [(B,\phi)]$ where $B \equiv 1 \mod L^2$ and $\phi \equiv \id \mod L^2$.

\subsubsection{Via transference}
Let $s$ be a choice of cocycle defined in \S\ref{sectPeriodsofAuto}.  Let $f$ be a cuspidal  Hecke eigenform of weight $2n+2$,  and let $k \geq 0$, and $a,b\geq 1$ such that $a+b=n+k$. Denote by 
$\pi: V^{dR}_{2n} \otimes V^{dR}_{2a} \otimes V^{dR}_{2b} \rightarrow V^{dR}_0$  an $S^{dR}$-equivariant projection onto $V^{dR}_0$. 
In order to simplify the notation, let us write $\ax,\bx,\fx$ for $\e_{2a+2}, \e_{2b+2} ,\e_f$ in this section, and write a subscript $w$ for `the coefficient of $w$'.

The cocycle $s$ satisfies the transference equation $(\ref{TransferenceACME})$ 
\begin{equation} \label{stransfer} \pi   \Big( \h   \big( s_{\ax}, s_{\bx\fx} \big)  +   \h  \big( s_{\ax\bx} , s_{\fx} \big)  +   s_{\ax\bx\fx}(T)   \Big) =0  \ . 
\end{equation}
The cocycle $s' = s \circ [(B,\phi)]$ also satisfies transference. If $w$ is a word of length two,
$$ s'(\gamma)_{w} = (B^{-1}\big|_{\gamma}  \Phi(s(\gamma)) B)_{w} =B_w^{-1}\big|_{\gamma}  + \phi(s(\gamma))_w+ B^{\gamma}_w $$
since the coefficients of $B$ in length one vanish.  For the same reason, the shuffle product (or the fact that $B$ is group-like) implies  that $B^{-1}_w = - B_w$.  The right-hand side of the previous expression is the cochain
$$\phi(s)_w-  \partial B_w  $$
 where $\partial B_w$ is the coboundary $\gamma \mapsto   B_w\big|_{\gamma} -  B_w $.  Since $\phi \equiv \id \mod L^2$,  
$s'_w = s_w$ for $w$ a word of length one. 
Therefore transference for $s'$ is the equation
  $$   \pi  \Big( \h (s_{\ax},  \phi(s)_{\bx \fx} -\partial B_{\bx \fx} )    +   \h(   \phi(s)_{\ax \bx} - \partial B_{\ax \bx}  ,s_{\fx}) + s_{\ax\bx\fx}(T)\Big)=0  $$
 using the fact that  $s'_T = s_T$ because $\exp(\sigma)$ preserves the local monodromy at the cusp. Write $\phi' = \phi-\id$, and  subtract the 
 transference equation $(\ref{stransfer})$ for $s$, to obtain
  $$   \pi  \Big( \h (s_{\ax},   \phi'(s)_{\bx \fx} - \partial B_{\bx \fx})    +   \h(   \phi'(s)_{\ax \bx}-  \partial B_{\ax \bx}  ,s_{\fx})\Big)=0  $$
  We know from the second equation of $(\ref{GeomLowestWeightAndDeltaepsilon})$  that $\phi'(s)_{\bx \fx}$
  is  a cuspidal cocycle. Using the  fact that cuspidal cocycles are orthogonal to coboundaries and Eisenstein cocycles, the previous expression reduces to the equation
  $$  \pi   \Big(   \h( \phi'(s)_{\ax \bx} ,s_{\fx})  -   \h(s_{\ax}, \partial B_{\bx \fx} )    \Big)=0 \ . $$
  The right-hand term is the coefficient of $\bx \fx$ in $b$, which is exactly $g^{(2b+2-k)}_{f,2b+2}(\sigma)$, multiplied by the inner product of the Eisenstein cocycle  $s_{\ax}$ with a coboundary, which is non-zero by  $(\ref{e0cuppedwithboundary})$. The left-hand term  is exactly ${}_fc^{(k)}_{2a+2,2b+2} (\sigma)$ multiplied by $\h(\s_{\e_f},\s_{\e_f})$, which is proportional to $\{\s_{\e_f}, \s_{\e_f}\}$ by lemma \ref{lemhiscurly} and is  also non-zero. 
 We deduce that $g^{(2b+2-k)}_{f,2b+2}(\sigma)$ is some (explicit) non-zero multiple of ${}_fc^{(k)}_{2a+2,2b+2} (\sigma)$.

 \subsubsection{Via the inertial condition} Let $\sigma = [(b,\delta)] \in \Lie\, \A^{dR}_{\U}$ be any element, and $f,n,k,d$ as above. The inertial condition  $(I)$ implies that 
 $$[b, \varepsilon_0^{\vee}] + [b, N^{dR}_+] + \delta(N^{dR}_+) =0 $$
 Reduce this equation modulo  $L^4$, and substitute $(\ref{Ndrfull})$.  This gives
 \begin{multline} [b, \varepsilon_0^{\vee}] +  \sum_{k\geq 1} \big(  [b, \e_{2k+2} \Xs^{2k} ] + \delta(\e_{2k+2} \Xs^{2k} )  \big){\Be_{2k+2} \over 4k+4} \\
 + \sum_{g} \delta(\PP_g (\Xs_1\Ys_2 - \Ys_1 \Xs_2)^{2w_g})  \equiv 0 \mod L^4 
 \end{multline}
 where the sum is over Hecke eigenforms $g$ of weight $2w_g+2$. 
  Project onto trivial $S^{dR}$-isotypical components, for example by first projecting onto highest weight vectors (which kills the first term), and then onto
 lowest weight vectors (which kills the third).  Denoting this projection by $\pi$, we deduce that
 $$ \sum_{k\geq 1}   \pi [b, {\Be_{2k+2} \over 4k+4}\e_{2k+2} \Xs^{2k} ]    +  \delta(\PP_f       (\Xs_1\Ys_2 - \Ys_1 \Xs_2)^{n}        )  \equiv 0 \mod L^4 \ .$$
 The right-hand term produces a linear combination of $S^{dR}$-isotypical terms of the form
 $$ {}_f c^{(k)}_{2a+2, 2b+2} (\sigma) \,  [\e_f, [\e_{2a+2}, \e_{2b+2}]]\ .$$
The left hand term produces a linear combination of $S^{dR}$-isotypical  terms of the form
$$ g^{(2m-k)}_{f,2m+2}(\sigma) \,  [   [\e_{f}, \e_{  2m+2}] , \e_{2n-2m+2k+2}]]  $$
 since $g^{(2m-k)}_{f,2m+2}(\sigma)$ is by definition the coefficient of a lowest-weight vector in $b$. By comparing types, and by the Jacobi identity, we obtain  a formula for $g^{(2b-k)}_{f,2b+2}(\sigma)$ and $g^{(2a-k)}_{f,2a+2}(\sigma)$
 in terms of $ {}_f c^{(k)}_{2a+2, 2b+2} (\sigma)$.  We deduce the

  \begin{cor} The images of  $(g^{(2b-k)}_{f, 2b+2})^{\uu}$  and $({}_fc^{(k)}_{2a+2,2b+2})^{\uu}$ in $\Or(\U^{dR}_{\Hc})\otimes K_f$ satisfy
  $$ \big( g^{(2b-k)}_{f, 2b+2}  \big)^{\uu} \quad \in  \quad \big( {}_fc^{(k)}_{2a+2,2b+2}\big)^{\uu}\, K^{\times}_f    \ .$$
   In particular, the elements $(2)$ of fixed type cannot all vanish. Indeed, by remark \ref{remnonzercab},  $\sigma_f(d)$ has a geometric head proportional to  $g^{(2)}_{f,4}$ if $d=2n+2$, 
to
  $$ g^{(0)}_{f,d-2n} \hbox{ if } d> 2n+2  \hbox{ even,} \qquad \hbox{ and  to } \quad
 g^{(1)}_{f,d+1-2n} \hbox{ if } d\geq 2n+3  \hbox{ odd}.
 $$
        \end{cor}

\section{Relations between the $\varepsilon_{2n}^{\vee}$ and modular forms} \label{sectRelations} 
At this point we can invoke the property $(5)$ of the Lie algebra $\Lie \, \A_{\U}^{dR}$ of automorphisms, namely  the stability of the space of relations $R^{dR}$. It follows that 
 the zeta and modular elements satisfy the condition: 
$$
\sigma_{2n+1} R^{dR} \subset  R^{dR}  \quad \hbox{ and } \quad 
\sigma_f(d)  \big(R^{dR}\otimes \overline{\Q}\big)    \subset   R^{dR}\otimes \overline{\Q}  
$$
Since the elements $\e_f \in R^{dR}\otimes{\overline{\Q}}$ we deduce in particular the equation 
\begin{equation} \label{sigmacusprel} 
\sigma_f(d) (\e_f)   \quad \in \quad R^{dR} \otimes \overline{\Q} \ .
\end{equation} 
Note that both $\e_f$  and $\sigma_f(d)$ are rank two $S^{dR}$-modules, but $(\ref{sigmacusprel})$ is indeed  a single relation. 
This  implies the existence of a relation corresponding to every Hecke eigenform  $f$  of weight $2n+2$ and an integer $d =2n+2+k$, where $k\geq 0$.  Combined with Pollack's explicit computations of the quadratic parts of these relations, we deduce an identity between 
periods of double Eisenstein series and a proof of theorem \ref{thmsigmasequal}.

\begin{rem}
Since any word containing a cuspidal element $\e_g$  lies in the space of relatons $R^{dR}$ by $(\ref{Liemonodromy})$, the action of  elements in $\Lie \, \A_{\U}^{dR}$ of modular degree $\geq 2$ will not give 
any information about the structure of $R^{dR}$. This is because, for such a $\sigma \in \Lie \, \A_{\U}^{dR}$, the equation $\sigma R^{dR} \subset R^{dR}$ is trivially satisfied for reasons of type. 

\end{rem}

\subsection{Pollack's computations}
Let us fix $M$ and $W$ splittings as in \S\ref{sectSplittings}.
Pollack \cite{Po}  computed the kernel in length two of the geometric monodromy map $(\ref{Liemonodromy})$:
\begin{equation} 
\gr^L_2 R^{dR}_{\eis}  = \mathrm{ker}\Big( \bigoplus_{a,b\geq 1}  [ \e_{2a+2} V^{dR}_{2a} ,  \e_{2b+2} V_{2b}^{dR}   ]  \To \mathrm{Der}\, \LL(\av,\bv) \Big) \  ,
\end{equation} 
where $R^{dR}_{\eis}$ was defined in $(\ref{RdREis})$.  It is bigraded by $M$ and $W$. Define 
\begin{equation} \label{ndasMW} 2n= W-M \qquad \hbox{ and } \quad 2d = - M  \ .\end{equation}
Let $\mathrm{lw}$ denote  lowest weight vectors. Pollack defined an injective linear map 
\begin{equation} \label{Pollackmap}
 \mathrm{lw} \,(\gr_{-2d}^{M} \, \gr^W_{2n-2d} \, \gr^L_2 \, R^{dR}_{\eis})  \To V_{2n} 
 \end{equation}
to the space of homogeneous polynomials in two variables $X, Y$ of degree $2n$. He  proved that the image is isomorphic 
to the subspace $W^{0, \pm}_{2n} \subset W_{2n}^{\pm}$ 
of the space of odd or even  period polynomials (\S\ref{sectPeriodPolys}) of weight $2n+2$ which satisfy $P(0,Y)=P(X,0)=0$. 
 The sign $\pm$ is $(-1)^k=(-1)^d$.  There is a canonically split exact sequence
 $$0 \To  \Q( X^{2n}-Y^{2n})^{\pm} \To W^{\pm}_{2n} \To W^{0, \pm}_{2n} \To 0 \ ,$$
 where we recall that $X^{2n}-Y^{2n}$ is the  image under $(\ref{Z1toW})$ of the coboundary cuspidal period polynomial in $Z^1_{\mathrm{cusp}}(\Gamma, V_{2n})$. 
 The third map sends a polynomial $P\in V_{2n}$ to $P- P(X,0)- P(0,Y)$.  It follows from the Eichler-Shimura isomorphism \S\ref{sectESisom} that 
 $$\mathrm{dim}_{\Q} \,W^{0, \pm}_{2n} = \mathrm{dim}_{\Q}\, S_{2n+2}(\Gamma)\ ,$$
 the  dimension of the space of (rational) cusp forms of weight $2n+2$.

\subsection{Compatibility with theorem  \ref{thmcomputesvperiods}}  \label{sectCompatibility}
Let $f$ be a normalised cuspidal Hecke eigenform  $f$ of weight $2n+2$, let $k\geq 0$  and let $d=2n+2+k$. 
Let $\sigma \in \Lie \,\A^{dR}_{\U}$ of $M$-degree $-2n-3-2k$.   Then for every $\sigma$ we obtain a relation :
$$\sigma(\e_f \Ys^{2n} ) \subset  R^{dR}\otimes \overline{\Q} $$
of bi-degrees $(M,W)$ as given by $(\ref{ndasMW})$.   To simplify notation, let us denote the left-hand side of $(\ref{Pollackmap})$ by $QR$.  The  projection $q$ of the previous relation onto words of length two in Eisenstein  generators  (i.e., the map $q$  means the quotient modulo all cuspidal generators $\e_g$ and modulo $L^3$, and lands in $L^2$ since the images of the $\e_{2n+2}$  under $(\ref{Liemonodromy})$ are linearly independent), 
lies in  $QR \otimes \overline{\Q}$. We shall
denote it by 
$$q(\sigma(\e_f \Ys^{2n}) ) \subset QR\otimes \overline{\Q} \ .$$
Explicitly, it is given by  the lowest weight vector
$$q(\sigma(\e_f \Ys^{2n} ))   = \sum_{a,b}  {}_f c^{(k)}_{2a+2,2b+2}(\sigma) [\e_{2a+2} \Ys^{2a-k}, \e_{2b+2}  \Ys^{2b-k}] (\Xs_1 \Ys_2 - \Xs_2 \Ys_1)^{k}   $$
by definition of ${}_f c^{(k)}_{2a+2,2b+2}$. 
Since this is true for all $\sigma$,  and since $\Pe^{\dR}_{\Hc} = \Or(\GG^{dR}_{\Hc}) \cong \Or(\U^{dR}_{\Hc}) \otimes \Or(S^{dR}_{\Hc})$ via our choice of splittings,  we deduce that there is a linear map
\begin{eqnarray}
S_{2n+2}(\Gamma) &  \overset{c^{(k)}}{\To} &  QR\otimes \Pe^{\dR}_{\Hc}\otimes \overline{\Q} \nonumber \\ 
f & \mapsto &   \sum_{a,b}  {}_f c^{(k)}_{2a+2,2b+2} [\e_{2a+2} \Ys^{2a-k}, \e_{2b+2}  \Ys^{2b-k}] (\Xs_1 \Ys_2 - \Xs_2 \Ys_1)^{k} \ .   \nonumber
\end{eqnarray} 
This holds because  the ${}_f c^{(k)}_{2a+2,2b+2}$ are of definite type, and hence uniquely determined by their restrictions to $ \Or(\U^{dR}_{\Hc})$.  By composing with  $(\ref{Pollackmap})$ we obtain  a linear map
\begin{equation}\label{ckmap} 
c^{(k)}:  S_{2n+2}(\Gamma) \To W^{0,(-1)^k}_{2n} \otimes \Pe^{\dR}_{\Hc}\otimes \overline{\Q}
\end{equation} 
Composing with  the single-valued period $\sv : \Pe^{\dR}_{\Hc} \rightarrow \C$  gives a linear map
\begin{equation} \label{svckmap} \sv(c^{(k)}):  S_{2n+2}(\Gamma) \To W^{0,(-1)^k}_{2n}\otimes \C\ .
\end{equation} 
The computations of $\sv(c^{(k)})$ given in theorems \ref{thmcomputesvperiods} and \ref{thmImEab} imply that
$$\sv(c^{(k)})(f)  \qquad \in \qquad  (2 i \pi )^k \Lambda(f,2n+2+k)\,  P_f^{(-1)^k}  K_f  $$
where $P_f^{\pm}$ is the image of the  period polynomial (\S\ref{sectPeriodPolys}) of $f$ in $W^{0,(-1)^k}_{2n}$. Since the Eichler-Shimura map \S\ref{sectESisom} is an isomorphism,
$$S_{2n+2}(\Gamma)\otimes \C  \overset{\sim}{\To}    W^{0,(-1)^k}_{2n+2}\otimes \C$$
 it follows that
$(\ref{svckmap})$ is surjective and that $\sv(c^{(k)})$ is equivariant for the action of Hecke operators.   Therefore $\sv(c^{(k)})$ is also injective, and in fact it is proportional, on each Hecke eigenspace,  to the Eichler-Shimura isomorphism.

\begin{rem} The methods of this paper do not quite allow us to prove that the $\sigma^{a,b}_f(d)$ are independent of $a,b$. 
For example, consider $n$  such that $\dim S_{2n+2}(\Gamma)=2$, and let $f ,g $ be a basis of normalised Hecke eigenforms of weight $2n+2$. 
We know that
$$c^{(k)}(f)  =  P^{\pm}_f  \alpha  + P^{\pm}_g \beta  \qquad \hbox{ and } \qquad  c^{(k)}(g) =  P^{\pm}_g \alpha' + P^{\pm}_f \beta'$$
where $\alpha, \beta,\alpha', \beta' \in \Pe^{\dR}_{\Hc}$, $\alpha$ and $\alpha'$ are non-zero,  the sign $\pm = (-1)^k$, and $\sv(\beta)= \sv(\beta')=0$, but  this is not enough to conclude  that $\beta=\beta'=0$, as we would expect.  
\end{rem}

Proposition \ref{thmsigmasequal} follows from $(\ref{ckmap})$, since in that case  $W^{0,(-1)^k}_{2n}$ is one-dimensional.

\subsection{Extension of Pollack's relations}
For every  $f$ and $d$ as above, choose a lifting of $\sigma_f(d)$ to  $ \Lie \,\A_{\U}^{dR} \times \overline{\Q}$ of type $\V_f^{dR}(d)$.
It provides a relation 
$$\sigma_f(d) (\e_f) \in R^{dR}_{\eis}\otimes \overline{\Q}\ .$$
In the previous paragraph, we showed that the map $(\ref{svckmap})$ is surjective.  This implies that the quadratic parts of these relations give rise to all of Pollack's relations.
\begin{thm}
The natural map
$$\mathrm{lw} (M_{-2d} W_{2n-2d} R^{dR}_{\eis})  \To  \mathrm{lw} \,(\gr_{-2d}^{M} \, \gr^W_{2n-2d} \, \gr^L_2 \, R^{dR}_{\eis})   $$
is surjective. In other words, every one of Pollack's quadratic relations arises from an actual  element in $R^{dR}_{\eis}$, and hence extends to all  lengths $L\geq 2$. 
\end{thm} 
 
A similar result was obtained in \cite{HaDBCA}, based also on theorem \ref{thmImEab} in this paper, but via a somewhat different method. 

\subsection{Zeta elements}
By the defining property $(R)$ of $\Lie \, \A^{dR}_{\U}$, there is a map 
\begin{eqnarray}  \label{LieAtoDerUe}
\Lie \, \A^{dR}_{\U}  &\To & \mathrm{Der} \, \ue    \\
{[}(b,\delta)] & \mapsto & \mathrm{ad}(b) + \delta  \ . \nonumber
\end{eqnarray} 
which is well-defined. The algebra $\ue$ was defined in \ref{defue}. 
Denote the  images of (a choice of) zeta elements $\sigma_{2n+1}$   by
\begin{equation}\label{imageofsigma}   \sigma_{2n+1} \in \mathrm{Der} \, \ue \  \qquad \hbox{ for all } n\geq 1 \  .
\end{equation}
The geometric part of these derivations were studied in \cite{BrSigma}.   

\begin{thm} The zeta element $(\ref{imageofsigma})$  can be written 
\begin{equation} \label{sigmaanatomy} \sigma_{2n+1} \equiv \mathrm{ad}(b_{2n+1}) \pmod {W_{-4n-2} \ue}  
\end{equation} 
for some $b_{2n+1} \in \ue$, where $b_{2n+1} \equiv \varepsilon_{2n+2}^{\vee} \pmod{W_{-2n-3}}$.  

The arithmetic  part of $\sigma_{2n+1}$ defines an element
$$\delta_{2n+1} \in  (\mathrm{Der} \, \ue)^{\mathrm{sl}_2}  /   (\ue)^{\mathrm{sl}_2}$$
whose initial terms are given explicitly   by  theorem \ref{thmphieis}. 
 \end{thm} 
\begin{proof}
The first part follows from the fact that if we write $\sigma_{2n+1} = [(b_{2n+1}, \delta_{2n+1})]$, where $\delta_{2n+1}$ is $S^{dR}$-equivariant, then by corollary \ref{cordeltaM=W}, $\delta_{2n+1}$ lies in $W_{-4n-2}$.  The second part is immediate.
\end{proof}
The  neck $(\ref{sigmaanatomy})$ of a zeta element in $\ue$ was called the `anatomy' in \cite{BrSigma}, remark 3.9, and stated without proof.
The action of  the arithmetic part $\delta_{2n+1}$ is encoded to first order  by  the coefficients $ \lambda^{2,n}_{n-1}$. They agree, up to a   normalisation  of the generators $\e_{2r}$, with the computations due to Pollack (\cite{Po}, \S5.3) for the action of $\delta_3$ on $\ue$. He conjectured correctly from a handful of examples that the coefficients of quadratic terms in this case are a quotient of two Bernoulli numbers $\Be_{2n}/ \Be_{2n-2}$.

\subsection{Speculation on the structure of $R$}
In \S\ref{sectFreeness} we shall prove that the $\sigma_{2n+1}$ and $\sigma_f(d)$ generate a free Lie algebra. For this reason it is natural to expect that 
$R^{dR}_{\eis}\otimes \overline{\Q}$  is the ideal generated by the  free  $\LL(\sigma_{2n+1} , n\geq 1)$-module spanned by  the images of the $\sigma_f(d) \e_f $.

\section{Freeness theorem} \label{sectFreeness}

Throughout this section, let us choose a splitting of the $M$ and $W$-filtrations on $\uu^{dR}_{1,1}$. These induce splittings on $\Lie \, \A_{\U}^{dR}$.
Recall that a derivation $\sigma \in \Lie \, \A_{\U}^{dR}$ has a \emph{geometric head} if it has a non-zero component in the region $W>M$.

\begin{prop} Let $\{ \sigma_i\}  \in \Lie\, \A_{\U}^{dR} \otimes \overline{\Q}$  denote a collection of elements which have geometric heads $h(\sigma_i) \in \uu^{dR}_{1,1}\otimes \overline{\Q}$. 
Suppose that the $h(\sigma_i)$ are independent in $\uu^{dR}_{1,1}\otimes \overline{\Q}$, i.e., generate a free Lie subalgebra  $\mathcal{B} \subset \uu^{dR}_{1,1}\otimes \overline{\Q}$. 
 Then the $\sigma_i$ generate a free Lie subalgebra $\mathcal{F}$ of $\Lie \, \A_{\U}^{dR}\otimes \overline{\Q}$, and  there is an isomorphism of Lie algebras: 
 \begin{eqnarray} h: \mathcal{F} & \overset{\sim}{\To}&  \mathcal{B} \nonumber  \\
 \sigma_i & \mapsto & h(\sigma_i) \nonumber \ . 
   \end{eqnarray} 
\end{prop} 

\begin{proof}  Let $\sigma_1,\ldots, \sigma_n$ denote any subset of these elements.  Suppose that   the head $h_i = h(\sigma_i)$ is of $W$-degree $w_i$, and write $\sigma_i$ in  the form $[(b_i, \delta_i)]$ \S\ref{sectNew1}. Then 
since $h_i$ lies in the region $W>M$ and $\delta_i$  in the region $M=W$ by corollary \ref{cordeltaM=W}, it follows that
$$h_i  = \gr^W_{w_i}  b_i  \qquad \hbox{ and } \qquad \delta_i \in W_{w_i-1} \mathrm{Der} \, \uu^{dR}_{1,1}\ .$$
It follows from the formula for a semi-direct product $(\ref{semidLieformula})$  that in the associated weight graded, the Lie algebra product 
is simply given by the Lie bracket on the geometric heads in  $\gr^W \uu^{dR}_{1,1}$ (i.e.,  the terms involving the action of the arithmetic parts $\delta_i$  do not contribute to the associated $W$-graded). Therefore:
$$\gr^W_{w}  [h_1, [h_2, \ldots [h_{n-1}, h_n] \cdots ] = [\sigma_1, [\sigma_2, \ldots [\sigma_{n-1}, \sigma_n] \cdots ]   $$
where $w= \sum_{i=1}^n w_i$.   Since the $h_i$ generate a free Lie algebra, the same is true of the $\sigma_i$. The last statement is clear.
\end{proof}
The proposition applies often since the Lie algebra $\uu^{dR}_{1,1}\otimes \overline{\Q}$ is  free, and so any Lie subalgebra of it is also free. For example, 
any elements $\sigma_i \in \Lie \, \A_{\U}^{dR}\otimes \overline{\Q}$ whose geometric  heads are in a  Hall or Lyndon basis \cite{Reut} of $\uu^{dR}_{1,1}\otimes \overline{\Q}$ necessarily generate a free Lie algebra.
The expected generators discussed in \S\ref{sectOutline} are indeed  images of Lyndon words with respect to any ordering of the generators of $\uu^{dR}_{1,1}\otimes \overline{\Q}$ for which the Eisenstein generators $\e_{2n+2}$ are smaller than the cuspidal ones $\e_f$. 

\begin{thm} Any choices of (lifts of) zeta and modular elements 
$$\sigma_{2n+1} \ , \sigma'_f(d)  \ , \ \sigma''_f(d)  \quad  \in \quad \Lie \, \A_{\U}^{dR}\otimes \overline{\Q}  $$
generate a free Lie algebra. Thus they act freely on $\uu^{dR}_{1,1}\otimes \overline{\Q} $. 
\end{thm} 
\begin{proof} Their geometric heads are proportional to $\e_{2n+2} \Ys^{2n}$ and 
 $$ [\e_{2a+2} \Ys^{2a-k} , \e'_f \Ys^{2n-k}] (\Ys_1 \Xs_2 - \Xs_1 \Ys_2)^k]   \quad \hbox{ or }  \quad [\e_{2a+2} \Ys^{2a-k} , \e''_f \Ys^{2n-k}] (\Ys_1 \Xs_2 - \Xs_1 \Ys_2)^k]$$
 for some $a$ and $k < \min\{2a, 2n\}$, respectively, by remark \ref{reme'fe''f}.   These are  independent in $\uu^{dR}_{1,1} \otimes \overline{\Q}.$ 
\end{proof}

\section{Decomposition of iterated Shimura integrals} \label{sectDecomp}

A useful technique in the theory of multiple zeta values is to replace motivic multiple zeta values with their so-called `f-alphabet' decomposition \cite{BrDec}, which assigns to any motivic multiple zeta value a word in an alphabet in letters $f_{2n+1}$, for $n\geq 1$.  It depends on some choices, but the longest
word in this decomposition is canonical. 

This was generalised in \cite{NotesMot} to give a general decomposition formula for all $\Hc$-periods, and takes the form of a canonical homomorphism
\begin{equation} \label{DecompMap} \Phi:  \gr^C \, \Pe^{\mm}_{\Hc} \To \Pe^{\mm}_{\Hc^{ss}} \otimes T^c (\gr^C_1 \Or(\U^{dR}_{\Hc} )) 
\end{equation}
where $C$ denotes the coradical filtration.
Using our results on the action of the Galois group of $\Hc$, we can compute the decomposition of, for example, iterated integrals of Eisenstein series. 
 Consider, for any $w\in \Or(\GG^{dR}_{1,1})$, the $\Hc$-period
$$  \int^{\mm}_S w=  [ \Or(\GG^{dR}_{1,1}),  S, w]^{\mm} \quad \in \quad \Pe^{\mm}_{\Hc}\ .$$
Choose a splitting of the $W$-filtration, which yields a linear map
$\Or(\U^{dR}_{1,1}) \rightarrow \Or(\GG^{dR}_{1,1})$, so we can take $w\in \Or(\U^{dR}_{1,1})$. If $w$ is totally holomorphic, then the period
$$ \mathrm{per} \,   \int^{\mm}_S w   =   \CC(S)_w \qquad \in \C  $$
is given by the regularised iterated Shimura integral of $w$ along $S$, i.e., the coefficient of $w$ in the canonical cocycle $\CC(S)$, after re-scaling 
$X$ and $Y$ as in lemma \ref{lemCanFromHol}.  

\begin{prop}  The Lie algebra $\uu^{dR}_{\Hc}$ acts on the left  via derivations on  $\Or(\U^{dR}_{1,1})$ and   strictly decreases the length. In particular we have an action
\begin{equation} \label{uudronOrUdr}  \uu^{dR}_{\Hc} \times \gr_{\ell}^L \Or(\U^{dR}_{1,1})\To \gr^L_{\ell-1} \Or(\U^{dR}_{1,1})\ .
\end{equation}
which factors through the abelianisation $(\uu^{dR}_{\Hc})^{ab}$.
This action  respects the increasing filtration on $ \gr^L  \Or(\U^{dR}_{1,1})$ given by the  number of Eisenstein generators $\e_{2n+2}$. It also preserves
the subspace of $\gr^L_{\ell} \Or(\U^{dR}_{1,1})$ generated by words which contain at least one cuspidal generator $\e_f$.
\end{prop} 

\begin{proof}  The first statement follows from the fact that the image of  $\uu^{dR}_{\Hc}$  is contained in  $L^1 \Lie \, \A^{dR}_{\U}$, and therefore
the action on the affine ring $\Or(\U^{dR}_{1,1})$ strictly decreases the length.  The action of a derivation $\sigma \in \Lie \,\A^{dR}_{\U}$, represented by $[(b,\delta)]$ via $(\ref{uudronOrUdr})$, only depends upon $b ,\delta  \mod L^2$. Suppose that $\sigma$ is of Tate type. We know by theorem $\ref{thmphieis}$  that 
$b \mod L^2$ can only be of the form $\mathrm{ad}(\e_{2n+2}\Ys^{2n})$, and $\delta$ of the general shape
$\e_{2a+2} \mapsto [\e_{2b+2},\e_{2c+2}] \mod L^2$. Therefore,  the  dual action of $\gr^L_1 \sigma$ via $(\ref{uudronOrUdr})$ 
strictly decreases the number of Eisenstein generators, and preserves the number of cuspidal generators.
Now, by lemma \ref{lemcusphead}, if $\sigma $ is not of Tate type, it satisfies $b \equiv 0 \mod L^2$, and by lemma $\ref{lemsimpleKerNargument}$,
$\delta$ satisfies  $\delta(\e_{2n+2}) \equiv 0 \mod L^3$. It could potentially act via 
$$\delta :  \e_f  \mapsto     [\e_{2a+2},\e_{2b+2}] \ \hbox{ or }   \ [\e_{2a+2}, \e_g]  \hbox{ or } \ [\e_h, \e_g] $$
depending on the type of $\sigma$, where $f,g,h$ are cusp forms.  The dual action $(\ref{uudronOrUdr})$ of such a derivation preserves or decreases the number of Eisenstein generators. Although it could potentially decrease the number of cuspidal generators (in the last  of the three cases), it cannot get rid of them altogether. \end{proof}

\begin{thm}
Suppose that $w$ is of length $\leq \ell$. Then 
$$ \int^{\mm}_S w   \quad  \in \quad  C_{\ell} \Pe^{\mm}_{\Hc}$$
is of coradical filtration $\leq \ell$. If $w$ contains a  cuspidal term $\e_f$, then  
$$ \int^{\mm}_S w  \quad   \in  \quad C_{\ell-1 } \Pe^{\mm}_{\Hc} \qquad \qquad \hbox{`Coradical drop'}\ .$$
Otherwise, if $w$ is a word in Eisenstein series of length $\ell$, then the decomposition $(\ref{DecompMap})$ in length $\ell$ is a linear combination of  powers of the Lefschetz period $\Lef^{\mm}$,  multiplied by  words of length $\ell$ in the $f_{2n+1}$, which were defined in \S\ref{defnofprimitiveelements}
$$ \Phi\Big(  \int^{\mm}_S w    \Big)  \in   \Q[\Lef^m ] \otimes  T^c(f_3,f_5,\ldots)     $$
In particular, we have the formula:
\begin{equation} \label{DecompofEis}  \Phi \Big(  \int^{\mm}_S \e_{2a_1+2} \Ys^{2a_1} \ldots \e_{2a_n+2} \Ys^{2a_n}     \Big)   =  {2^n \over (2a_1)! \ldots (2a_n)!}  f_{2a_1+1} \ldots f_{2a_n+1}\end{equation} 
\end{thm} 

\begin{proof}  The coradical filtration has the property that an element $\xi \in \Pe^{\mm}_{\Hc}$ lies in 
$C_{\ell} \Pe^{\mm}_{\Hc}$ if and only if $\sigma \xi \in C_{\ell-1} \Pe^{\mm}_{\Hc} $ for all  $\sigma \in \uu^{dR}_{\Hc}$, where $C_{-1} =0 $. 
Therefore the first statement is immediate from the previous proposition. For the second, it follows from lemma \ref{lemcusphead} that 
 if $w$ is any cuspidal generator in $ \Or(\U^{dR}_{1,1})$  dual to $\e_f$, then it satisfies  $\sigma w =0$ for every $\sigma \in \Lie\, \A^{dR}_{\U}$, and so the corresponding $\Hc$-period $ \int_S^{\mm}w \in C_0 \Pe^{\mm}_{\Hc}$.  The statement follows by  induction on the length of words  and the fact, proved in the previous proposition, that  if every term in  $w \in \gr^L_{\ell} \Or(\U^{dR}_{1,1})$ contains a cuspidal generator,  the same  must be true of the  image  $\sigma w \in \gr^L_{\ell-1} \Or(\U^{dR}_{1,1}) $.
The  proof of the third statement is similar. Suppose that $w$ is a word in Eisenstein generators of length $\ell$. Application of a non-Tate element $\sigma \in \uu^{dR}_{\Hc}$ via $(\ref{uudronOrUdr})$ will lead to an element $\sigma w$ of length $\ell-1$ which lies, for reasons of type, in the subspace generated by words with at least one cuspidal generator. By the previous argument then, $\sigma w \in C_{\ell-2} \Pe^{\mm}_{\Hc}$, and  the action 
of $\uu^{dR}_{\Hc}$ on the associated graded for the coradical filtration of the image of   $\Or(\U^{dR}_{1,1})$  in $\Pe^{\mm}_{\Hc}$ via the map $w \mapsto \int_{S}^{\mm}w$, factors through the action of the free Lie algebra generated by the Tate elements $\sigma_{2n+1}$.
Finally, formula $(\ref{DecompofEis})$ follows from the explicit formula for the geometric head of the $\sigma_{2n+1}$ given in  $(\ref{headsigma2n+1})$. 
\end{proof}
In principle, one could extend $(\ref{DecompofEis})$  to provide a formula for the decomposition of an arbitrary motivic  iterated integral 
of Eisenstein series.   The equation $(\ref{DecompofEis})$ exhibits the freeness of the Lie algebra generated by the zeta elements $\sigma_{2n+1}$. 

As an application, note that  image of the  dual of the monodromy representation
 $$\Or(\uu_{\Eq})\To \Or(\uu^{dR}_{1,1})$$
factors through $\Or(\ue)$ and  lands in the subspace generated by Eisenstein elements by $(\ref{Liemonodromy})$.  The previous theorem  implies a decomposition formula for $\Or(\uu_{\Eq})$. This   can also be deduced from the 
 formulae for the  zeta elements  on $\uu_{\Eq}$ given in \cite{BrSigma}.

  \subsection{Modular depth defect for double zeta values}
  Recall the exact sequence defined by the geometric monodromy
  $$ 0 \To \rr^{dR} \To \uu^{dR}_{1,1} \To \ue \To 0$$
  where $\ue$ is the Lie algebra generated by the derivations $\varepsilon_{2n}^{\vee}$ (definition \ref{defue}),
  for $n\geq 0$, and $\rr^{dR}$ the ideal of relations.  Let us bigrade the above Lie algebras with respect to $M$, $W$. Since cuspidal classes in $\uu^{dR}_{1,1}$ play no role, let us denote by 
  $$\mathcal{E} = \bigoplus_{n \geq 1}  \e_{2n+2} \Ys^{2n} \Q \quad \leq \quad \uu^{dR}_{1,1} $$
 These generators can be interpreted, via equation $(\ref{headsigma2n+1})$ as `zeta heads':
  $$\e_{2n+2} \Ys^{2n} =  { (2n)! \over 2} \,  \gr^W_{-2n-2}  \,  \sigma_{2n+1}$$   
   The geometric monodromy   gives  the
  exact sequence: 
 $$ 0 \To \rr_{\mathcal{E}}^{dR} \To \mathbb{L}(\mathcal{E})  \To \ue_+ \To 0$$
  where $\ue_+ \subset \ue$ is the Lie subalgebra generated by the $\varepsilon_{2n}^{\vee}$ for $n\geq 2$,  $\mathbb{L}$ denotes a free Lie algebra, and  $\rr^{dR}_{\mathcal{E}}$
  is defined to be the kernel of the third map.  .  It is the space of relations between generators of  $\mathcal{E}$. 
  This sequence is very closely related to the totally odd motivic multiple zeta values defined in \cite{BrDepth}. Indeed,  we have
  \begin{equation}   \label{epsheadofsigma}
   \varepsilon_{2n+2}^{\vee}  = \gr^1_B \, \sigma_{2n+1} 
  \end{equation} 
  where $\sigma_{2n+1}$ is the image of the zeta element under $(\ref{LieAtoDerUe})$, and $B$ is the $B$-filtration as defined in \cite{BrSigma}, \S3.2. On the bigraded Lie algebra $\mathbb{L}(\av, \bv)$ the $B$-filtration is induced by the degree in $\bv$. 
   It is proved in \cite{BrSigma}, corollary 3.7, that the $B$-filtration induces the depth filtration on the image of the fundamental groupoid of the projective line minus three points under the map $\Phi^{dR}$ of \S\ref{sectHainmorphism}:
   $$   \varepsilon_{2n+2}^{\vee}  = \gr^1_B \, \sigma_{2n+1} = \Phi^{dR} \gr^1_D \, \sigma_{2n+1}$$
   where the rightmost $\sigma_{2n+1}$  is the image of the zeta elements in the automorphisms of the motivic torsor of paths of the projective line minus three points \S\ref{sectHainmorphism}. 
     In fact, the Lie algebra $\ue_+$ is isomorphic to the Lie subalgebra of $\mathfrak{ls}$, defined in \cite{BrDepth}, generated by polynomials $x_1^{2n}$ for $n\geq 1$  under the Ihara bracket.

\subsubsection{Length two.}   In length two, we deduce an exact sequence
  \begin{equation} \label{seqexactWEgrB} 0 \To  \bigoplus_{n\geq 1} W^+_{2n}  \To  \textstyle{\bigwedge}^2 \mathcal{E} \To \gr^B_2 \ue_+ \To 0\ . 
  \end{equation} 
  The fact that the kernel is isomorphic to the space of even period polynomials  follows from Pollack's theorem \cite{Po}. A conceptual explanation is given in   section \S\ref{sectRelations}:  the kernel is given over $\overline{\Q}$ by the images of the following relations 
  $$ \sigma_f(d) \,  \e_f \subset R^{dR}_{\eis}\otimes \overline{\Q}$$  
  in length two,   for every normalised Hecke cusp form $f$ of weight $2n+2$. These are  in one-to-one correspondence with cusp forms.  The relations
  $\sigma_f(d+k)\e_f $ for $k >0$ play no role, since a glance at the definition \S\ref{defnofprimitiveelements} $(3)$ shows that they involve non-trivial coefficients of $\Xs_1, \Xs_2$ are therefore do not intersect $\bigwedge^2 \mathcal{E}$.

  Let $\mathfrak{d}^d$ be the depth $d$ component of the bigraded Lie algebra generated by the zeta elements $\sigma_{2n+1}$, for $n\geq 1$,  graded by weight and the depth filtration on the projective line minus three points (or, equivalently,  the $B$-filtration as defined above).  By the above remarks $\mathcal{E}$ is isomorphic to $\mathfrak{d}^1$, and $\gr^2_B \ue$ to $\mathfrak{d}^2$.
Thus      sequence $(\ref{seqexactWEgrB})$ is equivalent to the sequence
  \begin{equation} \label{Sd1d2}  
  0 \To \bigoplus_n W_{2n}^{+} \To  
  \textstyle{\bigwedge}^2 \mathfrak{d}^1 \To \mathfrak{d}^2 \To 0 
  \end{equation} 
 which can in turn be identified with the sequence of \S7.3 in \cite{BrDepth} for the double shuffle equations in depth two, since, in the notations of that paper,  $\mathfrak{d}^i \cong \ls^i$ for $i=1, 2$. 
  The sequence $(\ref{Sd1d2})$ gives precisely the Ihara-Takao relations.

\begin{example} In weight 12, the space  $W_{10}^+$ is spanned by the period polynomial 
\begin{equation} \label{ppw12} X^8Y^2 - 3X^6Y^4+3X^4Y^6 - X^2Y^8\ .
\end{equation} 
By $(\ref{Sd1d2})$ this provides the Ihara-Takao relation 
$$[\gr_D^1 \sigma_3, \gr_D^1 \sigma_9]  - 3 [\gr_D^1 \sigma_5, \gr_D^1 \sigma_7] =0  \quad \hbox{ in } \mathfrak{d}^2\ .$$ 
This in turn corresponds to the relation found by Pollack \cite{Po}:
$$ [\varepsilon_4^{\vee}, \varepsilon_{10}^{\vee}] - 3 [\varepsilon_6^{\vee}, \varepsilon_8^{\vee}] =0 \ .$$
Now consider the exact sequence which is graded dual to $(\ref{seqexactWEgrB})$. It gives 
$$0 \To \gr^L_2\Or(\ue_+)/ \gr^L_1\Or(\ue_+)^2   \To  \textstyle{\bigwedge^2}( \mathcal{E}^{\vee} )  \To  \bigoplus_n (W_{2n}^+)^{\vee}  \To 0 \ . $$
The decomposition theorem $(\ref{DecompofEis})$ embeds the  middle space in the tensor  coalgebra on the elements $f_{2n+1}$, for $n\geq 1$,  which are  dual to the $\sigma_{2n+1}$.

 In weight $12$,  its image is generated by $f_3\wedge f_9$ and $f_5\wedge f_7$. The left-hand space is the one-dimensional subspace dual to the period polynomial $(\ref{ppw12})$ given by 
\begin{equation} \label{IHrelation} \big(3  f_3\wedge f_9 + f_5 \wedge f_7 \big)\Q\ .
\end{equation}
From the previous discussion, we deduce two facts. First of all, the   decomposition  $(\ref{DecompMap})$  applied to  double motivic zeta values necessarily lands in  $(\ref{IHrelation})$.  For example, one checks using the method of  \cite{BrDec}, for example,  that indeed $\zetam(3,9) \mapsto - 9  (3  f_3\wedge f_9 + f_5 \wedge f_7)$ and $\zetam(4,8) \mapsto 16(3  f_3\wedge f_9 + f_5 \wedge f_7)$.  

 Secondly, 
 the image of $\gr^L_2\Or(\ue_+)$ modulo products in weight 12 lands in the  subspace of iterated integrals of Eisenstein series  spanned by the following element:
\begin{equation} \label{EELC} 9   [\e_4 \Ys^2,  \e_{10} \Ys^8]  +  14 [\e_6\Ys^4,  \e_{8} \Ys^6] \ .
\end{equation} 
This is entirely consistent with theorem \ref{thmImEab}. The iterated integrals of Eisenstein series
$[\underline{E}_4(\tau)(0,Y) ,  \underline{E}_{10}(\tau)(0,Y)]$  and $[\underline{E}_4(\tau)(0,Y),  \underline{E}_8(\tau)(0,Y)]$ individually involve  the $L$-value   $\Lambda(\Delta,12)$ of the cusp form $\Delta$ of weight $12$, which is a period of a non-Tate object of $\Hc\otimes \overline{\Q}$. This period precisely cancels out in the  linear combination $(\ref{EELC})$.
\end{example}

In this manner, double  motivic multiple zeta values are isomorphic to the subspace of double Eisenstein integrals (iterated integrals of length two in 
$E_{2k+2}(\tau) \tau^{2k} Y^{2k} d \tau$ ) 
which are orthogonal to all cusp forms. This provides a geometric explanation for  the cuspidal  defect  \cite{GKZ} for double  zeta values. Similarly, the  iterated integrals of Eisenstein series arising as regularised values of multiple elliptic polylogarithms evaluated at the origin are also orthogonal to all cusp forms.

\bibliographystyle{plain}
\bibliography{main}

\end{document}